\documentclass[11pt]{article}
\usepackage{authblk}
\usepackage[toc,page]{appendix}
\usepackage[top=2cm, bottom=2cm, left=2cm, right=2cm]{geometry}

\usepackage{color}
\usepackage{helvet}         
\usepackage{courier}        
\usepackage{type1cm}        

\usepackage{framed}
\usepackage{tikz}
\usepackage{makeidx}         
\usepackage{graphicx}        
\usepackage{multicol}        
\usepackage[bottom]{footmisc}

\usepackage{amsmath}
\usepackage{amssymb}
\usepackage{bbold}
\usepackage{amsthm}
\usepackage{subcaption}
\usepackage{sidecap}
\usepackage{floatrow}
\usepackage{wasysym}

\usepackage{pdflscape}
\usepackage{comment}
\usepackage[font=small]{caption}
\usepackage{enumitem}
\usepackage{esint}

\usepackage{scalerel}

\newtheorem{theorem}{Theorem}
\newtheorem{corollary}[theorem]{Corollary}
\newtheorem{lemma}[theorem]{Lemma}

\newtheorem{proposition}[theorem]{Proposition}
\newtheorem{proposition*}{Proposition}
\newtheorem{lemma*}{Lemma}

\theoremstyle{remark}

\newtheorem{definition}[theorem]{\bf Definition}

\numberwithin{theorem}{section}
\numberwithin{question}{section}
\numberwithin{figure}{section}
\numberwithin{equation}{section}

\begin{document}

\title{Tripod in uniform spanning tree and three-sided radial SLE$_2$}
\bigskip{}
\author[1]{Jiacheng Ding\thanks{dingjcthuprobab@gmail.com}}
\author[2]{Mingchang Liu\thanks{liumc\_prob@163.com}}
\author[1]{Hao Wu\thanks{hao.wu.proba@gmail.com}}
\affil[1]{Tsinghua University, China}
\affil[2]{Capital Normal University, China}
\date{}

%

%

\maketitle


\global\long\def\qnum#1{\left[#1\right]_q }
\global\long\def\qfact#1{\left[#1\right]_q! }
\global\long\def\qbin#1#2{\left[\begin{array}{c}
	#1\\
	#2 
	\end{array}\right]_q}

\global\long\def\defpatt{\shuffle}

\newcommand{\nradpartfn}[1]{\mathcal{Z}_{#1\mathrm{\textnormal{-}rad}}}
\newcommand{\nradE}[1]{\mathbb{E}_{#1\mathrm{\textnormal{-}rad}}}
\newcommand{\nradP}[1]{\mathbb{P}_{#1\mathrm{\textnormal{-}rad}}}

\global\long\def\covmap{h}

\global\long\def\mslitdriv{\omega}
\global\long\def\nnofloops{\mathscr{L}}


\global\long\def\U{\mathbb{U}}
\global\long\def\T{\mathbb{T}}
\global\long\def\HH{\mathbb{H}}
\global\long\def\R{\mathbb{R}}
\global\long\def\C{\mathbb{C}}
\global\long\def\N{\mathbb{N}}
\global\long\def\Z{\mathbb{Z}}
\global\long\def\E{\mathbb{E}}
\global\long\def\PP{\mathbb{P}}
\global\long\def\ratespiral{\mathcal{J}^{(\mu)}}
\global\long\def\QQ{\mathbb{Q}}
\global\long\def\A{\mathbb{A}}
\global\long\def\one{\mathbb{1}}

\newcommand{\PPspiral}[1]{\mathbb{P}^{(\mu)}_{#1}}
\newcommand{\PPnospiral}[1]{\mathbb{P}^{(0)}_{#1}}

\global\long\def\CR{\mathrm{CR}}
\global\long\def\ST{\mathrm{ST}}
\global\long\def\SF{\mathrm{SF}}
\global\long\def\cov{\mathrm{cov}}
\global\long\def\dist{\mathrm{dist}}
\global\long\def\SLE{\mathrm{SLE}}
\global\long\def\hSLE{\mathrm{hSLE}}
\global\long\def\CLE{\mathrm{CLE}}
\global\long\def\GFF{\mathrm{GFF}}
\global\long\def\inte{\mathrm{int}}
\global\long\def\ext{\mathrm{ext}}
\global\long\def\inrad{\mathrm{inrad}}
\global\long\def\outrad{\mathrm{outrad}}
\global\long\def\dimH{\mathrm{dim}}
\global\long\def\capa{\mathrm{cap}}
\global\long\def\diam{\mathrm{diam}}
\global\long\def\sign{\mathrm{sgn}}
\global\long\def\cat{\mathrm{Cat}}
\global\long\def\cst{\mathrm{C}}
\global\long\def\ck{\mathrm{C}_{\kappa}}
\global\long\def\free{\mathrm{free}}
\global\long\def\hF{{}_2\mathrm{F}_1}
\global\long\def\simple{\mathrm{simple}}
\global\long\def\even{\mathrm{even}}
\global\long\def\odd{\mathrm{odd}}
\global\long\def\st{\mathrm{ST}}
\global\long\def\ust{\mathrm{UST}}
\global\long\def\lerw{\mathrm{LERW}}
\global\long\def\usf{\mathrm{USF}}
\global\long\def\Leb{\mathrm{Leb}}
\global\long\def\LP{\mathrm{LP}}
\global\long\def\I{\mathrm{I}}
\global\long\def\II{\mathrm{II}}
\global\long\def\hcap{\mathrm{hcap}}
\global\long\def\trifurcation{\mathfrak{t}}
\global\long\def\out{\mathrm{out}}
\global\long\def\harmonic{\mathrm{H}}
\global\long\def\nharmonic{\mathrm{H}_{\mathrm{n}}}
\global\long\def\dharmonic{\mathrm{Q}}
\global\long\def\Poisson{\mathrm{P}}
\global\long\def\det{\mathrm{det}}
\global\long\def\LZtripod{\mathcal{Z}_{\mathrm{tri}}}
\global\long\def\Green{\mathrm{G}}
\global\long\def\constCW{\mathrm{C}_{\mathrm{CW}}}
\global\long\def\chordalSLE{\PP_{\mathrm{chord}}}
\global\long\def\chordalSLEpf{\mathcal{Z}_{\mathrm{chord}}}
\global\long\def\chordalSLEexp{\E_{\mathrm{chord}}}

\global\long\def\LA{\mathcal{A}}
\global\long\def\LB{\mathcal{B}}
\global\long\def\LC{\mathcal{C}}
\global\long\def\LD{\mathcal{D}}
\global\long\def\LF{\mathcal{F}}
\global\long\def\LK{\mathcal{K}}
\global\long\def\LE{\mathcal{E}}
\global\long\def\LG{\mathcal{G}}
\global\long\def\LGmu{\mathcal{G}_{\mu}}
\global\long\def\LI{\mathcal{I}}
\global\long\def\LJ{\mathcal{J}}
\global\long\def\LL{\mathcal{L}}
\global\long\def\LM{\mathcal{M}}
\global\long\def\LN{\mathcal{N}}
\global\long\def\OO{\mathcal{O}}
\global\long\def\LQ{\mathcal{Q}}
\global\long\def\LR{\mathcal{R}}
\global\long\def\LT{\mathcal{T}}
\global\long\def\LS{\mathcal{S}}
\global\long\def\LU{\mathcal{U}}
\global\long\def\LV{\mathcal{V}}
\global\long\def\LW{\mathcal{W}}
\global\long\def\LX{\mathcal{X}}
\global\long\def\LY{\mathcal{Y}}
\global\long\def\PartF{\mathcal{Z}}
\global\long\def\LH{\mathcal{H}}
\global\long\def\LJ{\mathcal{J}}

\global\long\def\blm{\mathfrak{m}}

\global\long\def\LZ{\mathcal{Z}}
\global\long\def\LZrp{\mathcal{Z}_{\alpha; \bs{s}}^{(p)}}
\global\long\def\LJrp{\mathcal{J}_{\alpha; \bs{s}}^{(p)}}
\global\long\def\chamberrp{\chamber_{\alpha; \bs{s}}^{(p)}}
\global\long\def\LErp{\mathcal{E}_{\alpha; \bs{s}}^{(p)}}
\global\long\def\Greenrp{G_{\alpha; \bs{s}}^{(p)}}
\global\long\def\Prp{P_{\alpha; \bs{s}}^{(p)}}
\global\long\def\norcst{\mathrm{C}_{\kappa}^{(\mathfrak{r})}}
\global\long\def\LZalphar{\mathcal{Z}_{\alpha}^{(\mathfrak{r})}}

\global\long\def\coulomb{\LH}
\global\long\def\auxcoulomb{\hat{\coulomb}}
\global\long\def\coulombGas{\LF}
\global\long\def\coulombnew{\LK}
\global\long\def\coulombLine{\LG}
\global\long\def\kfunc{p}

\global\long\def\eps{\epsilon}
\global\long\def\ov{\overline}
\global\long\def\QQrp{\QQ_{\alpha; \bs{s}}^{(p)}}

\global\long\def\bn{\mathbf{n}}
\global\long\def\MR{MR}
\global\long\def\cond{\,|\,}
\global\long\def\la{\langle}
\global\long\def\ra{\rangle}
\global\long\def\tree{\Upsilon}
\global\long\def\prob{\mathbb{P}}
\global\long\def\hm{\mathrm{Hm}}
%

\global\long\def\Im{\operatorname{Im}}
\global\long\def\Re{\operatorname{Re}}

\global\long\def\ud{\mathrm{d}}
\global\long\def\pder#1{\frac{\partial}{\partial#1}}
\global\long\def\pdder#1{\frac{\partial^{2}}{\partial#1^{2}}}
\global\long\def\pddder#1{\frac{\partial^{3}}{\partial#1^{3}}}
\global\long\def\der#1{\frac{\ud}{\ud#1}}

\global\long\def\bZnn{\mathbb{Z}_{\geq 0}}
\global\long\def\bZpos{\mathbb{Z}_{> 0}}
\global\long\def\bZneg{\mathbb{Z}_{< 0}}

\global\long\def\Vfunc{\LG}
\global\long\def\gfunc{g^{(\rr)}}
\global\long\def\hfunc{h^{(\rr)}}

\global\long\def\SimplexInt{\rho}
\global\long\def\CubeInt{\widetilde{\rho}}

\global\long\def\ii{\mathfrak{i}}
\global\long\def\ee{\mathrm{e}}
\global\long\def\rr{\mathfrak{r}}
\global\long\def\chamber{\mathfrak{X}}
\global\long\def\Wchamber{\mathfrak{W}}

\global\long\def\SimplexIntKappa8{\SimplexInt}

\global\long\def\nested{\boldsymbol{\underline{\Cap}}}
\global\long\def\unnested{\boldsymbol{\underline{\cap\cap}}}
\global\long\def\unnested{\boldsymbol{\underline{\cap\cap}}}

\global\long\def\acycle{\vartheta}
\global\long\def\bcycle{\tilde{\acycle}}

\global\long\def\metric{\mathrm{dist}}

\global\long\def\adj#1{\mathrm{adj}(#1)}

\global\long\def\bs{\boldsymbol}

\global\long\def\edge#1#2{\langle #1,#2 \rangle}
\global\long\def\graph{G}

\newcommand{\conn}{\varsigma}
\newcommand{\realacycle}{\smash{\mathring{\acycle}}}
\newcommand{\realpt}{\smash{\mathring{x}}}
\newcommand{\corrind}{\LC}
\newcommand{\bssymb}{\pi}
\newcommand{\coeff}{p}
\newcommand{\MainConst}{C}

\global\long\def\removeLink{/}

\global\long\def\domainofdef{\mathfrak{U}}
\global\long\def\Test_space{C_c^\infty}
\global\long\def\Distr_space{(\Test_space)^*}

\global\long\def\bs{\boldsymbol}
\global\long\def\cst{\mathrm{C}}

\newcommand{\red}{\textcolor{red}}
\newcommand{\blue}{\textcolor{blue}}
\newcommand{\green}{\textcolor{green}}
\newcommand{\magenta}{\textcolor{magenta}}

\newcommand{\coulombGasH}{\mathcal{H}}
\newcommand{\secondbeta}{\intloop}

\newcommand{\cev}[1]{\reflectbox{\ensuremath{\vec{\reflectbox{\ensuremath{#1}}}}}}

\global\long\def\anticonf{\zeta}
\global\long\def\intloop{\varrho}
\global\long\def\Gloop{\smash{\mathring{\intloop}}}

\global\long\def\SLEmeasure{\mathrm{P}}
\global\long\def\SLEmeasureEx{\mathrm{E}}

\global\long\def\fugacity{\nu}
\global\long\def\meanderMat{\mathcal{M}}
\global\long\def\LM{\mathcal{M}}
\global\long\def\meanderMatrix{\meanderMat_{\fugacity}}
\global\long\def\meanderMatrixPrime{\meanderMat_{\fugacity(\kappa')}}
\global\long\def\meanderRenorm{\widehat{\mathcal{M}}}

\global\long\def\PartFRenorm{\widehat{\PartF}}
\global\long\def\coulombGasRenorm{\widehat{\coulombGas}}

\global\long\def\hexa{\scalebox{1.3}{\hexagon}}

\global\long\def\np{p}

\global\long\def\FKdual{\mathcal{L}}

\global\long\def\fixedindex{\flat}

\begin{center}
\begin{minipage}{0.95\textwidth}
\abstract{
Fix a bounded $3$-polygon $(\Omega; x_1, x_2, x_3)$ with three marked boundary points $x_1, x_2, x_3\in\partial\Omega$ and suppose $(\Omega^{\delta}; x_1^{\delta}, x_2^{\delta}, x_3^{\delta})$ is an approximation of $(\Omega; x_1, x_2, x_3)$ on $\delta$-scaled hexagonal lattice. 
We consider uniform spanning tree (UST) in  $\Omega^{\delta}$ with wired boundary conditions. Conditional on the event that both branches from $x_1^{\delta}$ and $x_2^{\delta}$ hit the boundary through $x_3^{\delta}$, the two branches meet at a point $\trifurcation^{\delta}$ which we call trifurcation, and the union of the three branches from $x_j^{\delta}$ to $\trifurcation^{\delta}$ form a tripod in the UST. We compute the scaling limit of the tripod: the distribution of trifurcation is absolutely continuous with respect to Lebesgue measure with explicit density; given the trifurcation, the conditional law of the tripod is three-sided radial SLE$_2$. 
The proof relies on construction of a new observable for trifurcation in our key lemma--Lemma~\ref{lem::Fomin}--where we use 
Fomin's formula
and the geometry of the hexagonal lattice in an essential way.
Interestingly, the scaling limit of the observable for trifurcation coincides with the partition function for three-sided radial $\SLE_2$. 
Our result gives a probabilistic interpretation of the correlation function in CFT which has conformal weights $1$ at the three boundary points and has a spinless field of weights $(1,1)$ at the bulk point. 
 }

\bigskip{}


\end{minipage}
\end{center}
\newpage
\tableofcontents
\newpage
\section{Introduction}
\label{sec::intro}
Schramm–-Loewner evolution (SLE), introduced by O. Schramm [Sch00], is a family of random curves defined in simply connected domains, proposed as scaling limits of interfaces in two-dimensional critical lattice models. There are two main variants: chordal SLE, where the curve connects two boundary points, and radial SLE, where the curve runs from a boundary point to an interior point. To date, conformal invariance and convergence to SLE have been established for several key models: loop-erased random walk (LERW) converges to radial $\SLE_2$~\cite{SchrammFirstSLE, LawlerSchrammWernerLERWUST}, interfaces in critical Ising model converges to chordal $\SLE_3$~\cite{ChelkakSmirnovIsing, CDCHKSConvergenceIsingSLE}, 
level lines in discrete Gaussian free field converges to chordal $\SLE_4$~\cite{SchrammSheffieldDiscreteGFF}, 
interfaces in critical FK-Ising model converges to chordal $\SLE_{16/3}$~\cite{ChelkakSmirnovIsing, CDCHKSConvergenceIsingSLE}, 
interfaces in critical Bernoulli percolation converges to chordal $\SLE_6$~\cite{SmirnovPercolationConformalInvariance, CamiaNewmanPercolation}, 
and the Peano curve for uniform spanning tree (UST) converges to chordal $\SLE_8$~\cite{SchrammFirstSLE, LawlerSchrammWernerLERWUST}. 
Recent developments~\cite{BauerBernardKytolaMultipleSLE, FloresKlebanPDE3, PeltolaWuGlobalMultipleSLEs, KarrilaKytolaPeltolaCorrelationsLERWUST, PeltolaWuCrossingProbaIsing, FengPeltolaWuConnectionProbaFKIsing, LiuPeltolaWuUST} have extended this framework to the scaling limits of multiple interfaces in polygonal domains. 
In such settings, collections of interfaces have been shown to converge to multiple chordal SLEs~\cite{DubedatEulerIntegralsCommutingSLEs, KozdronLawlerMultipleSLEs, LawlerPartitionFunctionsSLE, KytolaPeltolaPurePartitionFunctions, WuHyperSLE, BeffaraPeltolaWuUniqueness, AngHoldenSunYu2023, ZhanExistenceUniquenessMultipleSLE}. 

In contrast, the theory of multiple radial SLEs remains less developed. In~\cite{HealeyLawlerNSidedRadialSLE}, V. Healey and G. Lawler introduced multi-sided radial SLE and derived its invariant density under common-time parameter. Subsequent works~\cite{FengWuYangIsing, KrusellWangWuCommutationRelation, HuangPeltolaWuMultiradialSLEResamplingBP} studied multi-sided radial SLE  under multi-time parameter, establishing properties including transience, resampling and boundary perturbation. In this article, we investigate the connection between discrete lattice model and multi-sided radial SLE and prove that a collection of triple branches in UST converges to a three-sided radial $\SLE_2$. 
The proof relies on the construction a new observable in UST.

\begin{figure}[ht!]
\begin{subfigure}[b]{0.5\textwidth}
\begin{center}
\includegraphics[width=0.9\textwidth]{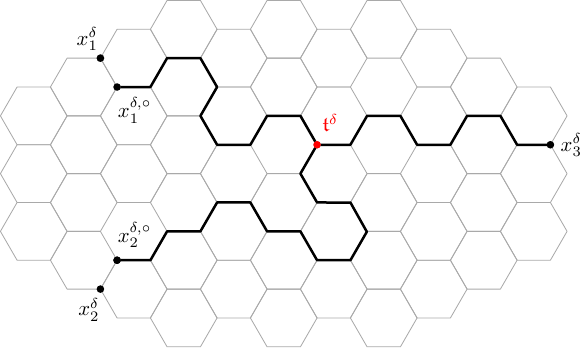}
\end{center}
\caption{}
\end{subfigure}
$\quad$
\begin{subfigure}[b]{0.4\textwidth}
\begin{center}
\includegraphics[width=\textwidth]{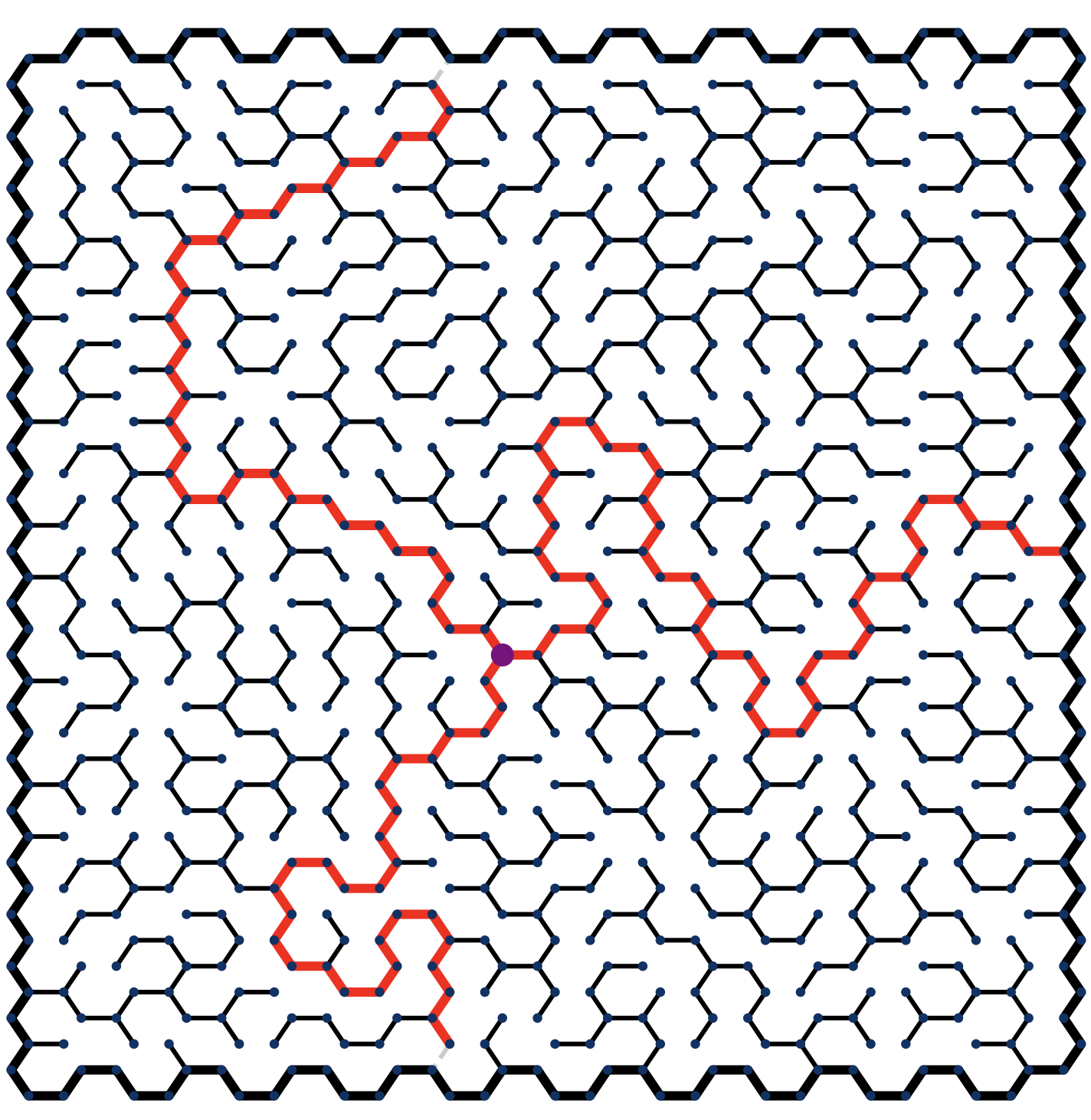}
\end{center}
\caption{}
\end{subfigure}
\caption{\label{fig::tripod} Illustration for tripod and trifurcation.}
\end{figure}

\subsection{Trifurcation and tripod in uniform spanning tree}
\label{subsec::intro_setup}
\paragraph*{Polygon.}
For $p\ge 1$, we say that 
$(\Omega; x_1, x_2, \ldots, x_p)$ 
is a (topological) $p$-\textit{polygon} if $\Omega\subsetneq\C$ is simply connected and $x_1, x_2, \ldots, x_p\in \partial\Omega$ are distinct points lying counterclockwise along the boundary. 
We will also assume throughout that $\partial\Omega$ is locally connected and the marked boundary points $x_1, x_2, \ldots, x_p$ lie on $C^{1+\eps}$-boundary segments, for some $\eps>0$, so that derivatives of conformal maps on $\Omega$ are defined there. For $x,y\in\partial\Omega$, we denote by $(xy)$ the counterclockwise boundary arc from $x$ to $y$. 

\paragraph*{Uniform spanning tree (UST) and its boundary branches.}
We consider a finite planar graph $\LG=(\LV, \LE)$, where a non-empty subset $ \LV^{\partial} \subset \LV$ is designated as the set of boundary vertices. All other vertices are referred to as interior vertices, and their set is denoted by $\LV^{\circ}=\LV \setminus  \LV^{\partial}$.
The set of boundary edges $ \LE^{\partial} \subset \LE$ consists of all edges $e=\langle e^{\partial}, e^{\circ}\rangle$ connecting a boundary vertex $e^{\partial} \in  \LV^{\partial}$ to an interior vertex $e^{\circ} \in \LV^{\circ}$.

A spanning tree of a connected graph is a subgraph which is connected and has no cycles and which contains every vertex. 
Let $\LG$ be a finite graph embedded in a planar domain. 
A uniform spanning tree with wired boundary condition on $\LG$ is defined as follows.
We collapse all boundary vertices in $\LV^{\partial}$ into a single vertex $v_{\partial}$, and denote the quotient graph by $\LG / \partial$.
Each spanning tree $\LT$ of $\LG / \partial$ can be viewed as a collection of edges of the original graph $\LG$, and the collection contains the edges identified from those of $\LG/\partial$ and the edges among boundary vertices.
We  give uniform distribution among all such collections.

Suppose $\LT$ is a spanning tree of $\LG/\partial$. 
If $v\in \LV^{\circ}$ is an interior vertex, there exists a unique path $\eta_v$  in $\LT$ from $v$ to the boundary and we call it \textit{boundary branch} from $v$.
We denote by $\eta_v=(u_0,u_1,\dots,u_{L})$ the boundary branch with distinct vertices $u_0,u_1,\dots,u_{L}$ such that $u_0=v$ and $\langle u_{j-1}, u_j\rangle\in \LT$ for all $j=1,\dots,L$.
The branch hits the boundary through the boundary edge $\langle u_{L-1}, u_{L}\rangle\in \LE^{\partial}$.
For a given $e_{\out} \in  \LE^{\partial}$, we denote by $\{v \rightsquigarrow e_{\out}\}$ the event $\{\langle u_{L-1}, u_{L}\rangle=e_{\out}\}$ that the boundary branch $\eta_v$ from $v$ hits the boundary through $e_{\out}$.

\paragraph*{Approximation by hexagonal lattices.}
Denote by $\hexagon$ the \textit{planar hexagonal lattice}. 
Fix a $p$-polygon $(\Omega; x_1, \ldots, x_p)$. Let us describe its approximation on hexagon lattice. 
Let $\Omega^\delta$ denote a domain such that $\partial \Omega^\delta$ consists of the edges of $\delta \hexagon$.
Denote by $\LV(\Omega^\delta)$ the set of the vertices in $\overline{\Omega^\delta}$,
and denote by $\LV^{\partial}(\Omega^\delta)$ the set of the vertices in $\partial \Omega^\delta$.
The set of the interior vertices is defined as $\LV^{\circ}(\Omega^\delta)=\LV(\Omega^\delta)\setminus \LV^{\partial}(\Omega^\delta)$.
The set $\LE(\Omega^{\delta})$ of edges consists of all pairs of adjacent vertices such that at least one of the vertices is an interior vertex,
and the set $\LE^{\partial}(\Omega^\delta)$ of boundary edges consists of all edges connecting a boundary vertex to an interior vertex. For $e\in\LE^{\partial}(\Omega^{\delta})$, we write $e=\langle e^{\circ}, e^{\partial}\rangle$ where $e^{\circ}, e^{\partial}$ are the two ends of $e$ and $e^{\circ}\in\LV^{\circ}(\Omega^{\delta})$ and $e^{\partial}\in\LV^{\partial}(\Omega^{\delta})$. 
For $x^{\delta}, y^{\delta}\in\partial\Omega^{\delta}$, we denote by $(x^{\delta}y^{\delta})$ the counterclockwise boundary arc from $x^{\delta}$ to $y^{\delta}$.  

Pick $p$ boundary edges $e_1^\delta, \ldots, e_p^\delta\in \LE^{\partial}(\Omega^\delta)$ and denote $x_j^{\delta}=e_j^{\delta,\partial}$ and $x_j^{\delta,\circ}=e_j^{\delta,\circ}$ for $1\le j\le p$.
We say that $(\Omega^{\delta}; x_1^{\delta}, \ldots, x_p^{\delta})$ is an approximation of $(\Omega; x_1, \ldots, x_p)$ on hexagonal lattice $\delta\hexagon$ in \textit{Carath\'eodory sense} if 
there exist conformal maps $\varphi^{\delta}: \U\to\Omega^{\delta}$ and conformal map $\varphi: \U\to \Omega$ such that $\varphi^{\delta}\to \varphi$ uniformly on compact subsets of $\U$ and $(\varphi^{\delta})^{-1}(x_j^{\delta})\to \varphi^{-1}(x_j)$ for $1\le j\le p$.  
The Carath\'eodory convergence guarantees local uniform convergence, but it still allows wild behavior of the boundaries. See more discussion in~\cite{KarrilaMultipleSLELocalGlobal}. 
In our setup, we need further assumptions on the boundaries of the domains: we assume further that $\partial\Omega^{\delta}$ converges to $\partial\Omega$ in Hausdorff distance
\begin{equation}\label{eqn::boundary_cvg_Hausdorff}
\lim_{\delta\to 0}\dist_H(\partial\Omega^{\delta}, \partial\Omega)=0. 
\end{equation}

\paragraph*{Trifurcation and tripod in UST.}
Fix a bounded $3$-polygon $(\Omega; x_1, x_2, x_3)$ and suppose $(\Omega^{\delta}; x_1^{\delta}, x_2^{\delta}, x_3^{\delta})$ is an approximation of $(\Omega; x_1, x_2, x_3)$ on $\delta\hexagon$ in Carath\'eodory sense. Suppose $\LT^{\delta}$ is a spanning tree on $\Omega^{\delta}$ with wired boundary condition. 
For $j=1,2$, denote by $\eta_j^{\delta}$ the boundary branch starting from $x_j^{\delta, \circ}$; 
and denote by $\LA_j^{\delta}$ the event that $\eta_j^{\delta}$ hits the boundary through $e_3^{\delta}=\langle x_3^{\delta, \circ}, x_3^{\delta}\rangle$, i.e. 
\begin{equation}\label{eqn::conditionalevent}
\LA_j^{\delta}=\{x_j^{\delta, \circ}\rightsquigarrow e_3^{\delta}\},\qquad \text{for }j=1,2.
\end{equation}
Conditioning on $\LA_1^{\delta}\cap\LA_2^{\delta}$, the branches $\eta_1^{\delta}$ and $\eta_2^{\delta}$ intersect for the first time at a vertex, denoted by $\trifurcation^{\delta}$, which we call \textit{trifurcation}. 
In the union $\eta_1^{\delta}\cup\eta_2^{\delta}$, there are three paths connecting $\trifurcation^{\delta}$ to $x_1^{\delta, \circ}, x_2^{\delta, \circ}, x_3^{\delta}$ respectively, denoted by $(\gamma_1^{\delta}, \gamma_2^{\delta}, \gamma_3^{\delta})$ and we view $\gamma_j^{\delta}$ as a path from $x_j^{\delta, \circ}$ to $\trifurcation^{\delta}$. We call $(\gamma_1^{\delta}, \gamma_2^{\delta}, \gamma_3^{\delta})$ \textit{tripod}.  See Figure~\ref{fig::tripod}.
The goal of this article is to derive the limiting distributions of the trifucartion and the tripod. 

\subsection{Limiting distribution for trifurcation}

\paragraph*{Partition function for tripod.}

The limiting distributions of the trifurcation and the tripod are encoded by the following partition function, which we call \textit{tripod partition function}: suppose $(\Omega; x_1, x_2, x_3)$ is a $3$-polygon with $z\in\Omega$, 
\begin{align}\label{eqn::LZtripod_def}
\LZtripod(\Omega; x_1, x_2, x_3; z)=\frac{\left(\Poisson(\Omega; z, x_1)\Poisson(\Omega; z, x_2)\Poisson(\Omega; z, x_3)\right)^2}{\CR(\Omega; z)^2\left(\Poisson(\Omega; x_1, x_2)\Poisson(\Omega; x_2, x_3)\Poisson(\Omega; x_3, x_1)\right)^{1/2}}, 
\end{align}
where $\CR(\Omega; z)$ is the conformal radius of $\Omega$ seen from $z$ (see~\eqref{eqn::CR_H}-\eqref{eqn::CR_cov}), $\Poisson(\Omega; z, x_j)$ is Poisson kernel (see~\eqref{eqn::Poisson_H}-\eqref{eqn::Poisson_cov}) and $\Poisson(\Omega; x_i, x_j)$ is boundary Poisson kernel (see~\eqref{eqn::bPoisson_H}-\eqref{eqn::bPoisson_cov}). 
Note that the tripod partition function is conformally covariant: for any confomral map $\varphi$ on $\Omega$, 
\begin{align}\label{eqn::LZtripod_cov}
\LZtripod(\Omega; x_1, x_2, x_3; z)=\prod_{j=1}^3 |\varphi'(x_j)|\times |\varphi'(z)|^2 \times\LZtripod(\varphi(\Omega); \varphi(x_1), \varphi(x_2), \varphi(x_3); \varphi(z)). 
\end{align}
In particular, when $\Omega=\HH$ is the upper-half plane and $x_1<x_2<x_3$ and $z\in\HH$, we have 
\begin{align}\label{eqn::LZtripod_H}
\LZtripod(\HH; x_1, x_2, x_3; z)=4\sqrt{2}\frac{(x_2-x_1)(x_3-x_2)(x_3-x_1)\mathrm{Im}(z)^4}{|(z-x_1)(z-x_2)(z-x_3)|^4}. 
\end{align}
Moreover, the tripod partition function is integrable (see Lemma~\ref{lem::LZtripod_integrable}):
\begin{align}\label{eqn::LZtripod_integrable}
\int_{\Omega}\LZtripod(\Omega; x_1, x_2, x_3; w)|\ud w|^2={\frac{3\pi}{4}} \left(\Poisson(\Omega; x_1, x_2)\Poisson(\Omega; x_2, x_3)\Poisson(\Omega; x_3, x_1)\right)^{1/2}.  
\end{align}

The limiting distribution of the trifurcation is given by the tripod partition function. 

\begin{theorem}\label{thm::trifurcation}
Fix a bounded $3$-polygon $(\Omega; x_1, x_2, x_3)$ and suppose $(\Omega^{\delta}; x_1^{\delta}, x_2^{\delta}, x_3^{\delta})$
 is an approximation of $(\Omega; x_1, x_2, x_3)$ on $\delta\hexagon$ in Carath\'eodory sense. 
We assume further that $\partial\Omega^{\delta}$ converges to $\partial\Omega$ in Hausdorff distance~\eqref{eqn::boundary_cvg_Hausdorff}. 
Consider the UST in $\Omega^{\delta}$ with wired boundary condition and 
define $\LA_1^{\delta}$ and $\LA_2^{\delta}$ as in~\eqref{eqn::conditionalevent} and consider the tripod $(\gamma_1^{\delta}, \gamma_2^{\delta}, \gamma_3^{\delta})$ and the trifurcation $\trifurcation^{\delta}$. 
Fix $z\in\Omega$ and suppose $z^{\delta}$ is a vertex in $\LV^{\circ}(\Omega^\delta)$ that is nearest to $z$. We have 
\begin{align}\label{eqn::deltasquare}
\lim_{\delta\to 0}\delta^{-2}\PP\left[\trifurcation^{\delta}=z^{\delta}\cond \LA_1^{\delta}\cap\LA_2^{\delta}\right]=\frac{3\sqrt{3}}{4} p(\Omega; x_1, x_2, x_3; z), 
\end{align}
where 
\begin{align}\label{eqn::density_def}
\begin{split}
p(\Omega; x_1, x_2, x_3; z)=&\frac{\LZtripod(\Omega; x_1, x_2, x_3; z)}{\int_{\Omega}\LZtripod(\Omega; x_1, x_2, x_3; w)|\ud w|^2}\\
=&{\frac{4}{3\pi}}\frac{\Poisson(\Omega; z, x_1)^2\Poisson(\Omega; z, x_2)^2\Poisson(\Omega; z, x_3)^2}{\CR(\Omega; z)^2\Poisson(\Omega; x_1, x_2)\Poisson(\Omega; x_2, x_3)\Poisson(\Omega; x_3, x_1)}. 
\end{split}
\end{align}
\end{theorem}

In~\cite{KenyonLongRangeSpanningTree}, R. Kenyon considered the distribution of the trifurcation in UST in a bounded $3$-polygon of $\delta\Z^2$ with free boundary condition conditional on a similar event $\LA_1^{\delta}\cap\LA_2^{\delta}$. The scaling limit of the distribution of the trifurcation in this case is given by~\cite[Theorem~3.1]{KenyonLongRangeSpanningTree}: 
\begin{align}\label{eqn::Kenyon_trifurcation}
\tilde{p}(\Omega; x_1, x_2, x_3; z)=\frac{2}{\pi^2}\frac{|(\varphi(x_2)-\varphi(x_1))(\varphi(x_3)-\varphi(x_2))(\varphi(x_3)-\varphi(x_1))|\Im(\varphi(z))}{|(\varphi(z)-\varphi(x_1))(\varphi(z)-\varphi(x_2))(\varphi(z)-\varphi(x_3))|^2}|\varphi'(z)|^2, 
\end{align}
where $\varphi$ is any conformal map from $\Omega$ onto $\HH$. Our answer~\eqref{eqn::density_def} can be written as
\begin{align}\label{eqn::density_H}
p(\Omega; x_1, x_2, x_3; z)=\frac{8}{3\pi}\frac{|(\varphi(x_2)-\varphi(x_1))(\varphi(x_3)-\varphi(x_2))(\varphi(x_3)-\varphi(x_1))|^2\Im(\varphi(z))^4}{|(\varphi(z)-\varphi(x_1))(\varphi(z)-\varphi(x_2))(\varphi(z)-\varphi(x_3))|^4}|\varphi'(z)|^2. 
\end{align}
As the boundary condition for~\eqref{eqn::Kenyon_trifurcation} is free, while our boundary condition in Theorem~\ref{thm::trifurcation} is wired, our answer~\eqref{eqn::density_H} is different from~\eqref{eqn::Kenyon_trifurcation}. The observable we construct in Section~\ref{sec::trifurcation} is different from the one in~\cite{KenyonLongRangeSpanningTree}. Moreover, our proof of the key lemma--Lemma~\ref{lem::Fomin}--relies on the hexagonal lattice in an crucial way and it does not seem to work with $\delta\Z^2$ lattice in an easy way.
We prove Theorem~\ref{thm::trifurcation} in Section~\ref{sec::trifurcation}. 
Our proof has three steps. 
\begin{itemize}
\item First, we derive the probability $\PP[\LA_1^{\delta}\cap\LA_2^{\delta}\cap\{\trifurcation^{\delta}=z^{\delta}\}]$ in Lemma~\ref{lem::Fomin}. In this step, we write the probability in terms of Poisson kernel using Fomin's formula~\cite{fomin2001loop}. This step uses the geometry of hexagonal lattice in an essential way. The constant $\frac{3\sqrt{3}}{4}$ in RHS of~\eqref{eqn::deltasquare} is also lattice-dependent: the area of the dual face of hexagonal lattice is $\frac{3\sqrt{3}}{4}\delta^2$, see Figure~\ref{fig::threeneighbors}. 
The proof for this step works for more general graphs, but it is important that each vertex of the graph has degree three. We include the calculation for 4-8 lattice in Appendix~\ref{appendix::48lattice}. However, the proof does not work for $\Z^2$ lattice directly.
\item Second, we derive the scaling limit of the probability $\PP[\LA_1^{\delta}\cap\LA_2^{\delta}\cap\{\trifurcation^{\delta}=z^{\delta}\}]$ in Proposition~\ref{prop::observable_cvg}. It turns out that the scaling limit of $\PP[\LA_1^{\delta}\cap\LA_2^{\delta}\cap\{\trifurcation^{\delta}=z^{\delta}\}]$ is the same as tripod partition function $\LZtripod$~\eqref{eqn::LZtripod_def} up to multiplicative constant. In CFT language, the partition function $\LZtripod$ can be thought of as a correlation function of conformal weights $h_{1,2}=\frac{6-\kappa}{2\kappa}=1$ at the boundary points $x_1, x_2, x_3$ and a spinless field of weights $(1,1)$ at the bulk point $z$. In our setup, we give a probabilistic interpretation of such correlation function via Proposition~\ref{prop::observable_cvg}.

\item Third, we derive the scaling limit of $\PP[\LA_1^{\delta}\cap \LA_2^{\delta}]$ in Proposition~\ref{prop::deltacube}. Moreover, we also derive the probability $\PP[\LA_1^{\delta}\cap\{z^{\delta}\rightsquigarrow e_3^{\delta}\}]$ using Fomin's formula in Proposition~\ref{prop::A1interior}. This conclusion has an interesting consequence for chordal $\SLE_2$, see Proposition~\ref{prop::chordalSLE2_Fomin}. In the proof of Propositions~\ref{prop::deltacube} and~\ref{prop::A1interior}, it is quite delicate to interchange integral and limit. The proof of this step relies crucially on tools developed in~\cite{ChelkakSmirnovDiscreteComplexAnalysis} and~\cite{ChelkakWanMassiveLERW} and careful analysis in Sections~\ref{subsec::proba_tight}-\ref{subsec::A1interior}. 
\end{itemize}

\subsection{Limiting distribution for tripod}

\paragraph*{Space of curves.} 
We denote by $\chamber$ the set of planar oriented curves, i.e. continuous mappings from $[0,1]$ to $\C$ modulo reparameterization. We equip $\chamber$ with the metric:
\begin{equation}\label{eqn::curves_metric}
\dist_{\chamber}(\eta_1, \eta_2)=\inf_{\varphi_1, \varphi_2}\sup_{t\in [0,1]}|\eta_1(\varphi_1(t))-\eta_2(\varphi_2(t))|,
\end{equation}
where the infimum is taken over all increasing homeomorphisms $\varphi_1, \varphi_2: [0,1]\to [0,1]$. 
The metric space $(\chamber, \dist_{\chamber})$ is complete and seperable, see~\cite{ABRandomcurve, KemppainenSmirnovRandomCurves}. Let $\mathcal{P}$ be a family of probability measures on $\C$. We say $\mathcal{P}$ is tight if for any $\eps>0$, there exists a compact set $K_{\eps}$ such that $\PP[K_{\eps}]\ge 1-\eps$ for all $\PP\in\mathcal{P}$. We say $\mathcal{P}$ is relatively compact if every sequence of elements in $\mathcal{P}$ has a weakly convergent subsequence. As the metric space is complete and separable, relative compactness is equivalent to tightness.

For an $1$-polygon $(\Omega; x_1)$ with $z\in\Omega$, denote by $\chamber(\Omega; x_1; z)$ the collection of continuous simple curves in $\Omega$ from $x_1$ to $z$ such that they only touch the boundary $\partial\Omega$ at $x_1$. Fix $p\ge 2$. For a $p$-polygon $(\Omega; x_1, \ldots, x_p)$ with $z\in\Omega$, denote by $\chamber(\Omega; x_1, \ldots, x_p; z)$ the collection of continuous simple curves $(\gamma_1, \ldots, \gamma_p)$ such that $\gamma_j\in\chamber(\Omega; x_j; z)$ for $j\in\{1, \ldots, p\}$ and $\gamma_i\cap\gamma_j=\{z\}$ for $i\neq j$. For the tripod in UST, we will consider its convergence in the space $\chamber(\Omega; x_1, x_2, x_3; z)$. 

\paragraph*{Three-sided radial SLE.}
Multi-sided radial SLE was introduced by V. Healey and G. Lawler in~\cite{HealeyLawlerNSidedRadialSLE} as radial process under common-time parameter. It was generalized as radial process under multi-time parameter later in~\cite{FengWuYangIsing, KrusellWangWuCommutationRelation} and~\cite{HuangPeltolaWuMultiradialSLEResamplingBP} where the authors provided an equivalent description of the process~\cite[Theorem~1.3]{HuangPeltolaWuMultiradialSLEResamplingBP} (see more details in Section~\ref{subsec::pre_SLE}). 
\begin{definition}\label{def::3sidedradialSLE}
Fix $\kappa\in (0,4]$ and a $3$-polygon $(\Omega; x_1, x_2, x_3)$ with $z\in\Omega$. Three-sided radial $\SLE_{\kappa}$ is a probability measure on $\chamber(\Omega; x_1, x_2, x_3; z)$ whose law is uniquely characterized by the following:
\begin{itemize}
\item the marginal law of $\gamma_3$ is radial $\SLE_{\kappa}(2,2)$ in $\Omega$ from $x_3$ to $z$ with force points $(x_1, x_2)$;
\item given $\gamma_3$, the conditional law of $\gamma_1$ chordal $\SLE_{\kappa}(2)$ in $\Omega\setminus\gamma_3$ from $x_1$ to $z$ with force point $x_2$; 
\item given $(\gamma_3, \gamma_1)$, the conditional law of $\gamma_2$ is chordal $\SLE_{\kappa}$ in $(\Omega\setminus(\gamma_3\cup\gamma_1); x_2, z)$. 
\end{itemize}
The partition function for three-sided radial $\SLE_{\kappa}$ in $(\Omega; x_1, x_2, x_3; z)$ is given by 
\begin{align}\label{eqn::3sided_pf}
\nradpartfn{3}(\Omega; x_1, x_2, x_3; z)=\frac{\left(\Poisson(\Omega; z, x_1)\Poisson(\Omega; z, x_2)\Poisson(\Omega; z, x_3)\right)^{\frac{(10-\kappa)}{2\kappa}}}{\CR(\Omega; z)^{\frac{(10-\kappa)(\kappa+2)}{8\kappa}}\left(\Poisson(\Omega; x_1, x_2)\Poisson(\Omega; x_2, x_3)\Poisson(\Omega; x_3, x_1)\right)^{\frac{1}{\kappa}}}. 
\end{align}
\end{definition}

The tripod in UST converges to three-sided radial $\SLE_2$.

\begin{theorem}\label{thm::tripod}
Assume the same setup as in Theorem~\ref{thm::trifurcation}. 
The law of the tripod $(\gamma_1^{\delta}, \gamma_2^{\delta}, \gamma_3^{\delta})$ conditional on $\LA_1^{\delta}\cap \LA_2^{\delta}$ converges weakly to three continuous simple curves $(\gamma_1, \gamma_2, \gamma_3)$ whose law is characterized by the following properties.
\begin{itemize}
\item There exists $\trifurcation\in\Omega$ such that $(\gamma_1, \gamma_2, \gamma_3)\in\chamber(\Omega; x_1, x_2, x_3; \trifurcation)$. We call $\trifurcation$ trifurcation.  
The law of the trifurcation $\trifurcation$ is absolutely continuous with respect to Lebesgue measure on $\Omega$ and the density is given by $p(\Omega; x_1, x_2, x_3; z)$ defined in~\eqref{eqn::density_def}, i.e. 
\begin{equation}\label{eqn::trifurcation_distribution}
\PP[\trifurcation\in K]=\int_K p(\Omega; x_1, x_2, x_3; w)|\ud w|^2, \qquad \text{for any compact subset $K$ of $\Omega$}.
\end{equation}
\item Given $\trifurcation$, the conditional law of the tripod $(\gamma_1, \gamma_2, \gamma_3)$ is three-sided radial $\SLE_2$ in $(\Omega; x_1, x_2, x_3; \trifurcation)$ in Definition~\ref{def::3sidedradialSLE}. 
\end{itemize}
\end{theorem}

There are several works about convergence of LERW or branches in UST in different setup~\cite{LawlerSchrammWernerLERWUST, ZhanLERW, KarrilaUSTBranches, ChelkakWanMassiveLERW, LiuWuLERW, HanLiuWuUST}.
The strategy in these proofs is as follows: 1st. find a proper observable for the model; 2nd. derive the tightness of the process; 3rd. derive the corresponding partition function for SLE process from the observable. We recall the conclusion of~\cite{LawlerSchrammWernerLERWUST} in Lemma~\ref{lem::LERW_radialSLE2} and recall the conclusion of~\cite{ZhanLERW}  for the simply connected domain in Lemma~\ref{lem::boundary_branch_cvg1} (our proof of Theorem~\ref{thm::tripod} relies on these two cases). 
In these two cases, the scaling limit of the observable is the same as Poisson kernels.
In~\cite{KarrilaUSTBranches}, the author derived the scaling limit of multiple boundary-to-boundary branches in UST using the observable derived in~\cite{KenyonWilsonBoundaryPartitionsTreesDimers, KarrilaKytolaPeltolaCorrelationsLERWUST}. 
In~\cite{ChelkakWanMassiveLERW}, the authors derived the scaling limit of massive LERW where the corresponding observable is massive Poisson kernel. 
Although, we do not use the main conclusion of~\cite{ChelkakWanMassiveLERW} in this article, we do use an important tool developed in~\cite{ChelkakWanMassiveLERW}, see Section~\ref{subsec::harmonic_limit}. 
In the latter works~\cite{LiuWuLERW, HanLiuWuUST}, the authors derived the scaling limit of certain branches in UST using the observable derived in~\cite{LiuPeltolaWuUST}. The observable in this case is more involved: it is holomorphic function on the domain with proper boundary data. One is able to derive the partition function of the SLE process from the observable, but the observable is distinct from the partition function in general. 
It is quite surprising for us that the observable $\LZtripod$~\eqref{eqn::LZtripod_def} given by trifurcation distribution coincides exactly with three-sided radial SLE partition function $\nradpartfn{3}$~\eqref{eqn::3sided_pf} for $\kappa=2$. 

We prove Theorem~\ref{thm::tripod} in Section~\ref{sec::tripod}. 
Section~\ref{sec::tripod} contains two main results---Propositions~\ref{prop::cvg_branch_eta1} and~\ref{prop::gamma3_cvg}---as well as the proof of Theorem~\ref{thm::tripod}. 
We first derive the scaling limit of the boundary branch $\eta_1^{\delta}$ conditional on $\LA_1^{\delta}\cap\LA_2^{\delta}$ in Proposition~\ref{prop::cvg_branch_eta1}. The proof relies on the convergence of $\eta_1^{\delta}$ conditional on $\LA_1^{\delta}$ derived in~\cite{ZhanLERW} (see Lemma~\ref{lem::boundary_branch_cvg1}) and tools from~\cite{ChelkakSmirnovDiscreteComplexAnalysis} and~\cite{ChelkakWanMassiveLERW}. Proposition~\ref{prop::cvg_branch_eta1} provides us with the tightness of the tripod (see Proposition~\ref{prop::tripod_tightness}) and it is also crucial in the proof of Theorem~\ref{thm::tripod}.
We then derive the scaling limit of $\gamma_3^{\delta}$ conditional on $\LA_1^{\delta}\cap\LA_2^{\delta}$ in Proposition~\ref{prop::gamma3_cvg}. 
The proof relies on the convergence of LERW derived in~\cite{LawlerSchrammWernerLERWUST} (see Lemma~\ref{lem::LERW_radialSLE2}) and the scaling limit of the observable derived in Section~\ref{sec::trifurcation}. 
We complete the proof of Theorem~\ref{thm::tripod} in Section~\ref{subsec::tripod_proof}. The proof relies on Propositions~\ref{prop::cvg_branch_eta1} and~\ref{prop::gamma3_cvg} and techniques about multi-sided radial SLE developed recently in~\cite{HuangPeltolaWuMultiradialSLEResamplingBP}. Although our proof of the key lemma---Lemma~\ref{lem::Fomin}---is sensitive to the geometry of the lattice, the convergence of the tripod in Theorem~\ref{thm::tripod} can be extended to more general graphs including $\Z^2$, see discussion in Section~\ref{subsec::tripod_otherlattice}.

\subsection{A consequence of Fomin's formula}
\begin{proposition}\label{prop::chordalSLE2_Fomin}
Fix a $3$-polygon $(\Omega; x_1, x_2, x_3)$ with $z\in\Omega$. 
Suppose $\ell$ is chordal $\SLE_2$ in $(\Omega; x_1, x_3)$ and denote its law by $\chordalSLE^{(\kappa=2)}(\Omega; x_1, x_3)$. Define 
\begin{equation}
g(\Omega; x_1, x_3; z)=\chordalSLEexp^{(\kappa=2)}(\Omega; x_1, x_3)\left[\harmonic(\Omega\setminus\ell; z, \ell)\right],
\end{equation}
where $\harmonic(\Omega\setminus\ell; z, \ell)$ is harmonic measure of $\ell$ in $\Omega\setminus\ell$ seen from $z$. Then $g$ is conformally invariant: for any conformal map $\varphi$ on $\Omega$, we have
\begin{equation}\label{eqn::chordalSLE2_harmonic_inv}
g(\Omega; x_1, x_3; z)=g(\varphi(\Omega); \varphi(x_1), \varphi(x_3); \varphi(z)); 
\end{equation}
and for $(\Omega; x_1, x_3)=(\HH; 0, \infty)$ and $z=r\ee^{\ii\theta}\in \HH$ with $\theta\in(0,\pi)$, we have
\begin{equation}\label{eqn::chordalSLE2_harmonic_H}
g(\HH; 0, \infty; r\ee^{\ii\theta})=\frac{4}{3\pi}\left(\theta(\pi-\theta)+(\frac{\pi}{2}-\theta)\sin\theta\cos\theta+2\sin^2\theta\right).
\end{equation}
Moreover, when $z\to x_2$, we have 
\begin{equation}\label{eqn::chordalSLE2_nharmonic}
\chordalSLEexp^{(\kappa=2)}(\Omega; x_1, x_3)\left[\nharmonic(\Omega\setminus\ell; x_2, \ell)\right]=\left(\frac{2\Poisson(\Omega; x_1, x_2)\Poisson(\Omega; x_2, x_3)}{\Poisson(\Omega; x_1, x_3)}\right)^{1/2}, 
\end{equation}
where $\nharmonic(\Omega\setminus\ell; x_2, \ell)$ is the renormalized harmonic measure of $\ell$ in $\Omega\setminus\ell$ seen from $x_2$ (see~\eqref{eqn::nharmonic_H}-\eqref{eqn::nharmonic_cov}) and $\Poisson(\Omega; x_i, x_j)$ is boundary Poisson kernel (see~\eqref{eqn::bPoisson_H}-\eqref{eqn::bPoisson_cov}). 
\end{proposition}
We will complete the proof of Proposition~\ref{prop::chordalSLE2_Fomin} in Section~\ref{subsec::prop14}. 
The proof is a direct consequence of Proposition~\ref{prop::A1interior} and Proposition~\ref{prop::cvg_branch_eta1}. 
It seems hard for us to derive~\eqref{eqn::chordalSLE2_harmonic_H} without Fomin's formula. 


\paragraph*{Acknowledgments.}
We thank Chongzhi Huang for helpful discussion about Section~\ref{subsec::pre_SLE}. 
We thank Yunhui Wu and Hong Zhang for helpful discussion about the proof  of Lemma~\ref{lem::g_solution_Laplace}. We thank Yueheng Li and Jinwoo Sung for helpful discussion on how to extend our conclusion from the hexagonal lattice to $\Z^2$ lattice.
H.W. is supported by New Cornerstone Investigator Program 100001127. H.W. is partly affiliated at Yanqi Lake Beijing Institute of Mathematical Sciences and Applications, Beijing, China.

\section{Preliminaries}
\label{sec::pre}
\subsection{Conformal radius, Green's function, Poisson kernel and harmonic measure}
\paragraph*{Conformal radius.}
For a simply connected domain $\Omega\subsetneq\C$ and $z\in \Omega$, the \emph{conformal radius} of $\Omega$ seen from $z$ is defined by $\CR(\Omega; z) := 1/\phi'(z)$  where $\phi : \Omega \to \U$ is the conformal map such that $\phi(z)=0$ and $\phi'(z)>0$. 
In particular, for $\Omega=\HH$ and $z\in\HH$, we have
\begin{align}\label{eqn::CR_H}
\CR(\HH; z)=2\Im(z). 
\end{align}
Conformal radius is conformally covariant: for any conformal map $\varphi$ on $\Omega$, we have
\begin{align} \label{eqn::CR_cov}
\CR(\Omega; z)=|\varphi'(z)|^{-1}\CR(\varphi(\Omega); \varphi(z)).
\end{align}

\paragraph*{Green's function.}
For $\HH$, \textit{Green's function} is defined as 
\begin{equation}\label{eqn::cgreen}
    \Green(\HH;z,w)=\log{\frac{|z-\overline{w}|}{|z-{w}|}}, \qquad z,w\in \HH.
\end{equation}
For a simply connected domain $\Omega\subsetneq \C$, we define the Green's function via conformal invariance:
\begin{equation}\label{eqn::cgreen_inv}
    \Green(\Omega;z,w)=\Green(\HH;\varphi(z),\varphi(w)), \qquad z,w\in \Omega,
\end{equation}
where $\varphi$ is any conformal map from $\Omega$ onto $\HH$.

\paragraph*{Poisson kernel.} 
For a $1$-polygon $(\HH; x)$ with $z\in\HH$, the \textit{Poisson kernel} is defined as 
\begin{equation}\label{eqn::Poisson_H}
\Poisson(\HH;z,x) = \frac{2\Im z}{|z - x|^2}.
\end{equation}
For a general $1$-polygon $(\Omega; x)$ with $z\in\Omega$, we define the Poisson kernel via conformal covariance: 
\begin{equation}\label{eqn::Poisson_cov}
\Poisson(\Omega;z,x) = |\varphi'(x)|  \Poisson(\HH;\varphi(z),\varphi(x)),
\end{equation}
where $\varphi$ is any conformal map from $\Omega$ onto $\HH$. The relation between Poisson kernel and Green's function is as follows:
\begin{equation}\label{eqn::cgreen_poisson}
    \Poisson(\HH;z,x)=\lim_{r\to 0+}\frac{1}{r}\Green(\HH; z,x+\ii r).
\end{equation}

\paragraph*{Boundary Poisson kernel.}
For a $2$-polygon $(\HH; x,y)$, the \textit{boundary Poisson kernel} is defined as
\begin{align}\label{eqn::bPoisson_H}
\Poisson(\HH; x,y) :=\lim_{r\to 0}\frac{1}{r}\Poisson(\HH; x+\ii r, y)=\frac{2}{(y-x)^2}. 
\end{align}
For a general $2$-polygon $(\Omega; x,y)$, we define the boundary Poisson kernel via conformal covariance: 
\begin{align}\label{eqn::bPoisson_cov}
\Poisson(\Omega; x,y) := |\varphi'(x)\varphi'(y)|\Poisson(\HH; \varphi(x),\varphi(y)), 
\end{align}
where $\varphi$ is any conformal map from $\Omega$ onto $\HH$.

\paragraph*{Harmonic measure.}
For the upper half-plane $\HH$ with $z\in\HH$, suppose $\ell\subset\partial\HH$ is an interval, the harmonic measure of $\ell$ seen from $z$ can be defined as 
\begin{align}\label{eqn::charnomic_H}
\harmonic(\HH; z, \ell)=\int_{\ell}\Poisson(\HH; z, y)|\ud y|. 
\end{align}
Suppose $x\in\R\setminus\ell$, we define (renormalized) harmonic measure by 
\begin{align}\label{eqn::nharmonic_H}
\nharmonic(\HH; x, \ell):=\lim_{r\to 0+}\frac{1}{r}\harmonic(\HH;x+\ii r, \ell)=\int_{\ell}\Poisson(\HH; x, y)|\ud y|. 
\end{align}
For a general $1$-polygon $(\Omega; x)$ with $z\in\Omega$, suppose $\ell\subset\partial\Omega$ is a boundary arc and $x\not\in\ell$, we define $\harmonic$ and $\nharmonic$ via conformal covariance: 
\begin{align}\label{eqn::nharmonic_cov}
\harmonic(\Omega; z, \ell)=\harmonic(\HH; \varphi(z), \varphi(\ell)), \qquad \nharmonic(\Omega; x, \ell)=|\varphi'(x)|\nharmonic(\HH; \varphi(x), \varphi(\ell)), 
\end{align}
where $\varphi$ is any conformal map from $\Omega$ onto $\HH$. Note that $\harmonic(\Omega; z, \ell)$ is the harmonic measure of $\ell$ seen from $z$, and we also say that $\nharmonic(\Omega; x, \ell)$ is the harmonic measure of $\ell$ seen from $x$.

\begin{lemma}\label{lem::LZtripod_integrable}
Recall that tripod partition function $\LZtripod$ is defined in~\eqref{eqn::LZtripod_def}: 
\begin{align*}
\LZtripod(\Omega; x_1, x_2, x_3; z)=\frac{\left(\Poisson(\Omega; z, x_1)\Poisson(\Omega; z, x_2)\Poisson(\Omega; z, x_3)\right)^2}{\CR(\Omega; z)^2\left(\Poisson(\Omega; x_1, x_2)\Poisson(\Omega; x_2, x_3)\Poisson(\Omega; x_3, x_1)\right)^{1/2}}, 
\end{align*}
It satisfies conformal covariance~\eqref{eqn::LZtripod_cov}; it has expression~\eqref{eqn::LZtripod_H} when $\Omega=\HH$; it is integrable and~\eqref{eqn::LZtripod_integrable} holds: 
\begin{align*}
\int_{\Omega}\LZtripod(\Omega; x_1, x_2, x_3; w)|\ud w|^2={\frac{3\pi}{4}} \left(\Poisson(\Omega; x_1, x_2)\Poisson(\Omega; x_2, x_3)\Poisson(\Omega; x_3, x_1)\right)^{1/2}.  
\end{align*}
\end{lemma}

\begin{proof}
The conformal covariance~\eqref{eqn::LZtripod_cov} follows from~\eqref{eqn::CR_cov}, \eqref{eqn::Poisson_cov} and~\eqref{eqn::bPoisson_cov}.
The expression~\eqref{eqn::LZtripod_H} when $\Omega=\HH$ follows from ~\eqref{eqn::CR_H}, \eqref{eqn::Poisson_H} and~\eqref{eqn::bPoisson_H}. 
It remains to show~\eqref{eqn::LZtripod_integrable}. As the two sides of~\eqref{eqn::LZtripod_integrable} satisfies the same conformal covariance, it suffices to show that, for $\Omega=\HH$ and $x_1<x_2<x_3$, 
\begin{align*}
\int_{\HH}\LZtripod(\HH; x_1, x_2, x_3; w)|\ud w|^2=&{4\sqrt{2}}\int_{\HH}\frac{(x_2-x_1)(x_3-x_2)(x_3-x_1)\Im(w)^4}{|(w-x_1)(w-x_2)(w-x_3)|^4}|\ud w|^2\notag\\
=&\frac{3\pi/\sqrt{2}}{(x_2-x_1)(x_3-x_2)(x_3-x_1)}. 
\end{align*} We write $x_{ji}=x_j-x_i$ to shorten the notation. Then, it suffices to show 
\begin{align}\label{eqn::LZtripod_integrable_H_goal}
\int_{\HH}\frac{\Im(w)^4|\ud w|^2}{|(w-x_1)(w-x_2)(w-x_3)|^4}=\frac{3\pi/8}{(x_{21}x_{32}x_{31})^2}.
\end{align}

We perform change of variables:  
\[
u=\frac{(w - x_1) (x_2 - x_3)}{(w - x_3) (x_2 - x_1)}. 
\]
Then
\begin{align}\label{eqn::LZtripod_integrable_H_aux1}
\text{LHS of}~\eqref{eqn::LZtripod_integrable_H_goal}=\frac{1}{(x_{21}x_{32}x_{31})^2}\underbrace{\int_{\HH}\frac{\Im(u)^4|\ud u|^2}{|u(u-1)|^4}}_{I_0:=}. 
\end{align}

Let us evaluate $I_0$. We use Euclidean coordinates $u=x+\ii y$ with $x\in\R$ and $y>0$, then 
\begin{align*}
I_0
=\int_{y>0}\int_{x\in\mathbb{R}}
\frac{y^4 \ud x   \ud y}{(x^2+y^2)^2 ((x-1)^2+y^2)^2}.
\end{align*}
Feynman-parameter identity for $A, B>0$ reads
\begin{align*}
\frac{1}{A^2B^2}
=6\int_0^1 \frac{t(1-t)\ud t}{\left(tA+(1-t)B\right)^4}.
\end{align*}
Setting $A=x^2+y^2$ and $B=(x-1)^2+y^2$, we obtain
\begin{align*}
\frac{1}{(x^2+y^2)^2((x-1)^2+y^2)^2}=6\int_0^1\frac{t(1-t)\ud t}{(t(x^2+y^2)+(1-t)((x-1)^2+y^2))^4}.
\end{align*}
Using this identity, we have 
\begin{align*}
I_0=& 6\int_0^1t(1-t) \ud t
\int_{y>0}\int_{x\in\R}\frac{y^4\ud x\ud y}{(t(x^2+y^2)+(1-t)((x-1)^2+y^2))^4}\\
=&6\int_0^1t(1-t) \ud t
\int_{y>0}\int_{x\in\R}\frac{y^4\ud x\ud y}{(x^2+y^2+t(1-t))^4}\tag{replace $x$ by $x+(1-t)$}\\
=& 6\int_0^1t(1-t) \ud t\int_0^{\pi}(\sin\theta)^4\ud\theta\int_0^{\infty}\frac{r^5\ud r}{(r^2+t(1-t))^4}.\tag{set $x=r\cos\theta, y=r\sin\theta$}
\end{align*}
Let us evaluate the integrals with $\ud \theta$ and $\ud r$. 
\begin{itemize}
\item Wallis formula gives
    \[\int_0^\pi(\sin\theta)^4 \ud \theta=\frac{3\pi}{8}.\]
\item For $p,q>0$, denote Beta-integral by 
\[\mathrm{Beta}(p,q)=\int_{0}^{\infty}\frac{s^{p-1} \ud s}{(1+s)^{p+q}}.\]
Then 
\begin{align*}
\int_0^\infty\frac{r^5\ud r}{(r^2+t(1-t))^4} 
=&\frac{1}{2t(1-t)}\int_0^\infty\frac{s^2\ud s}{(1+s)^4} \tag{set $s=\frac{r^2}{t(1-t)}$}\\
=&\frac{\mathrm{Beta}(3,1)}{2t(1-t)}=\frac{1}{6t(1-t)}. 
\end{align*}
\end{itemize}
Plugging these two integrals into $I_0$, we obtain $I_0=3\pi/8$. Plugging $I_0=3\pi/8$ into~\eqref{eqn::LZtripod_integrable_H_aux1}, we obtain~\eqref{eqn::LZtripod_integrable_H_goal} as desired. 
\end{proof}

\subsection{Random walk}

\paragraph*{Poisson kernel and harmonic measure.}
Let $\LG=(\mathcal{V}, \LE)$ be a graph and $\LV^{\partial} \subset \LV$ be the set of boundary vertices and $ \LE^{\partial} \subset \LE$ be the set of boundary edges. 
Consider the simple random walk $\beta$ starting from an interior vertex $\beta(0)=v \in \LV^{\circ}$, and stopped at the first time $T$ it hits the boundary $\LV^{\partial}$. 
The last step of the stopped random walk $\beta=(\beta(t))_{t=0}^T$ is a boundary edge $\langle\beta(T-1), \beta(T)\rangle \in \LE^{\partial}$. 
For any given boundary edge $e_{\out} \in \LE^{\partial}$, the Poisson kernel of $e_{\out}$ seen from $v \in \LV^{\circ}$ is the probability that
the random walk $\beta$ starting from $v$ hits $\LV^{\partial}$ through the edge $e_{\out}$:
\begin{align}\label{eqn::dPoisson_def}
\harmonic(\LG; v, e_{\out})=\PP\left[\langle\beta(T-1), \beta(T)\rangle=e_{\out}\cond \beta(0)=v\right]. 
\end{align}
For $W\subset \LV^{\partial}$, the harmonic measure of $W$ seen from $v\in\LV^{\circ}$ is the probability that the random walk starting from $v$ hits $\LV^{\partial}$ through vertices in $W$: 
\begin{align}\label{eqn::dharmonic_def}
\harmonic(\LG; v, W)=\sum_{e=\langle e^{\circ},e^{\partial}\rangle\in \LE^{\partial}}\harmonic(\LG; v, e)\one\{e^{\partial}\in W\}. 
\end{align}

\paragraph*{Green's function.}
Green's function for simple random walk $\beta$ is defined by 
\begin{equation}\label{eqn::dGreen_def}
\Green(\LG; v, u)=\sum_{t=0}^T \PP[\beta(t)=u\cond \beta(0)=v].
\end{equation}
For each directed edge $\langle v,v^\prime\rangle$, we assign a weight $w_{\langle v,v\prime \rangle}=1/\deg(v)$. 
For a finite walk $\chi=(v_0,\dots,v_{\ell})$ on the graph $\LG$, we denote  by $\LW(v,u)$ the set of paths starting from $v_0=v$ and ending at $v_{\ell}=u$, and assign the weight $w(\chi)= \prod_{s=1}^{\ell}w_{\langle v_{s-1},v_s\rangle}$.
Then Green's function can also be expressed as
\begin{equation*}
    \Green(\LG; v, u)=\sum_{\chi\in\LW(v,u)}w(\chi).
\end{equation*}
The Poisson kernel of a boundary edge $e=\langle e^{\circ}, e^{\partial}\rangle\in\LE^{\partial}$ can be written as a Green's function:
\begin{equation*}
\harmonic(\LG; v, e)=\Green(\LG; v, e^{\partial}). 
\end{equation*}
Furthermore, 
\begin{equation}\label{eqn::harmonic_green}
\harmonic(\LG; v, e)=\Green(\LG; v, e^{\partial})=\frac{1}{\deg(e^{\circ})}\Green(\LG; v, e^{\circ}),  
\end{equation}
because each path $\chi \in \LW(v,e^{\partial})$ can be decomposed into a path $\chi^\prime\in \LW(v,e^{\circ})$ and the edge $e=\langle e^{\circ}, e^{\partial}\rangle$, and the weight of $e$ is $1/\deg(e^{\circ})$. 

\paragraph*{Discrete Laplacian.}
For a function $g$ on $\LV$, define (discrete) Laplacian by 
\begin{align}\label{eqn::dLaplacian_def}
\Delta g(v)=\sum_{w\sim v}\frac{1}{\deg(v)}(g(w)-g(v)), \qquad v\in\LV^{\circ}. 
\end{align}
Green's function $\Green(\LG; \cdot, u)$~\eqref{eqn::dGreen_def} has the following property
\[-\Delta \Green(\LG; v, u)=0, \quad\text{for }v\neq u; \qquad \text{and}\qquad -\Delta\Green(\LG; v, u)=1\quad\text{for }v= u.\]
As a consequence, suppose $g$ is a function on $\LV$ such that $-\Delta g=q$ on $\LV^{\circ}$ and $g=0$ on $\LV^{\partial}$, then 
\begin{align}\label{eqn::dLaplacian_green}
g(v)=\sum_{u\in \LV^{\circ}} \Green(\LG; v, u)q(u). 
\end{align}

\subsection{Scaling limit of  Green's function, Poisson kernel and harmonic measure}
\label{subsec::harmonic_limit}

In this section, we collect properties and the scaling limit of Green's function in Lemmas~\ref{lem::Green_cvg} and~\ref{lem::Green_upperbound}. 
We collect  properties and scaling limit of Poisson kernel in Lemmas~\ref{lem::discretePoisson_max_control} and~\ref{lem::Poisson_cvg}. 
We collect the scaling limit of the harmonic measure and boundary Poisson kernel in Lemmas~\ref{lem::bPoisson_cvg} and~\ref{lem::nharmonic_cvg}.
The proof for these conclusions relies crucially on tools from~\cite{ChelkakSmirnovDiscreteComplexAnalysis, ChelkakWanMassiveLERW}. 
These conclusions play an essential role in Sections~\ref{sec::trifurcation} and~\ref{sec::tripod}. 

\begin{lemma}[\cite{ChelkakSmirnovDiscreteComplexAnalysis}]
\label{lem::Green_cvg}
   Suppose $\Omega\subset \C$ is a bounded simply connected domain and suppose $\Omega^\delta$ is an approximation of $\Omega$ on $\delta\hexagon$ in Carath\'eodory sense.  
   Fix two interior points $z,w\in \Omega$ and suppose $z^\delta$ (resp. $w^\delta$) is the vertex in $\LV^{\circ}(\Omega^\delta)$ that is nearest to $z$ (resp. to $w$).
The scaling limit of the Green's function~\eqref{eqn::dGreen_def} exists:
    \begin{equation}\label{eqn::Green_cvg}
        \lim_{\delta\to 0}\Green(\Omega^\delta; z^\delta,w^\delta)=\frac{\sqrt{3}}{2\pi}\Green(\Omega;z,w),
    \end{equation}
    where $\Green(\Omega;z,w)$ is Green's function~\eqref{eqn::cgreen}-\eqref{eqn::cgreen_inv}. Moreover, the convergence~\eqref{eqn::Green_cvg} is uniform as long as $z,w$ remain in a compact subset of $\Omega$ and $\dist(z,w)\ge \eps$ for some $\eps>0$.
    \end{lemma}
    
    \begin{proof}
    The uniform convergence~\eqref{eqn::Green_cvg} is guaranteed by~\cite[Corollary~3.11]{ChelkakSmirnovDiscreteComplexAnalysis}. Note that the Green's function in LHS of~\eqref{eqn::Green_cvg} is the Green's function for the simple random walk on hexagonal lattice $\delta\hexagon$ and  the constant $\sqrt{3}$ in RHS of~\eqref{eqn::Green_cvg} is lattice-dependent.
    \end{proof}
    \begin{lemma}\label{lem::Green_upperbound}
    Suppose $\Omega\subset \C$ is a bounded simply connected domain and suppose $\Omega^\delta$ is an approximation of $\Omega$ on $\delta\hexagon$ in Carath\'eodory sense. 
    We assume further that $\partial\Omega^{\delta}$ converges to $\partial\Omega$ in Hausdorff distance~\eqref{eqn::boundary_cvg_Hausdorff}. 
    Fix an interior point $z\in\Omega$ and suppose $z^{\delta}$ is the vertex in $\LV^{\circ}(\Omega^{\delta})$ that is nearest to $z$. 
    There exists a constant $\delta_0>0$ depending on $(\{\Omega^\delta\}_{\delta>0};\Omega;z)$ and a constant $C_{\eqref{eqn::Green_upperbound}}\in (0,\infty)$ depending on $(\Omega;z)$ such that, for all $ \delta\le \delta_0$, 
\begin{equation}\label{eqn::Green_upperbound}
    \Green(\Omega^\delta; z^{\delta}, y)
    \le \frac{\sqrt{3}}{2\pi}\log\frac{1}{|z^{\delta}-y|\vee\delta}
    + C_{\eqref{eqn::Green_upperbound}}, \qquad \text{for all }y \in \LV(\Omega^\delta).
\end{equation}
    \end{lemma}
\begin{proof}
     We first recall the definition of Green's function on the whole lattice $\delta\hexagon$ as in~\cite[Definition~2.3]{ChelkakSmirnovDiscreteComplexAnalysis}.
Fix $u\in\LV(\delta\hexagon)$, we define $\Green(\delta\hexagon;u,\cdot)$ to be the function on $\LV(\delta\hexagon)$ satisfying
\begin{itemize}
    \item $-\Delta\Green(\delta\hexagon;u,v)=0$ for all $v\neq u$ and $-\Delta\Green(\delta\hexagon;u,v)=1$ for $v=u$;
    \item $\Green(\delta\hexagon;u,v)=o(|u-v|)$ as $|u-v|\to\infty$;
    \item $\Green(\delta\hexagon;u,u)=\frac{\sqrt{3}}{2\pi}\left(\log(1/ \delta)+\gamma_{\mathrm{Euler}}+\log 2\right)$, where $\gamma_{\mathrm{Euler}}$ is the Euler constant.
\end{itemize}
The existence and uniqueness of such $\Green(\delta\hexagon; u, \cdot)$ is proved in~\cite[Theorem~2.5]{ChelkakSmirnovDiscreteComplexAnalysis}.
Moreover, as $\delta\to 0$, 
    \begin{equation}\label{eqn::Green_plane_asymp}
        \Green(\delta\hexagon;u,v)= \frac{\sqrt{3}}{2\pi}\log \frac{1}{|u-v|}+O\left(\frac{\delta^2}{|u-v|^2}\right), \qquad \text{for } u,v\in \LV(\delta\hexagon), u\neq v. 
    \end{equation}
The asymptotic~\eqref{eqn::Green_plane_asymp} for isoradial graphs was first obtained in~\cite[Theorem~7.3]{KenyonLaplacianDiracCriticalPlanarGraphs}. The version we cite here is from~\cite[Theorem~2.5]{ChelkakSmirnovDiscreteComplexAnalysis} whose proof is contained in~\cite[Appendix~A.1]{ChelkakSmirnovDiscreteComplexAnalysis}. 

Let us consider the function $\Green(\Omega^{\delta}; z^{\delta}, \cdot)-\Green(\delta\hexagon; z^{\delta}, \cdot)$. This function is harmonic in $\LV^{\circ}(\Omega^{\delta})$ 
and its boundary value along $\partial\Omega^{\delta}$ is bounded from above by 
\begin{align*}
-\min_{v\in\partial\Omega^{\delta}}\Green(\delta\hexagon; z^{\delta}, v)\le &-\frac{\sqrt{3}}{2\pi}\min_{v\in\partial \Omega^\delta}\log\frac{1}{|z^\delta-v|} +\frac{C_{\eqref{eqn::Green_plane_asymp}}\delta^2}{\dist(z^\delta,\partial\Omega^\delta)^2}\\
=&\underbrace{\frac{\sqrt{3}}{2\pi}\max_{v\in\partial \Omega^\delta}\log|z^\delta-v| +\frac{C_{\eqref{eqn::Green_plane_asymp}}\delta^2}{\dist(z^\delta,\partial\Omega^\delta)^2}}_{J(\delta):=},
\end{align*}
where $C_{\eqref{eqn::Green_plane_asymp}}\in (0,\infty)$ is  a universal constant due to~\eqref{eqn::Green_plane_asymp}.  
By the maximum principle for discrete harmonic functions, we have 
$\Green(\Omega^{\delta}; z^{\delta}, y)\leq \Green(\delta\hexagon; z^{\delta}, y)+J(\delta)$ for all $y\in \LV(\Omega^{\delta})$. 
When $y\neq z^{\delta}$, combining with~\eqref{eqn::Green_plane_asymp}, we obtain
\begin{align}\label{eqn::green_upperbound_aux2}
\Green(\Omega^{\delta}; z^{\delta}, y)\leq \frac{\sqrt{3}}{2\pi}\log\frac{1}{|z^{\delta}-y|}+C_{\eqref{eqn::Green_plane_asymp}}+J(\delta),\qquad \text{for all }y\in \LV(\Omega^{\delta}) \text{ and }y\neq z^{\delta}. 
\end{align}
When $y=z^{\delta}$, combining with the value of $\Green(\delta\hexagon;z^{\delta},z^{\delta})$, we obtain
\begin{align}\label{eqn::green_upperbound_aux3}
\Green(\Omega^{\delta}; z^{\delta}, z^{\delta})\leq \frac{\sqrt{3}}{2\pi}\left(\log(1/ \delta)+\gamma_{\mathrm{Euler}}+\log 2\right)+J(\delta). 
\end{align}

Let us evaluate $J(\delta)$. Fix $R>r>0$ such that $B(z,r)\subset \Omega\subset B(z,R)$. As $\dist_H(\partial\Omega^{\delta}, \partial\Omega)\to 0$, there exists $\delta_0\in (0,1)$ such that 
\[B(z^{\delta}, r/2)\subset \Omega^{\delta}\subset B(z^{\delta}, 2R), \qquad \text{for all }\delta\le \delta_0.\]
Thus, for $\delta\le \delta_0$, we have 
$\max_{v\in\partial\Omega^{\delta}}|z^{\delta}-v|\le 2R$ and $\dist(z^{\delta}, \partial\Omega^{\delta})\ge r/2$. Therefore,
\begin{align}\label{eqn::green_upperbound_aux1}
J(\delta)\le \frac{\sqrt{3}}{2\pi}\log(2R)+\frac{4C_{\eqref{eqn::Green_plane_asymp}}}{r^2},\qquad \text{for all }\delta\le\delta_0. 
\end{align}
Combining the two cases~\eqref{eqn::green_upperbound_aux2} and~\eqref{eqn::green_upperbound_aux3} and~\eqref{eqn::green_upperbound_aux1}, we obtain~\eqref{eqn::Green_upperbound} with 
\[C_{\eqref{eqn::Green_upperbound}}=\frac{\sqrt{3}}{2\pi}\left(\gamma_{\mathrm{Euler}}+\log 2\right)+C_{\eqref{eqn::Green_plane_asymp}}+\frac{\sqrt{3}}{2\pi}\log(2R)+\frac{4C_{\eqref{eqn::Green_plane_asymp}}}{r^2}.\]
\end{proof}

\begin{lemma}\label{lem::discretePoisson_max_control}
Suppose $(\Omega; x)$ is a bounded $1$-polygon and suppose $(\Omega^{\delta}; x^{\delta})$ is an approximation of $(\Omega; x)$ on $\delta\hexagon$ in Carath\'eodory sense.
Suppose $e^{\delta}=\langle x^{\delta, \circ}, x^{\delta}\rangle$ is the boundary edge whose boundary end point $x^{\delta}$ is nearest to $x$. 
Fix an interior point $v\in\Omega$ and suppose $v^{\delta}$ is the vertex in $\LV^{\circ}(\Omega^{\delta})$ that is nearest to $v$. 
Fix a compact subset $K\subset\overline{\Omega}$ such that $x\not\in K$. Then there exists a constant $C_{\eqref{eqn::discretePoisson_max_control}}\in (0,\infty)$ depending on $(\Omega; K; x; v)$ such that 
\begin{align}\label{eqn::discretePoisson_max_control}
\max_{w\in K\cap\LV^{\circ}(\Omega^{\delta})}\frac{\harmonic(\Omega^{\delta}; w, e^{\delta})}{\harmonic(\Omega^{\delta}; v^{\delta}, e^{\delta})}\le C_{\eqref{eqn::discretePoisson_max_control}}. 
\end{align}
\end{lemma}
\begin{proof}
From~\eqref{eqn::harmonic_green}, for $w\in K\cap\LV^{\circ}(\Omega^{\delta})$, we have
\begin{align}\label{eqn::dPoisson_max_control_aux1}
\frac{\harmonic(\Omega^{\delta}; w, e^{\delta})}{\harmonic(\Omega^{\delta}; v^{\delta}, e^{\delta})}=\frac{\Green(\Omega^{\delta}; w, x^{\delta, \circ})}{\Green(\Omega^{\delta}; v^{\delta}, x^{\delta, \circ})}. 
\end{align}
Denote $r_0=\dist(x,K)\wedge |x-v|$.
Both $\Green(\Omega^{\delta}; w, \cdot)$ and $\Green(\Omega^{\delta}; v^{\delta}, \cdot)$ are harmonic in $\LV^{\circ}(\Omega^{\delta})\cap B(x,r_0)$ and vanish on $\LV^{\partial}(\Omega^{\delta})\cap B(x,r_0)$. Then \cite[Corollary~3.8]{ChelkakWanMassiveLERW} asserts that there exists a universal constant $\constCW\in (0,\infty)$ such that 
\begin{align}\label{eqn::dPoisson_max_control_aux2}
\begin{split}
\frac{\Green(\Omega^{\delta}; w, x^{\delta, \circ})}{\Green(\Omega^{\delta}; v^{\delta}, x^{\delta, \circ})}\le \constCW\frac{\Green(\Omega^{\delta}; w, u)}{\Green(\Omega^{\delta}; v^{\delta}, u)},\quad \text{for all }w\in \LV^{\circ}(\Omega^{\delta})\cap K\text{ and }u\in\LV^{\circ}(\Omega^{\delta})\cap B(x, \tfrac{1}{2}r_0).
\end{split}
\end{align}
Combining~\eqref{eqn::dPoisson_max_control_aux1} and~\eqref{eqn::dPoisson_max_control_aux2}, we have
\begin{align}\label{eqn::dPoisson_max_control_aux3}
\frac{\harmonic(\Omega^{\delta}; w, e^{\delta})}{\harmonic(\Omega^{\delta}; v^{\delta}, e^{\delta})}\le &\constCW\frac{\Green(\Omega^{\delta}; w, u)}{\Green(\Omega^{\delta}; v^{\delta}, u)},\qquad
\begin{cases}
\text{for all }w\in \LV^{\circ}(\Omega^{\delta})\cap K\text{ and }u\in\LV^{\circ}(\Omega^{\delta})\cap B(x, \tfrac{1}{2}r_0),\\
\text{where }r_0=\dist(x, K)\wedge|x-v|.
\end{cases}
\end{align}
Combining with~\eqref{eqn::Green_cvg}, we obtain~\eqref{eqn::discretePoisson_max_control} as desired. 
\end{proof}


Next, we address the scaling limit of the Poisson kernel and its discrete derivative at an interior point. 
For $u\in\LV^{\circ}(\Omega^{\delta})$, it has three adjacent vertices $u_1, u_2, u_3$. There are two possible cases for the directions of the edges $\langle u,u_1\rangle, \langle u,u_2\rangle, \langle u,u_3\rangle$ (see Figure~\ref{fig::threeneighbors}):  
\begin{align}\label{eqn::threeneighbors_case1}
\text{either}\qquad &u_k=u+\delta\bs{\alpha}_k, \qquad \text{where }\bs{\alpha}_k=\ee^{2k\pi\ii/3}, \qquad \text{for }k=1,2,3; \tag{$\vartriangleleft$}\\
\text{or}\qquad 
&u_k=u+\delta\bs{\alpha}_k, \qquad \text{where }\bs{\alpha}_k=\ee^{(2k+1)\pi\ii/3}, \qquad \text{for }k=1,2,3. \tag{$\vartriangleright$}
\label{eqn::threeneighbors_case2}
\end{align}
Define the directional-derivative of the Poisson kernel $\harmonic(\Omega; \cdot, e^{\delta})$ as 
\begin{equation}\label{eqn::Poisson_derivative_def}
\dharmonic_{\bs{\alpha}_k}(\Omega^{\delta}; u, e^{\delta})=\delta^{-1}\left(\harmonic(\Omega^{\delta}; u +\delta\bs{\alpha}_k, e^{\delta})-\harmonic(\Omega^{\delta}; u, e^{\delta})\right), \qquad \text{for }k=1,2,3. 
\end{equation}

\begin{figure}[ht!]
\begin{subfigure}[b]{0.4\textwidth}
\begin{center}
\includegraphics[width=0.4\textwidth]{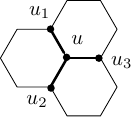}$\quad$
\includegraphics[width=0.4\textwidth]{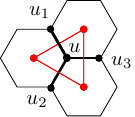}
\end{center}
\caption{Case~\eqref{eqn::threeneighbors_case1}.}
\end{subfigure}
\begin{subfigure}[b]{0.4\textwidth}
\begin{center}
\includegraphics[width=0.4\textwidth]{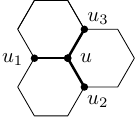}$\quad$
\includegraphics[width=0.4\textwidth]{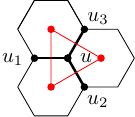}
\end{center}
\caption{Case~\eqref{eqn::threeneighbors_case2}.}
\end{subfigure}
\caption{\label{fig::threeneighbors} The directions of three adjacent edges of $u$ on hexagonal lattice have two possible cases: (a) Case~\eqref{eqn::threeneighbors_case1} and (b) Case~\eqref{eqn::threeneighbors_case2}. 
In Case~\eqref{eqn::threeneighbors_case1}, we denote by $\bigtriangleup^{\delta}(u)$ the triangle with three vertices $u+\delta\ee^{\ii\pi/3}, u-\delta, u+\delta\ee^{-\ii\pi/3}$ (red dots). In Case~\eqref{eqn::threeneighbors_case2}, we denote by $\bigtriangleup^{\delta}(u)$ the triangle with three vertices $u+\delta\ee^{2\ii\pi/3}, u+\delta\ee^{4\ii\pi/3}, u+\delta$ (red dots). Note that $\bigtriangleup^{\delta}(u)$ is the dual face of $u$ and its area is $\frac{3\sqrt{3}}{4}\delta^2$.}
\end{figure}

\begin{lemma}\label{lem::Poisson_cvg}
Suppose $(U; x)$ and $(\Omega; x)$ are bounded $1$-polygons and $U$ agrees with $\Omega$ in neighborhood of $x$. 
Suppose $(U^{\delta}; x^{\delta})$ (resp. $(\Omega^{\delta}; x^{\delta})$) is an approximation of $(U; x)$ (resp. of $(\Omega; x)$) on $\delta\hexagon$ in Carath\'eodory sense.
Suppose $e^{\delta}=\langle x^{\delta, \circ}, x^{\delta}\rangle$ is the boundary edge whose boundary end point $x^{\delta}$ is nearest to $x$. 
Fix two interior points $v\in\Omega$ and $z\in U$ and suppose $v^{\delta}$ is the vertex in $\LV^{\circ}(\Omega^{\delta})$ that is nearest to $v$.
\begin{itemize}
\item Suppose $z^{\delta}$ is the vertex in $\LV^{\circ}(\Omega^{\delta})$ that is nearest to $z$. The scaling limit of the Poisson kernel~\eqref{eqn::dPoisson_def} exists: 
\begin{align}\label{eqn::Poisson_cvg}
\lim_{\delta\to 0}\frac{\harmonic(U^{\delta}; z^{\delta}; e^{\delta})}{\harmonic(\Omega^{\delta}; v^{\delta}; e^{\delta})}=\frac{\Poisson(U; z, x)}{\Poisson(\Omega; v, x)}, 
\end{align}
where $\Poisson(U; z, x)$ is Poisson kernel~\eqref{eqn::Poisson_H}-\eqref{eqn::Poisson_cov}. 
\item 
Suppose $z^{\delta}$ is the vertex in $\LV^{\circ}(\Omega^{\delta})$ of Case~\eqref{eqn::threeneighbors_case1} that is nearest to $z$. 
The scaling limit of the directional-derivative of Poisson kernel~\eqref{eqn::Poisson_derivative_def} exists: 
\begin{align}\label{eqn::Poisson_derivative_cvg}
\lim_{\delta\to 0}\frac{\dharmonic_{\bs{\alpha}_k}(U^{\delta}; z^{\delta}, e^{\delta})}{\harmonic(\Omega^{\delta}; v^{\delta}, e^{\delta})}=\frac{\partial_{\bs{\alpha}_k}\Poisson(U; z, x)}{\Poisson(\Omega; v, x)}, \qquad\text{for }k=1,2,3,
\end{align}
    where $\partial_{\bs{\alpha}_k}\Poisson(U; z, x)$ is the directional-derivative of Poisson kernel given by
\begin{align*}
\partial_{\bs{\alpha}_k}\Poisson(U; z, x)= \lim_{\eps\to 0}\eps^{-1}\left(\Poisson(U; z+\eps \bs{\alpha}_k,x)-\Poisson(U; z,x)\right).
\end{align*}
\end{itemize}
\end{lemma}

\begin{proof}
We first prove~\eqref{eqn::Poisson_cvg}.
From~\eqref{eqn::harmonic_green}, we may write
\begin{equation}\label{eqn::Poisson_cvg_aux1}
    \frac{\harmonic(U^{\delta}; z^{\delta}; e^{\delta})}{\harmonic(\Omega^{\delta}; v^{\delta}; e^{\delta})}=\frac{\Green(U^{\delta}; z^{\delta},x^{\delta,\circ})}{\Green(\Omega^{\delta};  v^\delta,x^{\delta,\circ})}.
\end{equation}
We pick $r_0=r_0(U,\Omega)>0$ such that $B(x, 2r_0)\cap\Omega\subset U$ and $z,v\notin B(x,r_0)$. 
Both $\Green(U^\delta; z^\delta,\cdot)$ and $\Green(\Omega^{\delta}; v^{\delta},\cdot)$ are harmonic in $\LV^{\circ}(\Omega^{\delta}) \cap B(x, r_0)$ and vanish on $\LV^{\partial}(\Omega^{\delta}) \cap B(x, r_0)$. 
For any $\eps>0$, \cite[Corollary~3.8]{ChelkakWanMassiveLERW} asserts that there exists $r_\eps \in\left(0, r_0\right)$ such that
\begin{equation}\label{eqn::Poisson_cvg_aux2}
    1-\eps<\frac{\Green(U^\delta;  z^\delta,x^{\delta, \circ})}{\Green(\Omega^{\delta} ;  v^{\delta},x^{\delta, \circ})} \cdot \frac{\Green(\Omega^{\delta} ;  v^{\delta},u)}{\Green(U^\delta ;z^\delta,u)}<1+\eps, \quad \text { for all } u \in \LV^{\circ}(\Omega^{\delta}) \cap B(x, r_\epsilon) .
\end{equation}
Pick $u \in \LV^{\circ}(\Omega^{\delta}) \cap B(x, r_\epsilon)$ and $s_\eps<\frac{r_\eps}{4}$ such that $B(u, s_\eps) \subset B(x, r_\eps)$. 
Then by~\eqref{eqn::Green_cvg} and~\eqref{eqn::dPoisson_max_control_aux3}, we can pick $s_\eps>0$ small enough and there exists $\delta_\epsilon>0$  such that the following holds: for every $\delta <\delta_\epsilon$, there exists $\hat{u} \in B(u, s_\epsilon)$ such that
\begin{equation}\label{eqn::Poisson_cvg_aux3}
    \left|\frac{\Green(U^\delta ; z^\delta,u)}{\Green(\Omega^{\delta} ;v^{\delta},u)}-\frac{\Green(U;  z,\hat{u})}{\Green(\Omega ;  v,\hat{u})}\right|<\eps.
\end{equation}
Since $\Green(U; z,\cdot) / \Green(\Omega ;  v,\cdot)$ is continuous in $\overline{\Omega} \cap B(x_2, r_0)$, we can choose $r_0$ small enough such that
\begin{equation}\label{eqn::Poisson_cvg_aux4}
    \left|\frac{\Green(U ;  z,\hat{u})}{\Green(\Omega ;  v,\hat{u})}-\frac{\Poisson(U ; z,x)}{\Poisson(\Omega ; v, x)}\right|<\eps, \quad \text { for all } \hat{u} \in B(x, r_0) .
\end{equation}
Combining~\eqref{eqn::Poisson_cvg_aux1}-\eqref{eqn::Poisson_cvg_aux4}, we obtain~\eqref{eqn::Poisson_cvg} as desired.
\medbreak

Next, we prove~\eqref{eqn::Poisson_derivative_cvg}.
From~\cite[Theorem~3.13]{ChelkakSmirnovDiscreteComplexAnalysis},
we have the following convergence of discrete derivative of harmonic functions: as $\delta\to 0$,
\begin{equation}
    \frac{\dharmonic_{\bs{\alpha}_k}(U^\delta;z^\delta,e^\delta)}{\harmonic(U^\delta;z^\delta,e^\delta)}\to \frac{\partial_{\bs{\alpha}_k}\Poisson(U;z,x)}{\Poisson(U;z,x)}, \qquad\text{for }k=1,2,3.
\end{equation}
Combining this with~\eqref{eqn::Poisson_cvg}, we obtain~\eqref{eqn::Poisson_derivative_cvg} as desired.
\end{proof}

\begin{lemma}\label{lem::bPoisson_cvg}
Fix a bounded $2$-polygon $(\Omega; x_1, x_2)$ and suppose $(\Omega^{\delta}; x_1^{\delta}, x_2^{\delta})$ is an approximation of $(\Omega; x_1, x_2)$ on $\delta\hexagon$ in Carath\'eodory sense. 
We fix an interior point $v\in\Omega$ and suppose $v^{\delta}$ is the vertex in $\LV^{\circ}(\Omega^{\delta})$ that is nearest to $v$. 
\begin{itemize}
\item The scaling limit of the harmonic measure~\eqref{eqn::dharmonic_def} exists:
    \begin{equation}\label{eqn::charmonic_cvg}
        \lim_{\delta\to 0}\harmonic(\Omega^\delta; v^\delta, (x_1^\delta x_2^\delta))=\frac{1}{2\pi}\harmonic(\Omega;v,(x_1x_2)),
    \end{equation}
    where $\harmonic(\Omega;v,(x_1x_2))$ is harmonic measure~\eqref{eqn::charnomic_H}-\eqref{eqn::nharmonic_cov}.
\item The scaling limit of the Poisson kernel~\eqref{eqn::dPoisson_def} at boundary points exists:
\begin{align}\label{eqn::bPoisson_cvg}
\lim_{\delta\to 0}\frac{\harmonic(\Omega^{\delta}; x_2^{\delta,\circ}, e_1^{\delta})}{\harmonic(\Omega^{\delta}; v^{\delta}, e_1^{\delta})\harmonic(\Omega^{\delta}; v^{\delta}, e_2^{\delta})}=\frac{2\sqrt{3}\pi\Poisson(\Omega; x_1, x_2)}{\Poisson(\Omega; v, x_1)\Poisson(\Omega; v, x_2)}, 
\end{align}
where $\Poisson(\Omega; v, x)$ is Poisson kernel~\eqref{eqn::Poisson_H}-\eqref{eqn::Poisson_cov} and $\Poisson(\Omega; x_1, x_2)$ is boundary Poisson kernel~\eqref{eqn::bPoisson_H}-\eqref{eqn::bPoisson_cov}. 
\end{itemize}
\end{lemma}
\begin{proof}
The convergence~\eqref{eqn::charmonic_cvg} is guaranteed by~\cite[Theorem~3.12]{ChelkakSmirnovDiscreteComplexAnalysis}. 
The convergence~\eqref{eqn::bPoisson_cvg} can be proved similarly as~\eqref{eqn::Poisson_cvg} in Lemma~\ref{lem::Poisson_cvg}. 
As the constant $2\sqrt{3}\pi$ will be important later. 
We still include the details below. 

Fix $r_0>0$ such that $v\not\in B(x_2, 2r_0)$.
By~\eqref{eqn::harmonic_green}, we may write
\begin{equation}\label{eqn::bPoisson_cvg_aux1}
\frac{\harmonic(\Omega^{\delta}; x_2^{\delta,\circ}, e_1^{\delta})}{\harmonic(\Omega^{\delta}; v^{\delta}, e_1^{\delta})\harmonic(\Omega^{\delta}; v^{\delta}, e_2^{\delta})}=
\frac{3\harmonic(\Omega^{\delta}; x_2^{\delta,\circ}, e_1^{\delta})}{\harmonic(\Omega^{\delta}; v^{\delta}, e_1^{\delta})\Green(\Omega^{\delta}; x_2^{\delta,\circ}, v^{\delta})}.
\end{equation}
Both $\harmonic(\Omega^{\delta}; \cdot, e_1^{\delta})$ and $\Green(\Omega^{\delta}; \cdot, v^{\delta})$ are harmonic in $\LV^{\circ}(\Omega^\delta)\cap B(x_2,r_0)$ and vanish on $\LV^{\partial} (\Omega^\delta)\cap B(x_2,r_0)$. For any $\eps>0$, \cite[Corollary~3.8]{ChelkakWanMassiveLERW} asserts that there exists $r_{\eps}\in (0,r_0)$ such that
\begin{equation}\label{eqn::bPoisson_cvg_aux2}
    1-\eps<\frac{\harmonic(\Omega^{\delta}; x_2^{\delta,\circ}, e_1^{\delta})}{\Green(\Omega^{\delta}; x_2^{\delta,\circ}, v^{\delta})}\cdot\frac{\Green(\Omega^{\delta}; u, v^{\delta})}{\harmonic(\Omega^{\delta}; u, e_1^{\delta})}<1+\eps, \qquad\text{for all }u \in \LV^{\circ}(\Omega^\delta) \cap B(x_2, r_\eps). 
\end{equation}
Combination of~\cite[Corollary~3.8]{ChelkakWanMassiveLERW} and~\cite[Corollary~3.11, Theorem~3.13]{ChelkakSmirnovDiscreteComplexAnalysis}  also asserts that there exists a constant $C_{\eqref{eqn::bPoisson_cvg_boundness}}\in (0,\infty)$ depending on $(\Omega;x_1,x_2;v;r_0)$ such that:
\begin{equation}\label{eqn::bPoisson_cvg_boundness}
\frac{\harmonic(\Omega^{\delta}; u, e_1^{\delta})}{\harmonic(\Omega^\delta;v^\delta;e_1^\delta)\Green(\Omega^{\delta}; u, v^{\delta})}\le C_{\eqref{eqn::bPoisson_cvg_boundness}}, \quad \text{ for all } u \in \LV^{\circ}(\Omega^{\delta}) \cap B(x_2, r_0).
\end{equation}
Pick $u \in \LV^{\circ}(\Omega^\delta) \cap B(x_2, r_\eps)$ and $s_\eps<\frac{r_\eps}{4}$ such that $B(u,s_\eps)\subset B(x_2,r_\eps)$.
Then by~\eqref{eqn::bPoisson_cvg_boundness}, {~\eqref{eqn::Poisson_cvg} and~\eqref{eqn::Green_cvg}}, there exists $\delta_{\eps}>0$ such that the following holds: 
for every $\delta<\delta_\eps$, there exists $\hat{u} \in B(u, s_\eps)$ such that
\begin{equation}\label{eqn::bPoisson_cvg_aux3}
\left|\frac{\harmonic(\Omega^{\delta}; u, e_1^{\delta})}{\harmonic(\Omega^{\delta}; v^{\delta}, e_1^{\delta})\Green(\Omega^{\delta}; u, v^{\delta})}-
\frac{\Poisson(\Omega;\hat{u},x_1)}{\Poisson(\Omega; v, x_1)\frac{\sqrt{3}}{2\pi}\Green(\Omega; \hat{u}, v)}\right|<\eps.
\end{equation}
Since $\Poisson(\Omega; \cdot, x_1)/\Green(\Omega; \cdot, v)$ is continuous in $\overline{\Omega}\cap B(x_2, r_{\eps})$, we can choose $r_\eps$ small enough such that
\begin{equation}\label{eqn::bPoisson_cvg_aux4}
    \left| \frac{\Poisson(\Omega;\hat{u},x_1)}{\Green(\Omega;\hat{u},v)}-\frac{\Poisson(\Omega;x_2,x_1)}{\Poisson(\Omega;v,x_2)} \right|< \epsilon, \quad \text { for all } \hat{u} \in B(x_2, r_\eps).
\end{equation}
Combining~\eqref{eqn::bPoisson_cvg_aux1}-\eqref{eqn::bPoisson_cvg_aux4}, we obtain~\eqref{eqn::bPoisson_cvg} as desired.
\end{proof}

\begin{lemma}\label{lem::nharmonic_cvg}
Suppose $(U; x_1, x_2, x_3)$ is a bounded $3$-polygon and $(\Omega; x_2)$ is a bounded $1$-polygon and $U$ agrees with $\Omega$ in neighborhood of $x_2$. 
Suppose $(U^{\delta}; x_1^{\delta}, x_2^{\delta}, x_3^{\delta})$ (resp. $(\Omega^{\delta}; x_2^{\delta})$) is an approximation of $(U; x_1, x_2, x_3)$ (resp. of $(\Omega; x_2)$) on $\delta\hexagon$ in Carath\'eodory sense. 
We fix an interior point $v\in\Omega$ and suppose $v^{\delta}$ is the vertex in $\LV^{\circ}(\Omega^{\delta})$ that is nearest to $v$. 
The scaling limit of harmonic measure~\eqref{eqn::dharmonic_def} at boundary points exists:
\begin{align}
\label{eqn::nharmonic_cvg}
\lim_{\delta\to 0}\frac{\harmonic(U^{\delta}; x_2^{\delta, \circ}, (x_3^{\delta}x_1^{\delta}))}{\harmonic(\Omega^{\delta}; v^{\delta}, e_2^{\delta})}=&\sqrt{3}\frac{\nharmonic(U; x_2, (x_3x_1))}{\Poisson(\Omega; v, x_2)},
\end{align}
where $\nharmonic(U; x_2, (x_3x_1))$ is harmonic measure~\eqref{eqn::nharmonic_H}-\eqref{eqn::nharmonic_cov}. 
\end{lemma}
\begin{proof}
This can be done by similar proof of~\eqref{eqn::Poisson_cvg} in Lemma~\ref{lem::Poisson_cvg} where we replace the convergence~\eqref{eqn::Green_cvg} by~\eqref{eqn::charmonic_cvg}.
\end{proof}

\subsection{Schramm--Loewner Evolution}
\label{subsec::pre_SLE}

In this section, we collect properties of radial SLE and three-sided radial SLEs. They will be crucial in the proof of Theorem~\ref{thm::tripod} in Section~\ref{sec::tripod}. 

\paragraph*{Chordal SLE.}
For a $2$-polygon $(\Omega; x_1, x_2)$, we denote by $\chamber(\Omega; x_1, x_2)$ the set of continuous simple unparameterized curves in $\Omega$ connecting $x_1$ and $x_2$ such that they only touch the boundary $\partial\Omega$ in $\{x_1, x_2\}$.
Fix $\kappa\le 4$. Chordal $\SLE_{\kappa}$ in $(\Omega; x_1, x_2)$ is a probability measure on $\chamber(\Omega; x_1, x_2)$ that satisfies conformal invariance and domain Markov property. Its definition is usually given in the upper-half plane via chordal Loewner chain. As we mainly focus on the radial setting in this article, we do not plan to introduce notations for chordal Loewner chain. Readers may look at~\cite{WernerRandomPlanarcurves}. We denote by $\chordalSLE(\Omega; x_1, x_2)$ the law of chordal $\SLE_{\kappa}$ in $(\Omega; x_1, x_2)$. Its partition function is given by 
\begin{align*}
\chordalSLEpf(\Omega; x_1, x_2)=\Poisson(\Omega; x_1, x_2)^{\frac{(6-\kappa)}{2\kappa}}, 
\end{align*}
where $\Poisson(\Omega; x_1, x_2)$ is boundary Poisson kernel~\eqref{eqn::bPoisson_H}-\eqref{eqn::bPoisson_cov}.

\paragraph*{Radial Loewner chain.}
To introduce radial SLE, it is more convenient to work with the polygon $(\Omega; x_1, \ldots, x_p)=(\U; \ee^{\ii\theta_1}, \ldots, \ee^{\ii\theta_p})$. We denote 
\[\LX_p=\{(\theta_1, \ldots, \theta_p)\in\R^p: \theta_1<\theta_2<\cdots<\theta_p<\theta_1+2\pi\}.\]

Fix $\theta\in\R$ and $T \in (0,\infty)$.
Let $\gamma \colon [0,T]\to \overline{\U}$ be a continuous simple curve such that $\gamma_0=\ee^{\ii \theta}$ and $\gamma_{(0,T)} \subset \U \setminus \{0\}$. 
For each $t \in [0,T]$, let $g_t : \U\setminus \gamma_{[0,t]} \to \U$ be the unique conformal map 
normalized at the origin, i.e., 
$g_t(0)=0$ and $g'_t(0)>0$, which we call the \emph{mapping-out function of} $\gamma$.
We say that the curve $\gamma$ is \emph{parameterized by capacity} if $g'_t(0) = \ee^t$ for all $t \in [0,T]$. 
These maps $\{g_t \colon t \in [0,T] \}$ solve the \emph{radial Loewner equation}
\begin{align*}
\partial_t g_t(z)=g_t(z) \frac{\ee^{\ii \xi_t}+g_t(z)}{\ee^{\ii \xi_t}-g_t(z)}, \qquad g_0(z)=z,
\end{align*}
where $\xi \colon [0,T]\to \R$ is a continuous function called the Loewner \emph{driving function} of $\gamma$. 
It is convenient to use the \emph{covering map} 
$h_t \colon \R \to \R$ 
of $g_t$ defined via $g_t(\ee^{\ii \theta}) = \exp( \ii h_t(\theta) )$. The process of covering maps satisfies the Loewner equation
\begin{align*}
\partial_t h_t(\theta) = \cot \left(\frac{h_t(\theta) - \xi_t}{2}\right) , \qquad h_0(\theta)=\theta .
\end{align*}

\textit{Radial $\SLE_{\kappa}$}  in $(\U; \ee^{\ii\theta}; 0)$ is the radial Loewner process with driving function $\xi_t=\theta+\sqrt{\kappa} B_t$ where $B_t$ is standard one-dimensional Brownian motion.  
When $\kappa\in (0,4]$, radial $\SLE_{\kappa}$ is almost surely generated by a continuous simple curve $\gamma\in\chamber(\U; \ee^{\ii\theta}; 0)$. 
We define radial $\SLE_{\kappa}$ in a more general $1$-polygon $(\Omega; x)$ with $z\in\Omega$ as the pushforward measure of radial $\SLE_{\kappa}$ by $\varphi^{-1}$, where $\varphi: \Omega\to \U$ is the conformal map with $\varphi(z)=0$ and $\varphi(x)=\ee^{\ii\theta}$. 
We denote by $\nradP{1}(\Omega; x; z)$ the law of radial $\SLE_{\kappa}$ in $(\Omega; x; z)$. Its partition function is given by (see~\cite[Section~2.1]{HuangPeltolaWuMultiradialSLEResamplingBP})
\begin{align}\label{eqn::1rad_pf}
\nradpartfn{1}(\Omega; x; z)=|\varphi'(z)|^{\frac{(6-\kappa)(\kappa-2)}{8\kappa}}|\varphi'(x)|^{\frac{(6-\kappa)}{2\kappa}}. 
\end{align}

The following coordinate change of radial SLE will be useful in Section~\ref{sec::tripod}. 

\begin{lemma}
\label{lem::1rad_RN}
Fix $\kappa\in (0,4]$ and an $1$-polygon $(\Omega; x)$. For two interior points $z_1, z_2\in\Omega$, the law of $\gamma$ under $\nradP{1}(\Omega; x; z_1)$ is the same as $\nradP{1}(\Omega; x; z_2)$ weighted by the following Radon-Nikodym derivative:
\begin{align}\label{eqn::1rad_RN}
\frac{\ud \nradP{1}(\Omega; x; z_1)}{\ud \nradP{1}(\Omega; x; z_2)}(\gamma[0,t])=\frac{\nradpartfn{1}(\Omega\setminus\gamma[0,t]; \gamma(t); z_1)}{\nradpartfn{1}(\Omega\setminus\gamma[0,t]; \gamma(t); z_2)}\frac{\nradpartfn{1}(\Omega; x, z_2)}{\nradpartfn{1}(\Omega; x, z_1)}. 
\end{align}
\end{lemma}
\begin{proof}
This follows from coordinate change~\cite{SchrammWilsonSLECoordinatechanges}. 
Note that the boundary of $\Omega\setminus\gamma[0,t]$ is rough at the tip $\gamma(t)$ and the partition function is not well-defined at such rough point. Whereas, the ratio in RHS of~\eqref{eqn::1rad_RN} is well-defined: suppose $\Omega=\U$, then we have (see~\cite[Eq.~(2.19)]{HuangPeltolaWuMultiradialSLEResamplingBP} for general conclusion)
\begin{align*}
\frac{\nradpartfn{1}(\U\setminus\gamma[0,t]; \gamma(t); z_1)}{\nradpartfn{1}(\U\setminus\gamma[0,t]; \gamma(t); z_2)}=\left|\frac{g_t'(z_1)}{g_t'(z_2)}\right|^{\frac{(6-\kappa)(\kappa-2)}{8\kappa}}\frac{\nradpartfn{1}(\U; \ee^{\ii\xi_t}; g_t(z_1))}{\nradpartfn{1}(\U; \ee^{\ii\xi_t}; g_t(z_2))}. 
\end{align*}
\end{proof}

\paragraph*{Radial SLE with force points.}
Fix $(\theta_1, \theta_2, \theta_3)\in\LX_3$ and $(\rho_1, \rho_2)\in\R^2$. 
Radial $\SLE_{\kappa}(\rho_1, \rho_2)$ in $\U$ from $\ee^{\ii\theta_3}$ to $0$ with force points $(\ee^{\ii\theta_1}, \ee^{\ii\theta_2})$  
is defined as the radial Loewner chain whose driving function $\xi_t$ solves the SDE 
\begin{align}\label{eqn::SLEkapparhomu_SDE}
\begin{cases}
\ud \xi_t = \sqrt{\kappa} \, \ud B_t 
- \displaystyle \frac{\rho_1}{2} \cot \left(\frac{V_t^{1} - \xi_t}{2}\right) \ud t -\frac{\rho_2}{2} \cot \left(\frac{V_t^{2} - \xi_t}{2}\right) \ud t, 
\qquad \xi_0 = \theta_3; \\
\ud V_t^{j} =  \displaystyle  \cot \left(\frac{V_t^{j} - \xi_t}{2}\right) \ud t, \quad V_0^{j}=\theta_j, 
\qquad j\in \{1,2\},
\end{cases}
\end{align}
where $B_t$ is standard one-dimensional Brownian motion. 
When $\rho_1\ge 0$ and $\rho_2\ge 0$, radial $\SLE_{\kappa}(\rho_1, \rho_2)$ process is almost surely generated by a continuous curve $\gamma\in\chamber(\U; \ee^{\ii\theta_3}; 0)$, see~\cite[Proposition~A.1]{HuangPeltolaWuMultiradialSLEResamplingBP}.
As usual, we can define radial $\SLE_{\kappa}(\rho_1, \rho_2)$ processes in a more general $3$-polygon $(\Omega; x_1, x_2, x_3)$ with $z\in\Omega$
as the pushforward measure of radial $\SLE_{\kappa}(\rho_1, \rho_2)$ by $\varphi^{-1}$, 
 where $\varphi : \Omega \to \U$ is the conformal map with $\varphi(z)=0$, 
and $(\theta_1, \theta_2, \theta_3)\in\LX_3$ is defined via $\varphi(x_j) =: \ee^{\ii\theta_j}$ for each~$j$.

\paragraph*{Three-sided radial SLE.}
Fix a $3$-polygon $(\Omega; x_1, x_2, x_3)$ with $z\in\Omega$. 
Three-sided radial SLE is a probability measure on $(\gamma_1, \gamma_2, \gamma_3)\in \chamber(\Omega; x_1, x_2, x_3; z)$ in Definition~\ref{def::3sidedradialSLE}. 
We denote by $\nradP{3}(\Omega; x_1, x_2, x_3; z)$ the law of three-sided radial $\SLE_{\kappa}$ in $(\Omega; x_1, x_2, x_3; z)$. 
Its partition function is given by~\eqref{eqn::3sided_pf}. 
Note that the partition function $\nradpartfn{3}$ satisfies the following conformal covariance: for any conformal map $\varphi$ on $\Omega$, 
\begin{align*}
\nradpartfn{3}(\Omega; x_1, x_2, x_3; z)
:= |\varphi'(z)|^{\frac{(10-\kappa)(\kappa+2)}{8\kappa}}  
\prod_{j=1}^3 |\varphi'(x_j)|^{\frac{(6-\kappa)}{2\kappa}} \nradpartfn{3}(\varphi(\Omega); \varphi(x_1), \varphi(x_2), \varphi(x_2); \varphi(z)).  
\end{align*}

The following two properties of three-sided radial SLE will be useful in Section~\ref{sec::tripod}. 
\begin{lemma}\label{lem::3radvs1rad_RN}
Fix $\kappa\in (0,4]$ and a $3$-polygon $(\Omega; x_1, x_2, x_3)$ with $z\in\Omega$. 
The law of $\gamma_3$ under $\nradP{3}(\Omega; x_1, x_2, x_3; z)$ is the same as $\nradP{1}(\Omega; x_3; z)$ weighted by the following Radon-Nikodym derivative
\begin{align}\label{eqn::3radvs1rad_RN}
\frac{\ud \nradP{3}(\Omega; x_1, x_2, x_3; z)}{\ud \nradP{1}(\Omega; x_3; z)}(\gamma_3[0,t])=\frac{\nradpartfn{3}(\Omega\setminus\gamma_3[0,t]; x_1, x_2, \gamma_3(t); z)}{\nradpartfn{1}(\Omega\setminus\gamma_3[0,t]; \gamma_3(t); z)}\frac{\nradpartfn{1}(\Omega; x_3; z)}{\nradpartfn{3}(\Omega; x_1, x_2, x_3; z)}. 
\end{align}
\end{lemma}
\begin{proof}
This is a special case of~\cite[Lemma~2.11]{HuangPeltolaWuMultiradialSLEResamplingBP}. 
The boundary of $\Omega\setminus\gamma_3[0,t]$ is rough at the tip $\gamma_3(t)$, but the ratio in RHS of~\eqref{eqn::3radvs1rad_RN} is well-defined: suppose $\Omega=\U$, then we have 
\begin{align*}
\frac{\nradpartfn{3}(\U\setminus\gamma_3[0,t]; x_1, x_2, \gamma_3(t); z)}{\nradpartfn{1}(\U\setminus\gamma_3[0,t]; \gamma_3(t); z)}
=&\frac{\nradpartfn{3}(\U; g_t(x_1), g_t(x_2), \ee^{\ii\xi_t}; g_t(z))}{\nradpartfn{1}(\U; \ee^{\ii\xi_t}; g_t(z))}|g_t'(z)|^{\frac{4}{\kappa}}|g_t'(x_1)|^{\frac{(6-\kappa)}{2\kappa}}|g_t'(x_2)|^{\frac{(6-\kappa)}{2\kappa}}. 
\end{align*}
\end{proof}


\begin{lemma}\label{lem::3radvschordal}
Fix $\kappa\in (0,4]$ and a $3$-polygon $(\Omega; x_1, x_2, x_3)$ with $z\in\Omega$. 
We consider the relation between the following two measures.
\begin{itemize}
\item The marginal law of $(\gamma_1, \gamma_2)$ under three-sided radial $\SLE_{\kappa}\sim \nradP{3}(\Omega; x_1, x_2, x_3; z)$. 
\item Suppose $\ell$ is chordal $\SLE_{\kappa}\sim\chordalSLE(\Omega; x_1, x_2)$ from $x_1$ to $x_2$. Let $\gamma_1$ be $\ell$ and let $\gamma_2$ be the time-reversal of $\ell$. Consider the law of $(\gamma_1, \gamma_2)$ under $\chordalSLE(\Omega; x_1, x_2)$. 
 \end{itemize}
Suppose $U\subset\Omega$ such that $(U; x_1, x_2)$ is a $2$-polygon and it has a positive distance from $\{x_3, z\}$. 
For $j=1,2$, let $\tau_j^U$ be the first time $\gamma_j$ exits $U$. 
Let $g_U$ be the conformal map from $\Omega\setminus(\gamma_1[0,\tau_1^U]\cup\gamma_2[0,\tau_2^U])$ onto $\Omega$ such that
\[g_U(\gamma_1(\tau_1^U))=x_1, \qquad g_U(\gamma_2(\tau_2^U))=x_2,\qquad g_U(x_3)=x_3.\]
Then the law of $(\gamma_1[0,\tau_1^U], \gamma_2[0,\tau_2^U])$ under $\nradP{3}(\Omega; x_1, x_2, x_3; z)$ is the same as $\chordalSLE(\Omega; x_1, x_2)$ weighted by the following Radon-Nikodym derivative
\begin{align}\label{eqn::3radvschordal_RN}
&\frac{\ud \nradP{3}(\Omega; x_1, x_2, x_3; z)}{\ud \chordalSLE(\Omega; x_1, x_2)}(\gamma_1[0,\tau_1^U], \gamma_2[0,\tau_2^U])\notag\\
=&\one\{\gamma_1[0,\tau_1^U]\cap\gamma_2[0,\tau_2^U]=\emptyset\}|g_U'(x_3)|^{\frac{(6-\kappa)}{2\kappa}}|g_U'(z)|^{\frac{(10-\kappa)(\kappa+2)}{8\kappa}}\frac{\nradpartfn{3}(\Omega; x_1, x_2, x_3; g_U(z))}{\nradpartfn{3}(\Omega; x_1, x_2, x_3; z)}.
\end{align}
\end{lemma}
\begin{proof}
For $j=1,2$, suppose $\gamma_j$ is radial $\SLE_{\kappa}\sim\nradP{1}(\Omega; x_j; z)$ and denote by $\nradP{1}(\Omega; x_1; z)\otimes \nradP{1}(\Omega; x_2; z)$ the measure under which $\gamma_1$ and $\gamma_2$ are independent. 
\begin{itemize}
\item From coordinate change~\cite[Section~2.5]{HuangPeltolaWuMultiradialSLEResamplingBP}, we have
\begin{align*}
&\frac{\ud \nradP{3}(\Omega; x_1, x_2, x_3; z)}{\ud \nradP{1}(\Omega; x_1; z)\otimes \nradP{1}(\Omega; x_2; z)}(\gamma_1[0,\tau_1^U], \gamma_2[0,\tau_2^U])\\
=&\one\{\gamma_1[0,\tau_1^U]\cap\gamma_2[0,\tau_2^U]=\emptyset\}\exp\left(\frac{\mathfrak{c}}{2}\blm(\Omega; \gamma_1[0,\tau_1^U], \gamma_2[0,\tau_2^U])\right)\\
&\times \frac{\nradpartfn{3}(\Omega\setminus(\gamma_1[0,\tau_1^U]\cup\gamma_2[0,\tau_2^U]); \gamma_1(\tau_1^U), \gamma_2(\tau_2^U), x_3; z)}{\nradpartfn{1}(\Omega\setminus\gamma_1[0,\tau_1^U]; \gamma_1(\tau_1^U); z)\nradpartfn{1}(\Omega\setminus\gamma_2[0,\tau_2^U]; \gamma_2(\tau_2^U); z)}\frac{\nradpartfn{1}(\Omega; x_1; z)\nradpartfn{1}(\Omega; x_2; z)}{\nradpartfn{3}(\Omega; x_1, x_2, x_3; z)}, 
\end{align*}
where $\mathfrak{c}$ is central charge
\[\mathfrak{c}=\frac{(6-\kappa)(3\kappa-8)}{2\kappa},\]
and $\blm(\Omega; A, B)$ is Brownian loop measure on loops in $\Omega$ intersecting both $A$ and $B$ (see~\cite{LawlerWernerBrownianLoopsoup} for its definition and properties and~\cite{LawlerPartitionFunctionsSLE} for a survey). 
\item From coordinate change~\cite{SchrammWilsonSLECoordinatechanges}, we have 
\begin{align*}
&\frac{\ud \chordalSLE(\Omega; x_1, x_2)}{\ud \nradP{1}(\Omega; x_1; z)\otimes \nradP{1}(\Omega; x_2; z)}(\gamma_1[0,\tau_1^U], \gamma_2[0,\tau_2^U])\\
=&\one\{\gamma_1[0,\tau_1^U]\cap\gamma_2[0,\tau_2^U]=\emptyset\}\exp\left(\frac{\mathfrak{c}}{2}\blm(\Omega; \gamma_1[0,\tau_1^U], \gamma_2[0,\tau_2^U])\right)\\&\times \frac{\chordalSLEpf(\Omega\setminus(\gamma_1[0,\tau_1^U]\cup\gamma_2[0,\tau_2^U]); \gamma_1(\tau_1^U), \gamma_2(\tau_2^U))}{\nradpartfn{1}(\Omega\setminus\gamma_1[0,\tau_1^U]; \gamma_1(\tau_1^U); z)\nradpartfn{1}(\Omega\setminus\gamma_2[0,\tau_2^U]; \gamma_2(\tau_2^U); z)}\frac{\nradpartfn{1}(\Omega; x_1; z)\nradpartfn{1}(\Omega; x_2; z)}{\chordalSLEpf(\Omega; x_1, x_2)}. 
\end{align*}
\end{itemize}
Combining the two observations, we obtain
\begin{align*}
&\frac{\ud \nradP{3}(\Omega; x_1, x_2, x_3; z)}{\ud \chordalSLE(\Omega; x_1, x_2)}(\gamma_1[0,\tau_1^U], \gamma_2[0,\tau_2^U])\\
=&\one\{\gamma_1[0,\tau_1^U]\cap\gamma_2[0,\tau_2^U]=\emptyset\}\frac{\nradpartfn{3}(\Omega\setminus(\gamma_1[0,\tau_1^U]\cup\gamma_2[0,\tau_2^U]); \gamma_1(\tau_1^U), \gamma_2(\tau_2^U), x_3; z)}{\chordalSLEpf(\Omega\setminus(\gamma_1[0,\tau_1^U]\cup\gamma_2[0,\tau_2^U]); \gamma_1(\tau_1^U), \gamma_2(\tau_2^U))}\frac{\chordalSLEpf(\Omega; x_1, x_2)}{\nradpartfn{3}(\Omega; x_1, x_2, x_3; z)}. 
\end{align*}
Combining with conformal covariance of partition functions, we obtain~\eqref{eqn::3radvschordal_RN} as desired. 
\end{proof}

In this article, we focus on $\kappa=2$. In this case, we have
\begin{align}\label{eqn::13rad_kappa2}
\nradpartfn{1}^{(\kappa=2)}(\Omega; x; z)=\Poisson(\Omega; z, x), \qquad \nradpartfn{3}^{(\kappa=2)}(\Omega; x_1, x_2, x_3; z)=\LZtripod(\Omega; x_1, x_2, x_3; z), 
\end{align}
where $\Poisson(\Omega; z, x)$ is Poisson kernel~\eqref{eqn::Poisson_H}-\eqref{eqn::Poisson_cov} and $\LZtripod$ is tripod partition function~\eqref{eqn::LZtripod_def}.

\section{Scaling limit of trifurcation}
\label{sec::trifurcation}
The goal of this section is to prove Theorem~\ref{thm::trifurcation}. 
In Section~\ref{subsec::trifurcation_key}, we derive the probability $\PP[\LA_1^{\delta}\cap\LA_2^{\delta}\cap\{\trifurcation^{\delta}=z^{\delta}\}]$ in terms of discrete Poisson kernels using Fomin's formula.
In Section~\ref{subsec::observable_limit}, we show that the scaling limit of the probability $\PP[\LA_1^{\delta}\cap\LA_2^{\delta}\cap\{\trifurcation^{\delta}=z^{\delta}\}]$ is the same as tripod partition function $\LZtripod$~\eqref{eqn::LZtripod_def} up to multiplicative constant. In Section~\ref{subsec::trifurcation_proof}, we derive the scaling limit of $\PP[\LA_1^{\delta}\cap\LA_2^{\delta}]$ and complete the proof of Theorem~\ref{thm::trifurcation}. 
In Section~\ref{subsec::A1interior}, we derive the scaling limit of the probability $\PP\left[ \{z^{\delta}\rightsquigarrow e_3^{\delta}\}\cond \LA_1^{\delta}\right]$ in Proposition~\ref{prop::A1interior}. It will be the key step in the proof of Proposition~\ref{prop::chordalSLE2_Fomin}. 

\subsection{Key lemma for trifurcation}
\label{subsec::trifurcation_key}

\paragraph*{Loop-erased random walk and Wilson's algorithm.}
Let $z = (z(t))_{t=0}^n$ be a finite sequence of symbols.  
Its \emph{loop-erasure} $\operatorname{LE}(z) = (\lambda(s))_{s=0}^m$ is defined recursively by  
\begin{equation}
\lambda(0) = z(0), \qquad 
\lambda(s+1) = z(t_s + 1), \qquad
t_s = \max\{  t > t_{s-1} \mid z(t) = \lambda(s)  \}, \qquad t_{-1} = -1,
\end{equation}
where $m$ is the smallest index with $t_m = n$.  
The sequence $\lambda$ satisfies $\lambda(0) = z(0)$, $\lambda(m) = z(n)$, and is self-avoiding.
A \textit{loop-erased random walk} (LERW) is the loop-erasure of a finite random walk.  
Let $\eta$ be a simple random walk on $\LV$ started at $v$, stopped upon hitting a non-empty set $S \subset \LV$ at time $T$.  
The LERW from $v$ to $S$ is $\lambda = \mathrm{LE}(\eta_{[0,T]})$.  
We identify $\lambda$ with its vertex list, edge list, or induced subgraph. 
The following procedure of constructing a uniform spanning tree is known as Wilson's algorithm~\cite{WilsonUSTLERW}. 
Let $v_0, \dots, v_n$ enumerate the vertices of a finite connected graph.  
Set $\LT_0 = \{v_0\}$.  
For $k=1,\dots,n$, let $\LT_k$ be the union of $\LT_{k-1}$ and an independent LERW from $v_k$ to $\LT_{k-1}$.  
Then $\LT = \LT_n$ is a uniform spanning tree. 

Applying Wilson's algorithm for the UST with wired boundary condition, the boundary branch $\eta_v$ from $v \in \LV$ has the law of an LERW from $v$ to $\LV^{\partial}$.  The probability that the boundary branch $\eta_v$ from $v \in \LV$ hits the boundary through a boundary edge $e_{\out} \in \LE^{\partial}$ is the Poisson kernel:
\begin{align}\label{eqn::boundarybranch_harmonic}
    \PP\left[v \rightsquigarrow e_{\out}\right]=\harmonic(\LG; v, e_{\out}) .
\end{align}
Moreover, conditional on exiting through $e_{\out}$, the branch $\eta_v$ is the loop-erasure of a simple random walk from $v$ to $\LV^{\partial}$ conditioned to exit through $e_{\out}$.

\begin{lemma}\label{lem::Fomin}
Fix a bounded $3$-polygon $(\Omega; x_1, x_2, x_3)$ and suppose $(\Omega^{\delta}; x_1^{\delta}, x_2^{\delta}, x_3^{\delta})$ is an approximation of $(\Omega; x_1, x_2, x_3)$ on $\delta\hexagon$ in Carath\'eodory sense. 
Recall that $e_j^\delta$ is the boundary edge $e_j^{\delta}=\langle x_j^{\delta,\circ}, x_j^\delta\rangle$.
Consider UST on $\Omega^{\delta}$ with wired boundary condition. 
Define $\LA_1^{\delta}=\{x_1^{\delta, \circ}\rightsquigarrow e_3^{\delta}\}$ and $\LA_2^{\delta}=\{x_2^{\delta, \circ}\rightsquigarrow e_3^{\delta}\}$ as in~\eqref{eqn::conditionalevent} and consider the trifurcation $\trifurcation^{\delta}$.
For $u \in \LV^{\circ}(\Omega^{\delta})$, denote by $u_1, u_2, u_3$ the three adjacent vertices of $u$ arranged in counterclockwise order. 
Then we have
\begin{align}\label{eqn::fomin_main}
\PP\left[\LA_1^{\delta}\cap \LA_2^{\delta}\cap \{\trifurcation^{\delta}=u\}\right]=\det\left(\harmonic(\Omega^{\delta}; u_k, e_{j}^\delta)\right)_{k,j=1}^3, 
\end{align}
where $\harmonic(\Omega^{\delta}; u, e)$ is the Poisson kernel~\eqref{eqn::dPoisson_def}. 
\end{lemma}

The proof of Lemma~\ref{lem::Fomin} relies on Fomin's formula~\cite[Theorem~6.1]{fomin2001loop} and uses the geometry of the hexagonal lattice crucially. 
In an earlier work~\cite{BLpiecewise}, the authors derived the scaling limit of the boundary-to-boudary branches in Temperleyan tree arising from the dimer model under piecewise Temperleyan boundary conditions. In that work, Fomin's formula also plays an important role in constructing the corresponding observable.

\begin{proof}[Proof of Lemma~\ref{lem::Fomin}]
In this lemma, we view $\gamma_j^{\delta}$ as the path from $\trifurcation^{\delta}$ to $x_j^{\delta}$ exiting the domain through the boundary edge $e_j^{\delta}=\langle x_j^{\delta, \circ}, x_j^{\delta}\rangle$. For $\ell\in\{0,1,2\}$, define 
\begin{align*}
\LB_{\ell}^{\delta}=\{\gamma_j^{\delta}\text{ traverses the edge }\langle u, u_{j+\ell}\rangle\text{ for }j\in\{1,2,3\}\},
\end{align*}
where we use the convention that $u_4=u_1, u_5=u_2$. See Figure~\ref{fig::fomin}~(a). For configurations belonging to the event $\LA_1^{\delta}\cap\LA_2^{\delta}\cap\{\trifurcation^{\delta}=u\}\cap\LB_{\ell}^{\delta}$, the operation of removing two edges $\langle u, u_{1+\ell}\rangle$ and $\langle u, u_{2+\ell}\rangle$ and adding two edges $e_1^{\delta}$ and $e_2^{\delta}$ induces a bijection to configurations in the event 
\[\LC_{\ell}^{\delta}=\{u_{1+\ell}\rightsquigarrow e_1^\delta\}\cap\{u_{2+\ell}\rightsquigarrow e_2^\delta\}\cap \{u_{3+\ell}\rightsquigarrow e_3^\delta\}\cap\{\langle u, u_{3+\ell}\rangle \in\LT^{\delta}\}.\]
See Figure~\ref{fig::fomin}. Thus, 
\begin{equation}\label{eqn::Fomin_aux1}
            \PP\left[\LA_1^\delta\cap\LA_2^\delta\cap\{\trifurcation^\delta=u\}\right]=\sum_{\ell=0}^2 \PP\left[\LA_1^\delta\cap\LA_2^\delta\cap\{\trifurcation^\delta=u\} \cap \LB_{\ell}^\delta\right]
            =\sum_{\ell=0}^{2}\PP[\LC_{\ell}^\delta].
        \end{equation}
 
\begin{figure}[ht!]
\begin{subfigure}[b]{0.45\textwidth}
\begin{center}
\includegraphics[width=\textwidth]{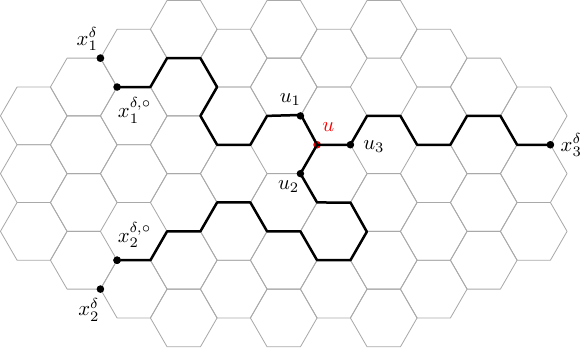}
\end{center}
\caption{}
\end{subfigure}
$\quad$
\begin{subfigure}[b]{0.45\textwidth}
\begin{center}
\includegraphics[width=\textwidth]{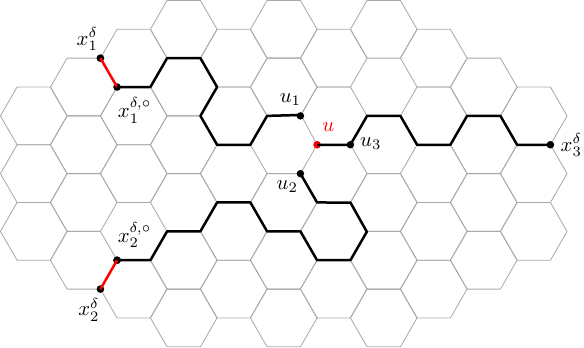}
\end{center}
\caption{}
\end{subfigure}
\caption{\label{fig::fomin} In tripod, suppose $\{u_1\rightsquigarrow e_1^{\delta}, u_2\rightsquigarrow e_2^{\delta}, u_3\rightsquigarrow e_3^{\delta}\}$ as in~(a). We delete two edges $\langle u, u_1\rangle$ and $\langle u, u_2\rangle$ and add two edges $\langle x_1^{\delta}, x_1^{\delta, \circ}\rangle$ and $\langle x_2^{\delta}, x_2^{\delta, \circ}\rangle$ (the two edges in red) as in~(b). 
Such operation induces a bijection from configurations in $\LA_1^{\delta}\cap\LA_2^{\delta}\cap\{\trifurcation^{\delta}=u\}\cap\LB_1^{\delta}$ to configurations in $\LC_1^{\delta}$.}
\end{figure}

We will evaluate $\PP[\LC_{\ell}^{\delta}]$ using Fomin's formula. Denote by $\LW(u, e)$ the set of all finite walks on the graph $\Omega^{\delta}$ from an interior vertex $u \in \LV^{\partial}(\Omega^{\delta})$ to the boundary through the boundary edge $e \in \LE^{\partial}(\Omega^{\delta})$. 
        Let $\beta_{1}^\delta, \beta_{2}^\delta, \beta_{3}^\delta$ be three independent simple random walks starting from $u_1, u_2, u_3$ respectively, and stopped upon hitting the boundary. 
\begin{itemize}
\item On the one hand, applying Wilson's algorithm, we have 
        \begin{equation}\label{eqn::fomin_aux2}
           \PP[\LC_{\ell}^\delta]=\PP\left[\beta_{j}^\delta\in \LW(u_j,e_{j-\ell}^\delta) \text{ for }j\in\{1,2,3\}, \text{ and } \beta_{j}^\delta\cap \mathrm{LE}(\beta_{i}^\delta)=\emptyset \text{ for all }i<j\right], 
        \end{equation}
        where we use the convention that $e_{0}=e_3, e_{-1}=e_2$.
        \item On the other hand, Fomin's formula~\cite[Theorem~6.1]{fomin2001loop} tells
        \begin{align}\label{eqn::fomin_aux3}
                    &\det\left(\harmonic(\Omega^{\delta}; u_k, e_{j}^\delta)\right)_{k,j=1}^3\notag\\
            =&
           \sum_{\sigma}\mathrm{sgn}(\sigma)\PP\left[\beta_{j}^\delta\in \LW(u_j,e_{\sigma(j)}^\delta) \text{ for }j\in\{1,2,3\}, \text{ and } \beta_{j}^\delta\cap \mathrm{LE}(\beta_{i}^\delta)=\emptyset \text{ for all }i<j\right], \end{align}
        where the sum ranges over all permutations $\sigma$ of order three. Since $u_1, u_2, u_3$ are three adjacent vertices of $u$ arranged in counterclockwise order, the requirement $\{\beta_{j}^\delta\cap \mathrm{LE}(\beta_{i}^\delta)=\emptyset \text{ for all }i<j\}$ tells that the only possible permutations  are $(123)$, $(231)$, or $(312)$ in RHS of~\eqref{eqn::fomin_aux3}. For these three permutations $\sigma$, we have $\mathrm{sgn}(\sigma)=1$. Thus, 
        \begin{align*}
                    &\det\left(\harmonic(\Omega^{\delta}; u_k, e_{j}^\delta)\right)_{k,j=1}^3\notag\\
            =&
           \sum_{\sigma}\mathrm{sgn}(\sigma)\PP\left[\beta_{j}^\delta\in \LW(u_j,e_{\sigma(j)}^\delta) \text{ for }j\in\{1,2,3\}, \text{ and } \beta_{j}^\delta\cap \mathrm{LE}(\beta_{i}^\delta)=\emptyset \text{ for all }i<j\right]\\
           =& \sum_{\ell=0}^2\PP\left[\beta_{j}^\delta\in \LW(u_j,e_{j-\ell}^\delta) \text{ for }j\in\{1,2,3\}, \text{ and } \beta_{j}^\delta\cap \mathrm{LE}(\beta_{i}^\delta)=\emptyset \text{ for all }i<j\right].\end{align*}
\end{itemize} 
Combining the two observations, we have 
\begin{equation}\label{eqn::fomin_aux4}
\sum_{\ell=0}^2\PP[\LC_{\ell}^{\delta}]=\det\left(\harmonic(\Omega^{\delta}; u_k, e_{j}^\delta)\right)_{k,j=1}^3.
\end{equation}
Plugging into~\eqref{eqn::Fomin_aux1}, we obtain~\eqref{eqn::fomin_main} as desired. 
\end{proof}

\subsection{Scaling limit of the observable}
\label{subsec::observable_limit}
The goal of this section is to take scaling limit of the RHS of~\eqref{eqn::fomin_main}. 
\begin{proposition}\label{prop::observable_cvg}
Fix a bounded $3$-polygon $(\Omega; x_1, x_2, x_3)$ and suppose $(\Omega^{\delta}; x_1^{\delta}, x_2^{\delta}, x_3^{\delta})$ is an approximation of $(\Omega; x_1, x_2, x_3)$ on $\delta\hexagon$ in Carath\'eodory sense. 
Consider UST on $\Omega^{\delta}$ with wired boundary condition. 
Define $\LA_1^{\delta}=\{x_1^{\delta, \circ}\rightsquigarrow e_3^{\delta}\}$ and $\LA_2^{\delta}=\{x_2^{\delta, \circ}\rightsquigarrow e_3^{\delta}\}$ as in~\eqref{eqn::conditionalevent} and consider the trifurcation $\trifurcation^{\delta}$.
We fix an interior point $v\in\Omega$ and suppose $v^{\delta}$ is the vertex in $\LV^{\circ}(\Omega^{\delta})$ that is nearest to $v$. 
Define 
\begin{equation}\label{eqn::deltafive_integrand}
f^{\delta}(u)=\frac{\PP\left[\LA_1^{\delta}\cap \LA_2^{\delta}\cap \{\trifurcation^{\delta}=u\}\right]}{\delta^2\prod_{j=1}^3 \harmonic(\Omega^\delta; v^\delta, e_j^\delta)},\qquad u\in\LV^{\circ}(\Omega^{\delta}). 
\end{equation}
We define the triangle $\bigtriangleup^{\delta}(u)$ covering vertex $u$ as in Figure~\ref{fig::threeneighbors}.  
For $u\in\LV^{\circ}(\Omega^{\delta})$, we set $f^{\delta}(z)=f^{\delta}(u)$ for $z\in\bigtriangleup^{\delta}(u)$. 
This is a step function defined on $\Omega^{\delta}$ and we set $f^{\delta}=0$ outside of $\Omega^{\delta}$. Define
\begin{equation}\label{eqn::deltafive_integrand_continuum}
f(z):=\frac{{6\sqrt{6}}\LZtripod(\Omega; x_1, x_2, x_3; z)}{\prod_{j=1}^{3}\Poisson(\Omega; v, x_j)}, \qquad z\in\Omega, 
\end{equation}
and set $f=0$ outside of $\Omega$. 
Then $f^{\delta}$ converges to $f$ uniformly on compact subsets of $\Omega$. 
\end{proposition}

Before we prove the local uniform convergence in Proposition~\ref{prop::observable_cvg}, we first address the pointwise convergence. 

\begin{lemma}\label{lem::deltafive_cvg}
Assume the same setup as in Proposition~\ref{prop::observable_cvg}. 
We fix two interior points $z,v\in\Omega$ and suppose $z^{\delta}$ (resp. $v^{\delta}$) is the vertex in $\LV^{\circ}(\Omega^{\delta})$ that is nearest to $z$ (resp. to $v$). 
For $z^{\delta}\in\LV^{\circ}(\Omega^{\delta})$, denote by $z_1^{\delta}, z_2^{\delta}, z_3^{\delta}$ the three adjacent vertices of $z^{\delta}$ arranged in counterclockwise order. 
We have
\begin{align}\label{eqn::det_convergence}
\lim_{\delta\to 0}\frac{\det\left(\harmonic(\Omega^{\delta}; z^\delta_k, e_{j}^\delta)\right)_{k,j=1}^3}{\delta^2\prod_{j=1}^3 \harmonic(\Omega^\delta; v^\delta, e_j^\delta)}=\frac{3\sqrt{3}}{\prod_{j=1}^{3}\Poisson(\Omega; v, x_j)}\det\left(M(\Omega; x_1, x_2, x_3; z)\right), 
\end{align}
where $M(\Omega;x_1,x_2,x_3,z)$ is the matrix given by
\begin{equation}\label{eqn::Mdef}
\begin{pmatrix}
  \partial_z\Poisson(\Omega; z, x_1) &    \partial_z \Poisson(\Omega; z, x_2)&     \partial_z \Poisson(\Omega; z, x_3) \\
-\ii\partial_{\overline{z}}\Poisson(\Omega; z, x_1) &  -\ii\partial_{\overline{z}}\Poisson(\Omega; z, x_2)&    -\ii\partial_{\overline{z}}\Poisson(\Omega; z, x_3) \\
\Poisson(\Omega; z, x_1) & \Poisson(\Omega; z, x_2) & \Poisson(\Omega; z, x_3) \\
\end{pmatrix}, 
\end{equation}
and $\Poisson(\Omega; z, x)$ is Poisson kernel~\eqref{eqn::Poisson_H}-\eqref{eqn::Poisson_cov}. 
\end{lemma}

\begin{proof}
We first prove~\eqref{eqn::det_convergence} via the subsequence of $\{z^\delta\}$ of Case~\eqref{eqn::threeneighbors_case1}.
For any subsequence of $\{z^\delta\}$ of Case~\eqref{eqn::threeneighbors_case1}, 
using~\eqref{eqn::Poisson_derivative_def}, we have 
\begin{align*}
&\delta^{-2}\det\left(\harmonic(\Omega^{\delta}; z^\delta_k, e_{j}^\delta)\right)_{k,j=1}^3\notag\\
=&\delta^{-2}\det
\begin{pmatrix}
    \harmonic(\Omega^\delta; z_1^\delta, e_1^\delta)-\harmonic(\Omega^\delta; z_3^\delta, e_1^\delta) & \harmonic(\Omega^\delta; z_1^\delta, e_2^\delta)-\harmonic(\Omega^\delta; z_3^\delta, e_2^\delta) & \harmonic(\Omega^\delta; z_1^\delta, e_3^\delta)-\harmonic(\Omega^\delta; z_3^\delta, e_3^\delta)\\
    \harmonic(\Omega^\delta; z_2^\delta, e_1^\delta)-\harmonic(\Omega^\delta; z_3^\delta, e_1^\delta) & \harmonic(\Omega^\delta; z_2^\delta, e_2^\delta)-\harmonic(\Omega^\delta; z_3^\delta, e_2^\delta) & \harmonic(\Omega^\delta; z_2^\delta, e_3^\delta)-\harmonic(\Omega^\delta; z_3^\delta, e_3^\delta)\\
    \harmonic(\Omega^\delta; z_3^\delta, e_1^\delta) & \harmonic(\Omega^\delta; z_3^\delta, e_2^\delta) & \harmonic(\Omega^\delta; z_3^\delta, e_3^\delta)\\
\end{pmatrix}\\
=&\det
\begin{pmatrix}
    \dharmonic_{\bs{\alpha}_1}(\Omega^\delta; z^\delta, e_1^\delta)-\dharmonic_{\bs{\alpha}_3}(\Omega^\delta; z^\delta, e_1^\delta) & \dharmonic_{\bs{\alpha}_1}(\Omega^\delta; z^\delta, e_2^\delta)-\dharmonic_{\bs{\alpha}_3}(\Omega^\delta; z^\delta, e_2^\delta) & \dharmonic_{\bs{\alpha}_1}(\Omega^\delta; z^\delta, e_3^\delta)-\dharmonic_{\bs{\alpha}_3}(\Omega^\delta; z^\delta, e_3^\delta)\\
    \dharmonic_{\bs{\alpha}_2}(\Omega^\delta; z^\delta, e_1^\delta)-\dharmonic_{\bs{\alpha}_3}(\Omega^\delta; z^\delta, e_1^\delta) & \dharmonic_{\bs{\alpha}_2}(\Omega^\delta; z^\delta, e_2^\delta)-\dharmonic_{\bs{\alpha}_3}(\Omega^\delta; z^\delta, e_2^\delta) & \dharmonic_{\bs{\alpha}_2}(\Omega^\delta; z^\delta, e_3^\delta)-\dharmonic_{\bs{\alpha}_3}(\Omega^\delta; z^\delta, e_3^\delta)\\
    \harmonic(\Omega^\delta; z_3^\delta, e_1^\delta) & \harmonic(\Omega^\delta; z_3^\delta, e_2^\delta) & \harmonic(\Omega^\delta; z_3^\delta, e_3^\delta)\\
\end{pmatrix}. 
    \end{align*}
Convergence~\eqref{eqn::Poisson_cvg}-\eqref{eqn::Poisson_derivative_cvg} gives
    \begin{equation*}
\lim_{\delta\to 0}\frac{\harmonic(\Omega^{\delta}; z_3^{\delta}; e_j^{\delta})}{\harmonic(\Omega^{\delta}; v^{\delta}; e_j^{\delta})}=\frac{\Poisson(\Omega; z, x_j)}{\Poisson(\Omega; v, x_j)}, \qquad \lim_{\delta\to 0}\frac{\dharmonic_{\bs{\alpha}_k}(\Omega^{\delta}; z^{\delta}, e_j^{\delta})}{\harmonic(\Omega^{\delta}; v^{\delta}, e_j^{\delta})}=\frac{\partial_{\bs{\alpha}_k}\Poisson(\Omega; z, x_j)}{\Poisson(\Omega; v, x_j)}, \quad \text{for }k,j=1,2,3.
\end{equation*}
Therefore, 
\begin{align*}
&\lim_{\delta\to 0}\frac{\det\left(\harmonic(\Omega^{\delta}; z^\delta_k, e_{j}^\delta)\right)_{k,j=1}^3}{\delta^2\prod_{j=1}^3 \harmonic(\Omega^\delta; v^\delta, e_j^\delta)}\\
=&\frac{1}{\prod_{j=1}^3\Poisson(\Omega; v, x_j)}\\
&\times \det
\begin{pmatrix}
    \partial_{\bs{\alpha}_1}\Poisson(\Omega; z, x_1)-\partial_{\bs{\alpha}_3}\Poisson(\Omega; z, x_1) & \partial_{\bs{\alpha}_1}\Poisson(\Omega; z, x_2)-\partial_{\bs{\alpha}_3}\Poisson(\Omega; z, x_2) & \partial_{\bs{\alpha}_1}\Poisson(\Omega; z, x_2)-\partial_{\bs{\alpha}_3}\Poisson(\Omega; z, x_3)\\
    \partial_{\bs{\alpha}_2}\Poisson(\Omega; z, x_1)-\partial_{\bs{\alpha}_3}\Poisson(\Omega; z, x_1) & \partial_{\bs{\alpha}_2}\Poisson(\Omega; z, x_2)-\partial_{\bs{\alpha}_3}\Poisson(\Omega; z, x_2) & \partial_{\bs{\alpha}_2}\Poisson(\Omega; z, x_3)-\partial_{\bs{\alpha}_3}\Poisson(\Omega; z, x_3)\\
    \Poisson(\Omega; z, x_1) & \Poisson(\Omega; z, x_2) & \Poisson(\Omega; z, x_3)\\
\end{pmatrix}\\
=&\ii \det\begin{pmatrix}
    \bs{\alpha}_1-\bs{\alpha}_3 & \overline{\bs{\alpha}_1}-\overline{\bs{\alpha}_3}\\
    \bs{\alpha}_2-\bs{\alpha}_3 & \overline{\bs{\alpha}_2}-\overline{\bs{\alpha}_3}
\end{pmatrix} \times
\frac{\det\left(M(\Omega; x_1, x_2, x_3; z)\right)}{\prod_{j=1}^{3}\Poisson(\Omega; v, x_j)}, 
\end{align*}
where, in the last equality, we use 
    \[\partial_{\bs{\alpha}_k}\Poisson(\Omega; z, x)=\bs{\alpha}_k\partial_z \Poisson(\Omega; z, x)+\overline{\bs{\alpha}_k}\partial_{\bar{z}}\Poisson(\Omega; z, x).\]   
    We denote by $\operatorname{Area}(\bs{\alpha}_1, \bs{\alpha}_2, \bs{\alpha}_3)$ the area of the triangle whose three vertices are $\bs{\alpha}_1, \bs{\alpha}_2, \bs{\alpha}_3$, then
\begin{equation}\label{eqn::lattice_dependence_area}
\ii \det\begin{pmatrix}
    \bs{\alpha}_1-\bs{\alpha}_3 & \overline{\bs{\alpha}_1}-\overline{\bs{\alpha}_3}\\
    \bs{\alpha}_2-\bs{\alpha}_3 & \overline{\bs{\alpha}_2}-\overline{\bs{\alpha}_3}
\end{pmatrix}=4\operatorname{Area}(\bs{\alpha}_1, \bs{\alpha}_2, \bs{\alpha}_3)=3\sqrt{3},
\end{equation} 
which gives the convergence in~\eqref{eqn::det_convergence} along any subsequence of Case~\eqref{eqn::threeneighbors_case1}. 

For the case \eqref{eqn::threeneighbors_case2}, we can use the same argument as above, and we still have $\operatorname{Area}(\bs{\alpha}_1, \bs{\alpha}_2, \bs{\alpha}_3)=\frac{3\sqrt{3}}{4}$. Thus we obtain the convergence of any subsequence to the same limit  in~\eqref{eqn::det_convergence}.
\end{proof}

\begin{lemma}\label{lem::detM_LZtripod}
The determinant of the matrix~\eqref{eqn::Mdef} equals $\LZtripod$~\eqref{eqn::LZtripod_def} up to multiplicative constant: 
\begin{equation}\label{eqn::detM_LZtripod}
\det(M(\Omega; x_1, x_2, x_3; z))={2\sqrt{2}}\LZtripod(\Omega; x_1, x_2, x_3; z). 
\end{equation}
\end{lemma}
\begin{proof}
The determinant $\det(M(\Omega; x_1, x_2, x_3; z))$ satisfies the conformal covariance~\eqref{eqn::LZtripod_cov} due to~\eqref{eqn::Poisson_cov}. It remains to show that the two sides of~\eqref{eqn::detM_LZtripod} are equal when $\Omega=\HH$. Note that, for $\Omega=\HH$, 
    \begin{align*}
    \Poisson(\HH; z, x)=\frac{2\Im(z)}{|z-x|^2}=\ii\left(\frac{1}{z-x}-\frac{1}{\overline{z}-x}\right), \qquad\text{for }x\in\R, z\in \HH.
    \end{align*}
    Thus,
    \begin{align}\label{eqn::Poisson_derivative}
        \partial_z\Poisson(\HH; z, x)=\frac{-\ii}{(z-x)^2},\qquad \partial_{\overline{z}}\Poisson(\HH; z,x)=\frac{\ii}{(\overline{z}-x)^2}.
    \end{align}
    With help of Mathematica, we have
    \begin{align}\label{eqn::M_H}
        \det(M(\HH;x_1,x_2,x_3,z))
        =&(-\ii)\det       
         \begin{pmatrix}
            (z-x_1)^{-2}   &(z-x_2)^{-2}   & (z-x_3)^{-2}   \\
            (\overline{z}-x_1)^{-2} &(\overline{z}-x_2)^{-2} & (\overline{z}-x_3)^{-2}\\
            \frac{2\Im(z)}{|z-x_1|^2}&\frac{2\Im(z)}{|z-x_2|^2}&\frac{2\Im(z)}{|z-x_3|^2}\\
        \end{pmatrix}\notag\\
        =&16\frac{(x_2-x_1)(x_3-x_2)(x_3-x_1)\Im(z)^4}{|(z-x_1)(z-x_2)(z-x_3)|^4}.
\end{align}
This gives~\eqref{eqn::detM_LZtripod} for $\Omega=\HH$ and completes the proof. 
\end{proof}

\begin{proof}[Proof of Proposition~\ref{prop::observable_cvg}]
The function $f^{\delta}$ converges to $f$ pointwise due to Lemmas~\ref{lem::Fomin} and~\ref{lem::deltafive_cvg} and~\ref{lem::detM_LZtripod}. 
In fact, \cite[Theorem~3.13]{ChelkakSmirnovDiscreteComplexAnalysis} tells uniform convergence. 
Fix a compact subset $K\subset\Omega$.  
Let $\tilde{K}^{\delta}$ be the union of all triangles $\bigtriangleup^{\delta}(u)$ with $u\in\LV^{\circ}(\Omega^{\delta})$ and $\bigtriangleup^{\delta}(u)\cap K\neq\emptyset$. 
\cite[Theorem~3.13]{ChelkakSmirnovDiscreteComplexAnalysis} tells that there exists $\eps(\delta, K, \Omega)\to 0$ as $\delta\to 0$ such that 
\[|f^{\delta}(u)-f(u)|\le \eps(\delta, K, \Omega), \qquad\text{for }u\in \tilde{K}^{\delta}\cap\LV^{\circ}(\Omega^{\delta}).\]
For any $z\in K$, suppose $u\in \tilde{K}^{\delta}\cap\LV^{\circ}(\Omega^{\delta})$ such that $z\in\bigtriangleup^{\delta}(u)$, then 
\[|f^{\delta}(z)-f(z)|= |f^{\delta}(u)-f(z)|\le \eps(\delta, K, \Omega)+|f(u)-f(z)|.\]
Consequently, 
$f^{\delta}$ converges to $f$ uniformly on $K$. 
\end{proof}

\subsection{Tightness of the renormalized probability}\label{subsec::proba_tight}
In this section, we derive three technical lemmas related to the tightness of the renormalization of $\PP[\LA_1^{\delta}\cap\LA_2^{\delta}]$. 
We emphasize that the conclusions and proofs for these three lemmas also hold for $\Z^2$ lattice.

   \begin{lemma}\label{lem::proba_tight2}
   Assume the same setup as in Theorem~\ref{thm::trifurcation}. We fix an interior point $v\in\Omega$ and suppose $v^{\delta}$ is the vertex in $\LV^{\circ}(\Omega^{\delta})$ that is nearest to $v$. For $r>0$, denote $V(r)=B(x_1, r)\cup B(x_2, r)\cup B(x_3, r)$.  There exists a universal constant $\alpha\in (0,\infty)$ (from weak Beurling-type estimate~\cite[Proposition 2.11]{ChelkakSmirnovDiscreteComplexAnalysis})
such that 
\begin{align}\label{eqn::deltacube_control_boundary1}
&\frac{\PP[\LA_1^{\delta}\cap \LA_2^{\delta}\cap\{\trifurcation^{\delta}\in V(r)\}]}{\prod_{j=1}^3\harmonic(\Omega^{\delta}; v^{\delta}, e_j^{\delta})}\le C_{\eqref{eqn::deltacube_control_boundary1}} r^{\alpha}, \qquad \text{for all }\delta>0. 
\end{align}
\end{lemma}

\begin{proof}
    Fix $r_0>0$ such that $B(x_j,2r_0)\cap B(x_k,2r_0)=\emptyset$ for $j\neq k$ and let $v\in \Omega\setminus(\cup_{j=1}^3 B(x_j,r_0))$ and suppose $0<r\le r_0$. 
    We first show that there exist a universal constant $\alpha\in(0,\infty)$ (from weak Beurling-type estimate~\cite[Proposition 2.11]{ChelkakSmirnovDiscreteComplexAnalysis}) and a constant $C_{\eqref{eqn::deltacube_control_boundary1_aux1}}\in (0,\infty)$ depending on $(\Omega; x_1, x_2, x_3; v; r_0)$ such that 
   \begin{equation}\label{eqn::deltacube_control_boundary1_aux1}
        \frac{\PP\left[\LA_1^\delta\cap \LA_2^\delta\cap\{\trifurcation^\delta \in B(x_1,r)\}\right]}{\prod_{j=1}^3\harmonic(\Omega^\delta; v^\delta,e_j^\delta)}\le C_{\eqref{eqn::deltacube_control_boundary1_aux1}}  r^\alpha,\qquad \text{for all }\delta>0, r\in (0,r_0].
    \end{equation}
For $j=1,2,3$, let $V_j^\delta(r)$ be the maximal domain contained in $B(x_j,r)\cap \Omega^\delta$ such that $\partial V_j^\delta(r)$ consists of the edges of $\delta \hexagon$. We set $S_j^{\delta}(r)=\LV^{\partial}(V_j^{\delta}(r))\setminus \LV^{\partial}(\Omega^{\delta})$.
    Recall that $\eta_1^\delta$ is the boundary branch from $x_1^{\delta,\circ}$. 
Denote by $\Omega_1^{\delta}$ the domain $\Omega^{\delta}\setminus\eta_1^{\delta}$. Note that its set of boundary vertices $\LV^{\partial}(\Omega_1^{\delta})$ consists of vertices in $\LV^{\partial}(\Omega^{\delta})$ and vertices in $\eta_1^{\delta}$.
    Given $(\LA_1^{\delta}, \eta_1^{\delta})$, the event $\{\LA_1^\delta\cap \LA_2^\delta\cap\{\trifurcation^\delta \in B(x_1,r)\}\}$ implies that the simple random walk from $x_2^{\delta, \circ}$ hits $\LV^{\partial}(\Omega_1^{\delta})$ through $\eta_1^{\delta}\cap V_1^{\delta}(r)$. 
    Using the notation~\eqref{eqn::dharmonic_def}, the probability of this event is bounded by    
     $\harmonic(\Omega_1^{\delta}; x_2^{\delta, \circ}, \eta_1^{\delta}\cap V_1^{\delta}(r))$. Thus, 
\begin{equation}\label{eqn::deltacube_control_boundary1_aux2}
        \PP\left[\LA_1^\delta\cap\LA_2^\delta\cap\{\trifurcation^\delta \in B(x_1,r)\}\right]\leq
        \E\left[ \harmonic(\Omega_1^\delta; x_2^{\delta,\circ}, \eta_1^\delta\cap V_1^{\delta}(r))\one\{\LA_1^\delta\}\right].
    \end{equation}

Let us evaluate $\harmonic(\Omega_1^\delta; x_2^{\delta,\circ}, \eta_1^{\delta}\cap V_1^{\delta}(r))$. 
Consider a simple random walk starting from $x_2^{\delta,\circ}$. In order to hit $\LV^{\partial}(\Omega_1^{\delta})$ through $\eta_1^{\delta}\cap V_1^{\delta}(r)$, it has to first exit $V_2^{\delta}(r_0)\cap\Omega_1^{\delta}$ through $S_2^{\delta}(r_0)\cap\Omega_1^{\delta}$ and then hit $\LV^{\partial}(\Omega_1^{\delta})$ through $\eta_1^{\delta}\cap V_1^{\delta}(r)$, see Figure~\ref{fig::deltacube_boundary}~(a). The probability of this event is bounded by 
\[\harmonic(V^\delta_2(r_0)\cap \Omega_1^{\delta};x_2^{\delta,\circ}, S^\delta_2(r_0)\cap \Omega_1^{\delta}) \le \harmonic(V^\delta_2(r_0);x_2^{\delta,\circ}, S^\delta_2(r_0)).\]
By the Markov property of random walk, given $(\LA_1^{\delta},\eta_1^{\delta})$,  we have 
\begin{align}\label{eqn::deltacube_control_boundary1_aux3}
        \harmonic(\Omega_1^\delta; x_2^{\delta,\circ}, \eta_1^{\delta}\cap V_1^{\delta}(r))        \leq & \harmonic(V^\delta_2(r_0);x_2^{\delta,\circ}, S^\delta_2(r_0)) 
        \times\max_{w\in S_2^\delta(r_0)}\harmonic(\Omega_1^\delta;w, \eta_1^{\delta}\cap V_1^{\delta}(r)).
    \end{align}
From weak Beurling-type estimate~\cite[Proposition 2.11]{ChelkakSmirnovDiscreteComplexAnalysis}, there exists a universal constant $\alpha\in (0,\infty)$ such that\footnote{We write $F\lesssim G$ if $F/G$ is bounded by a universal constant.}
    \begin{equation*}
\harmonic(\Omega_1^\delta;w, \eta_1^{\delta}\cap V_1^{\delta}(r))\lesssim  \left(\frac{\mathrm{diam}(\eta_1^{\delta}\cap V_1^{\delta}(r))}{\mathrm{dist}(w,\eta_1^{\delta}\cap V_1^{\delta}(r))}\right)^\alpha \lesssim \left(\frac{r}{r_0}\right)^\alpha, \qquad \text{for all }w\in S_2^\delta(r_0).
    \end{equation*} 
 Plugging it into~\eqref{eqn::deltacube_control_boundary1_aux3} and~\eqref{eqn::deltacube_control_boundary1_aux2}, we have 
\begin{align}\label{eqn::deltacube_control_boundary1_aux4}
\PP\left[\LA_1^\delta\cap\LA_2^\delta\cap\{\trifurcation^\delta \in B(x_1,r)\}\right]\lesssim \left(\frac{r}{r_0}\right)^\alpha\harmonic(V^\delta_2(r_0);x_2^{\delta,\circ}, S^\delta_2(r_0)) \PP[\LA_1^{\delta}]. 
\end{align} 
From~\eqref{eqn::nharmonic_cvg} and~\eqref{eqn::bPoisson_cvg}, the following ratios are bounded by a finite constant depending on $(\Omega; x_1, x_2, x_3; v; r_0)$: 
    \begin{align*}
    \frac{\harmonic(V^\delta_2(r_0);x_2^{\delta,\circ}, S^\delta_2(r_0))}{\harmonic(\Omega^{\delta}; v^{\delta}, e_2^{\delta})}, \qquad
            \frac{\PP[\LA_1^\delta]}{\harmonic(\Omega^\delta; v^\delta, e_1^\delta)\harmonic(\Omega^\delta; v^\delta, e_3^\delta)}. 
    \end{align*}
    Plugging into~\eqref{eqn::deltacube_control_boundary1_aux4}, we obtain~\eqref{eqn::deltacube_control_boundary1_aux1} as desired. 
\medbreak
   
   By symmetry, there exists a constant $C_{\eqref{eqn::deltacube_control_boundary1_aux9}}\in (0,\infty)$ depending on $(\Omega; x_1, x_2, x_3; v; r_0)$ such that
   \begin{equation}\label{eqn::deltacube_control_boundary1_aux9}
        \frac{\PP\left[\LA_1^\delta\cap \LA_2^\delta\cap\{\trifurcation^\delta \in B(x_2,r)\}\right]}{\prod_{j=1}^3\harmonic(\Omega^\delta; v^\delta,e_j^\delta)}\leq C_{\eqref{eqn::deltacube_control_boundary1_aux9}} r^\alpha.
    \end{equation}
    It remains to control $\PP\left[\LA_1^\delta\cap \LA_2^\delta\cap\{\trifurcation^\delta \in B(x_3,r)\}\right]$. By erasing the edge $\langle x_3^{\delta, \circ}, x_3^{\delta}\rangle$ and adding the edge $\langle x_1^{\delta,\circ}, x_1^{\delta}\rangle$, we switch the roles of $x_1$ and $x_3$. Thus, the above analysis also gives
    \begin{equation}\label{eqn::deltacube_control_boundary1_aux10}
        \frac{\PP\left[\LA_1^\delta\cap \LA_2^\delta\cap\{\trifurcation^\delta \in B(x_3,r)\}\right]}{\prod_{j=1}^3\harmonic(\Omega^\delta; v^\delta,e_j^\delta)}\leq C_{\eqref{eqn::deltacube_control_boundary1_aux10}} r^\alpha.
    \end{equation}
   Combining~\eqref{eqn::deltacube_control_boundary1_aux1}, \eqref{eqn::deltacube_control_boundary1_aux9} and~\eqref{eqn::deltacube_control_boundary1_aux10}, we obtain~\eqref{eqn::deltacube_control_boundary1} with $C_{\eqref{eqn::deltacube_control_boundary1}}=C_{\eqref{eqn::deltacube_control_boundary1_aux1}}+C_{\eqref{eqn::deltacube_control_boundary1_aux9}}+C_{\eqref{eqn::deltacube_control_boundary1_aux10}}$ as desired. 
\end{proof}

\begin{figure}[ht!]
\begin{subfigure}[b]{0.45\textwidth}
\begin{center}
\includegraphics[width=\textwidth]{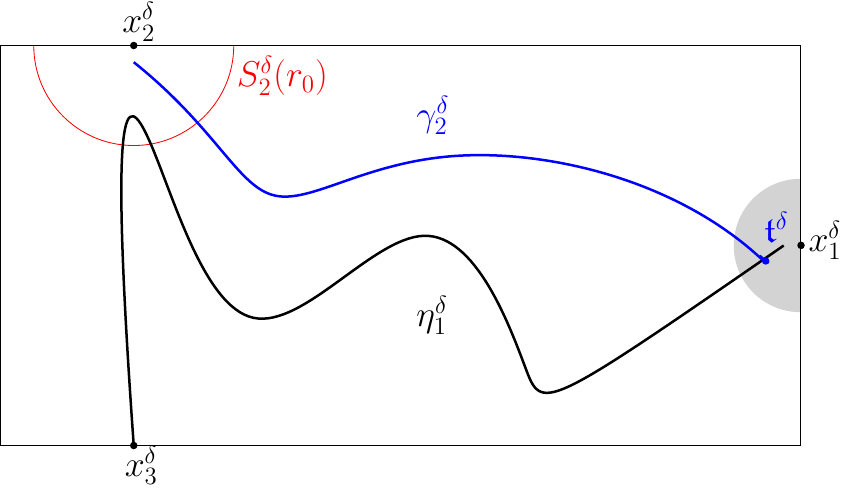}
\end{center}
\caption{The trifurcation $\trifurcation^{\delta}$ is inside $V_1^{\delta}(r)$ (the gray re-\\gion).}
\end{subfigure}
$\quad$
\begin{subfigure}[b]{0.45\textwidth}
\begin{center}
\includegraphics[width=\textwidth]{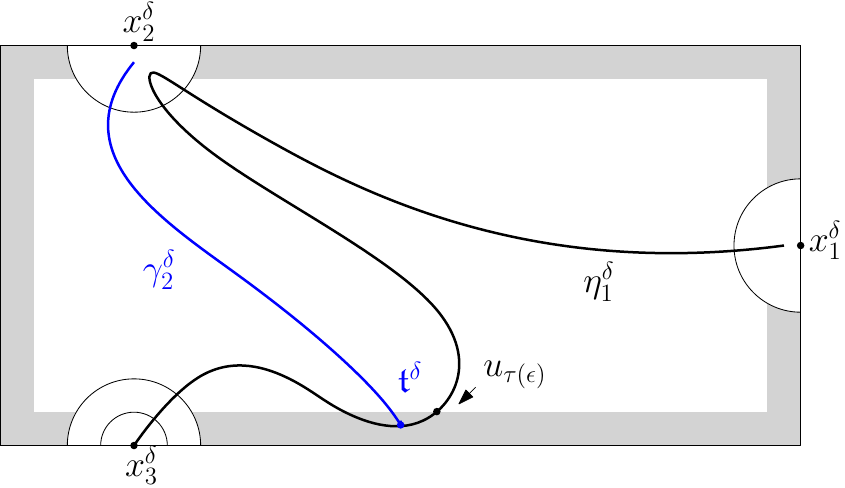}
\end{center}
\caption{The trifurcation $\trifurcation^{\delta}$ is inside $\Omega^{\delta}(\eps)\setminus V(r)$ (the gray region).}
\end{subfigure}
\caption{\label{fig::deltacube_boundary} The fat black curve indicates the boundary branch $\eta_1^{\delta}$ from $x_1^{\delta,\circ}$ to $e_3^{\delta}$. The fat blue curve indicates the boundary branch $\gamma_2^{\delta}$ from $x_2^{\delta, \circ}$ to $\eta_1^{\delta}$. }
\end{figure}

\begin{lemma}\label{lem::proba_tight3}
Assume the same setup as in Theorem~\ref{thm::trifurcation}. We fix an interior point $v\in\Omega$ and suppose $v^{\delta}$ is the vertex in $\LV^{\circ}(\Omega^{\delta})$ that is nearest to $v$. For $r>0$, denote $V(r)=B(x_1, r)\cup B(x_2, r)\cup B(x_3, r)$. 
For $\eps>0$, define $K(\eps)=\{z\in\Omega: \dist(z,\partial\Omega)\ge \eps\}$ and $\Omega^{\delta}(\eps)=\Omega^{\delta}\setminus K(\eps)$. 
There exists a universal constant $\alpha\in (0,\infty)$ (from weak Beurling-type estimate~\cite[Proposition 2.11]{ChelkakSmirnovDiscreteComplexAnalysis}), 
 a constant $C_{\eqref{eqn::deltacube_control_boundary1}}\in (0,\infty)$ depending on $(\Omega; x_1, x_2, x_3; v)$ and a constant $C_{\eqref{eqn::deltacube_control_boundary2}}(r)\in (0,\infty)$ depending on $(\Omega; x_1, x_2, x_3; v; r)$
and a constant $\delta_0(\eps)>0$ depending on $(\{\Omega^{\delta}\}_{\delta>0};\Omega; x_1, x_2, x_3; \eps)$
\begin{align}
    \label{eqn::deltacube_control_boundary2}
&\frac{\PP\left[\LA_1^\delta\cap\LA_2^\delta\cap \{\trifurcation^\delta\in \Omega^{\delta}(\eps)\setminus V(r)\}\right]}{\prod_{j=1}^{3}\harmonic(\Omega^\delta;v^\delta,e_j^\delta)}\le C_{\eqref{eqn::deltacube_control_boundary2}}(r) \left(\frac{\eps}{r}\right)^{\alpha},\quad \text{for all }\delta\le\delta_0(\eps).
\end{align}
\end{lemma}

\begin{proof}
    We use the same notation as in Proof of Lemma~\ref{lem::proba_tight2} and suppose $0<\eps<r$. 
    For $j=1,2,3$, let $V_j^\delta(r)$ be the maximal domain contained in $B(x_j,r)\cap \Omega^\delta$ such that $\partial V_j^\delta(r)$ consists of the edges of $\delta \hexagon$, and we set $S_j^{\delta}(r)=\LV^{\partial}(V_j^{\delta}(r))\setminus \LV^{\partial}(\Omega^{\delta})$ and denote by $\Omega_1^{\delta}$ the domain $\Omega^{\delta}\setminus \eta_1^{\delta}$ as before.
    Recall that $\eta_1^\delta=(u_0, u_1,\dots, u_L)$ is the boundary branch from $x_1^{\delta,\circ}$ to $e_3^{\delta}$. 
    Then $\LA_1^\delta\cap\LA_2^\delta\cap \{\trifurcation^\delta\in \Omega^{\delta}(\eps)\setminus V(r)\}$ implies the following three events: 
    \begin{itemize}
        \item The boundary branch $\eta_1^\delta$ first exits $V_1^\delta(r)$ through $S_1^\delta(r)$ and then hits $\Omega^{\delta}(\eps)\setminus V(r)$ before reaching $e_3^{\delta}$. The probability for this event is bounded by 
        \[\harmonic(V_1^{\delta}(r); x_1^{\delta, \circ}, S_1^{\delta}(r)).\]
         We denote by $\tau(\eps)$ the first time that $\eta_1^\delta$ hits $\Omega^{\delta}(\eps)\setminus V(r)$.
         \item The boundary branch from $u_{\tau(\eps)}$ in $\Omega^\delta(\eps)$ has to hit the boundary through the edge $e_3^\delta$.
         The probability of this event is bounded by 
         \[\max_{w\in \LV(\Omega^{\delta})\cap (\Omega^{\delta}(\eps)\setminus V(r))}\harmonic(\Omega^{\delta}; w, e_3^{\delta}).\]
         \item Given $\eta_1^{\delta}$, the boundary branch from $x_2^{\delta,\circ}$ in $\Omega_1^\delta$ has to exit $V_2^\delta(r)\cap\Omega_1^{\delta}$ through $S_2^\delta(r)\cap\Omega_1^{\delta}$. The probability for this event is bounded by 
         \[\harmonic(V_2^{\delta}(r)\cap\Omega_1^{\delta}; x_2^{\delta,\circ}, S_2^{\delta}(r)\cap\Omega_1^{\delta})\le \harmonic(V_2^{\delta}(r); x_2^{\delta, \circ}, S_2^{\delta}(r)).\]
    \end{itemize}
See Figure~\ref{fig::deltacube_boundary}~(b).    Combining the three events, by the Markov property of random walk, we have
    \begin{align}\label{eqn::deltacube_boundary2_aux1}
        \PP\left[\LA_1^\delta\cap\LA_2^\delta\cap \{\trifurcation^\delta\in \Omega^{\delta}(\eps)\setminus V(r)\}\right]\le &\prod_{j=1}^2\harmonic(V_j^\delta(r); x_j^{\delta,\circ},S_j^\delta(r))\times\max_{w\in \LV(\Omega^{\delta})\cap(\Omega^{\delta}(\eps)\setminus V(r))}\harmonic(\Omega^\delta; w, e_3^\delta).
    \end{align}
    Let us evaluate the three terms in RHS of~\eqref{eqn::deltacube_boundary2_aux1}. 
    \begin{itemize}
    \item 
    From~\eqref{eqn::nharmonic_cvg},
    the following ratios are bounded by finite constant depending on $(\Omega; x_1, x_2, x_3; v; r)$: 
        \begin{equation*}
                \frac{\harmonic(V_1^\delta(r); x_1^{\delta,\circ},S_1^\delta(r))}{\harmonic(\Omega^\delta; v^\delta, e_1^\delta)}, \qquad \frac{\harmonic(V_2^\delta(r); x_2^{\delta,\circ},S_2^\delta(r))}{\harmonic(\Omega^\delta; v^\delta, e_2^\delta)}.
        \end{equation*}
        \item 
        To evaluate $\harmonic(\Omega^\delta; w, e_3^\delta)$ with $w\in \Omega^{\delta}(\eps)\setminus V(r)$, consider a random walk starting from $w$. In order to hit the boundary through $e_3^{\delta}$, it has to first hit $V_3^{\delta}(r/2)$ through $S_3^{\delta}(r/2)$ and then hit the boundary through $e_3^{\delta}$. By the Markov property of the random walk, we have 
        \begin{equation}\label{eqn::deltacube_control_boundary2_aux1}
        \harmonic(\Omega^\delta; w, e_3^\delta)\le \harmonic(\Omega^\delta\setminus V_3^\delta(r/2); w, S^\delta_3(r/2))\times \max_{u\in S^\delta_3(r/2)}\harmonic(\Omega^\delta; u, e_3^\delta).
     \end{equation}
    From weak Beurling-type estimate~\cite[Proposition 2.11]{ChelkakSmirnovDiscreteComplexAnalysis}, there exists a universal constant $\alpha\in (0,\infty)$ such that 
        for all $w\in \LV(\Omega^{\delta})\cap (\Omega^{\delta}(\eps)\setminus V(r))$,
        \begin{equation*}
         \harmonic(\Omega^\delta\setminus V_3^\delta(r/2); w, S^\delta_3(r/2))\lesssim\left(\frac{\dist(w,\partial\Omega^\delta)}{\dist(w, S^\delta_3(r/2))}\right)^\alpha. 
\end{equation*}                
      We claim that there exists $\delta_0(\eps)>0$ depending on $(\{\Omega^{\delta}\}_{\delta>0}; \Omega; x_1, x_2, x_3; \eps)$ such that 
\begin{equation}\label{eqn::subtle_boundary_control}
\dist(w, \partial\Omega^{\delta})\le 2\eps, \qquad \text{for all }w\in \LV(\Omega^{\delta})\cap(\Omega^{\delta}(\eps)\setminus V(r))\text{ and }\delta\le\delta_0(\eps). 
\end{equation}
As $w\in \Omega^{\delta}(\eps)$, we have $\dist(w, \partial\Omega)\le \eps+\dist_H(\partial\Omega^{\delta}, \partial\Omega)$. Thus $\dist(w, \partial\Omega^{\delta})\le \eps+2\dist_H(\partial\Omega^{\delta}, \partial\Omega)$. This confirms~\eqref{eqn::subtle_boundary_control} due to~\eqref{eqn::boundary_cvg_Hausdorff}. 
Then 
         \begin{equation}\label{eqn::deltacube_control_boundary2_aux2}
          \harmonic(\Omega^\delta\setminus V_3^\delta(r/2); w, S^\delta_3(r/2))\lesssim\left(\frac{\dist(w,\partial\Omega^\delta)}{\dist(w, S^\delta_3(r/2))}\right)^\alpha \lesssim \left(\frac{\eps}{r}\right)^\alpha.
        \end{equation}
        From~\eqref{eqn::discretePoisson_max_control}, there exists a constant $C_{\eqref{eqn::deltacube_control_boundary2_aux3}}(r)\in (0,\infty)$ depending on $(\Omega; x_3; r)$ such that
        \begin{align}\label{eqn::deltacube_control_boundary2_aux3}
            \frac{\max_{u\in S_3^\delta(r/2)}\harmonic(\Omega^\delta; u,e_3^\delta)}{\harmonic(\Omega^\delta; v^\delta, e_3^\delta)}
            \le & C_{\eqref{eqn::deltacube_control_boundary2_aux3}}(r).
        \end{align}
        Plugging~\eqref{eqn::deltacube_control_boundary2_aux2} and~\eqref{eqn::deltacube_control_boundary2_aux3} into~\eqref{eqn::deltacube_control_boundary2_aux1}, there exists a constant $C_{\eqref{eqn::deltacube_control_boundary2_aux4}}(r)\in (0,\infty)$ depending on $(\Omega; x_1, x_2, x_3, v; r)$ such that 
      \begin{equation}\label{eqn::deltacube_control_boundary2_aux4}
            \max_{w\in\LV(\Omega^{\delta}\cap(\Omega^\delta(\eps)\setminus V(r)))}\frac{\harmonic(\Omega^\delta; w, e_3^\delta)}{\harmonic(\Omega^\delta; v^\delta, e_3^\delta)}\le C_{\eqref{eqn::deltacube_control_boundary2_aux4}}(r)\left(\frac{\eps}{r}\right)^\alpha,\qquad\text{for all }\delta\le\delta_0(\eps).
        \end{equation}
    \end{itemize}
     Combining the two observations above, we obtain~\eqref{eqn::deltacube_control_boundary2} as desired.
\end{proof}

\begin{lemma}\label{lem::proba_tight1}
Assume the same setup as in Theorem~\ref{thm::trifurcation}. We fix an interior point $v\in\Omega$ and suppose $v^{\delta}$ is the vertex in $\LV^{\circ}(\Omega^{\delta})$ that is nearest to $v$. The following ratio is uniformly bounded by finite constant depending on $(\Omega;x_1,x_2,x_3;v)$:
\begin{align}\label{eqn::deltacube_control_bulk}
   \frac{\PP[\LA_1^{\delta}\cap \LA_2^{\delta}]}{\prod_{j=1}^3\harmonic(\Omega^{\delta}; v^{\delta}, e_j^{\delta})}. 
   \end{align}
   \end{lemma}
   \begin{proof}
We use the same notation and setup as in Proof of Lemma~\ref{lem::proba_tight2}. 
Lemma~\ref{lem::proba_tight2} asserts the boundness of the probability of the event $\{\trifurcation^{\delta}\in V^{\delta}(r)\}\cap\LA_1^{\delta}\cap\LA_2^{\delta}$.
It remains to evaluate the probability of the event $\{\trifurcation^{\delta}\notin V^{\delta}(r)\}\cap\LA_1^{\delta}\cap\LA_2^{\delta}$.

The event $\{\trifurcation^{\delta}\notin V^{\delta}(r)\}$ implies that the simple random walk from $x_1^{\delta,\circ}$ has to exit $V_1^{\delta}(r)$ through $S_1^{\delta}(r)$, exit $\Omega^{\delta}$ via $e_3^{\delta}$, and 
the simple random walk from $x_2^{\delta}$ has to exit $V_2^{\delta}(r)$ through $S_2^{\delta}(r)$. 
   Thus by the Markov property of random walk, we have 
   \begin{equation}\label{eqn::boundary_branch_Dcvg_aux10}
       \PP\left[\{\trifurcation^{\delta}\notin V(r)\}\cap\LA_1^{\delta}\cap\LA_2^{\delta}\right]\le 
       \prod_{j=1}^2\harmonic(V_j^{\delta}(r);x_j^{\delta,\circ}, S_j^{\delta}(r))
\times\max_{w\in S_1^{\delta}(r)}\harmonic(\Omega^{\delta};w,e_3^{\delta}).
   \end{equation}
   From~\eqref{eqn::nharmonic_cvg} and~\eqref{eqn::discretePoisson_max_control}, the following ratios are bounded by a finite constant depending on $(\Omega; x_1, x_2, x_3; v; r)$: 
\begin{equation}\label{eqn::boundary_branch_Dcvg_aux11}
\frac{\harmonic(V_1^{\delta}(r);x_1^{\delta,\circ}, S^{\delta}_1(r))}{\harmonic(\Omega^{\delta}; v^{\delta},e_1^{\delta})}, \qquad \frac{\harmonic(V_2^{\delta}(r);x_2^{\delta,\circ}, S^{\delta}_2(r))}{\harmonic(\Omega^{\delta}; v^{\delta},e_2^{\delta})}, \qquad \frac{\max_{w\in S_2^{\delta}(r)}\harmonic(\Omega^{\delta};w,e_3^{\delta})}{\harmonic(\Omega^{\delta};v^{\delta},e_3^{\delta})}.
\end{equation}
Thus there exists a constant $C_{\eqref{eqn::boundary_branch_Dcvg_aux12}}\in(0,\infty)$ depending on $(\Omega;x_1,x_2,x_3;v;r)$ such that
   \begin{equation}\label{eqn::boundary_branch_Dcvg_aux12}
      \frac{\PP\left[\{\trifurcation^{\delta}\notin V(r)\}\cap\LA_1^{\delta}\cap\LA_2^{\delta}\right]}{\prod_{j=1}^3\harmonic(\Omega^{\delta}; v^{\delta}, e_j^{\delta})}\le  C_{\eqref{eqn::boundary_branch_Dcvg_aux12}}.
   \end{equation}
   Combining~\eqref{eqn::deltacube_control_boundary1} and~\eqref{eqn::boundary_branch_Dcvg_aux12}, we obtain the uniform bound for~\eqref{eqn::deltacube_control_bulk} as desired.
   \end{proof}

\subsection{Proof of Theorem~\ref{thm::trifurcation}}
\label{subsec::trifurcation_proof}
\begin{proposition}\label{prop::deltacube}
Assume the same setup as in Theorem~\ref{thm::trifurcation}. 
We fix an interior point $v\in\Omega$ and suppose $v^{\delta}$ is the vertex in $\LV^{\circ}(\Omega^{\delta})$ that is nearest to $v$. 
We have 
\begin{equation}\label{eqn::deltacube}
\lim_{\delta\to 0}\frac{\PP\left[\LA_1^\delta\cap\LA_2^\delta\right]}{\prod_{j=1}^{3}\harmonic(\Omega^\delta; v^\delta, e_j^\delta)}=\frac{{8\sqrt{2}}}{\prod_{j=1}^{3}\Poisson(\Omega; v, x_j)}\int_{\Omega}\LZtripod(\Omega; x_1, x_2, x_3; w)|\ud w|^2. 
\end{equation}
\end{proposition}

\begin{proof}
From~\eqref{eqn::fomin_main}, we have 
\begin{align*}
\frac{\PP\left[\LA_1^\delta\cap\LA_2^\delta\right]}{\prod_{j=1}^{3}\harmonic(\Omega^\delta; v^\delta, e_j^\delta)}
=&\frac{1}{\prod_{j=1}^{3}\harmonic(\Omega^\delta; v^\delta, e_j^\delta)}\sum_{u\in\LV^{\circ}(\Omega^{\delta})}\PP\left[\LA_1^\delta\cap\LA_2^\delta\cap\{\trifurcation^{\delta}=u\}\right]=\frac{4}{3\sqrt{3}}\sum_{u\in\LV^{\circ}(\Omega^{\delta})}f^{\delta}(u)\times\frac{3\sqrt{3}}{4}\delta^2, 
\end{align*}
where $f^{\delta}$ is defined by~\eqref{eqn::deltafive_integrand}. 
We define the triangle $\bigtriangleup^{\delta}(u)$ covering vertex $u$ as in Figure~\ref{fig::threeneighbors}.  
Note that the area of $\bigtriangleup^{\delta}(u)$ is $\frac{3\sqrt{3}}{4}\delta^2$. 
We extend $f^{\delta}$ as in Proposition~\ref{prop::observable_cvg}. With such extension, we have
\begin{align*}
\frac{\PP\left[\LA_1^\delta\cap\LA_2^\delta\right]}{\prod_{j=1}^{3}\harmonic(\Omega^\delta; v^\delta, e_j^\delta)}= &\frac{4}{3\sqrt{3}}\int_{\Omega^{\delta}}f^{\delta}(w)|\ud w|^2. 
\end{align*}
We define $f$ as in~\eqref{eqn::deltafive_integrand_continuum}, then the goal~\eqref{eqn::deltacube} can be written as 
\begin{equation}\label{eqn::deltacube_goal_repeat}
\lim_{\delta\to 0}\int_{\Omega^{\delta}}f^{\delta}(w)|\ud w|^2=\int_{\Omega}f(w)|\ud w|^2.
\end{equation}

Recall from Lemma~\ref{lem::proba_tight3} that for $\eps>0$, define $K(\eps)=\{z\in\Omega: \dist(z,\partial\Omega)\ge \eps\}$. Then $K(\eps)$ is a compact subset of $\Omega$ and Proposition~\ref{prop::observable_cvg} guarantees
\begin{align}\label{eqn::deltacube_integral_compact}
\lim_{\delta\to 0}\int_{K(\eps)}f^{\delta}(w)|\ud w|^2=\int_{K(\eps)}f(w)|\ud w|^2.
\end{align}
Comparing~\eqref{eqn::deltacube_goal_repeat} and~\eqref{eqn::deltacube_integral_compact}, it remains to show 
\[\limsup_{\eps\to 0}\limsup_{\delta\to 0}\int_{\Omega^{\delta}(\eps)}|f^{\delta}(w)-f(w)||\ud w|^2=0, \qquad \text{where }\Omega^{\delta}(\eps)=\Omega^{\delta}\setminus K(\eps). 
\]
As $f$ is integrable on $\Omega$ due to Lemma~\ref{lem::LZtripod_integrable}, we have $\lim_{\eps\to 0}\int_{\Omega^{\delta}(\eps)}f(w)|\ud w|^2=0$. Thus, to get~\eqref{eqn::deltacube_goal_repeat}, it remains to show that 
\begin{equation}\label{eqn::deltacube_control_boundary}
\limsup_{\eps\to 0}\limsup_{\delta\to 0}\int_{\Omega^{\delta}(\eps)}f^{\delta}(w)|\ud w|^2=0. 
\end{equation}
For $r>0$, denote $V(r)=B(x_1, r)\cup B(x_2, r)\cup B(x_3, r)$.  
From Lemmas~\ref{lem::proba_tight2} and~\ref{lem::proba_tight3}, there exists a universal constant $\alpha\in (0,\infty)$ (from weak Beurling-type estimate~\cite[Proposition 2.11]{ChelkakSmirnovDiscreteComplexAnalysis}), a constant $C_{\eqref{eqn::deltacube_control_boundary1}}\in (0,\infty)$ depending on $(\Omega; x_1, x_2, x_3; v)$ and a constant $C_{\eqref{eqn::deltacube_control_boundary2}}(r)\in (0,\infty)$ depending on $(\Omega; x_1, x_2, x_3; v; r)$
and a constant $\delta_0(\eps)>0$ depending on $(\{\Omega^{\delta}\}_{\delta>0};\Omega; x_1, x_2, x_3; \eps)$
such that 
\begin{align*}
\int_{V(r)}f^{\delta}(w)|\ud w|^2\le& C_{\eqref{eqn::deltacube_control_boundary1}} r^{\alpha}, \qquad \text{for all }\delta>0; \\
\int_{\Omega^{\delta}(\eps)\setminus V(r)}f^{\delta}(w)|\ud w|^2\le & C_{\eqref{eqn::deltacube_control_boundary2}}(r) \left(\frac{\eps}{r}\right)^{\alpha}, \qquad \text{for all }\delta\le \delta_0(\eps). 
\end{align*}
They give~\eqref{eqn::deltacube_control_boundary} as desired. 
\end{proof}

\begin{proof}[Proof of Theorem~\ref{thm::trifurcation}]
Combining~\eqref{eqn::fomin_main}, ~\eqref{eqn::det_convergence}, ~\eqref{eqn::detM_LZtripod} and~\eqref{eqn::deltacube}, we have 
\begin{align*}
\lim_{\delta\to 0}\delta^{-2}\PP\left[\trifurcation^{\delta}=z^{\delta}\cond\LA_1^{\delta}\cap\LA_2^{\delta}\right]=\frac{3\sqrt{3}}{4}\frac{\LZtripod(\Omega; x_1, x_2, x_3; z)}{\int_{\Omega}\LZtripod(\Omega; x_1, x_2, x_3; w)|\ud w|^2}, 
\end{align*}
as desired. 
\end{proof}


\subsection{A generalization of Proposition~\ref{prop::deltacube}}
\label{subsec::A1interior}
\begin{proposition}\label{prop::A1interior}
Fix a bounded $2$-polygon $(\Omega; x_1, x_3)$ and suppose $(\Omega^{\delta}; x_1^{\delta}, x_3^{\delta})$ is an approximation of $(\Omega; x_1, x_3)$ on $\delta\hexagon$ in Carath\'eodory sense. We assume further that $\partial\Omega^{\delta}$ converges to $\partial\Omega$ in Hausdorff distance~\eqref{eqn::boundary_cvg_Hausdorff}. 
We fix an interior points $z\in\Omega$ and suppose $z^{\delta}$ is the vertex in $\LV^{\circ}(\Omega^{\delta})$ that is nearest to $z$. 
The following limit exists 
\begin{align}\label{eqn::A1interior}
g(\Omega; x_1, x_3; z):=\lim_{\delta\to 0}\PP\left[ \{z^{\delta}\rightsquigarrow e_3^{\delta}\}\cond \LA_1^{\delta}\right].
\end{align}
The limit $g(\Omega; x_1, x_3; z)$ is characterized by~\eqref{eqn::chordalSLE2_harmonic_inv}-\eqref{eqn::chordalSLE2_harmonic_H}. 
\end{proposition}

\begin{lemma}
Fix a bounded $2$-polygon $(\Omega; x_1, x_3)$ and suppose $(\Omega^{\delta}; x_1^{\delta}, x_3^{\delta})$ is an approximation of $(\Omega; x_1, x_3)$ on $\delta\hexagon$ in Carath\'eodory sense. 
For $u \in \LV^{\circ}(\Omega^{\delta})$, denote by $u_1, u_2, u_3$ the three adjacent vertices of $u$ arranged in counterclockwise order. 
We define 
\begin{equation}
g^{\delta}(u)=\PP\left[\{u\rightsquigarrow e_3^{\delta}\}\cond\LA_1^{\delta}\right], \qquad u\in\LV^{\circ}(\Omega^{\delta}). 
\end{equation}
Recall that Laplacian is defined in~\eqref{eqn::dLaplacian_def}. 
Then 
\begin{align}\label{eqn::qdiscrete_def}
q^{\delta}(u):=&-\delta^{-2}\Delta g^{\delta}(u)\notag\\
=&\frac{1}{3\delta^2\PP[\LA_1^{\delta}]}\det 
        \begin{pmatrix}
        \harmonic(\Omega^{\delta}; u_1, e_1^\delta) & \harmonic(\Omega^{\delta}; u_1, (x_1^\delta x_3^\delta))-\harmonic(\Omega^\delta;u_1,(x_3^\delta x_1^\delta)) & \harmonic(\Omega^{\delta}; u_1, e_3^\delta)\\     
        \harmonic(\Omega^{\delta}; u_2, e_1^\delta) & \harmonic(\Omega^{\delta}; u_2, (x_1^\delta x_3^\delta))-\harmonic(\Omega^\delta;u_2,(x_3^\delta x_1^\delta)) & \harmonic(\Omega^{\delta}; u_2, e_3^\delta)\\ 
        \harmonic(\Omega^{\delta}; u_3, e_1^\delta) & \harmonic(\Omega^{\delta}; u_3, (x_1^\delta x_3^\delta))-\harmonic(\Omega^\delta;u_3,(x_3^\delta x_1^\delta)) & \harmonic(\Omega^{\delta}; u_3, e_3^\delta)   
        \end{pmatrix}.
\end{align}
\end{lemma}
\begin{proof}
On $\LA_1^{\delta}$, recall that $\eta_1^{\delta}$ is the boundary branch from $x_1^{\delta, \circ}$ to $x_3^{\delta}$ and we denote by $\Omega_1^{\delta}=\Omega^{\delta}\setminus\eta_1^{\delta}$. 
We have
\begin{align*}
g^{\delta}(u)=&\PP[\LA_1^{\delta}\cap \{u\rightsquigarrow e_3^{\delta}\}\cond\LA_1^{\delta}]=\E\left[\harmonic(\Omega_1^\delta;u, \eta_1^\delta)\cond\LA_1^\delta\right];\\
\text{and}\quad 
\Delta g^{\delta}(u)=&\frac{1}{3}\sum_{k=1}^3\left(g^\delta(u_k)-g^{\delta}(u)\right)=\E\left[\underbrace{\frac{1}{3}\sum_{k=1}^3\left(\harmonic(\Omega_1^{\delta}; u_k; \eta_1^{\delta})-\harmonic(\Omega_1^{\delta}; u, \eta_1^{\delta})\right)}_{\Delta\harmonic(\Omega_1^{\delta}; u, \eta_1^{\delta}):=}\cond \LA_1^{\delta}\right]. 
\end{align*}
As $\harmonic(\Omega_1^{\delta}; \cdot, \eta_1^{\delta})$ is harmonic in $\LV^{\circ}(\Omega_1^{\delta})$, we have
\begin{align*}
\Delta g^{\delta}(u)=\E\left[\Delta\harmonic(\Omega_1^{\delta}; u, \eta_1^{\delta})\one\{u\in\eta_1^{\delta}\}\cond \LA_1^{\delta}\right]. 
\end{align*}
When $u\in\eta_1^{\delta}$ and all its three neighbors $u_1, u_2, u_3$ are contained in $\eta_1^{\delta}$, we have $\harmonic(\Omega_1^{\delta}; u_k, \eta_1^{\delta})=\harmonic(\Omega_1^{\delta}; u, \eta_1^{\delta})=1$ for $k=1,2,3$. Thus, $\Delta\harmonic(\Omega_1^{\delta}; u, \eta_1^{\delta})=0$ in this case. 
When $u\in\eta_1^{\delta}$ and there are exactly two of its three neighbors contained in $\eta_1^{\delta}$, there are three cases: $u_{\ell}\not\in\eta_1^{\delta}$ and $u_{\ell+1}, u_{\ell+2}\in\eta_1^{\delta}$ for $\ell=1,2,3$ where we use the convention $u_4=u_1, u_5=u_2$. 
In this case, we have $\Delta\harmonic(\Omega_1^{\delta}; u, \eta_1^{\delta})=\frac{1}{3}(\harmonic(\Omega_1^{\delta}; u_{\ell}, \eta_1^{\delta})-1)$. 
Therefore, 
\begin{align*}
\Delta g^{\delta}(u)=\sum_{\ell=1}^3\E\left[\frac{1}{3}\left(\harmonic(\Omega_1^{\delta}; u_{\ell}, \eta_1^{\delta})-1\right)\one\{\{u, u_{\ell+1}, u_{\ell+2}\in\eta_1^{\delta}\}\cap\{u_{\ell}\not\in\eta_1^{\delta}\}\}\cond \LA_1^{\delta}\right].
\end{align*}
Note that 
\[\harmonic(\Omega_1^{\delta}; u_{\ell}, \eta_1^{\delta})-1=-\harmonic(\Omega_1^{\delta}; u_{\ell}, \partial\Omega^{\delta})=-\sum_{e\in\LE^{\partial}(\Omega^{\delta})}\harmonic(\Omega_1^{\delta}; u_{\ell}, e).\]
Thus,
\begin{align*}
-3\Delta g^{\delta}(u)=&\sum_{\ell=1}^3\sum_{e\in\LE^{\partial}(\Omega^{\delta})}\frac{1}{\PP[\LA_1^{\delta}]}\E\left[\harmonic(\Omega_1^{\delta}; u_{\ell}, e)\one\{\LA_1^{\delta}\cap\{u, u_{\ell+1}, u_{\ell+2}\in\eta_1^{\delta}\}\cap\{u_{\ell}\not\in\eta_1^{\delta}\}\}\right]\\
=&\sum_{\ell=1}^3\sum_{e\in\LE^{\partial}(\Omega^{\delta})}\frac{1}{\PP[\LA_1^{\delta}]}\PP\left[\LA_1^{\delta}\cap\{u, u_{\ell+1}, u_{\ell+2}\in\eta_1^{\delta}\}\cap\{u_{\ell}\rightsquigarrow e\}\right]. 
\end{align*}
Denote by $\LE^{\partial}_{13}(\Omega^{\delta})$ (resp. $\LE^{\partial}_{31}(\Omega^{\delta})$) the collection of boundary edges $e=\langle x^{\circ}, x^{\partial}\rangle\in\LE^{\partial}(\Omega^{\delta})$ such that $x^{\partial}\in (x_1^{\delta}x_3^{\delta})$ (resp. such that $x^{\partial}\in (x_3^{\delta}x_1^{\delta})$). Then 
\begin{align*}
-3\Delta g^{\delta}(u)
=&\sum_{\ell=1}^3\sum_{e\in\LE_{13}^{\partial}(\Omega^{\delta})}\frac{1}{\PP[\LA_1^{\delta}]}\PP\left[\LA_1^{\delta}\cap\{u, u_{\ell+1}, u_{\ell+2}\in\eta_1^{\delta}\}\cap\{u_{\ell}\rightsquigarrow e\}\right]\\
&+\sum_{\ell=1}^3\sum_{e\in\LE_{31}^{\partial}(\Omega^{\delta})}\frac{1}{\PP[\LA_1^{\delta}]}\PP\left[\LA_1^{\delta}\cap\{u, u_{\ell+1}, u_{\ell+2}\in\eta_1^{\delta}\}\cap\{u_{\ell}\rightsquigarrow e\}\right]. 
\end{align*}
Fix $e\in\LE_{13}^{\partial}(\Omega^{\delta})$, for configurations belonging to the event $\LA_1^{\delta}\cap\{u, u_{\ell+1}, u_{\ell+2}\in\eta_1^{\delta}\}\cap\{u_{\ell}\rightsquigarrow e\}$,
the operation of removing the edge $\langle u, u_{\ell+1}\rangle$ and adding the edge $e_1^\delta$ induces a bijection to configurations in the event 
\[\{u_{\ell+1}\rightsquigarrow e_3^\delta\}\cap\{u_{\ell+2}\rightsquigarrow e_1^\delta\}\cap\{u_\ell\rightsquigarrow e\}\cap\{\langle u, u_{\ell+2}\rangle\in \LT^\delta\}.\]
Fix $e\in\LE_{31}^{\partial}(\Omega^{\delta})$, for configurations belonging to the event $\LA_1^{\delta}\cap\{u, u_{\ell+1}, u_{\ell+2}\in\eta_1^{\delta}\}\cap\{u_{\ell}\rightsquigarrow e\}$,
the operation of removing the edge $\langle u, u_{\ell+1}\rangle$ and adding the edge $e_1^\delta$ induces a bijection to configurations in the event 
\[\{u_{\ell+1}\rightsquigarrow e_1^\delta\}\cap\{u_{\ell+2}\rightsquigarrow e_3^\delta\}\cap\{u_\ell\rightsquigarrow e\}\cap\{\langle u, u_{\ell+2}\rangle\in \LT^\delta\}.\]
See Figure~\ref{fig::fomin_laplace}.
Therefore, 
\begin{align}\label{eqn::A1interior_laplace_aux1}
-3\Delta g^{\delta}(u)
=&\sum_{\ell=1}^3\sum_{e\in\LE_{13}^{\partial}(\Omega^{\delta})}\frac{1}{\PP[\LA_1^{\delta}]}\PP\left[\{u_{\ell+1}\rightsquigarrow e_3^\delta\}\cap\{u_{\ell+2}\rightsquigarrow e_1^\delta\}\cap\{u_\ell\rightsquigarrow e\}\cap\{\langle u, u_{\ell+2}\rangle\in \LT^\delta\}\right]\notag\\
&+\sum_{\ell=1}^3\sum_{e\in\LE_{31}^{\partial}(\Omega^{\delta})}\frac{1}{\PP[\LA_1^{\delta}]}\PP\left[\{u_{\ell+1}\rightsquigarrow e_1^\delta\}\cap\{u_{\ell+2}\rightsquigarrow e_3^\delta\}\cap\{u_\ell\rightsquigarrow e\}\cap\{\langle u, u_{\ell+2}\rangle\in \LT^\delta\}\right]. 
\end{align}
Let us check the two cases. 
\begin{itemize}
    \item When $e\in \LE_{13}^{\partial}(\Omega^{\delta})$, the consequence of Fomin's formula~\eqref{eqn::fomin_aux4} tells
\begin{align*}
&\sum_{\ell=1}^{3}\PP\left[\{u_{\ell+1}\rightsquigarrow e_3^\delta\}\cap\{u_{\ell+2}\rightsquigarrow e_1^\delta\}\cap\{u_\ell\rightsquigarrow e\}\cap\{\langle u, u_{\ell+2}\rangle\in \LT^\delta\}\right]\\
=&\det\begin{pmatrix}
\harmonic(\Omega^{\delta}; u_1, e_1^\delta) & \harmonic(\Omega^{\delta}; u_{1}, e) & \harmonic(\Omega^{\delta}; u_{1}, e_3^\delta)\\
\harmonic(\Omega^{\delta}; u_2, e_1^\delta) & \harmonic(\Omega^{\delta}; u_{2}, e) & \harmonic(\Omega^{\delta}; u_{2}, e_3^\delta)\\
\harmonic(\Omega^{\delta}; u_3, e_1^\delta) & \harmonic(\Omega^{\delta}; u_{3}, e) & \harmonic(\Omega^{\delta}; u_{3}, e_3^\delta)
\end{pmatrix}.
\end{align*}
\item When $e\in \LE_{31}^{\partial}(\Omega^{\delta})$, the consequence of Fomin's formula~\eqref{eqn::fomin_aux4} tells
\begin{align*}
&\sum_{\ell=1}^{3}\PP\left[\{z_{\ell+1}^\delta\rightsquigarrow e_1^\delta\}\cap\{z_{\ell+2}^\delta\rightsquigarrow e_3^\delta\}\cap\{z_\ell^\delta\rightsquigarrow e\}\cap\{\langle z^\delta, z_{\ell+2}^\delta\rangle\in \LT^\delta\}\right]\\
=&\det\begin{pmatrix}
\harmonic(\Omega^{\delta}; u_{1}, e_3^\delta) & \harmonic(\Omega^{\delta}; u_{1}, e) & \harmonic(\Omega^{\delta}; u_{1}, e_1^\delta)\\
\harmonic(\Omega^{\delta}; u_{2}, e_3^\delta) & \harmonic(\Omega^{\delta}; u_{2}, e) & \harmonic(\Omega^{\delta}; u_{2}, e_1^\delta)\\
\harmonic(\Omega^{\delta}; u_{3}, e_3^\delta) & \harmonic(\Omega^{\delta}; u_{3}, e) & \harmonic(\Omega^{\delta}; u_{3}, e_1^\delta)
\end{pmatrix}\\
=&-\det\begin{pmatrix}
\harmonic(\Omega^{\delta}; u_{1}, e_1^\delta) & \harmonic(\Omega^{\delta}; u_{1}, e) & \harmonic(\Omega^{\delta}; u_{1}, e_3^\delta)\\
\harmonic(\Omega^{\delta}; u_{2}, e_1^\delta) & \harmonic(\Omega^{\delta}; u_{2}, e) & \harmonic(\Omega^{\delta}; u_{2}, e_3^\delta)\\
\harmonic(\Omega^{\delta}; u_{3}, e_1^\delta) & \harmonic(\Omega^{\delta}; u_{3}, e) & \harmonic(\Omega^{\delta}; u_{3}, e_3^\delta)     
\end{pmatrix}.
\end{align*}
\end{itemize}
Plugging the two cases into~\eqref{eqn::A1interior_laplace_aux1}, we obtain
\begin{align*}
 -3\Delta g^{\delta}(u)=&\sum_{e\in\LE_{13}^{\partial}(\Omega^{\delta})}\frac{1}{\PP[\LA_1^{\delta}]}\det
        \begin{pmatrix}
        \harmonic(\Omega^{\delta}; u_1, e_1^\delta) & \harmonic(\Omega^{\delta}; u_1, e) & \harmonic(\Omega^{\delta}; u_1, e_3^\delta)\\
        \harmonic(\Omega^{\delta}; u_2, e_1^\delta) & \harmonic(\Omega^{\delta}; u_2, e) & \harmonic(\Omega^{\delta}; u_2, e_3^\delta)\\
        \harmonic(\Omega^{\delta}; u_3, e_1^\delta) & \harmonic(\Omega^{\delta}; u_3, e) & \harmonic(\Omega^{\delta}; u_3, e_3^\delta)
        \end{pmatrix}\notag\\
&-\sum_{e\in\LE_{31}^{\partial}(\Omega^{\delta})}\frac{1}{\PP[\LA_1^{\delta}]}\det
        \begin{pmatrix}
        \harmonic(\Omega^{\delta}; u_1, e_1^\delta) & \harmonic(\Omega^{\delta}; u_1, e) & \harmonic(\Omega^{\delta}; u_1, e_3^\delta)\\
        \harmonic(\Omega^{\delta}; u_2, e_1^\delta) & \harmonic(\Omega^{\delta}; u_2, e) & \harmonic(\Omega^{\delta}; u_2, e_3^\delta)\\
        \harmonic(\Omega^{\delta}; u_3, e_1^\delta) & \harmonic(\Omega^{\delta}; u_3, e) & \harmonic(\Omega^{\delta}; u_3, e_3^\delta)
        \end{pmatrix}\notag\\
        =&\frac{1}{\PP[\LA_1^{\delta}]}\det 
        \begin{pmatrix}
        \harmonic(\Omega^{\delta}; u_1, e_1^\delta) & \harmonic(\Omega^{\delta}; u_1, (x_1^\delta x_3^\delta))-\harmonic(\Omega^\delta;u_1,(x_3^\delta x_1^\delta)) & \harmonic(\Omega^{\delta}; u_1, e_3^\delta)\\     
        \harmonic(\Omega^{\delta}; u_2, e_1^\delta) & \harmonic(\Omega^{\delta}; u_2, (x_1^\delta x_3^\delta))-\harmonic(\Omega^\delta;u_2,(x_3^\delta x_1^\delta)) & \harmonic(\Omega^{\delta}; u_2, e_3^\delta)\\ 
        \harmonic(\Omega^{\delta}; u_3, e_1^\delta) & \harmonic(\Omega^{\delta}; u_3, (x_1^\delta x_3^\delta))-\harmonic(\Omega^\delta;u_3,(x_3^\delta x_1^\delta)) & \harmonic(\Omega^{\delta}; u_3, e_3^\delta)   
        \end{pmatrix},
\end{align*}
as desired in~\eqref{eqn::qdiscrete_def}. 
\end{proof}

\begin{figure}[ht!]
\begin{subfigure}[b]{0.45\textwidth}
\begin{center}
\includegraphics[width=\textwidth]{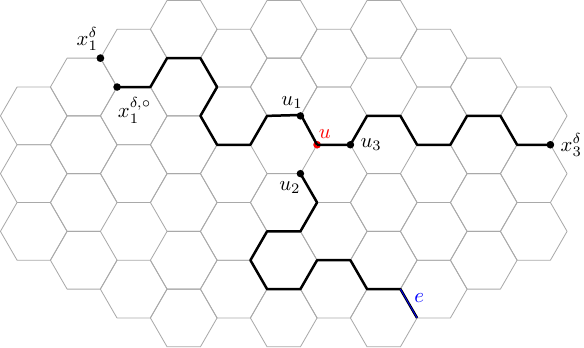}
\end{center}
\caption{}
\end{subfigure}
$\quad$
\begin{subfigure}[b]{0.45\textwidth}
\begin{center}
\includegraphics[width=\textwidth]{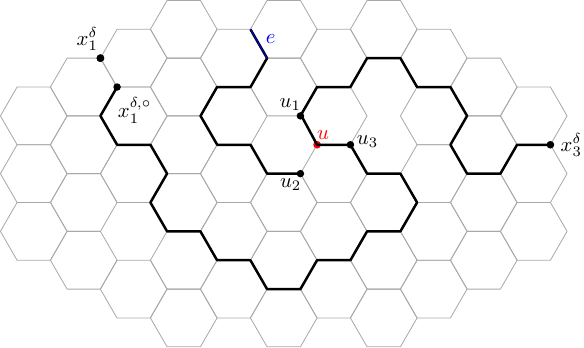}
\end{center}
\caption{}
\end{subfigure}
\caption{\label{fig::fomin_laplace} 
Suppose $u, u_1, u_3\in \eta_1^\delta$ and $u_2\not\in\eta_1^{\delta}$. 
In~(a), we have  $u_2\rightsquigarrow e$ for some $e \in \LE_{13}^\partial(\Omega^\delta)$. 
We delete the edge $\langle u, u_3\rangle$  and add the edge $e_1^{\delta}$. 
Such operation induces a bijection from configurations in $\LA_1^{\delta}\cap\{u,u_1,u_3\in \eta_1^\delta\}\cap\{u_2\rightsquigarrow e\}$ to configurations in $\{u_3\rightsquigarrow e_3^\delta\}\cap\{u_1\rightsquigarrow e_1^\delta\}\cap\{u_2\rightsquigarrow e\}\cap\{\langle u, u_1\rangle\in\LT^\delta\}$.
In~(b), we have  $u_2\rightsquigarrow e$ for some $e \in \LE_{31}^\partial(\Omega^\delta)$. 
We delete the edge $\langle u, u_3\rangle$  and add the edge $e_1^{\delta}$. 
Such operation induces a bijection from configurations in $\LA_1^{\delta}\cap\{u,u_1,u_3\in \eta_1^\delta\}\cap\{u_2\rightsquigarrow e\}$ to configurations in $\{u_3\rightsquigarrow e_1^\delta\}\cap\{u_1\rightsquigarrow e_3^\delta\}\cap\{u_2\rightsquigarrow e\}\cap\{\langle u, u_1\rangle\in\LT^\delta\}$.
}
\end{figure}

\begin{lemma}\label{lem::q_cvg}
Fix a bounded $2$-polygon $(\Omega; x_1, x_3)$ 
and suppose $(\Omega^{\delta}; x_1^{\delta}, x_3^{\delta})$ is an approximation of $(\Omega; x_1, x_3)$ on $\delta\hexagon$ in Carath\'eodory sense. 
\begin{itemize}
\item We define $q^{\delta}$ as in~\eqref{eqn::qdiscrete_def} and extend it in the same way as before: for $u\in\LV^{\circ}(\Omega^{\delta})$, we set $q^{\delta}(z)=q^{\delta}(u)$ for $z\in\bigtriangleup^{\delta}(u)$. 
\item Define $q(\Omega; x_1, x_3; z)$ as follows: for any conformal map $\varphi$ on $\Omega$, 
\begin{align}\label{eqn::q_cov}
q(\Omega; x_1, x_3; z)=|\varphi'(z)|^2 q(\varphi(\Omega); \varphi(x_1), \varphi(x_3); \varphi(z)); 
\end{align} 
and for $(\Omega; x_1, x_3)=(\HH; 0, \infty)$ and $z=r\ee^{\ii\theta}\in\HH$ with $\theta\in (0,\pi)$, we have
\begin{align}\label{eqn::q_H}
q(r, \theta):=q(\HH; 0, \infty; z)=\frac{2\sin(2\theta)}{\pi r^2}\left(\tan\theta-\theta+\frac{\pi}{2}\right). 
\end{align}
\end{itemize}
Then $q^{\delta}(\cdot)$ converges to $q(\Omega; x_1, x_3; \cdot)$ uniformly on compact subsets of $\Omega$. 
\end{lemma}

\begin{proof}
    For $z\in\Omega$, suppose $z^\delta$ the vertex in $\LV^\circ(\Omega^\delta)$ that is nearest to $z$, and denote by $z_1^\delta, z_2^\delta,z_3^\delta$ the three adjacent vertices of $z^\delta$ arranged in counterclockwise order.
Using the same analysis as in Proof of Lemma~\ref{lem::deltafive_cvg}, and combining with the convergence~\eqref{eqn::charmonic_cvg}-\eqref{eqn::bPoisson_cvg},
 we have
\begin{align*}
\lim_{\delta\to 0}q^{\delta}(z^\delta)=
    &\lim_{\delta \to 0}\frac{1}{\delta^2\harmonic(\Omega^\delta;x_1^{\delta, \circ},e_3^\delta)}\det
        \begin{pmatrix}
        \harmonic(\Omega^{\delta}; z_1^\delta, e_1^\delta) & \harmonic(\Omega^{\delta}; z_1^\delta, (x_1^\delta x_3^\delta))-\harmonic(\Omega^\delta;z_1^\delta,(x_3^\delta x_1^\delta)) & \harmonic(\Omega^{\delta}; z_1^\delta, e_3^\delta)\\     
        \harmonic(\Omega^{\delta}; z_2^\delta, e_1^\delta) & \harmonic(\Omega^{\delta}; z_2^\delta, (x_1^\delta x_3^\delta))-\harmonic(\Omega^\delta;z_2^\delta,(x_3^\delta x_1^\delta)) & \harmonic(\Omega^{\delta}; z_2^\delta, e_3^\delta)\\ 
        \harmonic(\Omega^{\delta}; z_3^\delta, e_1^\delta) & \harmonic(\Omega^{\delta}; z_3^\delta, (x_1^\delta x_3^\delta))-\harmonic(\Omega^\delta;z_3^\delta,(x_3^\delta x_1^\delta)) & \harmonic(\Omega^{\delta}; z_3^\delta, e_3^\delta)   
        \end{pmatrix}\\
=&\frac{1}{2\pi\Poisson(\Omega;x_1, x_3)}
    \det
    \begin{pmatrix}
      \partial_z\Poisson(\Omega; z, x_1) & \partial_z\left(\harmonic(\Omega;z,(x_1x_3))-\harmonic(\Omega;z,(x_3x_1))\right)  &   \partial_z \Poisson(\Omega; z, x_3) \\  
        -\ii\partial_{\overline{z}}\Poisson(\Omega; z, x_1) & -\ii\partial_{\overline{z}}\left(\harmonic(\Omega;z,(x_1x_3))-\harmonic(\Omega;z,(x_3x_1))\right)  &   -\ii\partial_{\overline{z}} \Poisson(\Omega; z, x_3) \\   
        \Poisson(\Omega; z, x_1) & \harmonic(\Omega;z,(x_1x_3))-\harmonic(\Omega;z,(x_3x_1)) & \Poisson(\Omega; z, x_3)
    \end{pmatrix}.
\end{align*}
Combining with $\harmonic(\Omega; z, (x_1x_3))+\harmonic(\Omega; z, (x_3x_1))=2\pi$, the pointwise limit of $q^\delta(\cdot)$ is given by 
\begin{equation}\label{eqn::q_limit}
q(\Omega; x_1, x_3; z):=\lim_{\delta\to 0}q^\delta(z^{\delta})=\frac{\det\left(N(\Omega; x_1, x_3; z)\right)}{2\pi\Poisson(\Omega; x_1, x_3)}, 
\end{equation}
where $N(\Omega;x_1,x_3,z)$ is the matrix given by
\begin{equation}\label{eqn::Ndef}
    \begin{pmatrix}
      \partial_z\Poisson(\Omega; z, x_1) & 2\partial_z\harmonic(\Omega;z,(x_1x_3))  &   \partial_z \Poisson(\Omega; z, x_3) \\  
        -\ii\partial_{\overline{z}}\Poisson(\Omega; z, x_1) & -2\ii\partial_{\overline{z}}\harmonic(\Omega;z,(x_1x_3))  &   -\ii\partial_{\overline{z}} \Poisson(\Omega; z, x_3) \\   
        \Poisson(\Omega; z, x_1) & 2\harmonic(\Omega;z,(x_1x_3))-2\pi & \Poisson(\Omega; z, x_3)
    \end{pmatrix}, 
\end{equation}
and $\Poisson(\Omega; z, x)$ is Poisson kernel~\eqref{eqn::Poisson_H}-\eqref{eqn::Poisson_cov}. 
Using the same analysis as in Proof of Proposition~\ref{prop::observable_cvg}, we obtain the unifrom convergence of~\eqref{eqn::q_limit} on compact sets of $\Omega$.

It remains to show that $q(\Omega; x_1, x_3; z)$ satisfies~\eqref{eqn::q_cov} and~\eqref{eqn::q_H}. 
The conformal covariance~\eqref{eqn::q_cov} is due to~\eqref{eqn::Poisson_cov} and~\eqref{eqn::bPoisson_cov}. 
   Let us derive~\eqref{eqn::q_H}.
   Recall from~\eqref{eqn::charnomic_H} and~\eqref{eqn::Poisson_derivative} that for $x\in \mathbb{R}$ and $x_1<x_3$,
   \begin{align*}
  & \Poisson(\HH;z,x) = \frac{2\Im z}{|z - x|^2},\qquad 
 \partial_z\Poisson(\HH; z, x)=\frac{-\ii}{(z-x)^2},\qquad \partial_{\overline{z}}\Poisson(\HH; z,x)=\frac{\ii}{(\overline{z}-x)^2};\notag\\
&\harmonic(\HH; z, (x_1x_3))=\int_{x_1}^{x_3}\Poisson(\HH; z, y)|\ud y|=\frac{1}{\ii}\log\frac{(x_3-z)(x_1-\overline{z})}{(x_3-\overline{z})(x_1-z)},\notag\\
& \partial_z \harmonic(\HH ; z,(x_1 x_3))=\ii\left(\frac{1}{x_3-z}-\frac{1}{x_1-z}\right), \quad \partial_{\bar{z}} \harmonic(\HH ; z,(x_1 x_3))=\ii\left(\frac{1}{x_1-\bar{z}}-\frac{1}{x_3-\bar{z}}\right),
\end{align*}
where the logarithm is taken on the principal branch $[0,2\pi)$.
Thus for $(\Omega;x_1,x_3)=(\HH; x_1,x_3)$ with $x_1<x_3$ and $z\in\HH$, we have
\begin{align*}
\displaystyle     \frac{\det\left(N(\HH; x_1, x_3; z)\right)}{2\pi\Poisson(\HH; x_1, x_3)}=&\frac{(x_3-x_1)^2}{2\pi}\det
    \begin{pmatrix}
         \frac{-\ii}{(z-x_1)^2} &  \ii(\frac{1}{x_3-z}-\frac{1}{x_1-z}) &  \frac{-\ii}{(z-x_3)^2} \\[8pt]
\frac{1}{\left(\overline{z}-x_1\right)^2} & \frac{1}{x_1-\overline{z}}-\frac{1}{x_3-\overline{z}} & \frac{1}{(\bar{z}-x_3)^2} \\[8pt]
\frac{2\Im(z)}{|z-x_1|^2} & \frac{1}{\ii} \log \frac{(x_3-z)(x_1-\overline{z})}{(x_3-\overline{z})(x_1-z)}-\pi & \frac{2\Im(z)}{|z-x_3|^2}
    \end{pmatrix}.
\end{align*}
Setting $x_1=0$ and $x_3=+\infty$, we obtain
\begin{align*}
    q(\HH;0,\infty;z)=&\frac{1}{2\pi}\det
    \begin{pmatrix}
-\ii /z^{2} & \ii /z & -\ii \\[8pt]
1/\overline{z}^{2} & -1/\overline{z} & 1 \\[8pt]
2|z|^{-2}\Im(z) & \pi - 2\arg z & 2\Im(z)
\end{pmatrix}\\
    =&\frac{4\Im(z)}{\pi|z|^4}\left(\Im(z)+\Re(z)(\frac{\pi}{2}-\arg(z))\right),
\end{align*}
which is~\eqref{eqn::q_H} as desired.
\end{proof}

\begin{lemma}\label{lem::g_solution_Laplace}
Fix a $2$-polygon $(\Omega; x_1, x_3)$.
Define 
\begin{equation}\label{eqn::g_Green}
g(z)=g(\Omega; x_1, x_3; z)=\frac{2}{3\pi}\int_{\Omega}\Green(\Omega; z, w)q(\Omega; x_1, x_3; w)|\ud w|^2, 
\end{equation}
where $\Green(\Omega; z, w)$ is Green's function~\eqref{eqn::cgreen}-\eqref{eqn::cgreen_inv} and $q(\Omega; x_1, x_3; w)$ is the function defined by~\eqref{eqn::q_cov}-\eqref{eqn::q_H}.
Then $g(\cdot)$ is the unique \textnormal{bounded} solution in $C(\overline{\Omega}\setminus\{x_1,x_3\})\cap C^2(\Omega)$ to the following Poisson equation:
\begin{align}\label{eqn::Laplace_Dirichlet}
\begin{cases}
-\Delta g(z)=\frac{4}{3}q(\Omega; x_1, x_3; z), \quad &z\in\Omega; \\
g(z)=0, \quad &z\in \partial\Omega\setminus\{x_1, x_3\}.
\end{cases}
\end{align}
Moreover, $g$ is characterized by~\eqref{eqn::chordalSLE2_harmonic_inv}-\eqref{eqn::chordalSLE2_harmonic_H}: for any conformal map $\varphi$ on $\Omega$, we have
\begin{equation}\label{eqn::chordalSLE2_harmonic_inv_repeat}
g(\Omega; x_1, x_3; z)=g(\varphi(\Omega); \varphi(x_1), \varphi(x_3); \varphi(z)); 
\end{equation}
and for $(\Omega; x_1, x_3)=(\HH; 0, \infty)$ and $z=r\ee^{\ii\theta}\in \HH$ with $\theta\in(0,\pi)$, we have
\begin{equation}\label{eqn::chordalSLE2_harmonic_H_repeat}
g(\theta):=g(\HH; 0, \infty; z)=\frac{4}{3\pi}\left(\theta(\pi-\theta)+(\frac{\pi}{2}-\theta)\sin\theta\cos\theta+2\sin^2\theta\right).
\end{equation}
\end{lemma}

Before the proof, let us give two remarks. First, the function $q(\Omega; x_1, x_3; z)$ explodes at $z=x_1$ and $z=x_3$ and the solution given by~\eqref{eqn::chordalSLE2_harmonic_inv_repeat}-\eqref{eqn::chordalSLE2_harmonic_H_repeat} is not continuous at $z=x_1$ or $z= x_3$. We need to analyze the solution to~\eqref{eqn::Laplace_Dirichlet} around $x_1, x_3$ carefully. 
Second, 
we emphasize that bounded solution to~\eqref{eqn::Laplace_Dirichlet} is unique. The ``boundedness" is important, as the solution to~\eqref{eqn::Laplace_Dirichlet} is not unique without the assumption of boundedness: any solution plus linear combinations of the Poisson kernels $\Poisson(\Omega; z, x_1)$ and $\Poisson(\Omega; z, x_3)$ is also a solution to~\eqref{eqn::Laplace_Dirichlet}. 
\begin{proof}
It suffices to prove the conclusion for the case when $(\Omega; x_1, x_3)=(\HH; 0,\infty)$. 
In this case, for $z\in \HH$ and $w=r\ee^{\ii\theta}$ with $\theta\in (0,\pi)$, we have 
\[q(\HH; 0, \infty; r\ee^{\ii\theta})=\frac{2\sin(2\theta)}{\pi r^2}\left(\tan\theta-\theta+\frac{\pi}{2}\right), \qquad \Green(\HH;z,w)=\log{\frac{|z-\overline{w}|}{|z-{w}|}}. 
\]
For $q$, we have: for $\theta\in (0,\pi)$,
\begin{equation}\label{eqn::q_upperbound}
|q(\HH;0,\infty;r\ee^{\ii \theta})|\lesssim \frac{\theta(\pi-\theta)}{r^2};
\end{equation}
and for Green's function, we have: for $\theta\in (0,\pi)$, 
\begin{align}\label{eqn::Green_upperbound1}
&|\Green(\HH;z,r\ee^{\ii\theta})|\lesssim \frac{|z|}{r},  \quad \text{for }r\ge 2|z|; \qquad |\Green(\HH;z,r\ee^{\ii \theta})|\lesssim \frac{r}{|z|}, \quad \text{for }r\le |z|/2; \\
&|\Green(\HH;z,w)|\lesssim\frac{\Im(w)}{\Im(z)},\quad \text{for }\Im(w)\leq \Im(z)/4,\label{eqn::Green_upperbound2}
\end{align}
where we use the following estimate to get~\eqref{eqn::Green_upperbound2}:
for $\Im(w)\leq \Im(z)/4$, we have
\begin{align*}
|\Green(\HH;z,w)|=&\frac{1}{2}\log \frac{\Re(z-w)^2+\Im(z)^2+2\Im(w)\Im(z)+\Im(w)^2}{\Re(z-w)^2+\Im(z)^2-2\Im(w)\Im(z)+\Im(w)^2}\\
\le &\frac{1}{2}\log \frac{\Im(z)^2+2\Im(w)\Im(z)}{\Im(z)^2-2\Im(w)\Im(z)}\\
= & \frac{1}{2}\log \frac{1+\frac{2\Im(w)}{\Im(z)}}{1-\frac{2\Im(w)}{\Im(z)}}
\lesssim \frac{\Im(w)}{\Im(z)}.
\end{align*}
\medbreak

Let us first check that the RHS of~\eqref{eqn::g_Green} is bounded: from~\eqref{eqn::q_upperbound}-\eqref{eqn::Green_upperbound1}, 
\begin{align*}
&\int_\HH |\Green(\HH; z, w)q(\HH; 0, \infty; w)| |\ud w|^2\notag\\
=&\int_0^{\pi}\int_0^{\infty} |\Green(\HH; z, r\ee^{\ii\theta})q(\HH; 0, \infty; r\ee^{\ii\theta})| r \ud r \ud \theta \notag\\
\lesssim & \int_{2|z|}^{\infty} \frac{|z|}{r} \frac{1}{r^2} r \ud r + \int_0^{|z|/2} \frac{r}{|z|} \frac{1}{r^2} r \ud r + \sup_{\theta\in(0,\pi)}\int_{|z|/2}^{2|z|} |\Green(\HH; z, r\ee^{\ii\theta})|  \frac{1}{r^2} r \ud r \notag\\
=&\frac{1}{2}+\frac{1}{2}+\sup_{\theta\in(0,\pi)}\int_{1/2}^{2}\frac{1}{u}\left|\log \frac{|\ee^{\ii \arg(z)}-u\ee^{\ii\theta}|}{|\ee^{\ii\arg(z)}-u\ee^{-\ii\theta}|}\right|\ud u \notag\\
\leq& 1+ 4\sup_{v\in \partial \mathbb{D}} \int_{1/2}^{2} \left|\log|u - v |\right| \ud u<\infty. 
\end{align*}
This shows that RHS of~\eqref{eqn::g_Green} is well-defined and bounded. We denote it by $g(z)$ as in~\eqref{eqn::g_Green}. 
\medbreak
Second, let us check $-\Delta g(z)=\frac{4}{3}q(\HH;0,\infty;z)$. 
For $\eps>0$, define a compact set $A_\eps$ in $\HH$ by $A_\eps:=\{w: \Im(w)\geq \eps^2, \eps\leq |w|\leq \eps^{-1}\}$.
Let $\chi_\eps\in C_c^\infty(\mathbb{\HH})$ be a smooth cut-off function such that $\mathrm{supp}(\chi_\eps) \subset A_{\eps/2}$ and $0 \le \chi_\eps \le 1$ and $\chi_\eps \equiv 1 \text{ on } A_\eps$.
Define
\[q_\eps(\HH;0,\infty;z)=\chi_\eps(z)q(\HH;0,\infty;z),\quad g_\eps(z):=\frac{2}{3\pi}\int_\HH \Green(\HH; z, w) q_\eps(\HH; 0, \infty; w)|\ud w|^2.\]
Then $-\Delta g_\eps(\cdot)=\frac{4}{3}q_\eps(\HH;0,\infty;\cdot)$ since $-\Delta_z \Green(\HH; z, w) = 2\pi\mathrm{Dirac}(z - w)$ in the sense of distributions.

Let us show that $g_\eps\to g$ uniformly on any compact subset of $\HH$ as $\eps\to 0$.
Fix $\eps_0>0$, let us evaluate $|g(z)-g_\eps(z)|$ for $\eps_0\leq |z|\leq \eps_0^{-1}, \Im(z)\geq \eps_0^2$: for any $\eps\le \eps_0/2$,
\begin{align*}
|g(z)-g_\eps(z)|\leq&\frac{2}{3\pi}\int_\HH |\Green(\HH; z, w) (q(\HH; 0, \infty; w)-q_\eps(\HH; 0, \infty; w))||\ud w|^2\\
\leq&\frac{2}{3\pi}\left(\int_{|w|\geq \eps^{-1}}+\int_{0<|w|\leq \eps}+\int_{\substack{\Im(w)\leq\eps^2\\\eps\leq |w|\leq \eps^{-1}}}\right) |\Green(\HH; z, w) q(\HH; 0, \infty; w)||\ud w|^2\\
\lesssim& \int_{\eps^{-1}}^\infty \frac{1}{\eps_0 r}\frac{1}{r^2}r\ud r+\int_{0}^{\eps} \frac{r}{\eps_0} \frac{1}{r^2} r\ud r+\int_{\eps}^{\eps^{-1}}\frac{\eps^2}{\eps_0^2}\frac{\eps}{r^2} r\ud r\tag{due to~\eqref{eqn::q_upperbound}-\eqref{eqn::Green_upperbound2}}\\
=& 2\frac{\eps}{\eps_0}+\frac{\eps^3}{\eps_0^2}\log(\eps^{-2})\to 0, \quad \text{as }\eps\to 0.
\end{align*}
This gives the desired uniform convergence. 
As $\lim_{\eps\to 0}g_\eps = g$ uniformly on any compact subset of $\HH$, the continuity of Laplacian on distributions asserts
\[-\Delta g(\cdot)=-\Delta\left(\lim_{\eps\to 0}g_\eps(\cdot)\right)=\lim_{\eps\to 0}\frac{4}{3}q_\eps(\HH;0,\infty;\cdot)=\frac{4}{3}q(\HH;0,\infty;\cdot)\] 
in the sense of distributions.
As Laplacian is hypoelliptic, $g\in C^2(\HH)$ and $-\Delta g(z)=\frac{4}{3}q(\HH;0,\infty;z)$ holds pointwise.
\medbreak

Third, let us check $\lim_{z\to x}g(z)=0$ for $x\in\R\setminus\{0\}$. Fix $x \in \R\setminus\{0\}$, for $\eps<(\frac{|x|}{4}\wedge \frac{1}{4|x|})$ and $|z-x|\le \eps$, we have
\begin{align*}
|g(z)|\le&\frac{2}{3\pi}\left(\int_{|w|\le \eps}+\int_{\eps\le |w|\leq \eps^{-1}}+\int_{|w|\ge \eps^{-1}}\right) |\Green(\HH; z, w) q(\HH; 0, \infty; w)||\ud w|^2\\
\lesssim & \int_0^{\eps}\frac{r}{|z|}\frac{1}{r^2}r\ud r+\int_{\eps\le |w|\leq \eps^{-1}}|\Green(\HH; z, w) q(\HH; 0, \infty; w)||\ud w|^2+\int_{\eps^{-1}}^{\infty}\frac{|z|}{r}\frac{1}{r^2}r\ud r\tag{due to~\eqref{eqn::q_upperbound}-\eqref{eqn::Green_upperbound1}}\\
\le & \frac{\eps}{|z|}+\int_{\eps\le |w|\leq \eps^{-1}}|\Green(\HH; z, w) q(\HH; 0, \infty; w)||\ud w|^2+\eps|z|\\
=&\frac{\eps}{|z|}+\underbrace{\int_{\substack{\eps\le |w|\leq \eps^{-1}\\ |w-x|\ge 2\eps}}|\Green(\HH; z, w) q(\HH; 0, \infty; w)||\ud w|^2}_{I_{\eps}(z):=}+\underbrace{\int_{|w-x|\le 2\eps}|\Green(\HH; z, w) q(\HH; 0, \infty; w)||\ud w|^2}_{J_{\eps}(z):=}+\eps|z|. 
\end{align*}
As $\lim_{z\to x}\Green(\HH; z, w)=0$, dominated convergence theorem guarantees $\lim_{z\to x}I_{\eps}(z)=0$. 
From~\eqref{eqn::q_upperbound}, we have 
\[J_{\eps}(z)\lesssim \frac{\eps}{|x|^3}\int_{|w-x|\le 2\eps}|\Green(\HH; z, w)||\ud w|^2 \le \frac{\eps}{|x|^3}\int_{|w-z|\leq 4\eps}\log\frac{1}{|w-z|}|\ud w|^2\lesssim\frac{\eps^3}{|x|^3}\log\frac{1}{\eps}. \] 
Thus, $\limsup_{z\to x}|g(z)|\lesssim \frac{\eps}{|x|}+\frac{\eps^3}{|x|^3}\log\frac{1}{\eps} +\eps|x|$. Let $\eps\to 0$, we obtain $\lim_{z\to x}|g(z)|=0$ as desired. 
\medbreak

From the above three steps, $g$ is a bounded solution to~\eqref{eqn::Laplace_Dirichlet} when $(\Omega; x_1, x_3)=(\HH; 0,\infty)$. Let us check that $g$ is the unique such solution.
Suppose there is another bounded solution $\tilde{g}(\cdot)$ to~\eqref{eqn::Laplace_Dirichlet} when $(\Omega; x_1, x_3)=(\HH; 0,\infty)$. 
Then $F(z):=g(z)-\tilde{g}(z)$ is a bounded harmonic function in $\HH$ and equals $0$ on $\mathbb{R}\setminus\{0\}$. 
We extend $F$ to lower-half plane $\{z\in\C: \Im(z)<0\}$ by reflection: $F(z)=-F(\overline{z})$.  
Then $F$ is a bounded harmonic function in $\mathbb{C}\setminus\{0\}$ due to mean value theorem. 
Consequently, $f(z)=F(\ee^z)$ is a bounded harmonic function on $\C$. 
By Liouville's theorem for harmonic function, $f(z)$ is constant.
As $f(1)=0$, we have $f\equiv 0$ and $F\equiv 0$, which gives the desired uniqueness.
\medbreak

Finally, let us check that $g$ is characterized by~\eqref{eqn::chordalSLE2_harmonic_inv_repeat}-\eqref{eqn::chordalSLE2_harmonic_H_repeat}.
The conformal invariance~\eqref{eqn::chordalSLE2_harmonic_inv_repeat} follows from~\eqref{eqn::cgreen_inv} and~\eqref{eqn::q_cov}. 
We denote by $\hat{g}(\theta)$ the RHS of~\eqref{eqn::chordalSLE2_harmonic_H_repeat}, then it is the unique solution to~\eqref{eqn::Laplace_Dirichlet} with $(\Omega; x_1, x_3)=(\HH; 0, \infty)$, because 
\begin{align*}
\begin{cases}
\Delta \hat{g}(\theta)=\left(\frac{\partial^2}{\partial r^2}+\frac{1}{r}\frac{\partial}{\partial r}+\frac{1}{r^2}\frac{\partial^2}{\partial\theta^2}\right)\hat{g}(\theta)=\frac{1}{r^2}\frac{\partial^2}{\partial\theta^2}\hat{g}(\theta)=-\frac{4}{3}q(r,\theta); \\
\hat{g}(0)=\hat{g}(\pi)=0. 
\end{cases}
\end{align*}
This gives~\eqref{eqn::chordalSLE2_harmonic_H_repeat} and completes the proof.
\end{proof}

\begin{proof}[Proof of Proposition~\ref{prop::A1interior}]
Recall from~\eqref{eqn::qdiscrete_def} that $-\Delta g^{\delta}(u)=\delta^{2}q^{\delta}(u)$ for $u\in \LV^{\circ}(\Omega^{\delta})$. 
We extend $q^{\delta}$ as in Lemma~\ref{lem::q_cvg}. 
For $z\in\Omega$, suppose $z^{\delta}$ is the vertex in $\LV^{\circ}(\Omega^{\delta})$ that is nearest to $z$.
We extend Green's function $\Green(\Omega^{\delta}; z^{\delta}, \cdot)$~\eqref{eqn::dGreen_def} in the same way: for $u\in\LV^{\circ}(\Omega^{\delta})$, we set $\Green(\Omega^\delta;z^\delta,w)=\Green(\Omega^\delta;z^\delta,u)$ for $w\in\bigtriangleup^{\delta}(u)$.
We have
\begin{align*}
        g^\delta(z^{\delta})=&\sum_{u\in\LV^\circ(\Omega^\delta)}\delta^2\Green(\Omega^\delta;z^{\delta},u)q^\delta(u)\tag{due to~\eqref{eqn::dLaplacian_green}}\\
        =&\frac{4}{3\sqrt{3}}\sum_{u\in\LV^\circ(\Omega^\delta)}\frac{3\sqrt{3}}{4}\delta^2 \Green(\Omega^\delta;z^{\delta},u) q^\delta(u)\notag\\
        =&\frac{4}{3\sqrt{3}}\int_{\Omega^\delta} \Green(\Omega^\delta;z^{\delta},w)q^\delta(w) |\ud w|^2.
    \end{align*}
We will show in Lemma~\ref{lem::g_cvg} that
        \begin{equation}\label{eqn::Gq_int_cvg}
            \lim_{\delta\to 0}\int_{\Omega^\delta} \Green(\Omega^\delta;z^{\delta},w)q^\delta(w) |\ud w|^2=\frac{\sqrt{3}}{2\pi}\int_\Omega \Green(\Omega;z,w)q(\Omega;x_1,x_3;w)|\ud w|^2.
        \end{equation}
Combining with~\eqref{eqn::g_Green}, we have $g^{\delta}(z^{\delta})\to g(z)$ as $\delta\to 0$.
\end{proof}

\begin{lemma}\label{lem::g_cvg}
Assume the same notation as in Proof of Proposition~\ref{prop::A1interior}. The convergence~\eqref{eqn::Gq_int_cvg} holds. 
\end{lemma}
\begin{proof} 
For $0<\eps<r$, we define $K(\eps)=\{z\in\Omega: \dist(z,\partial\Omega)>\eps\}$ and denote $V_j(r)=B(x_j,r)\cap \Omega$ for $j=1,3$.
Let $V_j^\delta(r)$ be the maximal domain contained in $V_j(r)$ such that $\partial V_j^{\delta}(r)$ consist of the edges of $\delta\hexagon$,
 and denote $S^\delta_j(r)=\partial V^\delta_j(r)\setminus \partial \Omega^\delta$ as before.
Denote $\Omega_{r,\eps}=K(\eps)\setminus(V_1(r)\cup V_3(r))$, and let $\Omega^\delta_{r,\eps}$ be the maximal domain contained in $\Omega_{r,\eps}$ such that $\partial \Omega^{\delta}_{r,\eps}$ consist of the edges of $\delta\hexagon$.
Let $g_{r,\eps}^\delta(\cdot)$ be the harmonic function on $\LV(\Omega^{\delta}_{r,\eps})$ with boundary value $g^\delta(\cdot)$ on $\partial \Omega^{\delta}_{r,\eps}$.
Then 
\begin{equation}
    g^\delta(z^\delta)-g_{r,\eps}^\delta(z^\delta)=\int_{\Omega_{r,\eps}^\delta}\Green(\Omega^\delta_{r,\eps};z^\delta,w)q^\delta(w)|\ud w|^2, \qquad\text{for }z^\delta\in \LV^\circ(\Omega^\delta_{r,\eps}).
\end{equation}
We have the following claims.
\begin{itemize}
\item Fix $r,\eps>0$, pick $s>0$ small such that $\overline{B(z,2s)}\subset\Omega_{r,\eps}$.
Lemma~\ref{lem::Green_cvg} asserts that Green's function $\Green(\Omega^\delta_{r,\eps};z^\delta,\cdot)\to \frac{\sqrt{3}}{2\pi}\Green(\Omega_{r,\eps};z,\cdot)$ uniform in $\Omega_{r,\eps}\setminus B(z,s)$. 
Lemma~\ref{lem::q_cvg} asserts that $q^\delta(\cdot)\to q(\Omega;x_1,x_3;\cdot)$ uniform in $\Omega_{r,\eps}$. 
Thus,
\begin{equation}\label{eqn::Greenq_localuniform}
    \lim_{\delta\to 0}\int_{\Omega_{r,\eps}^\delta\setminus B(z,s)}\Green(\Omega^\delta_{r,\eps};z^\delta,w)q^\delta(w)|\ud w|^2=\frac{\sqrt{3}}{2\pi}\int_{\Omega_{r,\eps}\setminus B(z,s)}\Green(\Omega_{r,\eps};z,w)q(\Omega;x_1,x_3;w)|\ud w|^2.
\end{equation}
\item
We will show 
\begin{equation}\label{eqn::Greenq_int_to_zero}
\limsup_{\delta\to 0}\int_{B(z,s)}|\Green(\Omega^\delta_{r,\eps};z^\delta,w)q^\delta(w)||\ud w|^2\to 0,\quad \text{as }s\to 0.
\end{equation} 
\item As $q(\Omega;x_1,x_3;w)$ is non-negative and $\Green(\Omega_{r,\eps};z,w)$ is monotone as $r,\eps\to 0$, we have
\begin{equation}\label{eqn::g_re_to_g_limit}
    \lim_{r\to 0}\lim_{\eps\to 0}\int_{\Omega_{r,\eps}}\Green(\Omega_{r,\eps};z,w)q(\Omega;x_1,x_3;w)|\ud w|^2=\int_{\Omega}\Green(\Omega;z,w)q(\Omega;x_1,x_3;w)|\ud w|^2.
\end{equation}
\item We will show \begin{equation}\label{eqn::g_re_to_g}
    \lim_{r\to 0}\lim_{\eps\to 0}\sup_{\delta>0}|g^\delta_{r,\eps}(z^\delta)|=0.    
\end{equation}
\end{itemize}
Combining~\eqref{eqn::Greenq_localuniform}-\eqref{eqn::g_re_to_g} gives~\eqref{eqn::Gq_int_cvg} as desired.
\medbreak
Let us prove~\eqref{eqn::Greenq_int_to_zero}.
The definition of $\Omega_{r,\eps}^\delta$ ensures that $\dist_H(\partial\Omega_{r,\eps},\partial\Omega_{r,\eps}^\delta)\to 0$ as $\delta \to 0$.
From Lemma~\ref{lem::Green_upperbound}, there exist constants $\delta_0\in(0,1)$ (depending on $(\{\Omega_{r,\eps}^\delta\}_{\delta>0};\Omega;r,\eps;z)$) and $C_{\eqref{eqn::Green_domain_plane_upperbound}}\in(0,\infty)$ (depending on $(\Omega;r,\eps;z)$) such that, for all $\delta\in(0,\delta_0)$,
 \begin{equation}\label{eqn::Green_domain_plane_upperbound}
 \Green(\Omega^\delta_{r,\eps};z^\delta,y)\leq \frac{\sqrt{3}}{2\pi}\log \frac{1}{|z^\delta-y|\vee\delta}+C_{\eqref{eqn::Green_domain_plane_upperbound}}, \quad \text{for all }y\in \LV(\Omega^\delta_{r,\eps}) .
 \end{equation}
 Note that for any $w\in \bigtriangleup^\delta(y)$ with $y\neq z^\delta$, we have $|z-w|\leq 3|z^\delta-y|$, while for $w\in \bigtriangleup^\delta(z^\delta)$, we have $|z-w|\leq \sqrt{3}\delta$.
Thus,
\begin{equation}\label{eqn::Green_re_step_aux1}
 \Green(\Omega^\delta_{r,\eps};z^\delta,w)\leq \frac{\sqrt{3}}{2\pi}\log \frac{1}{|z-w|}+C_{\eqref{eqn::Green_domain_plane_upperbound}}+\frac{\sqrt{3}}{2\pi}\log 3, \quad \text{for all }w\in B(z,s)\text{ and all }\delta\le \delta_0.
 \end{equation}
Consequently, 
\begin{align*}
    &\limsup_{\delta\to 0}\int_{B(z,s)}|\Green(\Omega^\delta_{r,\eps};z^\delta,w)q^\delta(w)||\ud w|^2\\
    \leq&\int_{B(z,s)}\limsup_{\delta\to 0}|\Green(\Omega^\delta_{r,\eps};z^\delta,w)q^\delta(w)||\ud w|^2\tag{due to Fatou's lemma}\\
    \leq&\int_{B(z,s)}\left(\frac{\sqrt{3}}{2\pi}\log \frac{1}{|z-w|}+C_{\eqref{eqn::Green_domain_plane_upperbound}}+\frac{\sqrt{3}}{2\pi}\log 3\right)|q(\Omega;x_1,x_3;w)||\ud w|^2\tag{due to~\eqref{eqn::Green_re_step_aux1} and Lemma~\ref{lem::q_cvg}}\\
    &\xrightarrow{s\to 0} 0.
\end{align*}
This completes the proof of~\eqref{eqn::Greenq_int_to_zero}.
\medbreak
It remains to show~\eqref{eqn::g_re_to_g}. 
Let us evaluate the probability of $\{u\rightsquigarrow e_3^\delta\}\cap\{x_1^{\delta,\circ}\rightsquigarrow e_3^\delta\}$ for $u\in \partial \Omega^\delta_{r,\eps}\setminus (V_1(2r)\cup V_3(2r))$.
Denote by $\eta_{u}^\delta$  the boundary branch from $u$, and denote by $R_{u}$ the event that $\eta_{u}^\delta$ hits $S_1^\delta(r)$.
We will analyze the following two events separately (see Figure~\ref{fig::ginterior_boundary}):
\[
\{u\rightsquigarrow e_3^\delta\}\cap\{x_1^{\delta,\circ}\rightsquigarrow e_3^\delta\}\cap R_{u}\quad \text{and}\quad \{u\rightsquigarrow e_3^\delta\}\cap\{x_1^{\delta,\circ}\rightsquigarrow e_3^\delta\}\cap R_{u}^c.
\]
\begin{itemize}
\item For the first case (see Figure~\ref{fig::ginterior_boundary}~(a)), recall that we denote by $\eta_1^\delta$ the boundary branch from $x_1^\delta$ and denote $\Omega_1^\delta=\Omega^\delta\setminus\eta_1^\delta$. The event 
$\{u\rightsquigarrow e_3^\delta\}\cap\{x_1^{\delta,\circ}\rightsquigarrow e_3^\delta\}\cap R_{u}$
implies that $\LA_1^\delta$ occurs, and given $\eta_1^\delta$, the random walk started from $u$ has to hit $S_1^\delta(r)$ before hitting $\partial \Omega_1^\delta$.
Thus,
\begin{align*}
    \PP\left[\{u \rightsquigarrow e_3^\delta\}\cap\{x_1^{\delta,\circ}\rightsquigarrow e_3^\delta\}\cap R_{u}\right]
    \leq &\E\left[\one\{\LA_1^\delta\}\harmonic(\Omega_1^\delta\setminus V_1^\delta(r);u,S_1^\delta(r))\right]\\
\leq &\PP[\LA_1^\delta]\times\harmonic(\Omega^\delta\setminus V_1^\delta(r);u,S_1^\delta(r)).
\end{align*}
\item For the second case (see Figure~\ref{fig::ginterior_boundary}~(b)), the event $\{u\rightsquigarrow e_3^\delta\}\cap\{x_1^{\delta,\circ}\rightsquigarrow e_3^\delta\}\cap R_{u}^c$
implies that the boundary branch $\eta_u^\delta$ reaches $S^\delta_3(r)$ and then hit $\partial\Omega^\delta$ through $e_3^\delta$; 
given $(R_u^c, \eta_u^{\delta})$, the random walk started from $x_1^{\delta,\circ}$ has to exit $V_1^\delta(r)$ through $S^\delta_1(r)$.
Thus, 
\begin{align*}
    &\PP\left[\{u \rightsquigarrow e_3^\delta\}\cap\{x_1^{\delta,\circ}\rightsquigarrow e_3^\delta\}\cap R_{u}^c\right]\\
    \leq &
    \harmonic(\Omega^\delta\setminus (V_1^\delta(r)\cup V_3^\delta(r));u,S_3^\delta(r))\times\max_{w\in S_3^\delta(r)}\harmonic(\Omega^\delta;w,e_3^\delta)\times\harmonic(V_1^\delta(r);x_1^{\delta,\circ},S_1^\delta(r))\\
    \leq&\harmonic(\Omega^\delta\setminus V_3^\delta(r);u,S_3^\delta(r))\times \max_{w\in S_3^\delta(r)}\harmonic(\Omega^\delta;w,e_3^\delta)\times\harmonic(V_1^\delta(r);x_1^{\delta,\circ},S_1^\delta(r)).
\end{align*}
\end{itemize}

\begin{figure}[ht!]
\begin{subfigure}[b]{0.45\textwidth}
\begin{center}
\includegraphics[width=\textwidth]{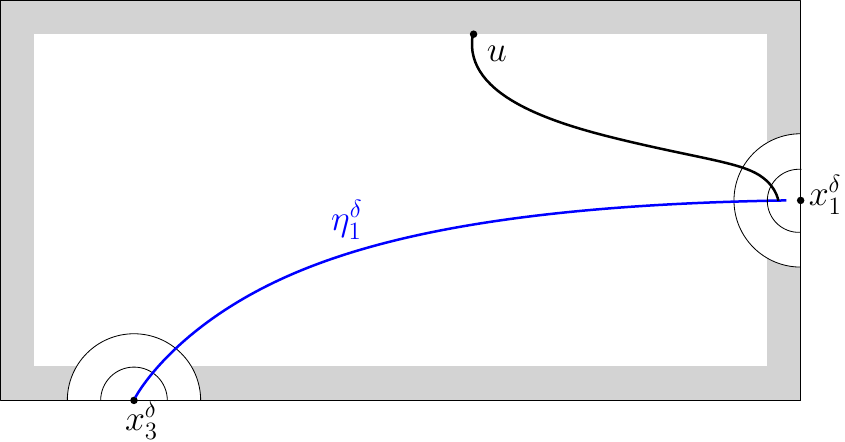}
\end{center}
\caption{The boundary branch $\eta_{u}^\delta$ reaches $S_1^\delta(r)$.}
\end{subfigure}
$\quad$
\begin{subfigure}[b]{0.45\textwidth}
\begin{center}
\includegraphics[width=\textwidth]{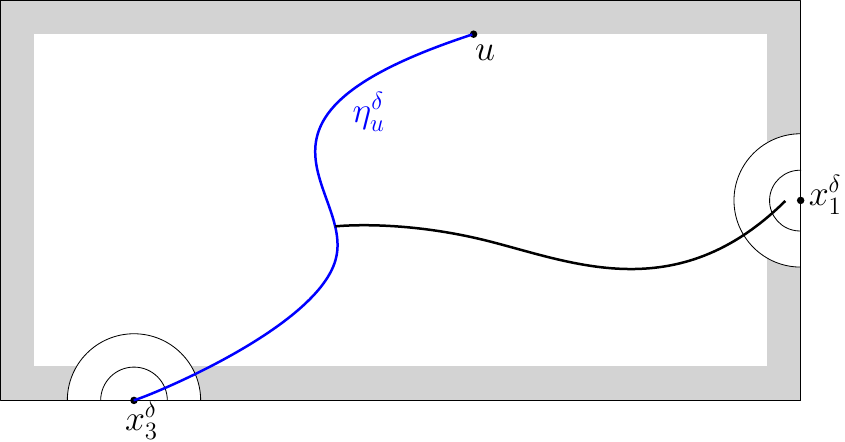}
\end{center}
\caption{The boundary branch $\eta_{u}^\delta$ does not reach $S_1^\delta(r)$.}
\end{subfigure}
\caption{\label{fig::ginterior_boundary}
In (a), the fat blue curve indicates the boundary branch $\eta_1^\delta$. 
In (b), the fat blue curve indicates the boundary branch $\eta_u^\delta$. 
The gray region is $\Omega_{r,\eps}^\delta\setminus(V_1(2r)\cup V_2(2r))$.}  
\end{figure}

Combining the two cases, we have
\begin{align}\label{eqn::g_re_to_g_aux1}
    g^\delta(u)=&\frac{\PP[\{u \rightsquigarrow e_3^\delta\}\cap\{x_1^{\delta,\circ}\rightsquigarrow e_3^\delta\}]}{\PP[\LA_1^\delta]}\\
    \le & \harmonic(\Omega^\delta\setminus V_1^\delta(r);u,S_1^\delta(r))+
    \harmonic(\Omega^\delta\setminus V_3^\delta(r);u,S_3^\delta(r)) \frac{\max_{w\in S_3^\delta(r)}\harmonic(\Omega^\delta;w,e_3^\delta)\times\harmonic(V_1^\delta(r);x_1^{\delta,\circ},S_1^\delta(r))}{\PP[\LA_1^\delta]}.\notag
\end{align}
Fix $v\in \Omega_{r,\eps}$ and let $v^\delta$ be the vertex in $\LV^\circ(\Omega^\delta_{r,\eps})$ nearest to $v$.
From~\eqref{eqn::discretePoisson_max_control},\eqref{eqn::bPoisson_cvg} and~\eqref{eqn::nharmonic_cvg}, the following ratios are bounded by a finite constant depending on $(\Omega;x_1,x_3;v;r)$:
\[\frac{\max_{w\in S_3^\delta(r)}\harmonic(\Omega^\delta;w,e_3^\delta)}{\harmonic(\Omega^\delta;v^\delta,e_3^\delta)}, \quad 
\frac{\harmonic(V_1^\delta(r);x_1^{\delta,\circ},S_1^\delta(r))}{\harmonic(\Omega^\delta;v^\delta,e_1^\delta)},\quad \frac{\harmonic(\Omega^\delta;v^\delta;e_1^\delta)\harmonic(\Omega^\delta;v^\delta;e_3^\delta)}{\PP[\LA_1^\delta]}.
\]
From weak Beurling-type estimate~\cite[Proposition 2.11]{ChelkakSmirnovDiscreteComplexAnalysis}, there exists a universal constant $\alpha\in (0,\infty)$ such that for $j=1,3$ and $\delta$ small enough
\[
\harmonic(\Omega^\delta\setminus V_j^\delta(r);u,S_j^\delta(r))\lesssim\left(\frac{\dist(u,\partial\Omega^\delta)}{\dist(u,S_j^\delta(r))}\right)^\alpha\lesssim\left(\frac{\eps}{r}\right)^{\alpha},\qquad \text{for all }u\in \partial \Omega^\delta_{r,\eps}\setminus (V_1(2r)\cup V_3(2r)). 
\]
Plugging these two estimates into~\eqref{eqn::g_re_to_g_aux1}, there exists a constant $C_{\eqref{eqn::g_re_to_g_aux2}}(r)\in(0,\infty)$ depending on $(\Omega;x_1,x_3;v;r)$ such that, for $\delta$ small enough, 
\begin{equation}\label{eqn::g_re_to_g_aux2}
    g^\delta(u)\leq C_{\eqref{eqn::g_re_to_g_aux2}}(r)\eps^{\alpha}, \qquad \text{for all }u\in \partial \Omega^\delta_{r,\eps}\setminus (V_1(2r)\cup V_3(2r)).
\end{equation}
For $u\in \partial\Omega^\delta_{r,\eps}\cap (V_1(2r)\cup V_3(2r))$, we have the trivial bound $g^\delta(u)\leq 1$.
Thus by the maximal principle for discrete harmonic functions, for $z^{\delta}\in \LV^{\circ}(\Omega^\delta_{r,\eps})$, we have
\begin{equation*}
    g^\delta_{r,\eps}(z^\delta)\leq C_{\eqref{eqn::g_re_to_g_aux2}}(r)\eps^{\alpha}\harmonic(\Omega^\delta_{r,\eps};z^\delta,\partial \Omega^\delta_{r,\eps}\setminus (V_1(2r)\cup V_3(2r)))+\harmonic(\Omega^\delta_{r,\eps};z^\delta,\partial\Omega^\delta_{r,\eps}\cap (V_1(2r)\cup V_3(2r))),
\end{equation*}
which implies~\eqref{eqn::g_re_to_g} by sending $\eps\to 0$ and then $r\to 0$. This completes the proof. 
\end{proof}

\section{Scaling limit of tripod}
\label{sec::tripod}
The goal of this section is to prove Theorem~\ref{thm::tripod}. 
In Section~\ref{prop::cvg_branch_eta1}, we derive the scaling limit of $\eta_1^{\delta}$ conditional on $\LA_1^{\delta}\cap\LA_2^{\delta}$ in Proposition~\ref{prop::cvg_branch_eta1}. 
In Section~\ref{subsec::tightness}, we prove the tightness of the tripod in Proposition~\ref{prop::tripod_tightness}. It follows from Proposition~\ref{prop::cvg_branch_eta1}. 
We emphasize that the conclusions and proofs in Sections~\ref{subsec::boundary_branch_cvg} and~\ref{subsec::tightness} also hold for $\Z^2$ lattice approximation.
In Section~\ref{subsec::prop14}, we complete the proof of Proposition~\ref{prop::chordalSLE2_Fomin}, it is a direct consequence of Proposition~\ref{prop::cvg_branch_eta1} and Proposition~\ref{prop::A1interior}.
In Section~\ref{subsec::gamma3}, we derive the scaling limit of $\gamma_3^{\delta}$ conditional on $\LA_1^{\delta}\cap\LA_2^{\delta}$ in Proposition~\ref{prop::gamma3_cvg}. 
In Section~\ref{subsec::tripod_proof}, we complete the proof of Theorem~\ref{thm::tripod}. The proof relies on Propositions~\ref{prop::cvg_branch_eta1} and~\ref{prop::gamma3_cvg} and tools of multi-sided radial SLE in Section~\ref{subsec::pre_SLE}. 
Finally, we extend the conclusion in Theorem~\ref{thm::tripod} for $\Z^2$ lattice approximation in Section~\ref{subsec::tripod_otherlattice}.

\subsection{Convergence of boundary branch}
\label{subsec::boundary_branch_cvg}
\begin{proposition}\label{prop::cvg_branch_eta1}
Fix a bounded $3$-polygon $(\Omega; x_1, x_2, x_3)$ and suppose $(\Omega^{\delta}; x_1^{\delta}, x_2^{\delta}, x_3^{\delta})$ is an approximation of $(\Omega; x_1, x_2, x_3)$ on $\delta\hexagon$ in Carath\'eodory sense. We assume further that $\partial\Omega^{\delta}$ converges to $\partial\Omega$ in Hausdorff distance~\eqref{eqn::boundary_cvg_Hausdorff}. 
Consider the UST in $\Omega^{\delta}$ with wired boundary condition and 
define $\LA_1^{\delta}=\{x_1^{\delta, \circ}\rightsquigarrow e_3^{\delta}\}$ and $\LA_2^{\delta}=\{x_2^{\delta, \circ}\rightsquigarrow e_3^{\delta}\}$ as in~\eqref{eqn::conditionalevent}. 
Let $\eta_1^{\delta}$ be the boundary branch starting from $x_1^{\delta, \circ}$. 
Suppose $\ell$ is chordal $\SLE_2\sim\chordalSLE^{(\kappa=2)}(\Omega; x_1, x_3)$. 
The law of $\eta_1^{\delta}$ conditional on $\LA_1^{\delta}\cap\LA_2^{\delta}$ converges weakly to the law of $\ell$ weighted by the following Radon-Nikodym derivative
\begin{equation}\label{eqn::RN_eta1_ell1}
\frac{\nharmonic(\Omega\setminus\ell; x_2, \ell)}{\chordalSLEexp^{(\kappa=2)}\left[\nharmonic(\Omega\setminus\ell; x_2, \ell)\right]}, 
\end{equation}
where $\nharmonic(\Omega\setminus\ell; x_2, \ell)$ is harmonic measure~\eqref{eqn::nharmonic_H}-\eqref{eqn::nharmonic_cov}.  
\end{proposition}

\begin{lemma}[\cite{ZhanLERW}]
\label{lem::boundary_branch_cvg1}
Fix a bounded $2$-polygon $(\Omega; x_1, x_3)$ and suppose $(\Omega^{\delta}; x_1^{\delta}, x_3^{\delta})$ is an approximation of $(\Omega; x_1, x_3)$ on $\delta\hexagon$ in Carath\'eodory sense. 
Consider UST on $\Omega^{\delta}$ with wired boundary condition. Let $\ell^{\delta}$ be the boundary branch starting from $x_1^{\delta, \circ}$ and define $\LA_1^{\delta}=\{x_1^{\delta, \circ}\rightsquigarrow e_3^{\delta}\}$ as in~\eqref{eqn::conditionalevent}. 
The law of $\ell^{\delta}$ conditional on $\LA_1^{\delta}$ converges weakly to the law of chordal $\SLE_2\sim\chordalSLE^{(\kappa=2)}(\Omega; x_1, x_3)$. 
\end{lemma}
\begin{proof}
This is a special case in~\cite{ZhanLERW}. This is a also a special case in~\cite{KarrilaUSTBranches}. 
\end{proof}

\begin{proof}[Proof of Proposition~\ref{prop::cvg_branch_eta1}]
For any bounded continuous function $F$ on curves, we have
\begin{align}
\E\left[F(\eta_1^{\delta})\cond \LA_1^{\delta}\cap \LA_2^{\delta}\right]=&\frac{1}{\PP[\LA_1^{\delta}\cap\LA_2^{\delta}]}\E\left[F(\eta_1^{\delta})\one\{ \LA_1^{\delta}\cap \LA_2^{\delta}\}\right]\notag\\
=&\frac{1}{\PP[\LA_2^{\delta}\cond \LA_1^{\delta}]}\E\left[F(\eta_1^{\delta})\one\{\LA_2^{\delta}\}\cond \LA_1^{\delta}\right]\notag\\
=&\frac{1}{\E[\harmonic(\Omega^{\delta}\setminus\ell^{\delta};x_2^{\delta},\ell^{\delta})\cond \LA_1^{\delta}]}
\E\left[F(\ell^{\delta}) \harmonic(\Omega^{\delta}\setminus\ell^{\delta}; x_2^{\delta, \circ}, \ell^{\delta})\cond \LA_1^{\delta}\right],
\label{eqn::boundary_branch_cvg_aux1}
\end{align}
where $\ell^{\delta}$ is defined as in Lemma~\ref{lem::boundary_branch_cvg1}. Let us derive the scaling limit of the terms in RHS of~\eqref{eqn::boundary_branch_cvg_aux1}. We fix an interior point $v\in\Omega$ and suppose $v^{\delta}$ is the vertex in $\LV^{\circ}(\Omega^{\delta})$ that is nearest to $v$.
From Lemma~\ref{lem::boundary_branch_cvg1}, the law of $\ell^{\delta}$ conditional on $\LA_1^{\delta}$ converges weakly to the law of $\ell$. 
Suppose $\{\delta_n\}_n$ is any subsequence such that $\delta_n\to 0$. 
By Skorokhod's representation theorem, we can couple $\{\ell^{\delta_n}\}_{n}$ and $\ell$ together such that $\ell^{\delta_n}$ converges to $\ell$ almost surely. 
We denote such coupling still by $\PP$. 
In this coupling, the convergence~\eqref{eqn::nharmonic_cvg} asserts the almost sure convergence:
\begin{align}\label{eqn::boundary_branch_cvg_as}
\lim_n \frac{\harmonic(\Omega^{\delta_n}\setminus\ell^{\delta_n}; x_2^{\delta_n, \circ}, \ell^{\delta_n})}{\harmonic(\Omega^{\delta_n}; v^{\delta_n}, e_2^{\delta_n})}=\sqrt{3}\frac{\nharmonic(\Omega\setminus\ell; x_2, \ell)}{\Poisson(\Omega; v, x_2)}.
\end{align}
We will improve the almost sure convergence to $L^1$ convergence  in Lemma~\ref{lem::boundary_branch_cvg_L1}:
\begin{equation}\label{eqn::boundary_branch_cvg_L1}
\lim_{n}\E\left[F(\ell^{\delta_n}) \frac{\harmonic(\Omega^{\delta_n}\setminus\ell^{\delta_n}; x_2^{\delta_n, \circ}, \ell^{\delta_n})}{\harmonic(\Omega^{\delta_n}; v^{\delta_n}, e_2^{\delta_n})}\cond \LA_1^{\delta_n}\right]=\E\left[F(\ell)\sqrt{3}\frac{\nharmonic(\Omega\setminus\ell; x_2, \ell)}{\Poisson(\Omega; v, x_2)}\right]. 
\end{equation}
Taking $F=1$, we have
\begin{align}\label{eqn::boundary_branch_cvg_L1_conseq}
\lim_{n}\frac{\E[\harmonic(\Omega^{\delta_n}\setminus\ell^{\delta_n};x_2^{\delta_n},\ell^{\delta_n})\cond \LA_1^{\delta_n}]}{\harmonic(\Omega^{\delta_n};v^\delta,e_2^{\delta_n})}
=\frac{\chordalSLEexp^{(\kappa=2)}\left[\sqrt{3}\nharmonic(\Omega\setminus\ell; x_2, \ell)\right]}{\Poisson(\Omega;v,x_2)}.
\end{align}
Plugging~\eqref{eqn::boundary_branch_cvg_L1} and~\eqref{eqn::boundary_branch_cvg_L1_conseq}  into~\eqref{eqn::boundary_branch_cvg_aux1}, we obtain~\eqref{eqn::RN_eta1_ell1} as desired. 
\end{proof}

\begin{lemma}\label{lem::boundary_branch_cvg_L1}
     Assume the same notation as in Proof of Proposition~\ref{prop::cvg_branch_eta1}. The $L^1$-convergence~\eqref{eqn::boundary_branch_cvg_L1} holds.
\end{lemma}

\begin{proof}
Note that $\nharmonic(\Omega\setminus\ell; x_2, \ell)$ is a continuous function on $\ell$, but it is not bounded because it can be large when $\ell$ is close to $x_2$. For the convergence~\eqref{eqn::boundary_branch_cvg_L1}, we need to control the distance between $\ell$ and $x_2$. For $\eps>0$, we write 
\begin{align}\label{eqn::boundary_branch_cvg_aux5}
   &\E\left[F(\ell^{\delta_n}) \frac{\harmonic(\Omega^{\delta_n}\setminus\ell^{\delta_n}; x_2^{\delta_n, \circ}, \ell^{\delta_n})}{\harmonic(\Omega^{\delta_n}; v^{\delta_n}, e_2^{\delta_n})}\cond \LA_1^{\delta_n}\right]\notag\\
   =&\E\left[F(\ell^{\delta_n}) \frac{\harmonic(\Omega^{\delta_n}\setminus\ell^{\delta_n}; x_2^{\delta_n, \circ}, \ell^{\delta_n})}{\harmonic(\Omega^{\delta_n}; v^{\delta_n}, e_2^{\delta_n})}\one\{\ell^{\delta_n}\cap B(x_2, \eps)=\emptyset\}\cond \LA_1^{\delta_n}\right]\notag\\
   &+\frac{\E[\harmonic(\Omega^{\delta_n}\setminus\ell^{\delta_n};x_2^{\delta_n},\ell^{\delta_n})\cond \LA_1^{\delta_n}]}{\harmonic(\Omega^{\delta_n}; v^{\delta_n}, e_2^{\delta_n})}\E\left[F(\eta_1^{\delta_n}) \one\{\eta_1^{\delta_n}\cap B(x_2, \eps)\neq\emptyset\}\cond \LA_1^{\delta_n}\cap\LA_2^{\delta_n}\right]. 
\end{align}
We claim that there exists a constant $C_{\eqref{eqn::boundary_branch_cvg_aux5_good}}\in (0,\infty)$ depending on $(\Omega; x_1, x_2, x_3; v; \eps)$ such that 
\begin{align}\label{eqn::boundary_branch_cvg_aux5_good}
\frac{\harmonic(\Omega^{\delta_n}\setminus\ell^{\delta_n}; x_2^{\delta_n, \circ}, \ell^{\delta_n})}{\harmonic(\Omega^{\delta_n}; v^{\delta_n}, e_2^{\delta_n})}\one\{\ell^{\delta_n}\cap B(x_2, \eps)=\emptyset\}\le C_{\eqref{eqn::boundary_branch_cvg_aux5_good}}, \qquad\text{for $n$ large enough}; 
\end{align}
and 
\begin{equation}\label{eqn::boundary_branch_cvg_aux5_bad}
   \lim_{\eps\to 0}\sup_{n\ge 1}\PP\left[\eta_1^{\delta_n}\cap B(x_2, \eps)\neq\emptyset\cond \LA_1^{\delta_n}\cap\LA_2^{\delta_n}\right]=0.
\end{equation}
Combining~\eqref{eqn::boundary_branch_cvg_as}, \eqref{eqn::boundary_branch_cvg_aux5_good}, \eqref{eqn::boundary_branch_cvg_aux5_bad}, we obtain the $L^1$-convergence~\eqref{eqn::boundary_branch_cvg_L1} as desired.  
\medbreak

It remains to show~\eqref{eqn::boundary_branch_cvg_aux5_good} and~\eqref{eqn::boundary_branch_cvg_aux5_bad}.
We use the same notation as in Section~\ref{subsec::proba_tight}.
Recall that, for $j=1,2,3$, let $V_j^\delta(r)$ be the maximal domain contained in $B(x_j,r)\cap \Omega^\delta$ such that $\partial V_j^\delta(r)$ consists of the edges of $\delta \hexagon$. We set $S_j^{\delta}(r)=\LV^{\partial}(V_j^{\delta}(r))\setminus \LV^{\partial}(\Omega^{\delta})$.
First, we prove~\eqref{eqn::boundary_branch_cvg_aux5_good}.
Given $(\LA_1^{\delta_n},\ell^{\delta_n})$ and on the event $\{\ell^{\delta_n}\cap B(x_2,\eps)=\emptyset\}$, we have $\dist(x_2^{\delta_n},\ell^{\delta_n})> \eps/2$ for $n$ large enough.
If a simple random walk starting from $x_2^{\delta_n, \circ}$ exits $\Omega^{\delta_n}\setminus\ell^{\delta_n}$ through $\ell^{\delta_n}$, 
it has to first exit $V_2^{\delta_n}(\eps/2)$ through $S_2^{\delta_n}(\eps/2)$. 
Thus,
\begin{equation*}
   \harmonic(\Omega^{\delta_n}\setminus \ell^{\delta_n}; x_2^{\delta_n,\circ}, \ell^{\delta_n})\le \harmonic(V_2^{\delta_n}(\eps/2);x_2^{\delta_n,\circ},S_2^{\delta_n}(\eps/2)).
\end{equation*}
From~\eqref{eqn::nharmonic_cvg}, the following ratio is bounded by a constant depending on $(\Omega;x_1,x_2,x_3;v;\eps)$: 
\[
   \frac{\harmonic(V_2^{\delta_n}(\eps/2);x_2^{\delta_n,\circ},S_2^{\delta_n}(\eps/2))}{\harmonic(\Omega^{\delta_n};v^{\delta_n},e_2^{\delta_n})}. 
\]
This completes the proof of~\eqref{eqn::boundary_branch_cvg_aux5_good}.
\medbreak

Next, we prove~\eqref{eqn::boundary_branch_cvg_aux5_bad}. 
Recall that $V(r)=B(x_1,r)\cup B(x_2,r)\cup B(x_3,r)$ as in~\eqref{eqn::deltacube_control_boundary1}.
Fix $r_0>0$ such that $B(x_j,2r_0)\cap B(x_k,2r_0)=\emptyset$ for $j\neq k$ and let $v\in \Omega\setminus(\cup_{j=1}^3 B(x_j,r_0))$ and suppose $0<\eps<r\le r_0$. 
Note that
\[\{\eta_1^{\delta_n}\cap B(x_2,\eps)\neq \emptyset\}\subset \{\trifurcation^{\delta_n}\in V(r)\}\cup\left(\{\trifurcation^{\delta_n}\notin V(r)\}\cap\{\eta_1^{\delta_n}\cap B(x_2,\eps)\neq \emptyset\}\right).\]
We have the following three estimates.
\begin{itemize}
   \item From Lemma~\ref{lem::proba_tight1}, the following ratio is uniformly bounded by a finite constant depending on $(\Omega; x_1, x_2, x_3; v)$: 
   \begin{align}\label{eqn::boundary_branch_Dcvg_aux91}
   \frac{\PP[\LA_1^{\delta_n}\cap \LA_2^{\delta_n}]}{\prod_{j=1}^3\harmonic(\Omega^{\delta_n}; v^{\delta_n}, e_j^{\delta_n})}.
   \end{align}
\item From~\eqref{eqn::deltacube_control_boundary1}, there exist a universal constant $\alpha\in (0,\infty)$ (from weak Beurling-type estimate~\cite[Proposition 2.11]{ChelkakSmirnovDiscreteComplexAnalysis}) and a constant $C_{\eqref{eqn::boundary_branch_Dcvg_aux5}}\in (0,\infty)$ depending on $(\Omega;x_1,x_2,x_3;v)$ such that, for $n$ large enough, 
   \begin{equation}\label{eqn::boundary_branch_Dcvg_aux5}
       \frac{\PP\left[\{\trifurcation^{\delta_n}\in V(r)\}\cap\LA_1^{\delta_n}\cap\LA_2^{\delta_n}\right]}{\prod_{j=1}^3\harmonic(\Omega^{\delta_n}; v^{\delta_n}, e_j^{\delta_n})}\le C_{\eqref{eqn::boundary_branch_Dcvg_aux5}}r^\alpha, \quad \text{ for all }0<r\le r_0.
   \end{equation}
   \item We claim that there exist a universal constant $\alpha\in (0,\infty)$ (from weak Beurling-type estimate~\cite[Proposition 2.11]{ChelkakSmirnovDiscreteComplexAnalysis}) and a constant $C_{\eqref{eqn::boundary_branch_Dcvg_aux6}}(r)\in (0,\infty)$ depending on $(\Omega; x_1,x_2,x_3;v;r)$ such that, for $n$ large enough, 
   \begin{equation}\label{eqn::boundary_branch_Dcvg_aux6}
       \frac{\PP\left[\{\trifurcation^{\delta_n}\notin V(r)\}\cap\{\eta_1^{\delta_n}\cap B(x_2,\eps)\neq \emptyset\}\cap\LA_1^{\delta_n}\cap\LA_2^{\delta_n}\right]}{\prod_{j=1}^3\harmonic(\Omega^{\delta_n}; v^{\delta_n}, e_j^{\delta_n})}\le C_{\eqref{eqn::boundary_branch_Dcvg_aux6}}(r)\eps^\alpha.
   \end{equation}
  Given $(\LA_1^{\delta_n}, \eta_1^{\delta_n})$, denote $\Omega_1^{\delta_n}=\Omega^{\delta_n}\setminus \eta_1^{\delta_n}$.
   The event $\{\trifurcation^{\delta_n}\notin V(r)\}\cap \LA_2^{\delta_n}$ implies that the simple random walk from $x_2^{\delta_n,\circ}$ has to exit $V_2^{\delta_n}(r) \cap \Omega_1^{\delta_n}$ through $S_2^{\delta_n}(r)\cap \Omega_1^{\delta_n}$.
   The probability of this event is bounded by \[\harmonic(V_2^{\delta_n}(r)\cap \Omega_1;x_2^{\delta_n,\circ},S_2^{\delta_n}(r)\cap \Omega_1^{\delta_n})\le \harmonic(V_2^{\delta_n}(r);x_2^{\delta_n,\circ},S_2^{\delta_n}(r)).\] Thus,
      \begin{align}\label{eqn::boundary_branch_Dcvg_aux7}
      &\PP\left[\{\trifurcation^{\delta_n}\notin V(r)\}\cap\{\eta_1^{\delta_n}\cap B(x_2,\eps)\neq \emptyset\}\cap\LA_1^{\delta_n}\cap\LA_2^{\delta_n}\right]\notag\\
      \le& \PP\left[\{\eta_1^{\delta_n}\cap B(x_2,\eps)\neq\emptyset\}\cap \LA_1^{\delta_n}\right]\harmonic(V_2^{\delta_n}(r);x_2^{\delta_n,\circ},S_2^{\delta_n}(r)). 
     \end{align}
   The event $\{\eta_1^{\delta_n}\cap B(x_2,\eps)\neq \emptyset\}\cap \LA_1^{\delta_n}$ implies that the simple random walk from $x_1^{\delta_n,\circ}$ has to exit $V_1^{\delta_n}(r)$ through $S_1^{\delta_n}(r)$,  hit $S_2^{\delta_n}(\eps)$, and then exit $V^{\delta_n}_2(r)$ through $S_2^{\delta_n}(r)$ and exit $\Omega^{\delta_n}$ via $e_3^{\delta_n}$. 
   Thus by the Markov property of random walk, we have 
\begin{align*}
   &\PP\left[\{\eta_1^{\delta_n}\cap B(x_2,\eps)\neq\emptyset\}\cap\LA_1^{\delta_n}\right]\notag\\
\le& \harmonic(V_1^{\delta_n}(r);x_1^{\delta_n,\circ}, S^{\delta_n}_1(r))\times\max_{u\in S_1^{\delta_n}(r)}\harmonic(\Omega^{\delta_n}\setminus V_2^{\delta_n}(\eps);u, S_2^{\delta_n}(\eps))\times\max_{w\in S_2^{\delta_n}(r)}\harmonic(\Omega^{\delta_n};w,e_3^{\delta_n}).
\end{align*}
Plugging into~\eqref{eqn::boundary_branch_Dcvg_aux7}, we have
\begin{align}\label{eqn::boundary_branch_Dcvg_aux9}
      &\PP\left[\{\trifurcation^{\delta_n}\notin V(r)\}\cap\{\eta_1^{\delta_n}\cap B(x_2,\eps)\neq \emptyset\}\cap\LA_1^{\delta_n}\cap\LA_2^{\delta_n}\right]\notag\\
\le& \prod_{j=1}^2\harmonic(V_j^{\delta_n}(r);x_j^{\delta_n,\circ}, S_j^{\delta_n}(r))
\times\max_{u\in S_1^{\delta_n}(r)}\harmonic(\Omega^{\delta_n}\setminus V_2^{\delta_n}(\eps);u, S_2^{\delta_n}(\eps))
\times\max_{w\in S_2^{\delta_n}(r)}\harmonic(\Omega^{\delta_n};w,e_3^{\delta_n}).
\end{align}
From~\eqref{eqn::nharmonic_cvg} and~\eqref{eqn::discretePoisson_max_control}, the following ratios are bounded by a finite constant depending on $(\Omega; x_1, x_2, x_3; v; r)$: 
\begin{equation}\label{eqn::eqn::boundary_branch_Dcvg_boundnessaux}
\frac{\harmonic(V_1^{\delta_n}(r);x_1^{\delta_n,\circ}, S^{\delta_n}_1(r))}{\harmonic(\Omega^{\delta_n}; v^{\delta_n},e_1^{\delta_n})}, \qquad \frac{\harmonic(V_2^{\delta_n}(r);x_2^{\delta_n,\circ}, S^{\delta_n}_2(r))}{\harmonic(\Omega^{\delta_n}; v^{\delta_n},e_2^{\delta_n})}, \qquad \frac{\max_{w\in S_2^{\delta_n}(r)}\harmonic(\Omega^{\delta_n};w,e_3^{\delta_n})}{\harmonic(\Omega^{\delta_n};v^{\delta_n},e_3^{\delta_n})}.
\end{equation}
From weak Beurling-type estimate~\cite[Proposition 2.11]{ChelkakSmirnovDiscreteComplexAnalysis}, there exists a universal constant $\alpha\in (0,\infty)$ such that, for $n$ large enough, 
\begin{equation*}
   \max_{u\in S_1^{\delta_n}(r)}\harmonic(\Omega^{\delta_n}\setminus V_2^{\delta_n}(\eps);u, S_2^{\delta_n}(\eps))\lesssim \eps^\alpha.
\end{equation*}
Plugging these into~\eqref{eqn::boundary_branch_Dcvg_aux9}, we obtain~\eqref{eqn::boundary_branch_Dcvg_aux6} as desired.
\end{itemize}
Combining the above three estimates, for any $0<\eps<r\le r_0$, there exist a universal constant  $\alpha\in(0,\infty)$ (from weak Beurling-type estimate~\cite[Proposition 2.11]{ChelkakSmirnovDiscreteComplexAnalysis}), a constant $C_{\eqref{eqn::boundary_branch_Dcvg_aux5}}$ depending on $(\Omega;x_1,x_2,x_3;v)$ and a constant $C_{\eqref{eqn::boundary_branch_cvg_aux6_bad}}(r)\in (0,\infty)$ depending on $(\Omega; x_1, x_2, x_3; v;r)$ such that
\begin{align}\label{eqn::boundary_branch_cvg_aux6_bad}
   \sup_{n\ge 1}\PP\left[\eta_1^{\delta_n}\cap B(x_2, \eps)\neq\emptyset\cond \LA_1^{\delta_n}\cap\LA_2^{\delta_n}\right]\le C_{\eqref{eqn::boundary_branch_Dcvg_aux5}}r^\alpha+ C_{\eqref{eqn::boundary_branch_cvg_aux6_bad}}(r)\eps^{\alpha}. 
\end{align}
By first sending $\eps \to 0$ and then sending $r\to 0$, we obtain~\eqref{eqn::boundary_branch_cvg_aux5_bad} as desired.
\end{proof}
\subsection{Tightness of tripod}
\label{subsec::tightness}

\begin{proposition}\label{prop::tripod_tightness}
Assume the same setup as in Theorem~\ref{thm::tripod}. 
Denote by $\LL^{\delta}(\Omega; x_1, x_2, x_3)$ the law of the tripod $(\gamma_1^{\delta}, \gamma_2^{\delta}, \gamma_3^{\delta})$ conditional on $\LA_1^{\delta}\cap\LA_2^{\delta}$. 
Then the family $\{\LL^{\delta}(\Omega; x_1, x_2, x_3)\}_{\delta>0}$ is tight.
 Furthermore, for any subsequential limit $(\gamma_1, \gamma_2, \gamma_3)$, there exists $\trifurcation\in\Omega$ such that $(\gamma_1, \gamma_2, \gamma_3)\in\chamber(\Omega; x_1, x_2, x_3; \trifurcation)$. 
\end{proposition}

\begin{proof}
Suppose that $\{\delta_n\}_n$ is any subsequence such that $\delta_n\to 0$ as $n\to\infty$. It suffices to prove that there exists a subsequence of $\{\delta_n\}_n$, which we still denote by $\{\delta_n\}_n$, such that the conditional law of $\{(\gamma_1^{\delta_n},\gamma_2^{\delta_n},\gamma_3^{\delta_n})\}_{n}$ converges. 
From Proposition~\ref{prop::cvg_branch_eta1}, the law of $\eta_j^{\delta}$ conditional on $\LA_1^{\delta}\cap\LA_2^{\delta}$ converges weakly for $j=1,2$. Denote by $\eta_j$ the limit for $j=1,2$. 
By Skorokhod's representation theorem, we can couple $\{(\eta_1^{\delta_n},\eta_2^{\delta_n})\}_{n
}$ and $(\eta_1,\eta_2)$ together, such that $\eta_j^{\delta_n}$ converges to $\eta_j$ in metric~\eqref{eqn::curves_metric} almost surely for $j=1,2$. 
We denote this coupling still by $\PP$. 
Since $(\gamma_1^{\delta_n},\gamma_2^{\delta_n},\gamma_3^{\delta_n})$ is a measurable function of $(\eta_1^{\delta_n},\eta_2^{\delta_n})$, 
the coupling $\PP$ is also a coupling of $\{(\gamma_1^{\delta_n},\gamma_2^{\delta_n},\gamma_3^{\delta_n})\}_{n}$.

We will prove that $(\gamma_1^{\delta_n},\gamma_2^{\delta_n},\gamma_3^{\delta_n})$ converges almost surely under $\PP$, 
and the limit $(\gamma_1, \gamma_2, \gamma_3)$ belongs to $\chamber(\Omega; x_1, x_2, x_3; \trifurcation)$ for some $\trifurcation\in\Omega$. 
For $j=1,2$, we fix a parameterization such that $\eta_j^{\delta_n}$ converges to $\eta_j$ uniformly on $[0,1]$. 
Denote by $T_1^{(n)}$ and $T_2^{(n)}$ the time such that $\eta_1^{\delta_n}(T_1^{(n)})=\eta_2^{\delta_n}(T_2^{(n)})=\trifurcation^{\delta_n}$. 
Denote by $T_1$ and $T_2$ the time such that $\eta_1(T_1)=\eta_2(T_2)$ and $\eta_1([0,T_1))\cap\eta_2([0,T_2))=\emptyset$.  
See Figure~\ref{fig::tightness}.
Denote by $E_j$ the event that $T_j^{\delta_n}$ does not converge to $T_j$ for $j=1,2$. It suffices to prove that $\PP[E_1\cup E_2]=0$. 
We only prove $\PP[E_1]=0$ since the proof for $\PP[E_2]=0$ is the same.

\begin{figure}[ht!]
   \begin{center}
   \includegraphics[width=0.5\textwidth]{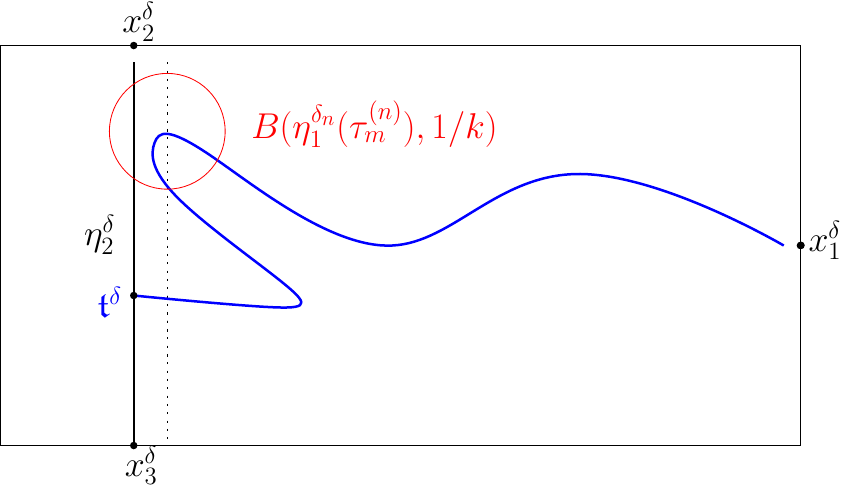}
   \end{center}
   \caption{The fat black curve indicates the boundary branch $\eta_2^{\delta_n}$ from $x_2^{\delta_n, \circ}$ to $e_3^{\delta_n}$. The fat blue curve indicates the boundary branch from $x_1^{\delta_n, \circ}$ to $\eta_2^{\delta_n}$. }
   \label{fig::tightness}
  \end{figure}

For each $m\ge 1$, denote by $\tau_m^{(n)}$ the first time that $\eta_1^{\delta_n}$ hits the $\frac{1}{m}$-neighbourhood of $\eta_2^{\delta_n}$. 
For $k\le m$, we define 
\[\sigma_{m,k}^{(n)}=\inf\left\{t\ge \tau_m^{(n)}: \eta_1^{\delta_n}(t)\notin B\left(\eta_1^{\delta_n}(\tau_m^{(n)}),\frac{1}{k}\right)\right\}.\]
Note that the event $E_1$ implies that 
there exists $k$ large enough such that
for any $m\ge k$, there exists $N_m$ such that $\sigma_{m,k}^{(n)}<T_1^{(n)}$ for all $n\ge N_m$.
In other words,
\[E_1\subset\bigcup_{k\ge 1}\bigcap_{m\ge k}\bigcup_{N\ge 1}\bigcap_{n\ge N}\{\sigma_{m,k}^{(n)}<T_1^{(n)}\}.\]
Thus, we have
\begin{equation}\label{eqn::tightness_aux1}
   \PP[E_1]\le \sup_{k\ge 1}\limsup_{m\to\infty}\limsup_{n\to\infty}\PP\left[\sigma_{m,k}^{(n)}<T_1^{(n)}\right].
\end{equation}

Let us evaluate $\PP\left[\sigma_{m,k}^{(n)}<T_1^{(n)}\right]$.
Denote by $\Omega_2^{\delta_n}$ the domain  $\Omega^{\delta_n}\setminus \eta_2^{\delta_n}$, and denote by $U^{\delta_n}_k$ the maximal domain contained in $\Omega^{\delta_n}_2\cap B(\eta_1^{\delta_n}(\tau_m^{(n)}),1/k)$ such that $\eta_1^{\delta_n}(\tau_m^{(n)})\in U_k^{\delta_n}$ and $\partial U^{\delta_n}_k$ consist of the edges of $\delta\hexagon$.
Then the event $\{\sigma_{m,k}^{(n)}<T_1^{(n)}\}$ implies that 
the random walk from $\eta_1^{\delta_n}(\tau_m^{(n)})$ has to exit $U_k^{\delta_n}$ through $\partial U_k^{\delta_n}\setminus \eta_2^{\delta_n}$.
The probability of this event is bounded by \[\harmonic(U_k^{\delta_n};\eta_1^{\delta_n}(\tau_m^{(n)}),\partial U_k^{\delta_n}\setminus \eta_2^{\delta_n}).\]
From weak Beurling-type estimate~\cite[Proposition~2.11]{ChelkakSmirnovDiscreteComplexAnalysis}, there exists a universal constant $\alpha\in(0,\infty)$ such that
\begin{equation*}
   \harmonic(U_k^{\delta_n};\eta_1^{\delta_n}(\tau_m^{(n)}),\partial U_k^{\delta_n}\setminus \eta_2^{\delta_n})\lesssim\left(\frac{\dist(\eta_1^{\delta_n}(\tau_m^{(n)}),\partial U_k^{\delta_n}\setminus \eta_2^{\delta_n})}{\dist(\eta_1^{\delta_n}(\tau_m^{(n)}),\partial U_k^{\delta_n})}\right)^\alpha \lesssim\left(\frac{k}{m}\right)^\alpha,\qquad\text{for $n$ large enough}.
\end{equation*}
Thus,  
\begin{equation*}
   \limsup_{n\to \infty}\PP\left[\sigma_{m,k}^{(n)}<T_1^{(n)}\right]\lesssim \left(\frac{k}{m}\right)^\alpha.
\end{equation*}
Plugging into~\eqref{eqn::tightness_aux1}, we obtain $\PP[E_1]=0$ as desired and complete the proof.
\end{proof}

In the following two lemmas, we give a first description of the scaling limit of the tripod in Theorem~\ref{thm::tripod}. 
Let us point out that Lemmas~\ref{lem::cvg_branch_gamma2} and~\ref{lem::tripod_ab} will not be used in the proof of Theorem~\ref{thm::tripod} and they will only be used in Section~\ref{subsec::tripod_otherlattice}. 
The advantage for this first description is that the conclusion and the proof also works for $\Z^2$ lattice approximation. 
The disadvantage is that it is hard to tell the marginal distribution of the trifurcation in such description. 
We will derive the marginal law of the trifurcation in Proposition~\ref{prop::gamma3_cvg} using the observable in Proposition~\ref{prop::observable_cvg}.

\begin{lemma}\label{lem::cvg_branch_gamma2}
Fix a bounded $3$-polygon $(\Omega; x_1, x_2, x_3)$ and suppose $(\Omega^{\delta}; x_1^{\delta}, x_2^{\delta}, x_3^{\delta})$ is an approximation of $(\Omega; x_1, x_2, x_3)$ on $\delta\hexagon$ in Carath\'eodory sense. We assume further that $\partial\Omega^{\delta}$ converges to $\partial\Omega$ in Hausdorff distance~\eqref{eqn::boundary_cvg_Hausdorff}. 
Consider the UST in $\Omega^{\delta}$ with wired boundary condition. Let $\gamma_2^{\delta}$ be the boundary branch starting from $x_2^{\delta, \circ}$ and denote by $\{x_2^{\delta, \circ}\rightsquigarrow (x_3^{\delta}x_1^{\delta})\}$ the event that the branch $\gamma_2^{\delta}$ exits the boundary through $(x_3^{\delta}x_1^{\delta})$. 
Then the law of $\gamma_2^{\delta}$ conditional on $\{x_2^{\delta, \circ}\rightsquigarrow (x_3^{\delta}x_1^{\delta})\}$ converges weakly to chordal $\SLE_2(-2;-2)$ in $\Omega$ from $x_2$ to $(x_3x_1)$ with force points $(x_1;x_3)$.\footnote{For $y\in (x_3x_1)$, suppose $\gamma_2$ is chordal $\SLE_2(-2;-2)$ in $\Omega$ from $x_2$ to $y$ with force points $(x_1;x_3)$. As $\kappa-6=-4$ when $\kappa=2$, such process is target-invariant: denote by $\tau$ the first hitting time of $(x_3x_1)$, the law of $\gamma_2[0,\tau]$ is independent of $y\in (x_3x_1)$. Thus we call such process $\SLE_2(-2;-2)$ in $\Omega$ from $x_2$ to $(x_3x_1)$ with force points $(x_1;x_3)$.}
\end{lemma}

\begin{proof}
This is a degenerate case of~\cite[Theorem~1.5]{HanLiuWuUST} and we give a brief strategy below. 
Suppose $(\Omega; x_1, \hat{x}_1, x_2, \hat{x}_3, x_3)$ is a 5-polygon and suppose 
$(\Omega^{\delta}; x_1^{\delta}, \hat{x}_1^{\delta}, x_2^{\delta}, \hat{x}_3^{\delta}, x_3^{\delta})$ is an approximation of $(\Omega; x_1, \hat{x}_1, x_2, \hat{x}_3, x_3)$ on $\delta\hexagon$ in Carath\'eodory sense. Consider UST in $\Omega^{\delta}$ with alternating boundary conditions: $(\hat{x}_1^{\delta}\hat{x}_3^{\delta})$ is wired and $(x_3^{\delta}x_1^{\delta})$ is wired (but $(\hat{x}_1^{\delta}\hat{x}_3^{\delta})$ and $(x_3^{\delta}x_1^{\delta})$ are not wired together). Let $\hat{\gamma}_2^{\delta}$ be the boundary branch starting from $x_2^{\delta, \circ}$ and denote by $\{x_2^{\delta, \circ}\rightsquigarrow (x_3^{\delta}x_1^{\delta})\}$ the event that the branch $\hat{\gamma}_2^{\delta}$ exits the boundary through $ (x_3^{\delta}x_1^{\delta})$. From~\cite[Theorem~1.5]{HanLiuWuUST} and the analysis in~\cite[Section~5]{HanLiuWuUST}, the law of $\hat{\gamma}^{\delta}$ conditional on $\{x_2^{\delta, \circ}\rightsquigarrow (x_3^{\delta}x_1^{\delta})\}$ converges weakly to the law of $\hat{\gamma}_2$ which is chordal $\SLE_2(-1,-1;-1,-1)$ in $\Omega$ from $x_2$ to $(x_3x_1)$ with force points $(x_1, \hat{x}_1; \hat{x}_3, x_3)$. \footnote{For $y\in (x_3x_1)$, suppose $\hat{\gamma}_2$ is chordal $\SLE_2(-1,-1;-1,-1)$ in $\Omega$ from $x_2$ to $y$ with force points $(x_1, \hat{x}_1; \hat{x}_3, x_3)$. As $\kappa-6=-4$ when $\kappa=2$, such process is target-invariant: denote by $\tau$ the first hitting time of $(x_3x_1)$, the law of $\hat{\gamma}_2[0,\tau]$ is independent of $y\in (x_3x_1)$. Thus we call such process $\SLE_2(-1,-1;-1,-1)$ in $\Omega$ from $x_2$ to $(x_3x_1)$ with force points $(\hat{x}_1, x_1;\hat{x}_3, x_3)$. Although, the statement in~\cite[Theorem~1.5]{HanLiuWuUST} also requires that $\partial\Omega$ is $C^1$, but this is not necessary. As discussed in~\cite{HanLiuWuUST}, the smoothness of $\partial\Omega$ is only used to guarantee the tightness of the Peano curves. In the setup of the boundary branch, the tightness does not require such smoothness: the tightness of the laws of $\{\gamma_2^{\delta}\}_{\delta>0}$ can be proved similarly as in Proposition~\ref{prop::tripod_tightness}.}
We have the following two observations. 
\begin{itemize}
\item The law of $\hat{\gamma}_2$ converges weakly to chordal $\SLE_2(-2;-2)$ in $\Omega$ from $x_2$ to $(x_3x_1)$ with force points $(x_1; x_3)$ as $\hat{x}_1\to x_1$ and $\hat{x}_3\to x_3$. 
\item Let us compare the laws of $\gamma_2^{\delta}$ and of $\hat{\gamma}_2^{\delta}$. 
We will construct a coupling $\LQ^{\delta}$ of $\gamma_2^{\delta}$ and $\hat{\gamma}_2^{\delta}$, and prove that there exists a constant $C_{\eqref{eqn::coupling}}\in (0,\infty)$ depending on $(\Omega;x_1,x_2,x_3)$ and a universal constant $\alpha\in(0,\infty)$ (from weak Beurling-type estimate) such that
\begin{equation}\label{eqn::coupling}
\LQ^{\delta}[\gamma_2^\delta\neq \hat{\gamma}_2^\delta]\le C_{\eqref{eqn::coupling}}\left(|x_1-\hat x_1|+|x_3-\hat x_3|\right)^{\alpha}.
\end{equation}
\end{itemize}
Assume we have obtained~\eqref{eqn::coupling}. Suppose $\gamma_2$ is any subsequential limit of $\{\gamma_2^\delta\}_{\delta>0}$ and suppose $\gamma_2^{\delta_n}$ converges to $\gamma_2$ in law as $n\to\infty$. By choosing a subsequence, we may assume that $\left(\gamma_2^{\delta_n},\hat\gamma_2^{\delta_n}\right)\sim\LQ^{\delta_n}$ converges to $(\gamma_2,\hat\gamma_2)\sim\LQ$. Then, for any bounded continuous function $F$ on the curve space, we have
\[\E[F(\gamma_2)]=\lim_{\hat x_1\to x_1,\hat x_3\to x_3}\LQ[F(\hat\gamma_2)].\]
Due to the first obervation, this implies that the law of $\gamma_2$ is given by chordal $\SLE_2(-2;-2)$ in $\Omega$ from $x_2$ to $(x_3x_1)$ with force points $(x_1; x_3)$.
\medbreak
It remains to prove~\eqref{eqn::coupling}. Denote by $\LR^{\delta}$ (resp. $\hat\LR^\delta$) the random walk on $\Omega^{\delta}$ starting from $x_2^{\delta,\circ}$ and stopped when hitting $\partial\Omega^{\delta}$ (resp. when hitting $(\hat x_1^\delta \hat x_3^\delta)\cup(x_3^\delta x_1^\delta)$). We construct a coupling $\LL^{\delta}$ of $(\LR^{\delta},\hat\LR^{\delta})$ such that $\LR^{\delta}=\hat\LR^{\delta}$ before hitting $(x_1^\delta\hat x_1^\delta)\cup(\hat x^\delta_3 x_3^\delta)$.
Denote by $\PP^{\delta}$ (resp. $\hat\PP^{\delta}$) the conditional law of $\LR^{\delta}$ (resp. $\hat\LR^{\delta}$) given the event $\{\LR^{\delta}\text{ hits }\partial\Omega^{\delta}\text{ at }(x_3^\delta x_1^\delta)\}$ (resp. given the event $\{\hat\LR^{\delta}\text{ hits }(\hat x_1^\delta \hat x_3^\delta)\cup(x_3^\delta x_1^\delta)\text{ at }(x_3^\delta x_1^\delta)\}$). Then, the law of $\gamma_2^\delta$ (resp. $\hat\gamma_2^\delta$) equals the loop-erasure of $\LR^{\delta}$ under $\PP^{\delta}$  (resp. the loop-erasure of $\hat\LR^{\delta}$ under $\hat{\PP}^{\delta}$). 
The total variance distance between the laws of $\gamma_2^{\delta}$ and of $\hat\gamma_2^{\delta}$ is bounded by 
\begin{align*}
&\max_{A}\left|\frac{\LL^{\delta}[\LR^{\delta}\in A,\LR^{\delta}\text{ hits }\partial\Omega^\delta\text{ at }(x_3^\delta x_1^\delta)]}{\LL^{\delta}[\LR^\delta\text{ hits }\partial\Omega^\delta\text{ at }(x_3^\delta x_1^\delta)]}-\frac{\LL^{\delta}[\hat\LR^\delta\in A,\hat\LR^\delta\text{ hits }(\hat x_1^\delta \hat x_3^\delta)\cup(x_3^\delta x_1^\delta)\text{ at }(x_3^\delta x_1^\delta)]}{\LL^\delta[\hat\LR^\delta\text{ hits }(\hat x_1^\delta \hat x_3^\delta)\cup(x_3^\delta x_1^\delta)\text{ at }(x_3^\delta x_1^\delta)]}\right|\notag\\
\le&\frac{\LL^{\delta}[\hat\LR^\delta\text{ hits }\partial\Omega^\delta\text{ at }(\hat x_3^\delta x_3^\delta)\cup(x_1^\delta \hat x_1^\delta)]}{\LL^{\delta}[\hat\LR^\delta\text{ hits }(\hat x_1^\delta \hat x_3^\delta)\cup(x_3^\delta x_1^\delta)\text{ at }(x_3^\delta x_1^\delta)]}+\left|1-\frac{\LL^{\delta}[\LR^\delta\text{ hits }\partial\Omega_\delta\text{ at }(x_3^\delta x_1^\delta)]}{\LL^{\delta}[\hat\LR^\delta\text{ hits }(\hat x_1^\delta \hat x_3^\delta)\cup(x_3^\delta x_1^\delta)\text{ at }(x_3^\delta x_1^\delta)]}\right|\notag\\
=&2\frac{\LL^\delta[\hat\LR^\delta\text{ hits }\partial\Omega_\delta\text{ at }(\hat x_3^\delta x_3^\delta)\cup(x_1^\delta \hat x_1^\delta)]}{\LL^\delta[\hat\LR^\delta\text{ hits }(\hat x_1^\delta \hat x_3^\delta)\cup(x_3^\delta x_1^\delta)\text{ at }(x_3^\delta x_1^\delta)]}.
\end{align*} 
From the same analysis as in Lemma~\ref{lem::proba_tight2}, there exists a constant $C_{\eqref{eqn::Beurling_twoRW}}\in(0,\infty)$ depending on $(\Omega;x_1,x_3,x_3)$ and a universal constant $\alpha\in(0,\infty)$ such that
\begin{equation}\label{eqn::Beurling_twoRW}
\frac{\LL^\delta[\hat\LR^\delta\text{ hits }\partial\Omega^\delta\text{ at }(\hat x_3^\delta x_3^\delta)\cup(x_1^\delta \hat x_1^\delta)]}{\LL^{\delta}[\hat\LR^\delta\text{ hits }(\hat x_1^\delta \hat x_3^\delta)\cup(x_3^\delta x_1^\delta)\text{ at }(x_3^\delta x_1^\delta)]}\le C_{\eqref{eqn::Beurling_twoRW}}\left(\max\{|x_1-\hat x_1|,|x_3-\hat x_3|\}\right)^{\alpha}.
\end{equation}
Thus, the total variance distance between the laws of $\gamma_2^{\delta}$ and of $\hat\gamma_2^{\delta}$ is bounded by RHS of~\eqref{eqn::Beurling_twoRW}. 
This gives the desired control in~\eqref{eqn::coupling}. 
\end{proof}
\begin{lemma}\label{lem::tripod_ab}
Fix a $3$-polygon $(\Omega; x_1, x_2, x_3)$ with $z\in\Omega$. We sample two random variables.
\begin{itemize}
\item[(a)] Suppose $\ell$ has the law of chordal $\SLE_2$ in $(\Omega; x_1, x_3)$ weighted by
\begin{equation}\label{eqn::tripod_twoequiv_RN1}
\frac{\nharmonic(\Omega\setminus\ell; x_2, \ell)}{\chordalSLEexp^{(\kappa=2)}[\nharmonic(\Omega\setminus\ell; x_2, \ell)]}, 
\end{equation}
where $\nharmonic(\Omega\setminus\ell; x_2, \ell)$ is harmonic measure~\eqref{eqn::nharmonic_H}-\eqref{eqn::nharmonic_cov}.  
\item[(b)] Given $\ell$, we sample $\gamma_2$ as chordal $\SLE_2(-2;-2)$ in $\Omega\setminus\ell$ from $x_2$ to $\ell$ with force points $(x_1;x_3)$. 
\end{itemize} 
Given $(\ell, \gamma_2)$, denote by $\gamma_1$ the part of $\ell$ from $x_1$ to $\trifurcation$, by $\gamma_3$ the reverse of $\ell$ from $x_3$ to $\trifurcation$.
Then the of the tripod $(\gamma_1^{\delta}, \gamma_2^{\delta}, \gamma_3^{\delta})$ conditional on $\LA_1^{\delta}\cap\LA_2^{\delta}$ in Theorem~\ref{thm::tripod} converges weakly to $(\gamma_1, \gamma_2, \gamma_3)$ whose law is characterized by (a)-(b) as above. 
\end{lemma}
\begin{proof}
The tripod $\{(\gamma_1^{\delta}, \gamma_2^{\delta}, \gamma_3^{\delta})\}_{\delta>0}$ is tight due to Proposition~\ref{prop::tripod_tightness}. Suppose $(\gamma_1^{\delta_n}, \gamma_2^{\delta_n}, \gamma_3^{\delta_n})$ is any convergent subsequence. By Skorokhod's representation theorem, we can couple them so that $(\gamma_1^{\delta_n}, \gamma_2^{\delta_n}, \gamma_3^{\delta_n})\to (\gamma_1, \gamma_2, \gamma_3)\in \chamber(\Omega; x_1, x_2, x_3; \trifurcation)$ almost surely. The conclusion follows from the following observations.
\begin{itemize}
\item From Proposition~\ref{prop::cvg_branch_eta1}, $\eta_1^{\delta_n}$ converges weakly to the law of $\ell$ in~(a). We can further couple the tripod with $\ell$ so that $\eta_1^{\delta_n}\to\ell$ almost surely. 
\item Given $\eta_1^{\delta_n}$, the conditional law of $\gamma_2^{\delta_n}$ is the same as the boundary branch for the UST in $\Omega^{\delta_n}\setminus\eta_1^{\delta_n}$ with wired boundary condition starting from $x_2^{\delta, \circ}$, conditional to exit the boundary through $\eta_1^{\delta_n}$. From Lemma~\ref{lem::cvg_branch_gamma2}, the law of $\gamma_2^{\delta_n}$ given $\eta_1^{\delta_n}$ converges weakly to chordal $\SLE_2(-2;-2)$ in $\Omega\setminus \ell$ from $x_2$ to $\ell$ with force points $(x_1;x_3)$ as desired.
\end{itemize}
\end{proof}

\subsection{Proof of Proposition~\ref{prop::chordalSLE2_Fomin}}
\label{subsec::prop14}
Proposition~\ref{prop::chordalSLE2_Fomin} follows from the convergence in Proposition~\ref{prop::cvg_branch_eta1} and Propositions~\ref{prop::deltacube} and~\ref{prop::A1interior}. 
\begin{lemma}\label{lem::cvg_branch_eta1_more}
Assume the same setup as in Proposition~\ref{prop::cvg_branch_eta1}. 
 \begin{itemize}
 \item The law of $\eta_1^{\delta}$ conditional on $\LA_1^{\delta}\cap\LA_2^{\delta}$ converges weakly to the law of $\ell$ weighted by the following Radon-Nikodym derivative
 \begin{equation}\label{eqn::RN_eta1_ell1_usingobservable}
 \left(\frac{\Poisson(\Omega; x_1, x_3)}{2\Poisson(\Omega; x_1, x_2)\Poisson(\Omega; x_2, x_3)}\right)^{1/2}\nharmonic(\Omega\setminus\ell; x_2, \ell), 
 \end{equation}
 where $\Poisson(\Omega; x_i, x_j)$ is boundary Poisson kernel~\eqref{eqn::bPoisson_H}-\eqref{eqn::bPoisson_cov} and $\nharmonic(\Omega\setminus\ell; x_2, \ell)$ is harmonic measure~\eqref{eqn::nharmonic_H}-\eqref{eqn::nharmonic_cov}.  
 \item Fix $z\in\Omega$ and suppose $z^{\delta}$ is the vertex in $\LV^{\circ}(\Omega^{\delta})$ that is nearest to $z$. The law of $\eta_1^{\delta}$ conditional on $\LA_1^{\delta}\cap\{z^{\delta}\rightsquigarrow e_3^{\delta}\}$ converges weakly to the law of $\ell$ weighted by the following Radon-Nikodym derivative
 \begin{equation}\label{eqn::RN_eta1_ell1_radial}
 \frac{1}{g(\Omega; x_1, x_3; z)}\harmonic(\Omega\setminus\ell; z, \ell),
 \end{equation}
 where $\harmonic(\Omega\setminus\ell; z, \ell)$ is harmonic measure~\eqref{eqn::charnomic_H}-\eqref{eqn::nharmonic_cov} and $g(\Omega; x_1, x_3; z)$ is the limiting function in~\eqref{eqn::A1interior} of Proposition~\ref{prop::A1interior}. 
 \end{itemize}
\end{lemma}

\begin{proof}[Proof of~\eqref{eqn::RN_eta1_ell1_usingobservable}]
 For any bounded continuous function $F$ on curves, we have
 \begin{align}
 \E\left[F(\eta_1^{\delta})\cond \LA_1^{\delta}\cap \LA_2^{\delta}\right]=&\frac{\PP[\LA_1^{\delta}]}{\PP[\LA_1^{\delta}\cap\LA_2^{\delta}]}
 \E\left[F(\ell^{\delta}) \harmonic(\Omega^{\delta}\setminus\ell^{\delta}; x_2^{\delta, \circ}, \ell^{\delta})\cond \LA_1^{\delta}\right],
 \label{eqn::boundary_branch_cvg_aux1_usingobservable}
 \end{align}
 where $\ell^{\delta}$ is defined as in Lemma~\ref{lem::boundary_branch_cvg1}. Let us derive the scaling limit of the terms in RHS of~\eqref{eqn::boundary_branch_cvg_aux1_usingobservable}. We fix an interior point $v\in\Omega$ and suppose $v^{\delta}$ is the vertex in $\LV^{\circ}(\Omega^{\delta})$ that is nearest to $v$.
From~\eqref{eqn::bPoisson_cvg} and~\eqref{eqn::deltacube} and~\eqref{eqn::LZtripod_integrable}, we have
 \begin{align}\label{eqn::boundary_branch_cvg_aux2}
 \lim_{\delta\to 0}\frac{\PP[\LA_1^{\delta}]\harmonic(\Omega^{\delta}; v^{\delta}, e_2^{\delta})}{\PP[\LA_1^{\delta}\cap\LA_2^{\delta}]}=&\frac{2\sqrt{3}\pi\Poisson(\Omega; x_1, x_3)\Poisson(\Omega; v, x_2)}{8\sqrt{2}\int_{\Omega}\LZtripod(\Omega; x_1, x_2, x_3; w)|\ud w|^2}\notag\\
 =&\frac{\Poisson(\Omega; v, x_2)\Poisson(\Omega;x_1,x_3)^{1/2}}{\sqrt{6}\left(\Poisson(\Omega; x_1, x_2)\Poisson(\Omega; x_2, x_3)\right)^{1/2}}. 
 \end{align}
 Plugging~\eqref{eqn::boundary_branch_cvg_aux2} and~\eqref{eqn::boundary_branch_cvg_L1} (proved in Lemma~\ref{lem::boundary_branch_cvg_L1}) into~\eqref{eqn::boundary_branch_cvg_aux1_usingobservable}, we obtain~\eqref{eqn::RN_eta1_ell1_usingobservable} as desired. 
\end{proof}

\begin{proof}[Proof of~\eqref{eqn::RN_eta1_ell1_radial}]
 For any bounded continuous function $F$ on curves, we have 
 \begin{align}\label{eqn::eta_RN_radial_aux1}
 \E\left[F(\eta_1^{\delta})\cond \LA_1^{\delta}\cap\{z^{\delta}\rightsquigarrow e_3^{\delta}\}\right]
 =&\frac{\E\left[F(\eta_1^{\delta})\one\{z^{\delta}\rightsquigarrow e_3^{\delta}\}\cond \LA_1^{\delta}\right]}{\PP\left[\{z^{\delta}\rightsquigarrow e_3^{\delta}\}\cond \LA_1^{\delta}\right]}
 =\frac{\E\left[F(\ell^{\delta})\harmonic(\Omega^{\delta}\setminus\ell^{\delta}; z^{\delta}, \ell^{\delta})\cond \LA_1^{\delta}\right]}{\E\left[\harmonic(\Omega^{\delta}\setminus\ell^{\delta}; z^{\delta}, \ell^{\delta})\cond \LA_1^{\delta}\right]}, 
 \end{align}
 where $\ell^{\delta}$ is defined as in Lemma~\ref{lem::boundary_branch_cvg1}. 
 From Lemma~\ref{lem::boundary_branch_cvg1}, the law of $\ell^{\delta}$ conditional on $\LA_1^{\delta}$ converges weakly to the law of $\ell$. Suppose $\{\delta_n\}_n$ is any subsequence such that $\delta_n\to 0$. By Skorokhod's representation theorem, we can couple $\{\ell^{\delta_n}\}_{n}$ and $\ell$ together such that $\ell^{\delta_n}$ converges to $\ell$ almost surely. We denote such coupling still by $\PP$. In this coupling, the convergence~\eqref{eqn::charmonic_cvg} asserts the almost sure convergence: 
 \begin{align*}
 \lim_n \harmonic(\Omega^{\delta_n}\setminus\ell^{\delta_n}; z^{\delta_n}, \ell^{\delta_n})=\frac{1}{2\pi}\harmonic(\Omega\setminus\ell; z, \ell). 
 \end{align*}
 Plugging the almost sure convergence into~\eqref{eqn::eta_RN_radial_aux1} and noting that the harmonic measure is bounded, we obtain
  \begin{align*}
  \lim_{\delta\to 0}\E\left[F(\eta_1^{\delta})\cond \LA_1^{\delta}\cap\{z^{\delta}\rightsquigarrow e_3^{\delta}\}\right]=\frac{\E\left[F(\ell)\harmonic(\Omega\setminus\ell; z, \ell)\right]}{\E\left[\harmonic(\Omega\setminus\ell; z, \ell)\right]}. 
  \end{align*}
 In other words, the law of $\eta_1^{\delta}$ conditional on $\LA_1^{\delta}\cap\{z^{\delta}\rightsquigarrow e_3^{\delta}\}$ converges weakly to the law of $\ell$ weighted by the following Radon-Nikodym derivative 
 \[\frac{\harmonic(\Omega\setminus\ell; z, \ell)}{\E[\harmonic(\Omega\setminus\ell; z, \ell)]}.\]
 Then the conclusion follows from~\eqref{eqn::A1interior}: 
 \begin{align*}
 \E[\harmonic(\Omega\setminus\ell; z, \ell)]=\lim_{\delta\to 0}\PP\left[\{z^{\delta}\rightsquigarrow e_3^{\delta}\}\cond \LA_1^{\delta}\right]=g(\Omega; x_1, x_3; z). 
 \end{align*}
\end{proof}

\begin{proof}[Proof of Proposition~\ref{prop::chordalSLE2_Fomin}]
From Lemma~\ref{lem::cvg_branch_eta1_more}, we have 
\begin{align*}
\chordalSLEexp^{(\kappa=2)}(\Omega; x_1, x_3)\left[\nharmonic(\Omega\setminus\ell; x_2, \ell)\right]=&\left(\frac{2\Poisson(\Omega; x_1, x_2)\Poisson(\Omega; x_2, x_3)}{\Poisson(\Omega; x_1, x_3)}\right)^{1/2}, \\
\chordalSLEexp^{(\kappa=2)}(\Omega; x_1, x_3)\left[\harmonic(\Omega\setminus\ell; z, \ell)\right]=&g(\Omega; x_1, x_3; z), 
\end{align*}
as desired. 
\end{proof}

\subsection{Convergence to radial SLE$_2(2,2)$}
\label{subsec::gamma3}

\begin{proposition}\label{prop::gamma3_cvg}
Assume the same setup as in Theorem~\ref{thm::tripod}. 
Fix a compact subset $K\subset\Omega$ and fix $r>0$ such that the $r$-neighborhood $K_r:=\{z\in\C: \dist(z, K)<r\}$ is still contained in $\Omega$.  
Recall that $\gamma_3^{\delta}$ is a path from $x_3^{\delta}$ to $\trifurcation^{\delta}$. 
We denote by $T_r^{\delta}$ its hitting time of $K_r$.  Then for any bounded continuous function $F$ on curves, we have 
\begin{align}\label{eqn::gamma3_cvg}
&\lim_{\delta\to 0}\E\left[F(\gamma_3^{\delta}[0,T_r^{\delta}])\one\{\trifurcation^{\delta}\in K\}\cond \LA_1^{\delta}\cap\LA_2^{\delta}\right]\notag\\
=&\int_{K}\nradE{3}^{(\kappa=2)}(\Omega; x_1, x_2, x_3; u)[F(\gamma_3[0,T_r])]
p(\Omega; x_1,x_2,x_3;u) |\ud u|^2, 
\end{align}
where the law of $\gamma_3$ under $\nradP{3}^{(\kappa=2)}(\Omega; x_1, x_2, x_3; u)$ is radial $\SLE_2(2,2)$ in $\Omega$ from $x_3$ to $u$ with force points $(x_1, x_2)$ and $T_r$ is  its hitting time of $K_r$,
and $p(\Omega; x_1,x_2,x_3;u)$ is the density function defined in~\eqref{eqn::density_def}.
\end{proposition}

\begin{lemma}[\cite{LawlerSchrammWernerLERWUST}]
\label{lem::LERW_radialSLE2}
Suppose $(\Omega; x)$ is a bounded $1$-polygon and suppose $(\Omega^{\delta}; x^{\delta})$ is an approximation of $(\Omega; x)$ on $\delta\hexagon$ in Carath\'eodory sense.
Suppose $e^{\delta}=\langle x^{\delta, \circ}, x^{\delta}\rangle$ is the boundary edge whose boundary end point $x^{\delta}$ is nearest to $x$. 
Fix an interior point $v\in\Omega$ and suppose $v^{\delta}$ is the vertex in $\LV^{\circ}(\Omega^{\delta})$ that is nearest to $v$. 
Consider UST on $\Omega^{\delta}$ with wired boundary condition. Let $\gamma^{\delta}$ be the branch starting from $v^{\delta}$ and we consider the event $\{v^{\delta}\rightsquigarrow e^{\delta}\}$. We view $\gamma^{\delta}$ as a path from $x^{\delta, \circ}$ to $v^{\delta}$. Then the law of $\gamma^{\delta}$ conditional on $\{v^{\delta}\rightsquigarrow e^{\delta}\}$ converges weakly to the law of radial $\SLE_2\sim \nradP{1}^{(\kappa=2)}(\Omega; x; v)$. 
\end{lemma}
\begin{proof}
This was proved in~\cite{LawlerSchrammWernerLERWUST} for $\delta\Z^2$ lattice, and the proof is not restricted to a particular choice of lattice. See also~\cite{YadinYehudayoffLERWPoissonKernel}. 
\end{proof}

To prove Proposition~\ref{prop::gamma3_cvg}, we fix the following notation for the observable in Proposition~\ref{prop::observable_cvg}. Fix a bounded $3$-polygon $(\Omega; x_1, x_2, x_3)$ and suppose $(\Omega^{\delta}; x_1^{\delta}, x_2^{\delta}, x_3^{\delta})$ is an approximation of $(\Omega; x_1, x_2, x_3)$ on $\delta\hexagon$ in Carath\'eodory sense. We denote the observable in Proposition~\ref{prop::observable_cvg} by 
\begin{equation}
\LZ^{\delta}(\Omega^{\delta}; x_1^{\delta}, x_2^{\delta}, x_3^{\delta}; u)=\PP[\LA_1^{\delta}\cap\LA_2^{\delta}\cap\{\trifurcation^{\delta}=u\}],\qquad u\in\LV^{\circ}(\Omega^{\delta}). 
\end{equation}
We also extend its definition in the same way as in Proposition~\ref{prop::observable_cvg}: recall that the triangle $\bigtriangleup^{\delta}(u)$ covering vertex $u$ is defined in Figure~\ref{fig::threeneighbors}.  
For $u\in\LV^{\circ}(\Omega^{\delta})$, we set $\LZ^{\delta}(\Omega^{\delta}; x_1^{\delta}, x_2^{\delta}, x_3^{\delta}; z)=\LZ^{\delta}(\Omega^{\delta}; x_1^{\delta}, x_2^{\delta}, x_3^{\delta}; u)$ for $z\in\bigtriangleup^{\delta}(u)$. 

\begin{proof}[Proof of Proposition~\ref{prop::gamma3_cvg}]
As the tripod $\{(\gamma_1^{\delta}, \gamma_2^{\delta}, \gamma_3^{\delta})\}_{\delta>0}$ is tight due to Proposition~\ref{prop::tripod_tightness}, for any sequence $\delta_n\to 0$, there exists a subsequence, which we still denote by $\{\delta_n\}$, such that the conditional law of $\{(\gamma_1^{\delta_n}, \gamma_2^{\delta_n}, \gamma_3^{\delta_n})\}_n$ converges. We can couple 
$\{(\gamma_1^{\delta_n}, \gamma_2^{\delta_n}, \gamma_3^{\delta_n})\}_n$ and its limit $(\gamma_1, \gamma_2, \gamma_3)$ together, such that $\gamma_j^{\delta_n}\to \gamma_j$ in metric~\eqref{eqn::curves_metric} almost surely for $j=1,2,3$. We denote this coupling by $\PP$. Proposition~\ref{prop::tripod_tightness} also asserts that there exists $\trifurcation\in\Omega$ such that $(\gamma_1, \gamma_2, \gamma_3)\in\chamber(\Omega; x_1, x_2, x_3; \trifurcation)$ almost surely. 

From the domain Markov property of the time-reversal of LERW~\cite[Lemma~3.2]{LawlerSchrammWernerLERWUST}, we may write
\begin{align}\label{eqn::cvg_radial_aux1}
&\E\left[F(\gamma_3^{\delta_n}[0,T_r^{\delta_n}])\one\{\trifurcation^{\delta_n}\in K\}\cond \LA_1^{\delta_n}\cap\LA_2^{\delta_n}\right]\notag\\
=&\sum_{u\in \LV^{\circ}(\Omega^{\delta_n})\cap K}
\E\left[F(\gamma_3^{\delta_n}[0,T_r^{\delta_n}])\PP\left[\{\trifurcation^{\delta_n}=u\}\cap\LA_1^{\delta_n}\cap\LA_2^{\delta_n}\cond \gamma_3^{\delta_n}[0,T_r^{\delta_n}]\right]\right]\big/\PP[\LA_1^{\delta_n}\cap\LA_2^{\delta_n}]\notag\\
=&
\E\left[F(\gamma_3^{\delta_n}[0,T_r^{\delta_n}])\sum_{u\in \LV^{\circ}(\Omega^{\delta_n})\cap K}\frac{\LZ^{\delta_n}(\Omega^{\delta_n}\setminus\gamma_3^{\delta_n}[0,T_r^{\delta_n}]; x_1^{\delta_n}, x_2^{\delta_n}, \gamma_3^{\delta_n}(T_r^{\delta_n}); u)}{\PP[\LA_1^{\delta_n}\cap\LA_2^{\delta_n}]}\right].
\end{align}
To simplify notations, we denote $\Omega_3^{\delta_n}(T_r^{\delta_n}):=\Omega^{\delta_n}\setminus\gamma_3^{\delta_n}[0,T_r^{\delta_n}]$ and $\Omega_3(T_r):=\Omega\setminus\gamma_3[0,T_r]$.
Fix one of the boundary edges adjacent to $\gamma_3^{\delta_n}(T_r^{\delta_n})$ in $\Omega_3^{\delta_n}(T_r^{\delta_n})$ and denote it by $e_3^{\delta_n}(T_r^{\delta_n})$.
Fix $v\in K$ and suppose $v^{\delta_n}$ is the vertex in $\LV^{\circ}(\Omega^{\delta_n})$ that is nearest to $v$.
We have the following observations.
\begin{itemize}
\item Proposition~\ref{prop::observable_cvg} gives the following almost sure convergence: as $n\to\infty$, 
\begin{align*}
&\frac{\LZ^{\delta_n}(\Omega_3^{\delta_n}(T_r^{\delta_n}); x_1^{\delta_n}, x_2^{\delta_n}, \gamma_3^{\delta_n}(T_r^{\delta_n}); u^{\delta_n})}{\delta^2\prod_{j=1}^2\harmonic(\Omega_3^{\delta_n}(T_r^{\delta_n}); v^{\delta_n}, e_j^{\delta_n})\harmonic(\Omega_3^{\delta_n}(T_r^{\delta_n}); v^{\delta_n}, e_3^{\delta_n}(T_r^{\delta_n}))}\notag\\
\to &\frac{6\sqrt{6}\LZtripod(\Omega_3(T_r); x_1, x_2, \gamma_3(T_r); u)}{\prod_{j=1}^2\Poisson(\Omega_3(T_r); v, x_j)\Poisson(\Omega_3(T_r); v, \gamma_3(T_r))}, \qquad\text{uniformly over }u\in K. 
\end{align*}
As a consequence, we have the following almost sure convergence: 
\begin{align}\label{eqn::cvg_radial_aux2}
&\lim_n\frac{\sum_{u\in \LV^{\circ}(\Omega^{\delta_n})\cap K}\LZ^{\delta_n}(\Omega_3^{\delta_n}(T_r^{\delta_n}); x_1^{\delta_n}, x_2^{\delta_n}, \gamma_3^{\delta_n}(T_r^{\delta_n}); u)}
{\prod_{j=1}^2\harmonic(\Omega_3^{\delta_n}(T_r^{\delta_n}); v^{\delta_n}, e_j^{\delta_n})\harmonic(\Omega_3^{\delta_n}(T_r^{\delta_n}); v^{\delta_n}, e_3^{\delta_n}(T_r^{\delta_n}))}\notag\\
= &\frac{8\sqrt{2}\int_{K}\LZtripod(\Omega_3(T_r); x_1, x_2, \gamma_3(T_r); u)|\ud u|^2}{\prod_{j=1}^2\Poisson(\Omega_3(T_r); v, x_j)\Poisson(\Omega_3(T_r); v, \gamma_3(T_r))}. 
\end{align}
\item Proposition~\ref{prop::deltacube} gives the following convergence:
\begin{align}\label{eqn::cvg_radial_aux3}
\lim_n\frac{\PP[\LA_1^{\delta_n}\cap\LA_2^{\delta_n}]}{\prod_{j=1}^3\harmonic(\Omega^{\delta_n}; v^{\delta_n}, e_j^{\delta_n})}= \frac{{8\sqrt{2}}}{\prod_{j=1}^{3}\Poisson(\Omega; v, x_j)}\int_{\Omega}\LZtripod(\Omega; x_1, x_2, x_3; w)|\ud w|^2. 
\end{align}
\item From~\eqref{eqn::Poisson_cvg}, we have the following almost sure convergence: 
\begin{align}\label{eqn::cvg_radial_aux4}
\lim_n\frac{\harmonic(\Omega_3^{\delta_n}(T_r^{\delta_n}); v^{\delta_n}, e_j^{\delta_n})}{\harmonic(\Omega^{\delta_n}; v^{\delta_n}, e_j^{\delta_n})}=\frac{\Poisson(\Omega_3(T_r); v, x_j)}{\Poisson(\Omega; v, x_j)},\qquad j=1,2. 
\end{align}

\end{itemize}

We denote 
\begin{align}\label{eqn::def_R3_S3}
\begin{split}
\LR_3^{\delta_n}=&\frac{\harmonic(\Omega_3^{\delta_n}(T_r^{\delta_n});v^{\delta_n},e_3^{\delta_n}(T_r^{\delta_n})])}{\harmonic(\Omega^{\delta_n};v^{\delta_n},e_3^{\delta_n})}, \\
\LS_3^{\delta_n}=&\frac{1}{\LR_3^{\delta_n}}\sum_{u\in \LV^{\circ}(\Omega^{\delta_n})\cap K}\frac{\LZ^{\delta_n}(\Omega_3^{\delta_n}(T_r^{\delta_n}); x_1^{\delta_n}, x_2^{\delta_n}, \gamma_3^{\delta_n}(T_r^{\delta_n}); u)}{\PP[\LA_1^{\delta_n}\cap\LA_2^{\delta_n}]}. 
\end{split}
\end{align}
Plugging into~\eqref{eqn::cvg_radial_aux1}, we have
\begin{align*}
\E\left[F(\gamma_3^{\delta_n}[0,T_r^{\delta_n}])\one\{\trifurcation^{\delta_n}\in K\}\cond \LA_1^{\delta_n}\cap\LA_2^{\delta_n}\right]
=&\E\left[F(\gamma_3^{\delta_n}[0,T_r^{\delta_n}])\LR_3^{\delta_n}\LS_3^{\delta_n}\right].
\end{align*}
The convergence in~\eqref{eqn::cvg_radial_aux2}, \eqref{eqn::cvg_radial_aux3} and~\eqref{eqn::cvg_radial_aux4} gives the following almost sure convergence:
\begin{align}\label{eqn::S3_cvg}
\LS_3:=\lim_n \LS_3^{\delta_n}=\frac{\int_{K}\LZtripod(\Omega_3(T_r); x_1, x_2, \gamma_3(T_r); u)|\ud u|^2}{\Poisson(\Omega_3(T_r); v, \gamma_3(T_r))}\frac{\Poisson(\Omega; v, x_3)}{\int_{\Omega}\LZtripod(\Omega; x_1, x_2, x_3; w)|\ud w|^2}. 
\end{align}

We will prove in Lemma~\ref{lem::gamma3_radial} that the law of $\gamma_3^{\delta_n}[0,T_r^{\delta_n}]$ weighted by $\LR_3^{\delta_n}$ is the same as the time-reversal of loop-erased random walk in $\Omega^{\delta_n}$ starting from $v^{\delta_n}$ and conditional on $\{v^{\delta_n}\rightsquigarrow e_3^{\delta_n}\}$. We denote the conditional law of such loop-erased random walk by $\PP_{\lerw}^{\delta_n}(\Omega; x_3; v)$. 
We will prove in Lemma~\ref{lem::S3_bounded} that $\{\LS_3^{\delta_n}\}_n$ is uniformly bounded. 
Then
\begin{align*}
&\lim_n\E\left[F(\gamma_3^{\delta_n}[0,T_r^{\delta_n}])\one\{\trifurcation^{\delta_n}\in K\}\cond \LA_1^{\delta_n}\cap\LA_2^{\delta_n}\right]\\
=&\lim_n\E_{\lerw}^{\delta_n}(\Omega; x_3; v)\left[F(\gamma_3^{\delta_n}[0,T_r^{\delta_n}])\LS_3^{\delta_n}\right]\tag{due to Lemma~\ref{lem::gamma3_radial}}\\
=&\nradE{1}^{(\kappa=2)}(\Omega; x_3; v)\left[F(\gamma_3[0,T_r])\LS_3\right]\tag{due to Lemmas~\ref{lem::LERW_radialSLE2} and~\ref{lem::S3_bounded} and~\eqref{eqn::S3_cvg}}\\
=&\int_{K}|\ud u|^2\nradE{1}^{(\kappa=2)}(\Omega; x_3; v)\left[F(\gamma_3[0,T_r])\frac{\LZtripod(\Omega_3(T_r); x_1, x_2, \gamma_3(T_r); u)}{\Poisson(\Omega_3(T_r); v, \gamma_3(T_r))}\right]\frac{\Poisson(\Omega; v, x_3)}{\int_{\Omega}\LZtripod(\Omega; x_1, x_2, x_3; w)|\ud w|^2}\\
=&\int_{K}|\ud u|^2\nradE{1}^{(\kappa=2)}(\Omega; x_3; u)\left[F(\gamma_3[0,T_r])\frac{\LZtripod(\Omega_3(T_r); x_1, x_2, \gamma_3(T_r); u)}{\Poisson(\Omega_3(T_r); u, \gamma_3(T_r))}\right]\frac{\Poisson(\Omega; u, x_3)}{\int_{\Omega}\LZtripod(\Omega; x_1, x_2, x_3; w)|\ud w|^2}\tag{due to Lemma~\ref{lem::1rad_RN} and~\eqref{eqn::13rad_kappa2}}\\
=&\int_{K}|\ud u|^2\nradE{3}^{(\kappa=2)}(\Omega; x_1, x_2, x_3; u)\left[F(\gamma_3[0,T_r])\right]\frac{\LZtripod(\Omega; x_1, x_2, x_3; u)}{\int_{\Omega}\LZtripod(\Omega; x_1, x_2, x_3; w)|\ud w|^2}
\tag{due to Lemma~\ref{lem::3radvs1rad_RN} and~\eqref{eqn::13rad_kappa2}}\\
=&\int_{K}|\ud u|^2\nradE{3}^{(\kappa=2)}(\Omega; x_1, x_2, x_3; u)[F(\gamma_3[0,T_r])]
p(\Omega; x_1,x_2,x_3;u), 
\end{align*}
as desired. 
\end{proof}
\begin{lemma}\label{lem::gamma3_radial}
Assume the same notation as in Proof of Proposition~\ref{prop::gamma3_cvg}. The law of $\gamma_3^{\delta_n}[0,T_r^{\delta_n}]$ weighted by $\LR_3^{\delta_n}$ defined in~\eqref{eqn::def_R3_S3} is the same as the time-reversal of loop-erased random walk in $\Omega^{\delta_n}$ starting from $v^{\delta_n}$ and conditional on $\{v^{\delta_n}\rightsquigarrow e_3^{\delta_n}\}$.
\end{lemma}
\begin{proof}
Let $F$ be any bounded continuous function on curves and it suffices to prove
\begin{equation}
\E\left[F(\gamma_3^{\delta_n}[0,T_r^{\delta_n}])\LR_3^{\delta_n}\right]=\E_{\lerw}^{\delta_n}(\Omega; x_3; v)\left[F(\gamma_3^{\delta_n}[0,T_r^{\delta_n}])\right].
\end{equation}
We denote by $\chamber_3^{\delta_n}$ the collection of all possible paths $\gamma_3^{\delta_n}$ from $x_3^{\delta_n}$ to $v^{\delta_n}$, 
and denote by $\chamber_3^{\delta_n}(T_r^{\delta_n})$ the collection of all possible paths $\gamma_3^{\delta_n}[0,T_r^{\delta_n}]$ starting from $x_3^{\delta_n}$ and stopped at the first hitting time of $K_r$. 
We denote by $\st(\Omega^{\delta_n})$ the collection of configurations of spanning trees in $\Omega^{\delta_n}$.
Given a path $\lambda\in\chamber_3^{\delta_n}$, 
we have $\lambda[0,T_r^{\delta_n}]\in \chamber_3^{\delta_n}(T_r^{\delta_n})$ and 
we denote by $\st(\Omega^{\delta_n}\setminus\lambda[0,T_r^{\delta_n}])$ the collection of configurations of spanning trees in $\Omega^{\delta_n}$ containing $\lambda[0,T_r^{\delta_n}]$. 
By Wilson's algorithm, we have
\begin{align*}
   \harmonic(\Omega^{\delta_n};v^{\delta_n},e_3^{\delta_n})=\frac{|\{T\in \st(\Omega^{\delta_n}): v^{\delta_n}\rightsquigarrow e_3^{\delta_n}\}|}{|\st(\Omega^{\delta_n})|}, 
\end{align*}
and
\begin{align*}
   \harmonic(\Omega_3^{\delta_n}(T_r^{\delta_n});v^{\delta_n},e_3^{\delta_n}(T_r^{\delta_n}))=\frac{|\{T\in \st(\Omega_3^{\delta_n}(T_r^{\delta_n}): v^{\delta_n}\rightsquigarrow e_3^{\delta_n}(T_r^{\delta_n})\}|}{|\st(\Omega_3^{\delta_n}(T_r^{\delta_n}))|}. 
\end{align*}
Thus,
\begin{align*}
\E\left[F(\gamma_3^{\delta_n}[0,T_r^{\delta_n}])\LR_3^{\delta_n}\right]
=&\sum_{\lambda[0,T_r^{\delta_n}]\in \chamber_3^{\delta_n}(T_r^{\delta_n})}\frac{|\st(\Omega^{\delta_n}\setminus\lambda[0,T_r^{\delta_n}])}{|\st(\Omega^{\delta_n})|} F(\lambda[0,T_r^{\delta_n}])
\frac{\harmonic(\Omega^{\delta_n}\setminus\lambda[0,T_r^{\delta_n}];v^{\delta_n},e_3^{\delta_n}(T_r^{\delta_n}))}{\harmonic(\Omega^{\delta_n};v^{\delta_n},e_3^{\delta_n})}\\
=&\sum_{\lambda[0,T_r^{\delta_n}]\in\chamber_3^{\delta_n}(T_r^{\delta_n})}F(\lambda[0,T_r^{\delta_n}])\frac{|\{T\in \st(\Omega^{\delta_n}\setminus \lambda[0,T_r^{\delta_n}]): v^{\delta_n}\rightsquigarrow e_3^{\delta_n}(T_r^{\delta_n})\}|}{|\{T\in \st(\Omega^{\delta_n}): v^{\delta_n}\rightsquigarrow e_3^{\delta_n}\}|}\\
=&\sum_{\lambda\in\chamber_3^{\delta_n}}F(\lambda[0,T_r^{\delta_n}])\frac{1}{|\{T\in \st(\Omega^{\delta_n}): v^{\delta_n}\rightsquigarrow e_3^{\delta_n}\}|}\\
=&\E_{\lerw}^{\delta_n}(\Omega; x_3; v)\left[F(\gamma_3^{\delta_n}[0,T_r^{\delta_n}])\right],
\end{align*}
as desired.
\end{proof}

\begin{lemma}\label{lem::S3_bounded}
Assume the same notation as in Proof of Proposition~\ref{prop::gamma3_cvg}. The sequence $\{\LS_3^{\delta_n}\}_{n}$ defined in~\eqref{eqn::def_R3_S3} is uniformly bounded. 
\end{lemma}
\begin{proof}
Note that
\begin{align*}
\LS_3^{\delta_n}=&\frac{\PP[\LA_1^{\delta_n}\cap\LA_2^{\delta_n}\cap\{\trifurcation^{\delta_n}\in K\}\cond \gamma_3^{\delta_n}[0,T_r^{\delta_n}]]}{\PP[\LA_1^{\delta}\cap\LA_2^{\delta}]}\frac{\harmonic(\Omega^{\delta_n};v^{\delta_n},e_3^{\delta_n})}{\harmonic(\Omega_3^{\delta_n}(T_r^{\delta_n});v^{\delta_n},e_3^{\delta_n}(T_r^{\delta_n})])}\\
= &\underbrace{\frac{\PP\left[\{\trifurcation^{\delta_n}\in K\}\cap \LA_1^{\delta_n}\cap \LA_2^{\delta_n}\cond \gamma_3^{\delta_n}[0,T_r^{\delta_n}]\right]}{\prod_{j=1}^2\harmonic(\Omega^{\delta_n}; v^{\delta_n}, e_j^{\delta_n})\harmonic(\Omega_3^{\delta_n}(T_r^{\delta_n}); v^{\delta_n}, e_3^{\delta_n}(T_r^{\delta_n}))}}_{\LU_r^{\delta_n}:=}\underbrace{\frac{\prod_{j=1}^3\harmonic(\Omega^{\delta_n}; v^{\delta_n}, e_j^{\delta_n})}{\PP[\LA_1^{\delta_n}\cap\LA_2^{\delta_n}]}}_{\LU^{\delta_n}:=}. 
\end{align*}
Proposition~\ref{prop::deltacube} guarantees that the sequence $\{\LU^{\delta_n}\}_n$ is uniformly bounded. It suffices to show that the sequence $\{\LU_r^{\delta_n}\}_n$ is uniformly bounded as well. 

We use similar notation as in Proof of Lemma~\ref{lem::proba_tight2}. 
Fix $\eps>0$ such that $B(x_j,\eps)\cap K_r=\emptyset$ for $j=1,2$. 
Recall that $V_j^{\delta_n}(\eps)$ is the maximal domain contained in $B(x_j,\eps)\cap \Omega^{\delta_n}$ such that $\partial V_j^{\delta_n}(\eps)$ consists of the edges of $\delta_n \hexagon$, and $S_j^{\delta_n}(\eps)=\LV^{\partial}(V_j^{\delta_n}(\eps))\setminus \LV^{\partial}(\Omega^{\delta_n})$.
For $s>0$, let $V_3^{\delta_n}(T_r^{\delta_n},s)$ be the maximal domain contained in $B(\gamma_3^{\delta_n}(T_r^{\delta_n}),s)\cap \Omega_3^{\delta_n}(T_r^{\delta_n})$ such that $\partial V_3^{\delta_n}(T_r^{\delta_n},s)$ consists of edges of $\delta_n\hexagon$, 
and we set $S_3^{\delta_n}(T_r^{\delta_n},s)=\LV^{\partial} (V_3^{\delta_n}(T_r^{\delta_n},s))\setminus\LV^{\partial}(\Omega_3^{\delta_n}(T_r^{\delta_n}))$.
Denote by $\eta_{1,T_r}^{\delta_n}=(u_0,u_1,\dots,u_L)$ is the boundary branch from $x_1^{\delta_n,\circ}$ to $e_3^{\delta_n}(T_r^{\delta_n})$.
Given $\gamma_3^{\delta_n}[0,T_r^{\delta_n}]$, the event $\{\trifurcation^{\delta_n}\in K\}\cap \LA_1^{\delta_n}\cap \LA_2^{\delta_n}$ implies the following three events.
\begin{itemize}
   \item The boundary branch $\eta_{1,T_r}^{\delta_n}$ exits $V_1^{\delta_n}(\eps)\cap\Omega_3^{\delta_n}(T_r^{\delta_n})$ through $S_1^{\delta_n}(\eps)\cap\Omega_3^{\delta_n}(T_r^{\delta_n})$ and then hit $K$ before reaching $e_3^{\delta_n}(T_r^{\delta_n})$. 
   The probability for this event is bounded by 
   \begin{equation*}
      \harmonic(V_1^{\delta_n}(\eps)\cap\Omega_3^{\delta_n}(T_r^{\delta_n}); x_1^{\delta_n,\circ}, S_1^{\delta_n}(\eps)\cap\Omega_3^{\delta_n}(T_r^{\delta_n}))\le \harmonic(V_1^{\delta_n}(\eps); x_1^{\delta_n,\circ}, S_1^{\delta_n}(\eps)).
   \end{equation*}
We denote by $\tau(K)$ the first hitting time of $K$ for the boundary branch $\eta_{1,T_r}^{\delta_n}$.
 \item The boundary branch from $u_{\tau(K)}$ in $\Omega_3^{\delta_n}(T_r^{\delta_n})$ has to hit the boundary through the edge $e_3^{\delta_n}(T_r^{\delta_n})$. The probability of this event is bounded by
   \begin{equation*}
      \max_{z\in K\cap\LV^{\circ}(\Omega^{\delta_n})}\harmonic(\Omega_3^{\delta_n}(T_r^{\delta_n}); z, e_3^{\delta_n}(T_r^{\delta_n})).
   \end{equation*}
   \item Given $\eta_{1,T_r}^{\delta_n}$, the boundary branch from $x_2^{\delta_n,\circ}$ in $\Omega_3^{\delta_n}(T_r^{\delta_n})\setminus \eta_{1,T_r}^{\delta_n}$ has to exit $V_2^{\delta_n}(\eps)\cap\Omega_3^{\delta_n}(T_r^{\delta_n})$ through $S_2^{\delta_n}(\eps)\cap\Omega_3^{\delta_n}(T_r^{\delta_n})$.
   The probability of this event is bounded by
   \begin{equation*}
      \harmonic(V_2^{\delta_n}(\eps)\cap\Omega_3^{\delta_n}(T_r^{\delta_n}); x_2^{\delta_n,\circ}, S_2^{\delta_n}(\eps)\cap\Omega_3^{\delta_n}(T_r^{\delta_n}))\le \harmonic(V_2^{\delta_n}(\eps); x_2^{\delta_n,\circ}, S_2^{\delta_n}(\eps)).
   \end{equation*}
   \end{itemize}
   Combining the above three events, by the Markov property of random walk, we have
\begin{align}\label{eqn::rnd_bound_aux0}
&\PP\left[\{\trifurcation^{\delta_n}\in K\}\cap \LA_1^{\delta_n}\cap \LA_2^{\delta_n}\cond \gamma_3^{\delta_n}[0,T_r^{\delta_n}]\right]\notag\\
\le& \prod_{j=1}^{2}\harmonic(V_j^{\delta_n}(\eps); x_j^{\delta_n,\circ}, S_j^{\delta_n}(\eps))\times \max_{z\in K\cap\LV^{\circ}(\Omega^{\delta_n})}\harmonic(\Omega_3^{\delta_n}(T_r^{\delta_n}); z, e_3^{\delta_n}(T_r^{\delta_n})).
\end{align}
From~\eqref{eqn::nharmonic_cvg}, there exists a constant $C_{\eqref{eqn::rnd_bound_aux1}}\in (0,\infty)$ depending on $(\Omega;x_1,x_2,x_3;v;\eps)$ such that 
   \begin{equation}\label{eqn::rnd_bound_aux1}
      \prod_{j=1}^{2}\frac{\harmonic(V_j^{\delta_n}(\eps);x_j^{\delta_n,\circ},S_j^{\delta_n}(\eps))}{\harmonic(\Omega^{\delta_n};v^{\delta_n},e_j^{\delta_n})}\le C_{\eqref{eqn::rnd_bound_aux1}}.
   \end{equation}
From~\eqref{eqn::dPoisson_max_control_aux3} in Proof of Lemma~\ref{lem::discretePoisson_max_control},
there exists a universal constant $\constCW\in(0,\infty)$ such that for all $z\in K\cap\LV^{\circ}(\Omega^{\delta_n})$ and $w\in B(\gamma_3(T_r^{\delta_n}),r/2)\cap\LV^{\circ}(\Omega^{\delta_n})$,
\begin{equation}\label{eqn::rnd_bound_aux3}
   \frac{\harmonic(\Omega_3^{\delta_n}(T_r^{\delta_n}); z, e_3^{\delta_n}(T_r^{\delta_n}))}{\harmonic(\Omega_3^{\delta_n}(T_r^{\delta_n}); v^{\delta_n}, e_3^{\delta_n}(T_r^{\delta_n}))}
   \le \constCW \frac{\Green(\Omega_3^{\delta_n}(T_r^{\delta_n});z, w)}{\Green(\Omega_3^{\delta_n}(T_r^{\delta_n});v^{\delta_n},w)}.
\end{equation}
Now we fix $w$ to be a vertex in $B(\gamma_3(T_r^{\delta_n}),r/2)\cap K_{3r/4}\cap\LV^{\circ}(\Omega^{\delta_n})$. 
Denote by $K_{r}^{\delta_n}$ the maximal domain contained in $K_r\cap \Omega^{\delta_n}$ such that $\partial K_r^{\delta_n}$ consists of edges of $\delta_n\hexagon$.
As $\dist(w,\partial K_{7r/8})\ge r/8$,  we have the following two bounds:
\begin{equation*}
   \Green(\Omega_3^{\delta_n}(T_r^{\delta_n});z, w)\le \Green(\Omega^{\delta_n};z, w), \qquad 
   \Green(\Omega_3^{\delta_n}(T_r^{\delta_n});v^{\delta_n},w)\ge \Green(K_{7r/8}^{\delta_n};v^{\delta_n},w).
\end{equation*}
Plugging into~\eqref{eqn::rnd_bound_aux3}, we have 
\begin{align*}
\frac{\harmonic(\Omega_3^{\delta_n}(T_r^{\delta_n}); z, e_3^{\delta_n}(T_r^{\delta_n}))}{\harmonic(\Omega_3^{\delta_n}(T_r^{\delta_n}); v^{\delta_n}, e_3^{\delta_n}(T_r^{\delta_n}))}
   \le \constCW \frac{\Green(\Omega^{\delta_n};z, w)}{\Green(K_{7r/8}^{\delta_n};v^{\delta_n},w)}, \quad \text{ for all } z\in K\cap\LV^{\circ}(\Omega^{\delta_n}). 
\end{align*}
Combining with~\eqref{eqn::Green_cvg}, there exists a constant $C_{\eqref{eqn::rnd_bound_aux5}}\in(0,\infty)$ depending on $(\Omega,K;x_1,x_2,x_3;v;r)$ such that
\begin{equation}\label{eqn::rnd_bound_aux5}
   \frac{\harmonic(\Omega_3^{\delta_n}(T_r^{\delta_n}); z, e_3^{\delta_n}(T_r^{\delta_n}))}{\harmonic(\Omega_3^{\delta_n}(T_r^{\delta_n}); v^{\delta_n}, e_3^{\delta_n}(T_r^{\delta_n}))}\le C_{\eqref{eqn::rnd_bound_aux5}},\quad \text{ for all } z\in K\cap\LV^{\circ}(\Omega^{\delta_n}). 
\end{equation}
Plugging~\eqref{eqn::rnd_bound_aux5} and~\eqref{eqn::rnd_bound_aux1} into~\eqref{eqn::rnd_bound_aux0}, we obtain the uniform bound of $\{\LU_r^{\delta_n}\}_n$.
\end{proof}

\subsection{Proof of Theorem~\ref{thm::tripod}}
\label{subsec::tripod_proof}
In this section, we complete the proof of Theorem~\ref{thm::tripod}. It is a collection of various tools developed in previous sections. Before the proof, we still need one more lemma which are consequences of Proposition~\ref{prop::cvg_branch_eta1}. 
\begin{lemma}\label{lem::tripod_aux}
Assume the same setup as in Theorem~\ref{thm::tripod}. 
Fix a compact subset $K$ of $\Omega$. 
Suppose $\ell$ is chordal $\SLE_2\sim\chordalSLE^{(\kappa=2)}(\Omega; x_1, x_2)$. The law of $\gamma_1^{\delta}\cup\gamma_2^{\delta}$ conditional on $\LA_1^{\delta}\cap\LA_2^{\delta}\cap\{\trifurcation^{\delta}\in K\}$ converges weakly to the law of $\ell$ weighted by the following Radon-Nikodym derivative 
\begin{align}\label{eqn::tripod_aux_RN}
\frac{\nharmonic(\Omega\setminus\ell; x_3, \ell\cap K)}{\chordalSLEexp^{(\kappa=2)}[\nharmonic(\Omega\setminus\ell; x_3, \ell\cap K)]}. 
\end{align}
Moreover, 
\begin{align}\label{eqn::tripod_aux_RN_exp}
\chordalSLEexp^{(\kappa=2)}[\nharmonic(\Omega\setminus\ell; x_3, \ell\cap K)]=\left(\frac{2\Poisson(\Omega; x_2, x_3)\Poisson(\Omega; x_3, x_1)}{\Poisson(\Omega; x_1, x_2)}\right)^{1/2}\int_K p(\Omega; x_1, x_2, x_3; w)|\ud w|^2. 
\end{align}
\end{lemma}
\begin{proof}
Note that $\gamma_1^{\delta}\cup\gamma_2^{\delta}\cup \{e_2^{\delta}\}$ has the same law as the boundary branch $\ell^{\delta}$ in UST from $x_1^{\delta, \circ}$ to $x_2^{\delta}$.
For any continuous bounded function $F$ on curves, we have  
\begin{align}
&\E\left[F(\gamma_1^{\delta}\cup\gamma_2^{\delta}\cup \{e_2^{\delta}\})\cond \LA_1^{\delta}\cap\LA_2^{\delta}\cap\{\trifurcation^{\delta}\in K\}\right]\notag\\
=&\frac{\E[F(\gamma_1^{\delta}\cup\gamma_2^{\delta}\cup \{e_2^{\delta}\})\one\{\LA_1^{\delta}\cap\LA_2^{\delta}\cap\{\trifurcation^{\delta}\in K\}\}]}{\PP[\LA_1^{\delta}\cap\LA_2^{\delta}\cap\{\trifurcation^{\delta}\in K\}]}\notag\\
=&\frac{\E[F(\ell^{\delta})\harmonic(\Omega^{\delta}\setminus \ell^{\delta}; x_3^{\delta, \circ}, \ell^{\delta}\cap K)\one\{x_1^{\delta, \circ}\rightsquigarrow e_2^{\delta}\}]}{\PP[\LA_1^{\delta}\cap\LA_2^{\delta}\cap\{\trifurcation^{\delta}\in K\}]}\notag\\
=&\frac{\PP[x_1^{\delta, \circ}\rightsquigarrow e_2^{\delta}]\harmonic(\Omega^{\delta}; v^{\delta}, e_3^{\delta})}{\PP[\LA_1^{\delta}\cap\LA_2^{\delta}\cap\{\trifurcation^{\delta}\in K\}]}
\E\left[F(\ell^{\delta})\frac{\harmonic(\Omega^{\delta}\setminus \ell^{\delta}; x_3^{\delta, \circ}, \ell^{\delta}\cap K)}{\harmonic(\Omega^{\delta}; v^{\delta}, e_3^{\delta})}\cond \{x_1^{\delta, \circ}\rightsquigarrow e_2^{\delta}\}\right]. \label{eqn::final_aux1}
\end{align}
We have the following observations.
\begin{itemize}
\item From~\eqref{eqn::bPoisson_cvg}, we have 
\begin{align}\label{eqn::final_aux2}
\lim_{\delta\to 0}\frac{\PP[x_1^{\delta, \circ}\rightsquigarrow e_2^{\delta}]}{\harmonic(\Omega^{\delta}; v^{\delta}, e_1^{\delta})\harmonic(\Omega^{\delta}; v^{\delta}, e_2^{\delta})}=\frac{2\sqrt{3}\pi\Poisson(\Omega; x_1, x_2)}{\Poisson(\Omega; v, x_1)\Poisson(\Omega; v, x_2)}. 
\end{align}
\item From Proposition~\ref{prop::observable_cvg}, we have 
\begin{align}\label{eqn::final_aux3}
\lim_{\delta\to 0}\frac{\PP[\LA_1^{\delta}\cap\LA_2^{\delta}\cap\{\trifurcation^{\delta}\in K\}]}{\prod_{j=1}^3 \harmonic(\Omega^{\delta}; v^{\delta}, e_j^{\delta})}=\frac{8\sqrt{2}}{\prod_{j=1}^3 \Poisson(\Omega; v, x_j)}\int_K\LZtripod(\Omega; x_1, x_2, x_3; w)|\ud w|^2. 
\end{align}
\item From Lemma~\ref{lem::boundary_branch_cvg1}, the law of $\ell^{\delta}$ conditional on $\{x_1^{\delta, \circ}\rightsquigarrow e_2^{\delta}\}$ converges weakly to the law of $\ell$. The analysis in Lemma~\ref{lem::boundary_branch_cvg_L1} asserts
\begin{align}\label{eqn::final_aux4}
\lim_{\delta\to 0}\E\left[F(\ell^{\delta})\frac{\harmonic(\Omega^{\delta}\setminus \ell^{\delta}; x_3^{\delta, \circ}, \ell^{\delta}\cap K)}{\harmonic(\Omega^{\delta}; v^{\delta}, e_3^{\delta})}\cond \{x_1^{\delta, \circ}\rightsquigarrow e_2^{\delta}\}\right]=\E\left[F(\ell)\frac{\sqrt{3}\nharmonic(\Omega\setminus\ell; x_3, \ell\cap K)}{\Poisson(\Omega; v, x_3)}\right]. 
\end{align}
\end{itemize}
Plugging~\eqref{eqn::final_aux2}-\eqref{eqn::final_aux4} into~\eqref{eqn::final_aux1}, we obtain
\begin{align*}
\lim_{\delta\to 0}\E\left[F(\gamma_1^{\delta}\cup\gamma_2^{\delta})\cond \LA_1^{\delta}\cap\LA_2^{\delta}\cap\{\trifurcation^{\delta}\in K\}\right]=\frac{3\pi \Poisson(\Omega; x_1, x_2)}{4\sqrt{2}\int_K\LZtripod(\Omega; x_1, x_2, x_3; w)|\ud w|^2}\E\left[F(\ell)\nharmonic(\Omega\setminus\ell; x_3, \ell\cap K)\right]. 
\end{align*}
Combining with
\begin{align*}
\int_{K}\LZtripod(\Omega; x_1, x_2, x_3; w)|\ud w|^2
=&{\frac{3\pi}{4}} \left(\Poisson(\Omega; x_1, x_2)\Poisson(\Omega; x_2, x_3)\Poisson(\Omega; x_3, x_1)\right)^{1/2}\int_{K}p(\Omega; x_1, x_2, x_3; w)|\ud w|^2,
\end{align*}
we obtain the conclusion. 
\end{proof}

\begin{proof}[Proof of Theorem~\ref{thm::tripod}]
As the tripod $\{(\gamma_1^{\delta}, \gamma_2^{\delta}, \gamma_3^{\delta})\}_{\delta>0}$ is tight due to Proposition~\ref{prop::tripod_tightness}, for any sequence $\delta_n\to 0$, there exists a subsequence, which we still denote by $\{\delta_n\}$, such that the conditional law of $\{(\gamma_1^{\delta_n}, \gamma_2^{\delta_n}, \gamma_3^{\delta_n})\}_n$ converges. We can couple 
$\{(\gamma_1^{\delta_n}, \gamma_2^{\delta_n}, \gamma_3^{\delta_n})\}_n$ and its limit $(\gamma_1, \gamma_2, \gamma_3)$ together, such that $\gamma_j^{\delta_n}\to \gamma_j$ in metric~\eqref{eqn::curves_metric} almost surely for $j=1,2,3$. We denote this coupling by $\PP$. Proposition~\ref{prop::tripod_tightness} also asserts that there exists $\trifurcation\in\Omega$ such that $(\gamma_1, \gamma_2, \gamma_3)\in\chamber(\Omega; x_1, x_2, x_3; \trifurcation)$ almost surely. 
Let us derive the law of $(\gamma_1, \gamma_2, \gamma_3)$. 
By setting $F=1$ in~\eqref{eqn::gamma3_cvg}, we have
\begin{align}\label{eqn::trifurcation_distribution_repeat}
\lim_{\delta=0}\PP[\trifurcation^{\delta}\in K\cond \LA_1^{\delta}\cap\LA_2^{\delta}]=\int_K p(\Omega; x_1, x_2, x_3; u)|\ud u|^2,\qquad\text{for any compact subset $K$ of $\Omega$}. 
\end{align}
This implies~\eqref{eqn::trifurcation_distribution}.
\medbreak

Fix $z\in \Omega$ and suppose $0<\eps<r$ and $B(z,4r)\subset\Omega$. 
Let $T_r^{\delta}$ (resp. $T_r$) be the first time $\gamma_3^{\delta}$ (resp. $\gamma_3$) hits $B(z,r)$. 
Plugging~\eqref{eqn::trifurcation_distribution_repeat} into~\eqref{eqn::gamma3_cvg}, for any bounded continuous function $F$ on curves, we obtain
\begin{align*}
&\lim_{\delta\to 0}\E\left[F(\gamma_3^{\delta}[0,T_r^{\delta}])\cond \LA_1^{\delta}\cap\LA_2^{\delta}\cap\{\trifurcation^{\delta}\in \overline{B}(z,\eps)\}\right]\notag\\
=&\frac{1}{\int_{\overline{B}(z,\eps)}p(\Omega; x_1,x_2,x_3;w) |\ud w|^2}\int_{\overline{B}(z,\eps)}\nradE{3}^{(\kappa=2)}(\Omega; x_1, x_2, x_3; u)[F(\gamma_3[0,T_r])]
p(\Omega; x_1,x_2,x_3;u) |\ud u|^2. 
\end{align*}
This implies that the law of $\gamma_3$ given $\trifurcation$ is radial $\SLE_2(2,2)$ in $\Omega$ from $x_3$ to $\trifurcation$ with force points $(x_1, x_2)$. 
\medbreak

Finally, let us address the conditional law of $(\gamma_1, \gamma_2)$ given $(\gamma_3, \trifurcation)$. To this end, we will compare the relation between the following three measures.
\begin{itemize}
\item The conditional law of $(\gamma_1, \gamma_2)$ given $(\gamma_3, \trifurcation)$ under $\PP$ (the above coupling). 
\item Suppose $\ell$ is chordal $\SLE_2\sim\chordalSLE^{(\kappa=2)}(\Omega; x_1, x_2)$ from $x_1$ to $x_2$. Let $\gamma_1$ be $\ell$ and let $\gamma_2$ be the time-reversal of $\ell$. Consider the law of $(\gamma_1, \gamma_2)$ under $\chordalSLE^{(\kappa=2)}(\Omega; x_1, x_2)$. 
\item The law of $(\gamma_1, \gamma_2)$ given $\gamma_3$ under $\nradP{3}^{(\kappa=2)}(\Omega; x_1, x_2, 
x_3; z)$. 
\end{itemize}
Given $\gamma_3[0,T_r]$, let $\phi_r$ be the conformal map from $\Omega_3(T_r)$ onto $\Omega$ with \[\phi_r(x_1)=x_1, \quad \phi_r(x_2)=x_2, \quad\phi_r(\gamma_3(T_r))=x_3.\] 
Fix $U\subset\Omega$ such that $(U; x_1, x_2)$ is a $2$-polygon and it has a positive distance from $\{x_3\}$. Note that $\phi_r(z)\not\in U$ when $r>0$ small enough. Let $\tau_j^U$ be the first time $\phi_r(\gamma_j)$ exits $U$. 
Let $g_U$ be the conformal map from $\Omega\setminus(\phi_r(\gamma_1[0,\tau_1^U]\cup\gamma_2[0,\tau_2^U]))$ onto $\Omega$ such that
\[g_U(\phi_r(\gamma_1(\tau_1^U)))=x_1, \qquad g_U(\phi_r(\gamma_2(\tau_2^U)))=x_2,\qquad g_U(x_3)=x_3.\]
See Figure~\ref{fig::final_control}. 
\begin{itemize}
\item From the domain Markov property of three-sided radial $\SLE$, the conditional law of $(\phi_r(\gamma_1), \phi_r(\gamma_2))$ given $\gamma_3[0,T_r]$ under $\nradP{3}^{(\kappa=2)}(\Omega; x_1, x_2, x_3; z)$ is the same as their law under $\nradP{3}^{(\kappa=2)}(\Omega; x_1, x_2, 
x_3; \phi_r(z))$. From Lemma~\ref{lem::3radvschordal}, the law of $(\phi_r(\gamma_1[0,\tau_1^U]), \phi_r(\gamma_2[0,\tau_2^U]))$ given $\gamma_3[0,T_r]$ under $\nradP{3}^{(\kappa=2)}(\Omega; x_1, x_2, x_3; z)$ is the same as $\chordalSLE^{(\kappa=2)}(\Omega; x_1, x_2)$ weighted by the following Radon-Nikodym derivative
\begin{align}\label{eqn::final_aux5}
\one\{\phi_r(\gamma_1[0,\tau_1^U])\cap\phi_r(\gamma_2[0,\tau_2^U])=\emptyset\}|g_U'(x_3)||g_U'(\phi_r(z))|^{2}\frac{\LZtripod(\Omega; x_1, x_2, x_3; g_U(\phi_r(z)))}{\LZtripod(\Omega; x_1, x_2, x_3; \phi_r(z))}.
\end{align}

\item 
We claim that the law of $(\phi_r(\gamma_1[0,\tau_1^U]), \phi_r(\gamma_2[0,\tau_2^U]))$ given $\gamma_3[0,T_r]$ and $\{\trifurcation \in B(z,\eps)\}$ under $\PP$ is the same as $\chordalSLE^{(\kappa=2)}(\Omega; x_1, x_2)$ weighted by the following Radon-Nikodym derivative
\begin{align}\label{eqn::final_aux7}
\frac{|g_U'(x_3)|\int_{g_U(\phi_r(B(z,\eps))))}p(\Omega; x_1, x_2, x_3; w)|\ud w|^2}{\int_{\phi_r(B(z,\eps))}p(\Omega; x_1, x_2, x_3; w)|\ud w|^2}. 
\end{align}
From the domain Markov property of time-reversal of LERW~\cite[Lemma~3.2]{LawlerSchrammWernerLERWUST} and Lemma~\ref{lem::tripod_aux}, the law of $\gamma_1^{\delta}\cup\gamma_2^{\delta}$ conditional on $\gamma_3^{\delta}[0,T_r^{\delta}]$ and $\LA_1^{\delta}\cap\LA_2^{\delta}\cap\{\trifurcation^{\delta}\in \overline{B}(z,\eps)\}$ converges weakly to the law of $\tilde{\ell}$, chordal $\SLE_2$ in $(\Omega_3(T_r); x_1, x_2)$, weighted by the following Radon-Nikodym derivative 
\begin{align*}
\left(\frac{\Poisson(\Omega_3(T_r); x_1, x_2)}{2\Poisson(\Omega_3(T_r); x_2, \gamma_3(T_r))\Poisson(\Omega_3(T_r); \gamma_3(T_r), x_1)}\right)^{1/2}
\frac{\nharmonic(\Omega\setminus(\gamma_3[0,T_r]\cup\tilde{\ell}); \gamma_3(T_r), \tilde{\ell}\cap B(z,\eps))}{\int_{B(z,\eps)} p(\Omega_3(T_r); x_1, x_2, 
\gamma_3(T_r); w)|\ud w|^2}. 
\end{align*}
Therefore, the law of $(\phi_r(\gamma_1), \phi_r(\gamma_2))$ given $\gamma_3[0,T_r]$ and $\{\trifurcation \in B(z,\eps)\}$ under $\PP$ is the same as $\ell\sim\chordalSLE^{(\kappa=2)}(\Omega; x_1, x_2)$ weighted by 
\begin{align*}
\left(\frac{\Poisson(\Omega; x_1, x_2)}{2\Poisson(\Omega; x_2, x_3)\Poisson(\Omega; x_3, x_1)}\right)^{1/2}\frac{\nharmonic(\Omega\setminus\ell; x_3, \ell\cap \phi_r(B(z,\eps)))}{\int_{\phi_r(B(z,\eps))}p(\Omega; x_1, x_2, x_3; w)|\ud w|^2}. 
\end{align*}
Furthermore, the law of $(\phi_r(\gamma_1[0,\tau_1^U]), \phi_r(\gamma_2[0,\tau_2^U]))$ given $\gamma_3[0,T_r]$ and $\{\trifurcation \in B(z,\eps)\}$ under $\PP$ is the same as $\chordalSLE^{(\kappa=2)}(\Omega; x_1, x_2)$ weighted by the following Radon-Nikodym derivative
\begin{align}\label{eqn::final_aux6}
\left(\frac{\Poisson(\Omega; x_1, x_2)}{2\Poisson(\Omega; x_2, x_3)\Poisson(\Omega; x_3, x_1)}\right)^{1/2}\frac{\chordalSLEexp^{(\kappa=2)}[\nharmonic(\Omega\setminus\ell; x_3, \ell\cap \phi_r(B(z,\eps)))\cond \phi_r(\gamma_1[0,\tau_1^U]), \phi_r(\gamma_2[0,\tau_1^U])]}{\int_{\phi_r(B(z,\eps))}p(\Omega; x_1, x_2, x_3; w)|\ud w|^2}. 
\end{align}
Let us calculate the conditional expectation:
\begin{align*}
&\chordalSLEexp^{(\kappa=2)}[\nharmonic(\Omega\setminus\ell; x_3, \ell\cap \phi_r(B(z,\eps)))\cond \phi_r(\gamma_1[0,\tau_1^U]), \phi_r(\gamma_2[0,\tau_1^U])]\\
=&|g_U'(x_3)|\chordalSLEexp^{(\kappa=2)}[\nharmonic(\Omega\setminus\ell; x_3, \ell\cap g_U(\phi_r(B(z,\eps))))\cond \phi_r(\gamma_1[0,\tau_1^U]), \phi_r(\gamma_2[0,\tau_1^U])]\tag{due to~\eqref{eqn::nharmonic_cov}}\\
=&|g_U'(x_3)|\left(\frac{2\Poisson(\Omega; x_2, x_3)\Poisson(\Omega; x_3, x_1)}{\Poisson(\Omega; x_1, x_2)}\right)^{1/2}\int_{g_U(\phi_r(B(z,\eps))))}p(\Omega; x_1, x_2, x_3; w)|\ud w|^2. \tag{due to~\eqref{eqn::tripod_aux_RN_exp}}
\end{align*}
Plugging into~\eqref{eqn::final_aux6}, we obtain~\eqref{eqn::final_aux7} as desired. 
\end{itemize}
Combining~\eqref{eqn::final_aux5} and~\eqref{eqn::final_aux7}, the law of $(\phi_r(\gamma_1[0,\tau_1^U]), \phi_r(\gamma_2[0,\tau_2^U]))$ given $\gamma_3[0,T_r]$ and $\{\trifurcation \in B(z,\eps)\}$ under $\PP$ is the same as 
the law of $(\phi_r(\gamma_1[0,\tau_1^U]), \phi_r(\gamma_2[0,\tau_2^U]))$ given $\gamma_3[0,T_r]$ under $\nradP{3}^{(\kappa=2)}(\Omega; x_1, x_2, x_3; z)$ weighted by the following Radon-Nikodym derivative
\begin{align}\label{eqn::final_RN}
\LR(\eps, r):=\frac{\int_{g_U(\phi_r(B(z,\eps))))}p(\Omega; x_1, x_2, x_3; w)|\ud w|^2}{\int_{\phi_r(B(z,\eps))}p(\Omega; x_1, x_2, x_3; w)|\ud w|^2}\frac{p(\Omega; x_1, x_2, x_3; \phi_r(z))}{|g_U'(\phi_r(z))|^2 p(\Omega; x_1, x_2, x_3; g_U(\phi_r(z)))}. 
\end{align}
It is clear that $\lim_{\eps\to 0}\LR(\eps, r)=1$ almost surely; and we will prove in Lemma~\ref{lem::final_control} that 
\begin{equation}\label{eqn::final_control_RNuniformbound}
\LR(\eps, r)\lesssim 1, \qquad \text{for all }\eps\le r/16.
\end{equation}
This gives the conclusion that the law of $(\phi_r(\gamma_1[0,\tau_1^U]), \phi_r(\gamma_2[0,\tau_2^U]))$ given $(\gamma_3[0,T_r], \trifurcation)$ under $\PP$ is the same as 
the law of $(\phi_r(\gamma_1[0,\tau_1^U]), \phi_r(\gamma_2[0,\tau_2^U]))$ given $\gamma_3[0,T_r]$ under $\nradP{3}^{(\kappa=2)}(\Omega; x_1, x_2, x_3; \trifurcation)$ as desired.  
\end{proof}

\begin{figure}[ht!]
   \begin{center}
   \includegraphics[width=0.9\textwidth]{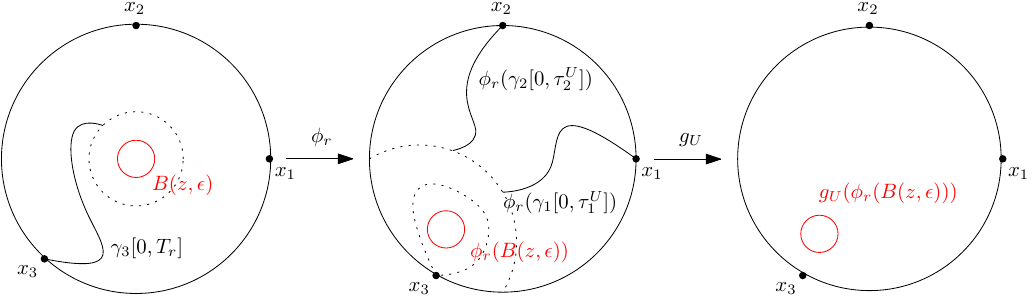}
   \end{center}
   \caption{The conformal map $\phi_r$ is from $\Omega_3(T_r)=\Omega\setminus\gamma_3[0,T_r]$ onto $\Omega$ such that $\phi_r(x_1)=x_1$, $\phi_r(x_2)=x_2$ and $\phi_r(\gamma_3(T_r))=x_3$. The conformal map $g_U$ is from $\Omega\setminus(\phi_r(\gamma_1[0,\tau_1^U]\cup\gamma_2[0,\tau_2^U]))$ onto $\Omega$ such that $g_U(\phi_r(\gamma_1(\tau_1^U)))=x_1$, $g_U(\phi_r(\gamma_2(\tau_2^U)))=x_2$ and $g_U(x_3)=x_3$.}
   \label{fig::final_control}
  \end{figure}

\begin{lemma}\label{lem::final_control}
Assume the same notation as in Proof of Theorem~\ref{thm::tripod}. The estimate~\eqref{eqn::final_control_RNuniformbound} holds. 
\end{lemma}
\begin{proof}
We denote 
\begin{align*}
&\varphi_1(\cdot)=\phi_r(\cdot),\quad z_1=\phi_r(z), \quad \eps_1=\frac{1}{4}\eps|\phi_r'(z)|;\\
&\varphi_2(\cdot)=g_U(\phi_r(\cdot)), \quad z_2=g_U(\phi_r(z)), \quad \eps_2=\frac{1}{4}\eps |g_U'(\phi_r(z))\phi_r'(z)|.
\end{align*}
Koebe’s one quarter theorem and the growth theorem assert that, for all $\eps\le r/16$, 
\begin{align*}
B(z_j, \eps_j)\subset \varphi_j(B(z, \eps))\subset B(z_j, 8\eps_j). 
\end{align*}
 Thus,
\begin{align}\label{eqn::final_control_aux1}
\LR(\eps, r)\le \frac{\int_{B(z_2, 8\eps_2)}p(\Omega; x_1, x_2, x_3; w)|\ud w|^2}{\int_{B(z_1, \eps_1)}p(\Omega; x_1, x_2, x_3; w)|\ud w|^2}\frac{p(\Omega; x_1, x_2, x_3; z_1)}{|g_U'(\phi_r(z))|^2 p(\Omega; x_1, x_2, x_3; z_2)}.
\end{align}

From~\eqref{eqn::density_def}, we have 
\begin{align*}
\begin{split}
p(\Omega; x_1, x_2, x_3; w)
=&{\frac{4}{3\pi}}\frac{\Poisson(\Omega; w, x_1)^2\Poisson(\Omega; w, x_2)^2\Poisson(\Omega; w, x_3)^2}{\CR(\Omega; w)^2\Poisson(\Omega; x_1, x_2)\Poisson(\Omega; x_2, x_3)\Poisson(\Omega; x_3, x_1)}. 
\end{split}
\end{align*}
For $j=1,2$ and $w\in B(z_j, 8\eps_j)$, let us control the ratio
\begin{align}\label{eqn::final_controal_density_ratio}
\frac{p(\Omega; x_1, x_2, x_3; w)}{p(\Omega; x_1, x_2, x_3; z_j)}=\left(\frac{\Poisson(\Omega; w, x_1)\Poisson(\Omega; w, x_2)\Poisson(\Omega; w, x_3)\CR(\Omega; z_j)}{\Poisson(\Omega; z_j, x_1)\Poisson(\Omega; z_j, x_2)\Poisson(\Omega; z_j, x_3)\CR(\Omega; w)}\right)^2. 
\end{align}
\begin{itemize}
\item For the conformal radius, Koebe’s one quarter theorem and the Schwarz lemma assert that 
\begin{align*}
\dist(w, \partial\Omega)\le \CR(\Omega; w)\le 4\dist(w, \partial\Omega). 
\end{align*}
Thus, for all $\eps\le r/16$, 
\begin{align}\label{eqn::CR_dist_eps}
\CR(\Omega; z_j)\ge \dist(z_j, \partial\Omega)\ge \frac{r}{4}|\varphi_j'(z)|\ge 16\eps_j.
\end{align}
Thus, for $w\in B(z_j, 8\eps_j)$, we have 
\begin{align*}
\CR(\Omega; w)\le& 4\dist(w, \partial\Omega)\le 4\dist(z_j, \partial\Omega)+32\eps_j\le 6\dist(z_j, \partial\Omega)\le 6\CR(\Omega; z_j); \\
\CR(\Omega; w)\ge& \dist(w, \partial\Omega)\ge \dist(z_j, \partial\Omega)-8\eps_j\ge \frac{1}{2}\dist(z_j, \partial\Omega)\ge \frac{1}{8}\CR(\Omega; z_j). 
\end{align*}
In conclusion,
\begin{align}\label{eqn::final_control_CR}
\frac{1}{8}\le \frac{\CR(\Omega; w)}{\CR(\Omega; z_j)}\le 6, \quad \text{for all }w\in B(z_j, 8\eps_j). 
\end{align}
\item For the Poisson kernel, for fixed $x\in\partial\Omega$, the function $\Poisson(\Omega; \cdot, x)$ is harmonic. Thus, Harnack's inequality asserts that
\begin{align*}
\frac{\dist(z_j, \partial\Omega)-8\eps_j}{\dist(z_j, \partial\Omega)+8\eps_j}\le \frac{\Poisson(\Omega; w, x)}{\Poisson(\Omega; z_j, x)}\le \frac{\dist(z_j, \partial\Omega)+8\eps_j}{\dist(z_j, \partial\Omega)-8\eps_j}, \qquad \text{for all }w\in B(z_j, 8\eps_j). 
\end{align*}
Combining with~\eqref{eqn::CR_dist_eps}, we have 
\begin{align}\label{eqn::final_control_Poisson}
\frac{1}{3}\le \frac{\Poisson(\Omega; w, x)}{\Poisson(\Omega; z_j, x)}\le 3, \qquad \text{for all }w\in B(z_j, 8\eps_j). 
\end{align}
\end{itemize}
Plugging~\eqref{eqn::final_control_CR} and~\eqref{eqn::final_control_Poisson} into~\eqref{eqn::final_controal_density_ratio}, we have\footnote{We write $F\asymp G$ if $F/G$ is bounded both sides by universal positive finite constants.}
\begin{align*}
\frac{p(\Omega; x_1, x_2, x_3; w)}{p(\Omega; x_1, x_2, x_3; z_j)}\asymp 1, \qquad \text{for all }w\in B(z_j, 8\eps_j). 
\end{align*}
Plugging into~\eqref{eqn::final_control_aux1}, we have
\begin{align*}
\LR(\eps, r)\lesssim \frac{\eps_2^2}{|g_U'(\phi_r(z))|^2\eps_1^2}=1, 
\end{align*}
as desired. 
\end{proof}

\subsection{Convergence of the tripod for $\Z^2$ lattice}
\label{subsec::tripod_otherlattice}
In this section, we explain that the conclusion in Theorem~\ref{thm::tripod} also holds for $\Z^2$ lattice approximation. 

\begin{corollary}\label{cor::tripod_Z2}
Fix a bounded $3$-polygon $(\Omega; x_1, x_2, x_3)$ and suppose $(\Omega^{\delta}; x_1^{\delta}, x_2^{\delta}, x_3^{\delta})$
 is an approximation of $(\Omega; x_1, x_2, x_3)$ on $\delta\Z^2$ in Carath\'eodory sense. 
We assume further that $\partial\Omega^{\delta}$ converges to $\partial\Omega$ in Hausdorff distance~\eqref{eqn::boundary_cvg_Hausdorff}. 
Consider the UST in $\Omega^{\delta}$ with wired boundary condition and 
define $\LA_1^{\delta}$ and $\LA_2^{\delta}$ as in~\eqref{eqn::conditionalevent} and consider the tripod $(\gamma_1^{\delta}, \gamma_2^{\delta}, \gamma_3^{\delta})$ and the trifurcation $\trifurcation^{\delta}$. 
The law of the tripod $(\gamma_1^{\delta}, \gamma_2^{\delta}, \gamma_3^{\delta})$ conditional on $\LA_1^{\delta}\cap \LA_2^{\delta}$ converges weakly to three continuous simple curves $(\gamma_1, \gamma_2, \gamma_3)$ whose law is characterized by the following properties. \begin{itemize}
\item[(1)] There exists $\trifurcation\in\Omega$ such that $(\gamma_1, \gamma_2, \gamma_3)\in\chamber(\Omega; x_1, x_2, x_3; \trifurcation)$. 
The law of the trifurcation $\trifurcation$ is absolutely continuous with respect to Lebesgue measure on $\Omega$ and the density is given by $p(\Omega; x_1, x_2, x_3; z)$ defined in~\eqref{eqn::density_def}. 
\item[(2)] Given $\trifurcation$, the conditional law of the tripod $(\gamma_1, \gamma_2, \gamma_3)$ is three-sided radial $\SLE_2$ in $(\Omega; x_1, x_2, x_3; \trifurcation)$ in Definition~\ref{def::3sidedradialSLE}. 
\end{itemize}
\end{corollary}
\begin{proof}
The conclusions and proofs in Sections~\ref{subsec::boundary_branch_cvg} and~\ref{subsec::tightness} hold for $\Z^2$ approximation. In particular, Lemma~\ref{lem::tripod_ab} holds for $\Z^2$ approximation. Thus, the limiting law of the tripod can be described by (a)-(b) in Lemma~\ref{lem::tripod_ab}. 
The law of $(\gamma_1, \gamma_2, \gamma_3)$ described by (a)-(b) in Lemma~\ref{lem::tripod_ab} is the same as the law described by (1)-(2) due to Theorem~\ref{thm::tripod}. This completes the proof. 
\end{proof}

Next, let us check how to derive the conclusion in Theorem~\ref{thm::trifurcation} for $\Z^2$ lattice approximation. Assume the same setup as in Corollary~\ref{cor::tripod_Z2}, the law of the trifurcation $\trifurcation^{\delta}$ converges weakly to the law of $\trifurcation$ given by~\eqref{eqn::trifurcation_distribution}. In particular, fix $z\in\Omega$ and for any $\eps>0$, 
\begin{equation}\label{eqn::trifurcation_Z2_aux1}
\lim_{\delta\to 0}\PP[\trifurcation^{\delta}\in B(z,\eps)\cond \LA_1^{\delta}\cap\LA_2^{\delta}]=\int_{B(z,\eps)} p(\Omega; x_1, x_2, x_3; w)|\ud w|^2.
\end{equation}
Suppose $z^{\delta}$ is a vertex in $\LV^{\circ}(\Omega^{\delta})$ that is nearest to $z$. We would like to show that 
\begin{align}\label{eqn::trifurcation_Z2_aux2}
\lim_{\delta\to 0}\delta^{-2}\PP\left[\trifurcation^{\delta}=z^{\delta}\cond \LA_1^{\delta}\cap\LA_2^{\delta}\right]=p(\Omega; x_1, x_2, x_3; z). 
\end{align}

The convergence~\eqref{eqn::trifurcation_Z2_aux2} would follow from~\eqref{eqn::trifurcation_Z2_aux1} and the following additional control: for every $\xi>0$, there exists $\eps=\eps(\xi)\to 0$ as $\xi\to 0$, such that for every $w^\delta\in B(z^\delta,\eps)$, we have
\begin{equation}\label{eqn::trifurcation_Z2_aux3}
\left|\frac{\PP[\trifurcation^{\delta}=w^\delta\cond \LA_1^{\delta}\cap\LA_2^{\delta}]}{\PP[\trifurcation^{\delta}=z^\delta\cond \LA_1^{\delta}\cap\LA_2^{\delta}]}-1\right|<\xi.
\end{equation}
Assume we have proved~\eqref{eqn::trifurcation_Z2_aux3}. Note that 
\[\PP[\trifurcation^{\delta}\in B(z,\eps)\cond \LA_1^{\delta}\cap\LA_2^{\delta}]=\sum_{w^\delta\in B(z,\eps)\cap V(\Omega_\delta)}\PP[\trifurcation^{\delta}=w^\delta\cond \LA_1^{\delta}\cap\LA_2^{\delta}].\]
Combining with~\eqref{eqn::trifurcation_Z2_aux1} and~\eqref{eqn::trifurcation_Z2_aux3}, we have
\[\frac{1}{(1+\xi)\operatorname*{Area}(B(z,\eps))}\int_{B(z,\eps)} p(\Omega; x_1, x_2, x_3; w)|\ud w|^2\le\liminf_{\delta\to 0}\delta^{-2}\PP\left[\trifurcation^{\delta}=z^{\delta}\cond \LA_1^{\delta}\cap\LA_2^{\delta}\right]\]
and
\[
\limsup_{\delta\to 0}\delta^{-2}\PP\left[\trifurcation^{\delta}=z^{\delta}\cond \LA_1^{\delta}\cap\LA_2^{\delta}\right]\le \frac{1}{(1-\xi)\operatorname*{Area}(B(z,\eps))}\int_{B(z,\eps)} p(\Omega; x_1, x_2, x_3; w)|\ud w|^2.
\]
By letting $\xi\to 0$, we obtain~\eqref{eqn::trifurcation_Z2_aux2}. 
The control~\eqref{eqn::trifurcation_Z2_aux3} might be proved via coupling of UST on different graphs, but it seems highly non-trivial for us.
\newpage
\appendix
\section{The trifurcation observable for 4-8 lattice}
\label{appendix::48lattice}
In this appendix, we first derive analogous conclusion of Proposition~\ref{prop::observable_cvg} for 4-8 lattice in Proposition~\ref{prop::observable_cvg_48}.

\begin{proposition}\label{prop::observable_cvg_48}
Fix a bounded $3$-polygon $(\Omega; x_1, x_2, x_3)$ and suppose $(\Omega^{\delta}; x_1^{\delta}, x_2^{\delta}, x_3^{\delta})$ is an approximation of $(\Omega; x_1, x_2, x_3)$ on 
$\delta$-scaled 4-8 lattice in Carath\'eodory sense. 
Consider UST on $\Omega^{\delta}$ with wired boundary condition. 
Define $\LA_1^{\delta}=\{x_1^{\delta, \circ}\rightsquigarrow e_3^{\delta}\}$ and $\LA_2^{\delta}=\{x_2^{\delta, \circ}\rightsquigarrow e_3^{\delta}\}$ as in~\eqref{eqn::conditionalevent} and consider the trifurcation $\trifurcation^{\delta}$.
We fix an interior point $v\in\Omega$ and suppose $v^{\delta}$ is the vertex in $\LV^{\circ}(\Omega^{\delta})$ that is nearest to $v$. 
Define 
\begin{equation}\label{eqn::deltafive_integrand_48}
f^{\delta}(u)=\frac{\PP\left[\LA_1^{\delta}\cap \LA_2^{\delta}\cap \{\trifurcation^{\delta}=u\}\right]}{\delta^2\prod_{j=1}^3 \harmonic(\Omega^\delta; v^\delta, e_j^\delta)},\qquad u\in\LV^{\circ}(\Omega^{\delta}). 
\end{equation}
We define the triangle $\bigtriangleup^{\delta}(u)$ covering vertex $u$ as in Figure~\ref{fig::48lattice}.  
For $u\in\LV^{\circ}(\Omega^{\delta})$, we set $f^{\delta}(z)=f^{\delta}(u)$ for $z\in\bigtriangleup^{\delta}(u)$. 
This is a step function defined on $\Omega^{\delta}$ and we set $f^{\delta}=0$ outside of $\Omega^{\delta}$. Define
\begin{equation}\label{eqn::deltafive_integrand_continuum_48}
f(z):=\frac{{4(2+\sqrt{2})}\LZtripod(\Omega; x_1, x_2, x_3; z)}{\prod_{j=1}^{3}\Poisson(\Omega; v, x_j)}, \qquad z\in\Omega, 
\end{equation}
and set $f=0$ outside of $\Omega$. 
Then $f^{\delta}$ converges to $f$ uniformly on compact subsets of $\Omega$. 
\end{proposition}
\begin{proof}
First, the key lemma---Lemma~\ref{lem::Fomin}---remains true since the proof is valid for any  graph such that $\deg(v)=3$ for all inner vertices $v$.
Thus for $z^\delta\in\LV^{\circ}(\Omega^{\delta})$, denote by $z_1^{\delta}, z_2^{\delta}, z_3^{\delta}$ the three neighbors of $z^\delta$ arranged in counterclockwise order, then we have
\begin{equation}\label{eqn::keylemma_48}
\PP\left[\LA_1^{\delta}\cap \LA_2^{\delta}\cap \{\trifurcation^{\delta}=z^\delta\}\right]=\det\left(\harmonic(\Omega^{\delta}; z_k^{\delta},e_j^\delta)\right)_{j,k=1}^3.
\end{equation}

Second, we consider the convergence of the harmonic measures.
For $u\in\LV^{\circ}(\Omega^{\delta})$, it has three adjacent vertices $u_1, u_2, u_3$. There are four possible cases for the directions of the edges $\langle u,u_1\rangle, \langle u,u_2\rangle, \langle u,u_3\rangle$ (see Figure~\ref{fig::48lattice}):
\begin{align*}
\text{Case 1:}\qquad &u_k=u+\delta\bs{\beta}_k, \qquad \text{where }\bs{\beta}_1=1,\; \bs{\beta}_2=\ee^{3\pi\ii/4},\; \bs{\beta}_3=\ee^{3\pi\ii/2}; \\
\text{Case 2:}\qquad 
&u_k=u+\delta\bs{\beta}_k, \qquad \text{where }\bs{\beta}_1=-\ii,\; \bs{\beta}_2=\ee^{\pi\ii/4},\; \bs{\beta}_3=-1; \\
\text{Case 3:}\qquad 
&u_k=u+\delta\bs{\beta}_k, \qquad \text{where }\bs{\beta}_1=-1,\; \bs{\beta}_2=\ee^{-\pi\ii/4},\; \bs{\beta}_3=\ii; \\
\text{Case 4:}\qquad 
&u_k=u+\delta\bs{\beta}_k, \qquad \text{where }\bs{\beta}_1=\ii,\; \bs{\beta}_2=\ee^{5\pi\ii/4},\; \bs{\beta}_3=1.
\end{align*}

\begin{figure}[ht!]
\begin{subfigure}[b]{0.2\textwidth}
\begin{center}
\includegraphics[width=\textwidth]{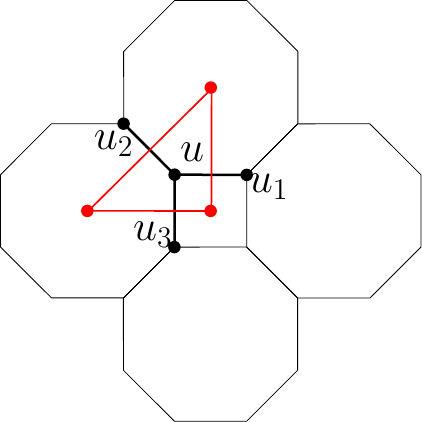}
\end{center}
\caption{Case 1.}
\end{subfigure}
\begin{subfigure}[b]{0.2\textwidth}
\begin{center}
\includegraphics[width=\textwidth]{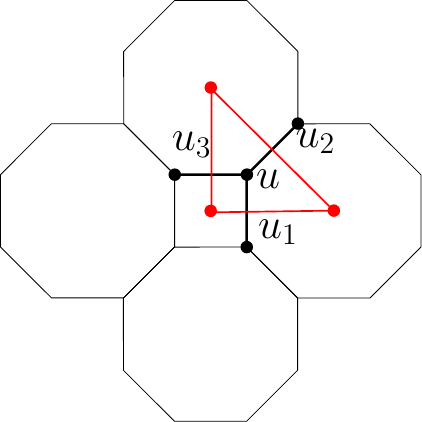}
\end{center}
\caption{Case 2.}
\end{subfigure}
\begin{subfigure}[b]{0.2\textwidth}
\begin{center}
\includegraphics[width=\textwidth]{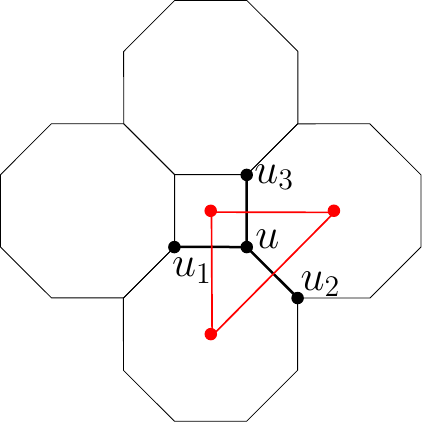}
\end{center}
\caption{Case 3.}
\end{subfigure}
\begin{subfigure}[b]{0.2\textwidth}
\begin{center}
\includegraphics[width=\textwidth]{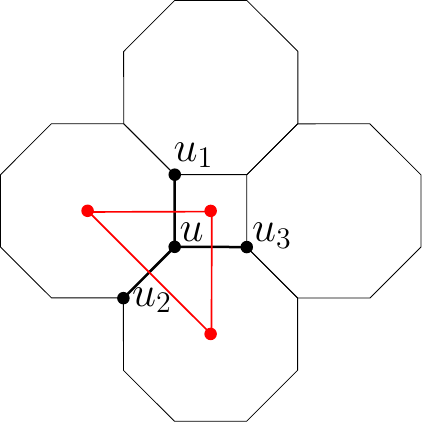}
\end{center}
\caption{Case 4.}
\end{subfigure}
\caption{\label{fig::48lattice}The directions of three adjacent edges of $u$ on 4-8 lattice have four possible cases. Denote by $\bigtriangleup^{\delta}(u)$ the dual face (in red) of $u$ in 4-8 lattice. } 
\end{figure}
For any subsequence of $\{z^{\delta}\}$ of Case 1, by the same proof of Lemma~\ref{lem::deltafive_cvg}, we obtain
\begin{equation}\label{eqn::Poisson_derivative_cvg_48}
\lim_{\delta\to 0}\frac{\det\left(\harmonic(\Omega^\delta;z_k^\delta,e_j^\delta)\right)}{\delta^2 \prod\harmonic(\Omega^\delta;v^\delta,e_j^\delta)}=\frac{4\operatorname*{Area}(\bs{\beta}_1,\bs{\beta}_2,\bs{\beta}_3)\det(M(\Omega;x_1,x_2,x_3;z))}{\prod_{j=1}^3\Poisson(\Omega;v,x_j)}, 
\end{equation}
where $\operatorname*{Area}(\bs{\beta}_1,\bs{\beta}_2,\bs{\beta}_3)$ denotes the area of the triangle whose three vertices are $\bs{\beta}_1,\bs{\beta}_2,\bs{\beta}_3$. 
For the 4-8 lattice, we have 
\[\operatorname*{Area}(\bs{\beta}_1,\bs{\beta}_2,\bs{\beta}_3)=\frac{1}{2}(1+\sqrt{2}).\]
Combining with~\eqref{eqn::keylemma_48} and~\eqref{eqn::Poisson_derivative_cvg_48} and Lemma~\ref{lem::detM_LZtripod}, we obtain the desired convergence in~\eqref{eqn::deltafive_integrand_continuum_48} for any subsequence of $\{z^{\delta}\}$ of Case 1. As $\operatorname*{Area}(\bs{\beta}_1,\bs{\beta}_2,\bs{\beta}_3)=\frac{1}{2}(1+\sqrt{2})$ for all the four cases, we obtain the convergence of the whole sequence. 
\end{proof}


\begin{thebibliography}{CDCH{\etalchar{+}}14}

\bibitem[AB99]{ABRandomcurve}
Michael Aizenman and Almut Burchard.
\newblock H\"{o}lder regularity and dimension bounds for random curves.
\newblock {\em Duke Math. J.}, 99(3):419--453, 1999.

\bibitem[AHSY23]{AngHoldenSunYu2023}
Morris Ang, Nina Holden, Xin Sun, and Pu~Yu.
\newblock Conformal welding of quantum disks and multiple {SLE}: the non-simple
  case.
\newblock Preprint in arXiv:2310.20583, 2023.

\bibitem[BBK05]{BauerBernardKytolaMultipleSLE}
Michel Bauer, Denis Bernard, and Kalle Kyt{\"o}l{\"a}.
\newblock Multiple {S}chramm-{L}oewner evolutions and statistical mechanics
  martingales.
\newblock {\em J. Stat. Phys.}, 120(5-6):1125--1163, 2005.

\bibitem[BL23]{BLpiecewise}
Nathanaël Berestycki and Mingchang Liu.
\newblock Piecewise Temperleyan dimers and a multiple SLE$_8$.
\newblock arXiv:2301.08513.

\bibitem[BPW21]{BeffaraPeltolaWuUniqueness}
Vincent Beffara, Eveliina Peltola, and Hao Wu.
\newblock On the uniqueness of global multiple {SLE}s.
\newblock {\em Ann. Probab.}, 49(1):400--434, 2021.

\bibitem[CDCH{\etalchar{+}}14]{CDCHKSConvergenceIsingSLE}
Dmitry Chelkak, Hugo Duminil-Copin, Cl\'{e}ment Hongler, Antti Kemppainen, and
  Stanislav Smirnov.
\newblock Convergence of {I}sing interfaces to {S}chramm's {SLE} curves.
\newblock {\em C. R. Math. Acad. Sci. Paris}, 352(2):157--161, 2014.

\bibitem[CN07]{CamiaNewmanPercolation}
Federico Camia and Charles~M. Newman.
\newblock Critical percolation exploration path and {${\rm SLE}\sb 6$}: a proof
  of convergence.
\newblock {\em Probab. Theory Related Fields}, 139(3-4):473--519, 2007.

\bibitem[CS11]{ChelkakSmirnovDiscreteComplexAnalysis}
Dmitry Chelkak and Stanislav Smirnov.
\newblock Discrete complex analysis on isoradial graphs.
\newblock {\em Adv. Math.}, 228(3):1590--1630, 2011.

\bibitem[CS12]{ChelkakSmirnovIsing}
Dmitry Chelkak and Stanislav Smirnov.
\newblock Universality in the 2{D} {I}sing model and conformal invariance of
  fermionic observables.
\newblock {\em Invent. Math.}, 189(3):515--580, 2012.

\bibitem[CW21]{ChelkakWanMassiveLERW}
Dmitry Chelkak and Yijun Wan.
\newblock {On the convergence of massive loop-erased random walks to massive
  SLE(2) curves}.
\newblock {\em Electron. J. Probab.}, 26:Paper No. 54, 2021.

\bibitem[Dub06]{DubedatEulerIntegralsCommutingSLEs}
Julien Dub\'edat.
\newblock Euler integrals for commuting {SLE}s.
\newblock {\em J. Stat. Phys.}, 123(6):1183--1218, 2006.

\bibitem[FK15]{FloresKlebanPDE3}
Steven~M. Flores and Peter Kleban.
\newblock A solution space for a system of null-state partial differential
  equations: {P}art 3.
\newblock {\em Comm. Math. Phys.}, 333(2):597--667, 2015.

\bibitem[Fom01]{fomin2001loop}
Sergey Fomin.
\newblock Loop-erased walks and total positivity.
\newblock {\em Trans. Amer. Math. Soc.}, 353(9):3563--3583, 2001.

\bibitem[FPW24]{FengPeltolaWuConnectionProbaFKIsing}
Yu~Feng, Eveliina Peltola, and Hao Wu.
\newblock Connection probabilities of multiple {FK-Ising} interfaces.
\newblock {\em Probab. Theory Related Fields}, 189(1-2):281--367, March
  2024.

\bibitem[FWY24]{FengWuYangIsing}
Yu~Feng, Hao Wu, and Lu~Yang.
\newblock {Multiple Ising interfaces in annulus and 2N-sided radial SLE}.
\newblock {\em Int. Math. Res. Not. IMRN}, 2024(6):5326--5372, 2024.

\bibitem[HL21]{HealeyLawlerNSidedRadialSLE}
Vivian~Olsiewski Healey and Gregory~F. Lawler.
\newblock N-sided radial {S}chramm-{L}oewner evolution.
\newblock {\em Probab. Theory Related Fields}, 181(1-3):451--488, 2021.

\bibitem[HLW25]{HanLiuWuUST}
Yong Han, Mingchang Liu, and Hao Wu.
\newblock Hypergeometric {SLE} with {$\kappa=8$}: convergence of {UST} and
  {LERW} in topological rectangles.
\newblock {\em Ann. Inst. Henri Poincar\'e{} Probab. Stat.}, 61(2):1163--1211,
  2025.

\bibitem[HPW25]{HuangPeltolaWuMultiradialSLEResamplingBP}
Chongzhi Huang, Eveliina Peltola, and Hao Wu.
\newblock Multiradial {SLE} with spiral: resampling property and boundary
  perturbation, 2025.
\newblock arXiv:2509.22045.

\bibitem[Kar19]{KarrilaMultipleSLELocalGlobal}
Alex Karrila.
\newblock Multiple {SLE} type scaling limits: from local to global, 2019.
\newblock arXiv:1903.10354.

\bibitem[Kar20]{KarrilaUSTBranches}
Alex Karrila.
\newblock U{ST} branches, martingales, and multiple {$\rm SLE(2)$}.
\newblock {\em Electron. J. Probab.}, 25:83, 2020.

\bibitem[Ken00]{KenyonLongRangeSpanningTree}
Richard Kenyon.
\newblock Long-range properties of spanning trees.
\newblock {\em J. Math. Phys.}, 41(3):1338--1363, 2000.

\bibitem[Ken02]{KenyonLaplacianDiracCriticalPlanarGraphs}
R.~Kenyon.
\newblock The {L}aplacian and {D}irac operators on critical planar graphs.
\newblock {\em Invent. Math.}, 150(2):409--439, 2002.

\bibitem[KKP20]{KarrilaKytolaPeltolaCorrelationsLERWUST}
Alex Karrila, Kalle Kyt\"{o}l\"{a}, and Eveliina Peltola.
\newblock Boundary correlations in planar {LERW} and {UST}.
\newblock {\em Comm. Math. Phys.}, 376(3):2065--2145, 2020.

\bibitem[KL07]{KozdronLawlerMultipleSLEs}
Michael~J. Kozdron and Gregory~F. Lawler.
\newblock The configurational measure on mutually avoiding {SLE} paths.
\newblock In {\em Universality and renormalization}, volume~50 of {\em Fields
  Inst. Commun.}, pages 199--224. Amer. Math. Soc., Providence, RI, 2007.

\bibitem[KP16]{KytolaPeltolaPurePartitionFunctions}
Kalle Kyt\"ol\"a and Eveliina Peltola.
\newblock Pure partition functions of multiple {SLE}s.
\newblock {\em Comm. Math. Phys.}, 346(1):237--292, 2016.

\bibitem[KS17]{KemppainenSmirnovRandomCurves}
Antti Kemppainen and Stanislav Smirnov.
\newblock Random curves, scaling limits and loewner evolutions.
\newblock {\em Ann. Probab.}, 45(2):698--779, 03 2017.

\bibitem[KW11]{KenyonWilsonBoundaryPartitionsTreesDimers}
Richard~W. Kenyon and David~B. Wilson.
\newblock Boundary partitions in trees and dimers.
\newblock {\em Trans. Amer. Math. Soc.}, 363(3):1325--1364, 2011.


\bibitem[KWW24]{KrusellWangWuCommutationRelation}
Ellen Krusell, Yilin Wang, and Hao Wu.
\newblock Commutation relations for two-sided radial {SLE}, 2024.
\newblock arXiv:2405.07082.

\bibitem[Law09]{LawlerPartitionFunctionsSLE}
Gregory~F. Lawler.
\newblock Partition functions, loop measure, and versions of {SLE}.
\newblock {\em J. Stat. Phys.}, 134(5-6):813--837, 2009.

\bibitem[LPW25]{LiuPeltolaWuUST}
Mingchang Liu, Eveliina Peltola, and Hao Wu.
\newblock Uniform spanning tree in topological polygons, partition functions
  for {SLE}(8), and correlations in c=--2 logarithmic {CFT}.
\newblock {\em Ann. Probab.}, 53(1):23--78, 2025.

\bibitem[LSW04]{LawlerSchrammWernerLERWUST}
Gregory~F. Lawler, Oded Schramm, and Wendelin Werner.
\newblock Conformal invariance of planar loop-erased random walks and uniform
  spanning trees.
\newblock {\em Ann. Probab.}, 32(1B):939--995, 2004.

\bibitem[LW04]{LawlerWernerBrownianLoopsoup}
Gregory~F. Lawler and Wendelin Werner.
\newblock The {B}rownian loop soup.
\newblock {\em Probab. Theory Related Fields}, 128(4):565--588, 2004.

\bibitem[LW23]{LiuWuLERW}
Mingchang Liu and Hao Wu.
\newblock Loop-erased random walk branch of uniform spanning tree in
  topological polygons.
\newblock {\em Bernoulli}, 29(2):1555--1577, 2023.

\bibitem[PW19]{PeltolaWuGlobalMultipleSLEs}
Eveliina Peltola and Hao Wu.
\newblock Global and local multiple {SLE}s for $\kappa \leq 4$ and connection
  probabilities for level lines of {GFF}.
\newblock {\em Comm. Math. Phys.}, 366(2):469--536, 2019.

\bibitem[PW23]{PeltolaWuCrossingProbaIsing}
Eveliina Peltola and Hao Wu.
\newblock Crossing probabilities of multiple {I}sing interfaces.
\newblock {\em Ann. Appl. Probab.}, 33(4):3169--3206, 2023.

\bibitem[Sch00]{SchrammFirstSLE}
Oded Schramm.
\newblock Scaling limits of loop-erased random walks and uniform spanning
  trees.
\newblock {\em Israel J. Math.}, 118:221--288, 2000.

\bibitem[Smi01]{SmirnovPercolationConformalInvariance}
Stanislav Smirnov.
\newblock Critical percolation in the plane: conformal invariance, {C}ardy's
  formula, scaling limits.
\newblock {\em C. R. Acad. Sci. Paris S\'er. I Math.}, 333(3):239--244, 2001.

\bibitem[SS09]{SchrammSheffieldDiscreteGFF}
Oded Schramm and Scott Sheffield.
\newblock Contour lines of the two-dimensional discrete {G}aussian free field.
\newblock {\em Acta Math.}, 202(1):21--137, 2009.

\bibitem[SW05]{SchrammWilsonSLECoordinatechanges}
Oded Schramm and David~B. Wilson.
\newblock S{LE} coordinate changes.
\newblock {\em New York J. Math.}, 11:659--669 (electronic), 2005.

\bibitem[Wer04]{WernerRandomPlanarcurves}
Wendelin Werner.
\newblock Random planar curves and {S}chramm-{L}oewner evolutions.
\newblock In {\em Lectures on probability theory and statistics}, volume 1840
  of {\em Lecture Notes in Math.}, pages 107--195. Springer, Berlin, 2004.

\bibitem[Wil96]{WilsonUSTLERW}
David~Bruce Wilson.
\newblock Generating random spanning trees more quickly than the cover time.
\newblock In {\em Proceedings of the {T}wenty-eighth {A}nnual {ACM} {S}ymposium
  on the {T}heory of {C}omputing ({P}hiladelphia, {PA}, 1996)}, pages 296--303.
  ACM, New York, 1996.

\bibitem[Wu20]{WuHyperSLE}
Hao Wu.
\newblock Hypergeometric {SLE}: conformal {M}arkov characterization and
  applications.
\newblock {\em Comm. Math. Phys.}, 374(2):433--484, 2020.

\bibitem[YY11]{YadinYehudayoffLERWPoissonKernel}
Ariel Yadin and Amir Yehudayoff.
\newblock Loop-erased random walk and {P}oisson kernel on planar graphs.
\newblock {\em Ann. Probab.}, 39(4):1243--1285, 2011.

\bibitem[Zha08]{ZhanLERW}
Dapeng Zhan.
\newblock The scaling limits of planar {LERW} in finitely connected domains.
\newblock {\em Ann. Probab.}, 36(2):467--529, 2008.

\bibitem[Zha24]{ZhanExistenceUniquenessMultipleSLE}
Dapeng Zhan.
\newblock Existence and uniqueness of nonsimple multiple {SLE}.
\newblock {\em J. Stat. Phys.}, 191(8):Paper No. 101, 15, 2024.

\end{thebibliography}

{\small
\newcommand{\etalchar}[1]{$^{#1}$}

}

\end{document}